\documentclass{amsproc}
\usepackage{euscript}
\usepackage{cases}
\usepackage{mathrsfs}
\usepackage{bbm}
\usepackage{amssymb}
\usepackage{txfonts}
\usepackage{amsfonts,amsmath,amsxtra,mathdots,mathabx}
\usepackage{color}
\usepackage{hyperref} 
\usepackage{tikz}
\usepackage{appendix}

\allowdisplaybreaks

\DeclareFontFamily{U}{matha}{\hyphenchar\font45}
\DeclareFontShape{U}{matha}{m}{n}{
	<5> <6> <7> <8> <9> <10> gen * matha
	<10.95> matha10 <12> <14.4> <17.28> <20.74> <24.88> matha12
}{}
\DeclareSymbolFont{matha}{U}{matha}{m}{n}

\DeclareMathSymbol{\Lt}{3}{matha}{"CE}
\DeclareMathSymbol{\Gt}{3}{matha}{"CF}

\DeclareSymbolFont{mathc}{OML}{txmi}{m}{it}
\DeclareMathSymbol{\varvv}{\mathord}{mathc}{118}
\DeclareMathSymbol{\varww}{\mathord}{mathc}{119}
\DeclareMathSymbol{\varnu}{\mathord}{mathc}{"17}

\DeclareSymbolFont{mathd}{OML}{ztmcm}{m}{it}
\DeclareMathSymbol{\varalpha}{\mathord}{mathd}{11}
\DeclareMathSymbol{\varlambda}{\mathord}{mathd}{21}
\def\valpha{\text{\scalebox{0.86}{$\varalpha$}}} 
\def\vlambda{\text{\scalebox{0.86}{$\varlambda$}}}
\def\ssnu{\text{\scalebox{0.9}{$\nu$}}} 
\def\ssmu{\text{\scalebox{0.9}{$\mu$}}}
\def\ssprime{\text{\scalebox{0.9}{$\prime$}}}

\DeclareMathSymbol{\depsilon}{\mathord}{mathd}{15}
\def\vepsilon{\text{\scalebox{0.88}{$\depsilon$}}}

\DeclareMathSymbol{\varchi}{\mathord}{mathd}{31}
\def\vchi{\text{\scalebox{0.9}{$\varchi$}}}

\newcommand{\overbar}[1]{\mkern 1mu\overline{\mkern-1mu#1\mkern-1mu}\mkern 1mu}

\newcommand{\BA}{{\mathbb {A}}} 
\newcommand{\BC}{{\mathbb {C}}}

\newcommand{\BQ}{{\mathbb {Q}}} \newcommand{\BR}{{\mathbb {R}}}

 \newcommand{\BZ}{{\mathbb {Z}}}

\newcommand{\RC}{{\mathrm {C}}} \newcommand{\RD}{{\mathrm {D}}}
\newcommand{\RE}{{\mathrm {E}}}

 \newcommand{\RN}{{\mathrm {N}}}

\newcommand{\GL}{{\mathrm {GL}}}
\newcommand{\PGL}{{\mathrm {PGL}}}
\newcommand{\SL}{{\mathrm {SL}}}

\newcommand{\sstyle}{\scriptstyle}
\newcommand{\ssstyle}{\scriptscriptstyle}

\newcommand{\ra}{\rightarrow} 
\def\viint{	\int \hskip -4 pt \int}
\def\nnmid{\hskip -1 pt \nmid \hskip -1 pt}

\def\fra{\mathfrak{a}}
\def\frm{\mathfrak{m}}
\def\frn{\mathfrak{n}}
\def\frf{\mathfrak{f}}
\def\frr{\mathfrak{r}}
\def\bfra{\text{\scalebox{1.08}{$\text{\usefont{U}{BOONDOX-frak}{m}{n}a}$}} }
\def\bfrf{\text{$\text{\usefont{U}{BOONDOX-frak}{m}{n}f}$} }
\def\-{^{-1}}

\def\sasymp{\text{ \small $\asymp$ }}
\def\mod{\mathrm{mod}\,  }

\def\sumx{\sideset{}{^\star}\sum}

\def\vv{\varv}

\def\tw{\varww}
\def\nd{\mathrm{d}}
\def\Tr{\mathrm{Tr}}
\def\oo{\mathrm{o}}
\def\Fx{F^{\times}}
\def\trh{ \mathrm{trh}}

\def\lp {\left (}
\def\rp {\right )}

\def\sstimes {\scalebox{0.55}{$\times$}}

\renewcommand{\Im}{{\mathrm{Im} }}
\renewcommand{\Re}{{\mathrm{Re} }}

\def\shskip{\hskip 0.5 pt}
\def\snatural{\text{\scalebox{0.85}{$\natural$}}} 
\def\ssharp{\text{\scalebox{0.85}{$\sharp$}}}

\def\frO {\text{\raisebox{- 2 \depth}{\scalebox{1.1}{$ \text{\usefont{U}{BOONDOX-calo}{m}{n}O} \hskip 0.5pt $}}}}
\def\frOO {\text{\raisebox{- 2 \depth}{\scalebox{1.1}{$ \text{\usefont{U}{BOONDOX-calo}{m}{n}O}$}}}}
\def\frp{\mathfrak{p}}
\def\frq{\mathfrak{q}}
\def\frb{\mathfrak{b}}
\def\frc{\mathfrak{c}}
\def\frd{\mathfrak{d}}

\def\frD{\mathfrak{D}}

\def\SB{\text{\raisebox{- 2 \depth}{\scalebox{1.1}{$ \text{\usefont{U}{BOONDOX-calo}{m}{n}B} \hskip 0.5pt $}}}}

\def\SS{\text{\raisebox{- 2 \depth}{\scalebox{1.1}{$ \text{\usefont{U}{BOONDOX-calo}{m}{n}S}\hskip 0.5pt $}}}}

\def\ST{\text{\raisebox{- 2 \depth}{\scalebox{1.1}{$ \text{\usefont{U}{BOONDOX-calo}{m}{n}T}\hskip 1pt $}}}}

\def\SDB{\text{\raisebox{- 1 \depth}{\scalebox{1.03}{$ \text{\usefont{U}{dutchcal}{m}{n}B}  $}}}}

\def\SDH{\text{\raisebox{- 1 \depth}{\scalebox{1.03}{$ \text{\usefont{U}{dutchcal}{m}{n}H}  $}}}}

 \def\SDJ{\text{{\scalebox{1.03}{$ \text{\usefont{U}{dutchcal}{m}{n}J}  $}}}}

\makeatletter
\g@addto@macro\normalsize{\setlength\abovedisplayskip{3pt}}
\makeatother

\makeatletter
\g@addto@macro\normalsize{\setlength\belowdisplayskip{3pt}}
\makeatother

\newcommand{\delete}[1]{}

\theoremstyle{plain}

\newtheorem{thm}{Theorem}[section] \newtheorem{cor}[thm]{Corollary}
\newtheorem{lem}[thm]{Lemma}  \newtheorem{prop}[thm]{Proposition}

 \newtheorem{defn}[thm]{Definition}

\newtheorem {rem}[thm]{Remark} 
\newtheorem {notation}[thm]{Notation}

\newtheorem*{acknowledgement}{Acknowledgements}

\numberwithin{equation}{section}

\begin{document}
	
	\title{Subconvexity for $L$-Functions on $\mathrm{GL}_3$ over Number Fields}
 
	\author{Zhi Qi}
	\address{School of Mathematical Sciences, Zhejiang University, Hangzhou, 310027, China}
	\email{zhi.qi@zju.edu.cn}
	

	\begin{abstract}
		In this paper, over an arbitrary number field, we prove subconvexity bounds for self-dual $\mathrm{GL}_3$ $L$-functions in the $t$-aspect and for self-dual $\mathrm{GL}_3 \times \mathrm{GL}_2$ $L$-functions in the $\mathrm{GL}_2$ Archimedean aspect.
	\end{abstract} 

\thanks{The  author is supported by the National Natural Science Foundation of China (Grant No. 12071420).}
	
	\maketitle

{\small \tableofcontents}

\renewcommand{\baselinestretch}{1.08}
	
\section{Introduction}

There is a great interest in  upper bounds for the central values of $L$-functions.   The subconvexity problem is concerned with improving over their convexity bound resulting from   the Phragm\'en-Lindel\"of convexity principle.

The subconvexity problem for $\GL_1$ and $\GL_2$ over arbitrary number fields was completely solved in the seminal work of Michel and Venkatesh \cite{Michel-Venkatesh-GL2}.  More recent work on the subconvexity for $\GL_2$ over number fields may be found in \cite{Blomer-Harcos-TR,Maga-Sub,Maga-Shifted,WuHan-GL2,WuHan-2,Nelson-Eisenstein}. 

Xiaoqing Li \cite{XLi2011} made the first progress on the subconvexity problem for $\GL_3$  in the  $t$-aspect and $\GL_3 \times \GL_2$ in the $\GL_2$ spectral aspect. For a {\it self-dual} Hecke--Maass form $\pi$ of $\SL_3 (\BZ)$ and the family $\SB     $ given by an orthonormal basis of  even  Hecke--Maass cusp forms for $\SL_2 (\BZ)$, she established the averaged Lindel\"of hypothesis for the first moment:
\begin{align}\label{0eq: XLi}
{\sum_{f \in  \SB   }  } e^{- (t_f-T)^2/M^2} L \big(\tfrac 1 2 , \pi \otimes f \big) + \frac 1 {4 \pi} \int_{\BR  } e^{- (t -T)^2/M^2} \left| L \big(\tfrac 1 2 + i t, \pi \big) \right|^2 \nd \shskip t \Lt_{    \shskip \pi, \shskip \vepsilon} M T^{1 + \vepsilon},
\end{align}
for $T^{3/8+\vepsilon} \leqslant M \leqslant T^{1/2}$, where $\frac 1 4 + t_f^2$ is the Laplace eigenvalue for $f$. As a consequence of \eqref{0eq: XLi} and the non-negativity of $L \big(\tfrac 1 2 , \pi \otimes f \big)$ due to the self-dual assumption, she obtained the subconvexity bounds
\begin{align} \label{0eq: Li}
L \big( \tfrac 1 2 + i t, \pi \big) \Lt_{    \shskip \pi, \shskip \vepsilon}   (1 + |t|   )^{11 /16 + \vepsilon}, \quad 
L \big(\tfrac 1 2 , \pi \otimes f \big) \Lt_{    \shskip \pi, \shskip \vepsilon}  (1 + |t_f|)^{11 /8 + \vepsilon}.
\end{align}
In a similar framework, Blomer \cite{Blomer} proved the subconvexity for twisted  $\GL_3$ and $\GL_3 \times \GL_2$ $L$-functions in the  $q$-aspect:
\begin{align}\label{0eq: Blomer} 
L \big( \tfrac 1 2 + it, \pi \otimes  \vchi \big) \Lt_{  \shskip t, \shskip \pi, \shskip \vepsilon} q^{5 /8 + \vepsilon} ,  \quad L \big( \tfrac 1 2 , \pi \otimes f \otimes \vchi \big) \Lt_{  \shskip f, \shskip \pi, \shskip \vepsilon} q^{5 /4 + \vepsilon}, 
\end{align} 
where $\vchi$ is a quadratic Dirichlet character  of prime modulus $q$.  
The work of Blomer clearly stems from the  remarkable paper of Conrey and Iwaniec \cite{CI-Cubic} on the cubic moment of twisted  $  \GL_2$ $L$-functions. A recent advance on the path of Conrey and Iwaniec is the work of Young \cite{Young-Cubic}, in which he introduced new analytic techniques, which are quite different from those of Xiaoqing Li, to prove  the hybrid Weyl-type subconvexity bound 
\begin{align}\label{0eq: Young}
L \big( \tfrac 1 2 + i t,  \vchi \big) \Lt_{\vepsilon} ( q \big(1 + |t|   ) )^{1 /6 + \vepsilon} , \quad L \big( \tfrac 1 2  , f \otimes \vchi \big) \Lt_{\vepsilon} ( q \big(1 + |t_f|   ) )^{1 /3 + \vepsilon} .
\end{align}
Later, in the spirit of Blomer and Young,  Nunes \cite{Nunes-GL3} improved Xiaoqing Li's exponent $\frac {11} {16}$ in \eqref{0eq: Li} into Blomer's $\frac {5} {8}$ in \eqref{0eq: Blomer}.   

In this paper, we shall prove subconvexity results in the $t$-aspect for $\GL_3$  and the $\GL_2$ Archimedean aspect for 
$\mathrm{GL}_3 \times \mathrm{GL}_2$  over arbitrary number fields. The Vorono\"i summation formula for $\GL_3$ of Ichino and Templier \cite{Ichino-Templier} is used in its full generality, with the aid of the asymptotic formulae of Bessel functions for $\GL_3$ in \cite{Qi-Bessel}.

As explained in \S \ref{sec: XQLi}, there are technical issues to generalize the analysis of Xiaoqing Li to number fields other than $\BQ$, so, instead, our approach is  inspired by the works of Blomer, Young, and Nunes. As for the strategy, briefly speaking, Xiaoqing Li uses the Vorono\"i summation twice, while we use  the Vorono\"i summation once, followed by the large sieve. 

For other related works, see for example \cite{Ivic-t-aspect,Petrow-Cubic,PY-Cubic,PY-Weyl2,PY-Weyl3,Ye-GL3,Huang-GL3,Qi-Gauss}. 

For subconvexity results for $\GL_3$ over $\BQ$ without the self-dual assumption, we refer the reader to the papers of Munshi, Holowinsky, Nelson, Yongxiao Lin, et. al.,  \cite{Munshi-Circle-III,Munshi-Circle-IV,Munshi-Circle-IV2,Munshi-GL3xGL2,HoNe-ZeroFr,Lin-GL3,SZ-Depth,Aggarwal-GL3,Munshi-GL3GL2-q,LinSun-GL3,Lin-Integral}.


\subsection{Statement of Results}  

Let $F    $ be a fixed number field of degree $N$. 
Let $S_{\infty}$ denote the set of Archimedean places of $F    $. As usual, write  $\vv | {\infty}$ in place of  $\vv \in S_{\infty}$. For $\vv | {\infty}$, let $N_{\vv} $ be the degree of $F_{\vv} / \BR$.

Let  $\pi$ be a {fixed} {\it self-dual} spherical automorphic cuspidal representation of $ \PGL_3$ over $F$. 
Let $ \SB $ be an orthonormal basis consisting of Hecke--Maass cusp forms for the spherical cuspidal spectrum for $ \PGL_2$ over $F$. For $f \in \SB$, let $\varnu_f \in \BC^{|S_{\infty}|}$ be its Archimedean parameter such that either $ \varnu_{f, \, \vv} $ is real or $i \varnu_{f, \, \vv} \in \lp - \frac 1 2, \frac 1 2 \rp$ for every $\vv | {\infty}$; we may readily assume that $ \varnu_{f, \, \vv} \in [0, \infty) \cup i \big[ 0, \frac 1 2 \big)$. 

We are concerned with the $\GL_3 \times \GL_2$ Rankin--Selberg $L$-functions $ L (s, \pi \otimes f) $ for $f$ in the family $\SB$ and the $\GL_3$ $L$-function $L (s, \pi)$. The following is our main theorem. 

\begin{thm}\label{thm: main}
Let notation be as above. Let $\vepsilon   > 0$. Let $T, M \in \BR_+^{|S_{\infty}|} $ be such that $1 \Lt T_{\vv}^{\vepsilon} \leqslant M_{\vv} \leqslant T_{\vv}^{1-\vepsilon} $ for every $\vv |\infty $.  Set $\RN (T) = \prod_{ \vv | {\infty} } T_{\vv}^{N_{\vv}}$ and $\RN^{\snatural} (M) = \prod_{ \vv | {\infty} } M_{\vv} $. Assume that $T_{\vv} \geqslant \RN (T)^{\vepsilon}$ for every $\vv | {\infty}$. 
Define $\SB_{T, \shskip M} $ to be the collection of $f \in \SB$ satisfying $ |\varnu_{f, \, \vv} - T_{\vv}| \leqslant M_{\vv} $ for all $\vv | {\infty}$, and define $ \BR_{\shskip T, \shskip M} $ to be the intersection of the intervals $ [T_{\vv}- M_{\vv}, T_{\vv} + M_{\vv} ] $ for all  $\vv | {\infty}$.  Then
\begin{align}\label{0eq: main}
 {\sum_{f \in \SB_{T, \shskip M}  }  } L \big(\tfrac 1 2 , \pi \otimes f \big) + \int_{\BR_{\shskip T, \shskip M} } \left| L \big(\tfrac 1 2 + i t, \pi \big) \right|^2 \nd \shskip t \Lt_{\shskip \vepsilon, \shskip  \pi, \shskip F} \RN^{\snatural} (M) \RN (T)^{5/4 + \vepsilon},
\end{align}
with the implied constant depending only on $\vepsilon$, $\pi$, and $F$. 
\end{thm}

\begin{rem}
	The assumption $T_{\vv} \geqslant \RN (T)^{\vepsilon}$ {\rm(}$\vv | {\infty}${\rm)} is used to make the contribution from exceptional forms negligible and to address some issues with the infinitude of units. 
\end{rem}

For $f \in \SB$, define its Archimedean conductor $\RC_{\infty} (f) = \RC (\varnu_f)^2$ by
\begin{align*}
\RC_{\infty} (f) = \RN (1 + |\varnu_f|)^2 =  \prod_{ \vv | {\infty} }  (1 + |\varnu_{f, \, \vv}|  )^{2 N_{ \vv}} .
\end{align*}

Since $\pi   $ is self-dual, by the non-negativity theorem of Lapid \cite{Lapid}, we have
\begin{align}\label{0eq: positivity}
L \big(\tfrac 1 2 , \pi \otimes f \big) \geqslant 0. 
\end{align}
As a consequence of \eqref{0eq: positivity}, we derive from \eqref{0eq: main} the following subconvexity bounds by taking $M = T^{\vepsilon}$.  

\begin{cor}
	 Let notation be as above. We have 
	 \begin{align}
	 L \big( \tfrac 1 2 + i t, \pi \big) \Lt_{\shskip \vepsilon, \shskip  \pi, \shskip F} \big(1 + |t|^{N} \big)^{5 /8 + \vepsilon}, 
	 \end{align}
	 and
	 \begin{align}\label{0eq: GL3 x GL2}
	 L \big(\tfrac 1 2 , \pi \otimes f \big) \Lt_{\shskip \vepsilon, \shskip  \pi, \shskip F} \RC_{\infty}(f)^{5 /8 + \vepsilon}, 
	 \end{align}
	 if $ |\varnu_{f, \, \vv}| \geqslant \RC_{\infty}(f)^{\vepsilon} $ for all $\vv| \infty$.
\end{cor}

\subsection{Subconvexity for $\GL_2$} With the Archimedean analysis of this paper,  following Young, one should be able to establish the Weyl-type bound (or even the hybrid Weyl-type  bound):
\begin{align}\label{0eq: GL2}
L \big(\tfrac 1 2 ,  f \big) \Lt_{\shskip \vepsilon,  \shskip F} \RC_{\infty}(f)^{1 /6 + \vepsilon}.
\end{align}
For   $F = \BQ$, this is a result of Ivi\'c \cite{Ivic-t-aspect}. For arbitrary $F$, Han Wu \cite{WuHan-2} has a uniform subconvexity bound with weaker exponent. 

Let  $\phi$ be a fixed spherical Hecke--Maass cusp form for $\PGL_2$ over $F$. By combining \eqref{0eq: GL3 x GL2} and \eqref{0eq: GL2},  with $\pi = \mathrm{Sym}^2 \phi$, we obtain 
\begin{align}\label{0eq: triple}
L \big(\tfrac 1 2 , \phi \otimes \phi \otimes f \big) \Lt_{\shskip \vepsilon, \shskip  \phi, \shskip F} \RC_{\infty}(f)^{19 /24 + \vepsilon}.  
\end{align} 
In comparison, when $F = \BQ$, Bernstein and Reznikov \cite{Bernstein-Reznikov} proved 
\begin{align}\label{0eq: BZ}
L \big(\tfrac 1 2 , \phi \otimes \phi' \otimes f \big) \Lt_{\shskip \vepsilon, \shskip  \phi , \shskip \phi' } |t_f|^{5 /3 +\vepsilon}. 
\end{align}

\subsection{Remarks on Hybrid Subconvexity}   
Let $\vchi$ be a quadratic Hecke character for $   F $ of prime conductor $\frq$.  In light of the works of Young \cite{Young-Cubic} and Blomer \cite{Blomer} in the case $F = \BQ$,   the following hybrid subconvexity bounds
\begin{align}\label{0eq: hybrid}
\begin{aligned}
& L \big( \tfrac 1 2 + i t, \pi \otimes \vchi \big) \Lt_{\shskip \vepsilon, \shskip  \pi, \shskip F} \big( \RN (\frq) \big(1 + |t|^{N} \big)\big)^{5 /8 + \vepsilon} , \\
& L \big( \tfrac 1 2, \pi \otimes f \otimes \vchi \big) \Lt_{\shskip \vepsilon, \shskip  \pi, \shskip F}  \big( \RN (\frq)^2  (1 + C_{\infty} (f)  ) \big)^{5 /8 + \vepsilon} ,
\end{aligned}
\end{align}
should hold, at least when $ |t|, C_{\infty} (f) \geqslant \RN (\frq)^{\vepsilon}  $. This is believed by Nunes \cite{Nunes-GL3} for $F = \BQ$, and has been confirmed privately by the author for $F = \BQ$ or $\BQ (i)$.  
It seems that \eqref{0eq: hybrid} can be verified whenever the class number $h_F = 1$, though the group of units might cause some troubles. In general, one has to compute certain non-Archimedean local integrals which should be transformed eventually to the character sums studied by Conrey and Iwaniec. The transformation could be quite intricate in view of Blomer's computations in \cite[\S 6]{Blomer}. Recently, Nelson \cite{Nelson-Eisenstein}  generalized Conrey--Iwaniec to general number fields, and he might provide us some hint for this. 

\subsection{Features of Analysis over Complex Numbers} 

The main difficulty in the analysis over $\BC$ is the lack of suitable stationary phase results in two dimensions. Considerably more efforts are needed particularly for the Hankel transform and the Mellin transform. We would rather not to discuss the technical details---It is more worthwhile and interesting to present here the features of certain trigonometric-hyperbolic functions  arising in the phases. 

To start with, the function $\trh (r, \omega) = \rho (r, \omega) e^{i \theta (r, \shskip \omega)}$ that occurs in the $\GL_2$ Bessel integral on $\BC$ is given by
\begin{align*}
\rho (r, \omega) =      \sqrt { \frac {\cosh 2 r + \cos 2 \omega} 2 }    , \qquad  \tan \theta (r, \omega) = \tanh r \tan \omega .
\end{align*}  
Since $\trh (r, 0) = \cosh r $ and $    \trh ( r, \pi/2) = i  \sinh r$, one expects that the $r$-integral behaves as if it is the Bessel integral on $\BR_+$ or $\BR_-$ when $\sin \omega$ or $\cos \omega$ is small, respectively. 

After the $\GL_3$ Hankel transform, we obtain a new function $ \trh^{\snatural} (r, \omega) = \rho^{\snatural} (r, \omega) e^{i \shskip \theta^{\snatural}  (r, \shskip \omega) } $ defined by
\begin{align*}
\rho^{\snatural } (r, \omega)   = \frac {\cosh 2 r -   \cos 2 \omega} {\cosh 2 r +   \cos 2 \omega}, \qquad
\tan  ( \theta^{\shskip \snatural}  (r, \omega)/ 2  ) =  \frac {\sin 2 \omega} {\sinh 2 r}   .
\end{align*}
It is certainly a pleasure to see the square-root sign gone. More important is the symmetry in $\trh^{\snatural} (r, \omega)$ reflected by the identities:
\begin{equation*} 
\frac {\partial^2 \log \rho^{\shskip \snatural}   } {\partial r^2}     = - \frac {\partial^2 \log \rho^{\shskip \snatural}   } {\partial \omega^2}  = \frac {\partial^2 \theta^{\shskip \snatural} } {\partial r \partial \omega}  , \qquad  
\frac {\partial^2 \theta^{\shskip \snatural}   } {\partial r^2}   = - \frac {\partial^2 \theta^{\shskip \snatural}  } {\partial \omega^2} = - \frac {\partial^2 \log \rho^{\shskip \snatural}   } {\partial r \partial \omega}  ,
\end{equation*}
which play a critical role in our analysis after the Mellin technique. 

Finally, we remark that, at the stage of the Mellin technique,   new phenomena emerge  in our analysis for $\cos 2 \omega $ in the vicinity of $0$ (for $|\sin \omega|$ and $|\cos \omega|$ nearly equal).

\subsection*{Notation} By $X \Lt Y$ or $X = O (Y)$ we mean that $|X| \leqslant c Y$ for some constant $c > 0$, and by $X \asymp Y$  we mean that  $X \Lt Y$ and $Y \Lt X$. We write $X \Lt_{P, \shskip Q, \, \dots} Y$ or $X = O_{P, \shskip Q, \, \dots} (Y)$ if the implied constant $c$ depends on  $P$, $Q, \dots$. Throughout this article $\RN (T) \Gt 1$ and each $T_{\vv} \geqslant \RN (T)^{\vepsilon} \Gt 1$ will be large, and we say  that $X$ is negligibly small if $X = O_A (\RN(T)^{-A})$ (or $O_A (T_{\vv}^{-A})$) for arbitrarily large  but fixed $A > 0$.

We adopt the usual $\vepsilon$-convention of analytic number theory; the value of $\vepsilon $ may differ from one occurrence to another.

\begin{acknowledgement}
	 I would like to  thank the referee for careful readings and  extensive	comments.
\end{acknowledgement}
	
 { \large \part{Number Theoretic Preliminaries}}
	
\section{Notation over Number Fields}\label{sec: notation}


\subsection*{Basic Notions}

Let $F    $ be a number field of degree $N$. Let $\frO $ be its ring of integers and $\frOO^{\times}$  be the group of units.  Let $\mathfrak{D}$ be the different ideal of $F    $. 
Let $\mathrm{N}$ and $\Tr$ denote the norm and the trace for $F    $, respectively. Let $d_F$ be the discriminant of $F$.  
Denote by $ \BA  $ the  ad\`ele ring of $F    $ and by $\BA^{\times}$ the id\`ele group of $F$.

For any place $\vv$ of $F    $, we denote by $F    _{\vv}$ the corresponding local field. 
When $\vv$ is non-Archimedean,  let $\frp_{\vv}$ be the corresponding prime ideal of   $\frO $ and let $\mathrm{ord}_{\vv}$   or $\vv$ itself 
denote the additive valuation; occasionally, $\frp_{  \vv}$   also stands for the prime ideal in $\frO_{\varv}$. Denote by $\mathfrak{D}_{\varv}$ the local different ideal.
Let  $N_{\vv}$ be the local degree of $F    _{\vv}$; in particular, $N_{\vv} = 1$ if $F    _{\vv} = \BR$ and  $N_{\vv} = 2$ if $F    _{\vv} = \BC$.  Let $| \hskip 3.5 pt |_{\vv}$  denote the normalized module of $F_{\vv}$. We have $| \hskip 3.5 pt |_{\vv} = | \hskip 3.5 pt |$ if $F    _{\vv} = \BR$ and  $| \hskip 3.5 pt |_{\vv} = | \hskip 3.5 pt |^2$ if $F    _{\vv} = \BC$, where $| \hskip 3.5 pt |$ is the usual absolute value. Let $r_1$ and $r_2$ be the number of real and  complex places of $F$, respectively. 

Let $S_{\infty}$ or $S_f$ denote the set of Archimedean or non-Archimedean  places of $F    $, respectively. Write $\varv | \infty$ and $\varv    \nnmid   \infty$ as the abbreviation for $\varv \in S_{\infty}$ and $\varv \in S_{f}$, respectively.
For a finite set of places $S$, 
denote by $\BA^S$, respectively $F    _{S}$, the sub-ring of ad\`eles with trivial component above $S$, respectively above the complement of $S$. The modules on $\BA^{\times S}$ and $ F_S $ will be denoted by $| \hskip 3.5 pt|^S$ and $| \hskip 3.5 pt |_S$ respectively. For brevity, write $\BA_f = \BA^{S_{\infty}}$ and $F    _{\infty} = F    _{S_{\infty}}$. 

\subsection*{Additive Characters and Haar Measures}
Fix the (non-trivial) standard additive character $\psi = \otimes_{\vv} \psi_{\vv}$ on $\BA/F    $ as in \cite[\S XIV.1]{Lang-ANT} such that $\psi_{\vv} (x) = e (- x)$ if $F    _\vv = \BR$,   $\psi_{\vv} (z) = e (- (z + \widebar z))$ if $F    _\vv = \BC$, 
and that $\psi_{\vv}$ has conductor $\mathfrak{D}_{\vv}\-$ for any non-Archimedean $F    _\vv$. For a finite set of places $S$, denote $\psi^S  = \prod_{\varv \shskip  \notin S} \psi_{  \varv}$ and $\psi_S  = \prod_{\varv \shskip  \in S} \psi_{  \varv}  $.  We split $\psi = \psi_{\infty} \psi_f  $ so that $\psi_{\infty} (x) = e (- \Tr_{F_{\infty}} (x))$ ($x \in F_{\infty}$). 


We choose the Haar measure  $\nd x$ of $F    _{\vv}$ self-dual with respect to $\psi_{\vv}$ as in  \cite[\S XIV.1]{Lang-ANT}; the Haar measure is the ordinary Lebesgue measure on the real line if $F    _{\vv} = \BR$, twice  the ordinary Lebesgue measure on the complex plane if $F    _{\vv} = \BC$, and that measure for which  $\frO_{\vv}$ has measure $\RN (\frD_{\vv})^{-1/2}$ 
 if $F    _{\vv}$ is non-Archimedean.  We slightly modify the Haar measure  $\nd^{\times} \hskip -1pt x$ of $F_{\vv}^{\times}$  defined in \cite[\S XIV.2]{Lang-ANT}: $\nd^{\times} \hskip -1pt x = \nd x / |x|_{\vv}$ if $\vv |\infty$, and $\nd^{\times} \hskip -1pt x = \RN (\frD_{\vv})^{1/2} \RN(\frp_{\vv}) / (\RN (\frp_{\vv}) - 1) \cdot \nd x / |x|_{\vv}$ if $\vv \nnmid \infty$, so that   $\frO_{\vv}^{\times}$ has mass $1$ (it is $\RN (\frD_{\vv})^{-1/2}$ in \cite{Lang-ANT}).

We remark that the module, measure and additive character on $F_{\infty} = \prod_{\vv|\infty} F_{\vv}$ are chosen differently in {\rm \cite{BM-Kuz-Spherical}}. For example, they use $| \hskip 3.5 pt |$ for both real and complex $\vv$, their additive measure differs from ours by a factor $1/\pi$ or $1/2\pi$ if $\vv$ is real or complex,  respectively, and their   additive character is $e (\Tr_{F_{\infty}} (x))$ instead of  $e (- \Tr_{F_{\infty}} (x))$.



\subsection*{Ideals}

In general, we use Gothic letters $\fra ,  \frb , \frc , \dots$ to denote {\it nonzero} fractional ideals of $F$, while we reserve $\frm$, $\frn$, $\frd$, $\frf$, $\frq$, and $\frr$ for nonzero integral ideals of $F$. Let $\frp$ always stand for a prime ideal. Let $\RN (\fra)$ denote the norm of $\fra$.
If $\valpha \in \Fx$, we denote by $(\valpha)$ the corresponding
principal ideal. If $\fra$ is a fractional ideal, we shall often just write $\valpha \fra$ for the product $(\valpha) \fra$. 

Given $ \theta \in \BR_+^{|S_{\infty}|}$ with $ \sum_{\vv|\infty} N_{\vv} \theta_{\vv} = 1$, for each principal fractional ideal $\fra$, we choose once and for all a generator, which
we denote $[\fra]$, so that 
\begin{align}\label{2eq: choice of [a]}
\big|\shskip [\fra] \shskip \big|_{\vv} \asymp \RN(\fra)^{N_{\vv} \theta_{\vv}  }, \qquad  \vv |\infty;
\end{align} 
such a choice is guaranteed by Dirichlet's units theorem (see \cite[\S V.1]{Lang-ANT}). Later, we shall set $\theta_{\vv} =   \log T_{\vv} / \log \RN (T)$. 

For each nonzero fractional ideal $\fra$, we fix a corresponding $\pi_{\fra} \in \BA_f^{\times}$ so that    $\mathrm{ord}_{\vv}(\fra)  = \mathrm{ord}_{\vv} (\pi_{\fra, \, \vv})$ for all $\vv \nnmid \infty$. Set $\delta = \pi_{\frD}$. 
For brevity, write $\fra_{\vv} = \fra \frO_{\vv}$. 
	
Let $ {C}_F$ be the class group and $h_F$ be the class number of $F$. We shall use the notation $\fra \sim \frb$ to mean that $\fra$ and $\frb$ are in the same ideal class. We choose a set $  \widetilde{C}_F$ of integral ideals that represent the class group.  

\subsection*{Characters and Mellin Transforms}

For  $\vv$ Archimedean, define the (unitary) character $\vchi_{i \varnu, \shskip m} (x) = |x|_{\vv}^{i \varnu} (x/|x|)^{m}$ ($x \in F_{\vv}^{\times}$) for $\varnu \in \BR$ and $m \in  \{0, 1\}$($=\BZ / 2\BZ$) if $F$ is real, and $m \in \BZ $ if $F$ is complex. Let $ \widehat{\bfra}_{\vv} $ denote the unitary dual of $F_{\vv}^{\times}$. We shall identify $ \widehat{\bfra}_{\vv}  $ with $\BR \times \{0, 1 \}$ or $\BR \times \BZ$ according as $\vv$ is real or complex, respectively. Let $\nd \mu (\varnu, m) $  denote the usual Lebesgue measure on $\BR \times \{0, 1\}$ or $\BR \times \BZ$, respectively. For notational simplicity, we shall write the summation on $ \{0, 1\} $ or $\BZ$ as integration. 

For $f (x) \in L^1 (F_{\vv}^{\times}) \cap L^2 (F_{\vv}^{\times})$, define its (local) Mellin transform $\breve{f} (\varnu, m)$ by 
\begin{align}\label{2eq: Mellin}
\breve{f} (\varnu, m) = \int_{ F_{\vv}^{\times} } f(x) \vchi_{i \varnu, \shskip m} (x) \nd^{\times} x .
\end{align} 
The Mellin inversion formula reads
\begin{align}\label{2eq: Mellin inverse}
f (x) = \frac 1 {2\pi c_{\vv}} \viint_{\widehat{\bfra}_{\vv} }  \breve{f} (\varnu, m) \overline{\vchi_{  i \varnu, \shskip  m} (x)} \nd \mu (\varnu, m),
\end{align}
where $c_{\vv} = 2$ if $\vv$ is real and $c_{\vv} = 2 \pi$ if $\vv$ is complex. Moreover, by Plancherel's theorem,
\begin{align}\label{2eq: Plancherel} 
\int_{ F_{\vv}^{\times} } | {f} (x)|^2 \nd^{\times} x = \frac {1} {2\pi c_{\vv}} \viint_{ \widehat{\bfra}_{\vv} } \big|\breve{f} (\varnu, m)\big|^2 \nd \mu (\varnu, m) . 
\end{align}

Let $ \widehat{\bfra} = \prod_{\vv |\infty} \widehat{\bfra}_{\vv}$ be  the unitary dual of $F^{\times}_{\infty}$. For $ (\varnu, m) \in \widehat{\bfra} $, define $ \vchi_{i \varnu, \shskip m} $ to be the product of $ \vchi_{i \varnu_{\vv}, \shskip m_{\vv}} $. Let $\nd \mu (\varnu, m) $ be the Lebesgue measure on $\widehat{\bfra}$. In an obvious way, the formulae in \eqref{2eq: Mellin}--\eqref{2eq: Plancherel} extend onto $F^{\times}_{\infty}$ and $\widehat{\bfra}$.



\section{\texorpdfstring{Automorphic Forms on $\GL_2$}{Automorphic Forms on GL(2)}}

In this section, we briefly recollect some   notation and preliminaries, mostly for the statement of   the Kuznetsov formula of Bruggeman and Miatello for $\PGL_2$. For simplicity, only spherical automorphic forms on $\PGL_2 (F) \backslash \PGL_2 (\BA)$ are considered. The reader is referred to \cite{Venkatesh-BeyondEndoscopy} for further discussions.  

Define   
\begin{align*}
N = \left\{  \begin{pmatrix}
1 & r \\ & 1
\end{pmatrix}\right\}, \quad A = \left\{  \begin{pmatrix}
x &   \\ & 1
\end{pmatrix}\right\}, \quad P = \left\{  \begin{pmatrix}
x & r \\ & y
\end{pmatrix}\right\}.
\end{align*}

We denote by $K = K_f K_{\infty}$ the standard maximal compact subgroup of $\PGL_2 (\BA)$. For each non-Archimedean $\vv$, let  
$ K_{\vv} =  \PGL_2 (\frO_{\vv}) $. 
Note that $  \PGL_2 (F) \cap K_f = \PGL_2 (\frO)$ (the intersection is taken in $\GL_2 (\BA_f)$). 
Let $K_{\vv} = \mathrm{O}_2 (\BR) / \{\pm 1_{2} \}$ if $\vv$ is real. Let $K_{\vv} = \mathrm{U}_2 (\BC)/  \{\pm 1_{2} \}$ if $\vv$ is complex. 

We identify $N (F_{\infty})$ with $F_{\infty}$ and $ A (F_{\infty}) $ with $F_{\infty}^{\times}$, and define their measures 
accordingly. For $\vv | \infty$, we normalize the Haar measure on $K_{\vv}$ so that $K_{\vv} /  A (F_{\vv}) \cap K_{\vv}   $ has measure $1$.  Thus the measure of $K_{\vv}$  is $2$ or $2\pi$ according as $\vv$ is real or complex. The Haar measure on $\PGL_2 (F_{\infty})$ is defined  via the Iwasawa decomposition $ \PGL_2 (F_{\infty}) = N (F_{\infty}) A(F_{\infty}) K_{\infty} $.  Again, our measure on the hyperbolic space $ \PGL_2 (F_{\vv}) / K_{\vv} $ is different from that in \cite{BM-Kuz-Spherical} or \cite{Venkatesh-BeyondEndoscopy}. 

\vskip 5 pt

\subsection{Archimedean Representations}\label{sec: Archimedean} In this paper, we shall be concerned
only with {spherical} representations of $\PGL_2 (F_{\infty})$. 

Let $\bfra$ be the vector space $\BR^{r_1+r_2}$. We usually identify $\BR$ with its image under the diagonal embedding $\BR \hookrightarrow \bfra$. 
Let $\bfra_{\BC}$ be its complexification. Let $Y \subset \bfra_{\BC}$ be the set of $\varnu$ such that $\varnu_{\vv} \in \BR$  or $i \varnu_{\vv} \in \hskip -1pt \left(- \frac 1 2, \frac 1 2 \right)$.  We associate to $\varnu = (\varnu_{\vv})_{\vv | \infty} $ in $Y$ a unique spherical unitary irreducible representation $\pi (i \varnu)$ of $\PGL_2 (F_{\infty})$. Namely, $\varnu $ determines a character of the  diagonal torus $A (F_{\infty})$ via
\begin{align*}
\begin{pmatrix}
 x & \\ & 1
\end{pmatrix} \ra \prod_{\vv \shskip \mid \infty} |x  |_{\vv}^{i \varnu_{\vv}} , \qquad x \in F_{\infty}^{\times}, 
\end{align*}
and we let $\pi (i \varnu)$ be the irreducible spherical constituent of the representation unitarily induced
from this character.   The spherical $ \pi (i \varnu) $ is   tempered if and only if $\varnu \in \bfra$.  Let $\nd \shskip \varnu$ be the usual Lebesgue measure on $\bfra$. We equip $\bfra$ with the Plancherel measure $\nd   \mu (\varnu)  $ defined as the product of 
\begin{equation}\label{1eq: defn Plancherel measure}
\nd  \mu (\varnu_{\vv}) =
 \left\{ \begin{aligned}
& \varnu_{\vv} \tanh (\pi \varnu_{\vv}) \nd \shskip \varnu_{\vv} , \  & & \text{ if }  \vv \text{ is real}, \\
&   \varnu_{\vv}^2 \nd \shskip \varnu_{\vv} ,   & & \text{ if }   \vv \text{ is complex}.
\end{aligned}\right.
\end{equation}
Moreover, we define the function ${\mathrm{Pl}} (\varnu) $ to be the product of 
\begin{equation}\label{1eq: defn of Pl(t)}
\mathrm{Pl}_{\vv} (\varnu_{\vv}) = \left\{ \begin{aligned}
& \cosh (\pi \varnu_{\vv})   , & & \text{ if }   \vv \text{ is real}, \\
& \sinh (2\pi \varnu_{\vv}) /   \varnu_{\vv}   , \  & & \text{ if }   \vv \text{ is complex}.
\end{aligned}\right.
\end{equation}
Note that $ 2^{r_2-r_1} \cdot {\mathrm{Pl}} (\varnu) \shskip \nd \shskip \mu (\varnu) $ is the measure used in \cite{BM-Kuz-Spherical,Venkatesh-BeyondEndoscopy}. 

We must fix, once and for all,
a spherical Whittaker vector corresponding to each $\pi (i \varnu)$, with respect to the character $\psi_{\infty}$ on $N (F_{\infty})$. We choose   $W_{ i \varnu} $ to be the product of $W_{i \varnu_\vv}$ so that
\begin{equation}
W_{i \varnu_\vv} \begin{pmatrix}
x_{\vv} & \\ & 1
\end{pmatrix} = \left\{ \begin{aligned}
& |x_{\vv} |_{\vv}^{  1 /2} K_{i \varnu_{\vv}} (2\pi |x_{\vv} |) , & & \text{ if }   \vv \text{ is real}, \\
& |x_{\vv} |_{\vv}^{  1 /2} K_{2 i \varnu_{\vv}} (4\pi |x_{\vv} |)  , \  & & \text{ if }   \vv \text{ is complex}.
\end{aligned}\right.
\end{equation}

Finally,  for $V \in \bfra_+ = \BR_+^{r_1 + r_2}$, we define 
\begin{align}\label{2eq: norm on a+}
\RN (V) = \prod_{ \varv | \infty } V_{\vv}^{N_{\vv}}, \qquad \RN^{\snatural} (V) = \prod_{ \varv | \infty } V_{\vv}. 
\end{align} 
We define the conductor of $\varnu \in Y$ to be $\RC  (\varnu) = \RN (1 + |\varnu|)$, namely, 
\begin{align}\label{1eq: defn conductor}
\RC  (\varnu) = \prod_{ \varv | \infty }  (1 + |\varnu_{\vv}|  )^{N_{\vv}}. 
\end{align}

\vskip 5 pt

\subsection{Automorphic Forms}

We shall be interested in the space of spherical automorphic forms, that is,  functions in  $L^2 (\PGL_2 (F) \backslash \PGL_2 (\BA)  )$ that are (right) invariant under  $K $.

We fix an orthonormal basis $ \SB $ for the   cuspidal subspace  that consists of eigenforms for the Hecke algebra as well as the Laplacian operators (Hecke--Maass cusp forms). Each $f \in \SB $ transforms under a certain representation $\pi (i \varnu_f)$ of $\PGL_2 (F_{\infty})$, for some $ \varnu_f \in Y $. We have the Kim--Sarnak bound in \cite{Blomer-Brumley} over the field $F$:
\begin{align}\label{2eq: Kim-Sarnak}
|\Im (\varnu_{f, \, \vv})| \leqslant \frac 7 {64}, \qquad  \vv | \infty . 
\end{align}
The Fourier coefficients $a_f (\fra, \valpha)$ are indexed by a
 pair consisting of an ideal class  and an element  of $F$.  To be precise, 
$a_f (\fra, \valpha ) $ are defined so that
\begin{align}\label{2eq: Fourier expansion}
f (\begin{pmatrix}
\pi_{\fra} & \\ & 1
\end{pmatrix} g_{\infty}) = \sum_{ \valpha \shskip \in F^{\sstimes}} \frac {a_f (\fra, \valpha)} {\sqrt{\RN (\valpha \fra \frD)}} W_{i \varnu_f} (\begin{pmatrix}
\valpha & \\ & 1
\end{pmatrix} g_{\infty}), \quad g_{\infty} \in \GL_2 (F_{\infty}), 
\end{align}
where $\pi_{\fra} \in \BA_f^{\times}$ is a representative of $\fra$.  It should be kept in mind that $a_f (\fra, \valpha)$ vanishes unless $\valpha \in \fra\- \frD\-$ and only depends on the ideal $\valpha \fra$. We therefore set $a_f (\frm) = a_f (\fra, \valpha)$ if $\frm = \valpha \fra \frD $. These $ a_f (\frm) $ may be interpreted in terms of the non-Archimedean spherical Whittaker function with respect to the additive character $\psi_f$ on $N (\BA_f)$. 
We denote $\lambdaup_f (\frm)$ the $\frm$-th Hecke eigenvalue of $f$. We normalize
 in such a way that the Ramanujan conjecture corresponds to $|\lambdaup_f (\frp)| \leqslant 2 $. As usual, there is a constant $C_f$ so that 
 \begin{align}\label{2eq: af = lambda f}
 a_f (\frm) = C_f \lambdaup_f (\frm)
 \end{align}
 for any nonzero integral ideal $\frm$. See \cite[\S 2.5]{Venkatesh-BeyondEndoscopy} for more details.
 
For $s \in \BC$, define via the Iwasawa decomposition the function
\begin{align*}
f_{s} (\begin{pmatrix}
x & r \\ & y 
\end{pmatrix} k ) = |x / y|^{s +   1 / 2} , \qquad x, y \in \BA^{\times}, \, r \in \BA, \, k \in K,
\end{align*} and define the Eisenstein series $E (g; s)$ by
\begin{align*}
 E (g; s) = \sum_{ \gamma \shskip \in P(F) \backslash \GL_2 (F) } f_s (\gamma g), \qquad g \in \GL_2 (\BA),
\end{align*}
for $\Re  ( s) > \frac 1 2$ and by a process of meromorphic continuation in general. For our purpose, we only need the knowledge of its Fourier coefficients $a_{E (s)} (\frm) = a_{E (s)} (\fra, \valpha)$  ($\frm = \valpha \fra \frD$) for  $\frm$ nonzero. Precisely, we have 
\begin{align}
a_{E (s)} (\frm) = C_{E(s)} \shskip \tau_s  (\frm  ) ,
\end{align}
where 
\begin{align}\label{1eq: defn of eta (m, s)}
\tau_s  (\frm  ) =  \sum_{ \sstyle \frb |\frm  }  \RN \big(\frm \frb^{-2} \big)^{ s  }, \qquad   C_{E(s)} = \frac {P (s)   } {\zeta_F (2s+1) \RN (\mathfrak{D})^{ s} }, 
\end{align}
$\zeta_F (s)$ is the Dedekind $\zeta$ function for $F$, and $P (s)$ is the product of 
\begin{equation}\label{2eq: defn of G(s)}
P_{\vv} (s) = \left\{ \begin{aligned}
& {2 \pi^{s + 1 / 2}  } / {\Gamma \big(s + \tfrac 1 2 \big)} ,  & & \text{ if } \vv \text{ is real}, \\
& {2   (2 \pi)^{2 s   + 1}  } / {\Gamma (2 s + 1 ) }  , \   & & \text{ if } \vv \text{ is complex}.
\end{aligned}\right.
\end{equation}
See for example \cite[\S \S 3.7, 4.6]{Bump}. A subtle issue is that the results in \cite[\S 4.6]{Bump} are proven for $\psi_{ \vv }$ of conductor $\frO_{\vv}$, but this may be easily addressed by re-scaling the character $\psi_{ \vv } (x_{\vv})$ and the Haar measure $\nd \shskip x_{\vv}$ (say,  $\psi_{ \vv } (x_{\vv}) = \psi_{ \vv }^{    \snatural } (\delta_{\vv} x_{\vv})$ and $\nd \shskip x_{\vv} = \hskip -1 pt \sqrt{|\delta_{\vv}|_{\vv}} \shskip \nd^{    \snatural } x_{\vv}$). Moreover, the $P_{\vv} (s)$ in  \eqref{2eq: defn of G(s)} arises from a computation of Jacquet's integral. Note that the definition of $P  $ may be extended from $\BC$ to $    \bfra_{\BC} = \BC^{r_1 + r_2}$ and that 
\begin{align}\label{1eq: P^2 = Pl}
| P (i \varnu) |^2 = 2^{2r_1+3r_2} \pi^{r_2} \cdot  \mathrm{Pl}  (\varnu), \qquad   \varnu \in \bfra. 
\end{align}

\subsection{Kuznetsov--Bruggeman--Miatello Formula}

We first make a preliminary definition.

\begin{defn}\label{def: x inverse}
	Let $\frb$, $\frq$ be fractional ideals with $\frb | \frq$. We set $(\frb/\frq)^{\times}$ to be those elements $x \in \frb/\frq$ which generate $\frb/\frq$ as an $\frO$-module. For $x \in (\frb/\frq)^{\times}$ define $x^{-1}$ to be the unique class $y \in (\frb^{-1}/\frq \frb^{-2})^{\times}$ such that $x y \in 1 + \frq \frb\-$. 
\end{defn}

We now define the Kloosterman sum. By necessity, we must include ideal classes as parameters. 

\begin{defn}[Kloosterman sum]\label{defn: Kloosterman KS} Let $\fra_1, \fra_2$ be nonzero fractional ideals of $F$, and $\frc$ be any ideal so that $\frc^2 \sim \fra_1 \fra_2$. Let $c \in \frc^{-1}$, $\valpha_1 \in \fra_1\- \frD\-$ and $\valpha_2 \in \fra_1 \frc^{-2} \frD\- $. We define the Kloosterman
	sum
\begin{align}\label{2eq: defn Kloosterman KS}
\mathrm{KS} (\valpha_1, \fra_1; \valpha_2, \fra_2; c, \frc) = \sum_{x \, \in (\fra_1 \frc\-/ \fra_1 (c))^{\sstimes} } \psi_{\infty} \bigg( \frac {\valpha_1 x + \valpha_2 x\-} {c} \bigg) ,
\end{align}
where $ (\fra_1 \frc\-/ \fra_1 (c))^{\times} $ and $x^{-1} \in  (\fra_1^{-1} \frc / \fra_1^{-1} (c)\frc^2  )^{\times}$ are defined as in Definition {\rm\ref{def: x inverse}}.
\end{defn}

We should view  the ideals $\valpha_1 \fra_1$ and $\valpha_2 \fra_2$ as the parameters of this Kloosterman sum, and the ideal $c \frc$ as the modulus. However, $\mathrm{KS}$ does depend on
the choice of generator; it is not invariant under the substitution $\valpha \ra \epsilon \valpha$, if $\epsilon \in \frOO^{\times}$ is a unit. To relate the definition to the usual Kloosterman sum, we note if $\fra_1 = \fra_2 = \frc = \frO$, then for $\valpha_1, \valpha_2 \in \frD^{-1}$ and $c \in \frO$ we have
\begin{align*}
\mathrm{KS} (\valpha_1, \fra_1; \valpha_2, \fra_2; c, \frc) = \sum_{x \, \in (\frO /c)^{\sstimes} } \psi_{\infty} \bigg( \frac {\valpha_1 x + \valpha_2 x\-} {c} \bigg). 
\end{align*}
We have the Weil bound for Kloosterman sums:
\begin{align}\label{3eq: Weil}
\mathrm{KS} (\valpha_1, \fra_1; \valpha_2, \fra_2; c, \frc) \Lt \RN (\valpha_1\fra_1 \frD, \valpha_2 \fra_1^{-1} \frc \frD, c \frc)^{1/2} \RN (c\frc)^{1/2+\vepsilon},
\end{align}
where the  brackets $(\cdot\, , \cdot\, , \cdot)$ denote greatest common divisor (of ideals).

\begin{defn}[Space of test functions] \label{defn: test functions}
	Let $S > \frac 1 2$.\footnote{The condition $S > \frac 1 2$ is borrowed from \cite{BM-Kuz-Spherical}, while it requires $S > 2$ in \cite{Venkatesh-BeyondEndoscopy} for some convergence issues. Note that the space of test functions  in \cite{Kuznetsov,B-Mo,B-Mo2} is much larger.} We set
	$ \mathscr{H} (S ) $ to be the space of functions $h : \bfra \ra \BC $ that are of the   form $
	h (\varnu) = \prod_{\varv | \infty} h_{\varv} (\varnu_{\varv}), $  
	where each $h_{\varv} : \BR \ra \BC $ extends to an even holomorphic function on the strip
	$\big\{ s    : |\Im  ( s) | \leqslant S \big\}$ such that,  on the horizontal line $\Im ( s) = \sigma$ {\rm(}$|\sigma| \leqslant S${\rm)}, we have uniformly 
	\begin{align*}
	h_{\varv} (t + i \sigma) \Lt e^{-\pi |t|} (|t|+1)^{- N}, 
	\end{align*}
	for some $N > 6$. 
	
\end{defn}

Next, we define the Bessel kernel   as follows.

\begin{defn}
	[Bessel kernel] \label{defn: Bessel kernel}
	
	Let  $\varnu  \in \bfra_{\BC}$. 
	
	{\rm(1)} When  $F_\varv = \BR$, for $x  \in \BR_+$ we define
	\begin{align*}
	B_{\varnu _{\varv}} (x) & = \frac {\pi} {\sin (\pi \varnu _{\varv}) } \big( J_{-2 \varnu _{\varv}} (4 \pi \sqrt {x }) - J_{2 \varnu _{\varv}} (4 \pi \sqrt {x }) \big), \\
	B_{\varnu _{\varv}} (-x ) & = \frac {\pi} {\sin (\pi \varnu _{\varv}) } \big( I_{-2 \varnu _{\varv}} (4 \pi \sqrt {x }) - I_{2 \varnu _{\varv}} (4 \pi \sqrt {x }) \big) =   {4 \cos (\pi \varnu _{\varv})}    K_{2 \varnu _{\varv}} (4 \pi \sqrt {x }) .
	\end{align*}
	
	{\rm(2)} When  $F_\varv = \BC$, for $z \in \BC^{\times}$ we define
	\begin{equation*}
	B_{\varnu _{\varv}} (z ) =  \frac {2\pi^2} {\sin (2\pi \varnu _{\varv}) } \big( { \textstyle  J_{-2 \varnu _{\varv}} (4 \pi \sqrt {z}) J_{- 2\varnu _{\varv}} (4 \pi \sqrt { \widebar z}) - J_{2 \varnu _{\varv}} (4 \pi \sqrt {z}) J_{ 2\varnu _{\varv}} (4 \pi \sqrt { \widebar z}) } \big). 
	\end{equation*}  
	
	For $x \in F_{\infty}^{\times}$, we define
	\begin{align*}
	\SDB_{\varnu} (x ) = \prod_{ \varv | \infty } B_{\varnu _{\varv}} (x_{\varv} ).
	\end{align*}
\end{defn}	

Note that  the $I$-Bessel functions in \cite{BM-Kuz-Spherical} or \cite{Venkatesh-BeyondEndoscopy} should be changed into $J$-Bessel functions in the complex case. Moreover,  according to \cite{Qi-Bessel}, we have normalized the Bessel kernel by a factor $\pi$ or $2\pi^2$ for real or complex $\vv$, respectively.

\begin{prop}[Kuznetsov formula]\label{prop: Kuznetsov}
Let $h$ be a test function on $\bfra_{\BC}$ belonging to 	$ \mathscr{H} (S) $ {\rm(}defined in Definition {\rm\ref{defn: test functions}}{\rm)} and define 
\begin{align}\label{1eq: defn Bessel integral}
 \SDH = \int_{\bfra} h (\varnu) \nd \shskip \mu (\varnu), \quad \SDH (x) = \int_{\bfra} h (\varnu) \SDB_{i \varnu} (x ) \nd \shskip \mu (\varnu), \qquad   x\in F_{\infty}^{\times}, 
\end{align}
where $ \SDB_{\varnu} (x )$ is the Bessel kernel defined in Definition {\rm\ref{defn: Bessel kernel}} and $\nd \shskip \mu (\varnu)$ is the Plancherel measure in {\rm\eqref{1eq: defn Plancherel measure}}. Let $ \SB $ be an orthonormal basis  of the spherical cuspidal spectrum on $  \PGL_2 (F) \backslash \PGL_2 (\BA) / K$, so that each $f \in \SB$ has Archimedean parameter $ \varnu_f \in  Y$ {\rm(}$f$ transforms under the $\GL_2 (F_{\infty})$-representation $\pi (i\varnu_f)${\rm)} and Hecke eigenvalues $\lambdaup_f (\frm)$. Let $\fra_1, \fra_2$ be fractional ideals. 
Let $\valpha_1 \in \fra_1\- \frD\-, \valpha_2 \in \fra_2\- \frD\- $. Set $\frm_1 = \valpha_1 \fra_1 \frD,   \frm_2 = \valpha_2 \fra_2 \frD$, and
\begin{align}
\beta  = \beta_{\frc, \shskip \fra_1 \fra_2} =   [\frc^2  ( \fra_1 \fra_2 )\-  ]
\end{align}
for every $\frc \in \widetilde{C}_F$ with $\frc^2 \sim \fra_1 \fra_2$ {\rm(}here $ [\frc^2  ( \fra_1 \fra_2 )\-  ]$ is a chosen generator for this principal ideal{\rm)}. We have
	\begin{equation}\label{1eq: Kuznetsov} 
	\begin{split}
& 	\sum_{f \in \SB  } \hskip -1pt \omega_f      h  ( \varnu_f ) \lambdaup_f  ( \frm_1 )     {\lambdaup_f  ( \frm_2 )} + \frac {1}  {4\pi} c_0 \int_{-\infty}^{\infty} \hskip -2pt  \omega (t) h ( t ) 
	\tau_{it}  (\frm_1  ) \tau_{ it}  (\frm_2  )         \nd t  \\
	   = & \, c_1 \delta_{\frm_1, \shskip \frm_2}  \SDH  + c_2 \mathop{\sum_{\frc \, \in \shskip \widetilde{C}_F}}_{ \frc^2 \sim \fra_1 \fra_2 } \sum_{ \epsilon \, \in \shskip \frOO^{\sstimes} \hskip -1pt / \frOO^{\sstimes 2} } \sum_{c \, \in \shskip \frc^{-1} } \frac {\mathrm{KS} ( \valpha_1, \fra_1; \epsilon \valpha_2 / \beta , \fra_2; c, \frc ) } { \RN (c \frc) } \SDH \bigg( \frac {\epsilon \valpha_1 \valpha_2 } {\beta c^2    }  \bigg),
	\end{split}
	\end{equation}
	where 
	\begin{align}\label{1eq: omegas}
	\omega_f = \frac {|C_f|^2} {\mathrm{Pl}(\varnu_f)}, \qquad  \omega (t) =  \frac {2^{2r_1+3r_2} \pi^{r_2}} {|\zeta_F(1+2it)|^2}, 
	\end{align} 
{\rm(}see {\rm \eqref{1eq: defn of Pl(t)}} and {\rm\eqref{2eq: af = lambda f}} for the definition of $\mathrm{Pl} (\varnu_f)$ and $C_f${\rm)},	$\tau_{s}  (\frm   )  $ is defined in {\rm\eqref{1eq: defn of eta (m, s)}}, $\delta_{\frm_1, \shskip \frm_2} $ is the Kronecker $\delta$ that detects $\frm_1 = \frm_2$, $\mathrm{KS}$ is the Kloosterman sum as in Definition {\rm\ref{defn: Kloosterman KS}}, and  the constants $c_0$, $c_1$  and $c_2$ are given by
	\begin{align}
	c_0 = \frac  {2^{r_1} (2\pi)^{r_2} R_F }  {\varw_F \sqrt{|d_F|} }, \quad c_1 = \frac { 2^{r_2} \sqrt{|d_F|}} {2 \pi^{2 r_1 + 2 r_2} h_F}  , \quad c_2 =   \frac {2^{r_2}} {4 \pi^{2 r_1 + 2 r_2} h_F }  ,
	\end{align}
	in which $\varw_F$, $d_F$, $R_F$ and $h_F$ are the number of roots of unity in $F$, the discriminant, the regulator and the class number of $F$, respectively. 
	
\end{prop}

	The formula \eqref{1eq: Kuznetsov}  is just a rewriting of (15) in \cite{Venkatesh-BeyondEndoscopy} in the fashion of (3.17) in \cite{CI-Cubic}. Some remarks are in order. 
	The test function  in \cite{Venkatesh-BeyondEndoscopy} has been modified here by  $   \mathrm{Pl} (\varnu)$ (see \eqref{1eq: defn of Pl(t)} and \eqref{1eq: P^2 = Pl}). 
	 The $1/4\pi$ in the first line of \eqref{1eq: Kuznetsov}  is adopted from (7.15) in \cite{Iw-Spectral}, and in the ad\`elic setting it also arises from  applying (5.16) to (4.21), (4.25) in \cite{Gelbart-Jacquet}. 
	The constant $c_0$ accounts for the translation of the spectral decomposition  from the ad\`elic to the classical setting. The constants $c_1$ and $c_2$ are adapted from (16) in \cite{Venkatesh-BeyondEndoscopy}, with extra factors due to our normalization of measures, modules and Bessel kernels.

\begin{lem}\label{lem: lower bounds for omega}
	Let $f $ be a be a Hecke--Maass cusp form whose  $L^2$ norm is $1$. Let $\omega_f$ and $\omega (t)$ be defined as in {\rm\eqref{1eq: omegas}}. Then 
	\begin{align}\label{1eq: bounds for omega}
	\omega_f \Gt \RC  (\varnu_f)^{-\vepsilon} , \qquad \omega (t) \Gt (1 + |t|)^{-\vepsilon},
	\end{align}
where $\RC  (\varnu_f)$ is the conductor of $\varnu_f$  as defined in {\rm\eqref{1eq: defn conductor}}, and the implied constants depend only on $\vepsilon$ and $F$. 
\end{lem}

By the Rankin--Selberg method,   $\omega_f$ may be expressed as a multiple of $1 / L (1, \mathrm{Sym}^2 f)$: 
\begin{align}
	\omega_f = \frac {2^{2r_1+3r_2 } \pi^{r_2}}  {2 L (1, \mathrm{Sym}^2 f)}.
\end{align} Then \eqref{1eq: bounds for omega} follows from \cite[Theorem 1]{Molteni-L(1)}. For $F = \BQ$, the lower bound for $\omega_f$ was first proven by Iwaniec \cite{Iwaniec-L(1)}. 

	\section{\texorpdfstring{Vorono\"i Summation for $\GL_3$}{Voronoi Summation for GL(3)}} \label{sec: Voronoi GL(3)}
	
The purpose of this section is to derive a Vorono\"i summation formula for $\GL_3$ in the classical terms from that of Ichino and Templier   \cite{Ichino-Templier} in the ad\`elic setting (see Appendix \ref{app: adelic Voronoi, proof}). 
	
\subsection*{Notation in the Ad\`elic Setting} 

We first recollect some notation from \cite{Ichino-Templier}.	Let $\pi = \otimes_{\varv} \, \pi_{\varv}$ be   an irreducible cuspidal automorphic representation of $\GL_3(\BA)$.  Let $\widetilde {\pi} = \otimes_{\varv} \, \widetilde \pi_{\varv}$ be the contragradient  representation of $\pi$. Let $\omega$ denote the central character of $\pi$. 

Let $S$ be a finite set of places of $F$ including the ramified places of $\pi$ and all the Archimedean places.  
 Denote by $W^S_{\text{o}} = \prod_{\varv \shskip  \notin S} W_{\text{o}  \shskip  \varv}$ the normalized unramified Whittaker
function of $\pi^S = \otimes_{\varv \shskip  \notin S} \pi_{\varv}$ above the complement of $S$. Let  $\widetilde W^S_{\mathrm{o}} = \prod_{\varv \shskip  \notin S} \widetilde W_{\mathrm{o}  \shskip  \varv}$ be the unramified (spherical) Whittaker function of   $\widetilde \pi^S = \otimes_{\varv \shskip  \notin S} \widetilde \pi_{\varv}$.

For any place $\varv$ of $F$, to a smooth compactly supported function $w_{\varv} \in C_c^{\infty} (F_{\varv}^{\times})$ is associated a dual function $\widetilde \varww_{\varv}$  of $ \varww_{\varv} $ such that
\begin{equation}\label{2eq: Ichino-Templier}
\begin{split}
\int_{ F_{\varv}^{\sstimes}} \widetilde \varww_{\varv} (x) \vchi (x)\-   |x|_{\varv}^{s - 1} \nd^\times \hskip -1pt x  =   \gamma (1-s, \pi_{\varv} \otimes \vchi, \psi_{\varv} ) \int_{ F_{\varv}^{\sstimes}} \varww_{\varv} (x) \vchi (x) |x|_{\varv}^{ - s } \nd^\times \hskip -1pt x,
\end{split}
\end{equation}
for all $s$ of real part sufficiently large and all unitary multiplicative characters $\vchi $ of $ F_{\varv}^\times$.
The equality \eqref{2eq: Ichino-Templier} is independent on the chosen Haar measure $\nd^\times \hskip -1pt x$ on $ F_{\varv}^\times$ and 
defines  $\widetilde w_{\varv}$ uniquely in terms of $\pi_{\varv}$, $\psi_{\varv}$ and $w_{\varv}$. For $S$ as above, we put $
w_S = \prod_{ \varv \shskip  \in S }  w_{\varv}  $, $ \widetilde w_S = \prod_{ \varv \shskip  \in S } \widetilde w_{\varv}  $.

Let $\varv$ be an unramified place of $\pi$. It should be kept in mind that, since the additive character $\psi_{ \vv }$ has conductor $\mathfrak{D}_{\vv}\-$,   $W_{\mathrm{o}\shskip  \vv} (a (x_1, x_2))$ vanishes unless both $x_1, x_2 \in \mathfrak{D}_{\vv}\-$, where 
\begin{align*}
a (x_1, x_2) = \begin{pmatrix}
x_1 x_2 & & \\
& x_1 & \\
& & 1 
\end{pmatrix} .
\end{align*}  

\begin{defn}
	[Kloosterman sum on a local field]  \label{defn: Kloosterman, local}
	Let $\vv$ be non-Archimedean. For $\valpha,  1/\nu  \in \frO_{\vv}$ and $ \beta  \in     \nu^{2} \frD_{\vv}^{-2} $, define  the local Kloosterman sum  
	\begin{align}\label{2eq: Kloosterman sum}
	\mathrm{Kl}_{\vv} (\valpha, \beta; \nu  ) =  \sum_{   \delta_{\vv}  x  \shskip \in \shskip   \nu   \shskip \frOO_{\,\varv}^{\sstimes} /   \frO_\varv } 
	\psi_\varv   \big( \valpha  x + \beta  x^{-1}   \big)  .
	\end{align} 
	In the quotient $ \nu  \frOO^\times_{\, \varv} / \frO_\varv $ above, the group $ \frO_{\varv}$ acts additively on $\nu  \frOO^\times_{\,\varv}$ if $|\nu |_{\varv} > 1$, and $ \nu  \frOO^\times_{\,\varv} / \frO_\varv = \{1\}$ if $|\nu|_{\varv} = 1$ so that the Kloosterman sum is equal to $1$.
\end{defn}

 Let $R$ be a finite set of places where $\pi$ is unramified. We define $   \widetilde W_{\mathrm{o} \shskip  R} = \prod_{\varv \shskip \in R} \widetilde W_{\mathrm{o} \shskip  \varv} $ and $ \mathrm{Kl}_{R} (\valpha, \beta; \nu  ) = \prod_{\varv \shskip  \in R} \mathrm{Kl}_{\varv} (\valpha_{\varv} , \beta_{\varv} ; \nu_{\varv}   ) $ for $\valpha, \beta, \nu \in F_R $ satisfying $|\valpha|_{\varv} \leqslant 1 \leqslant |\nu|_{\varv}$ and $|  \beta|_{\varv} \leqslant  |  \nu/ \delta   |_{\varv}^2 $ for all $\vv \in R$.

\subsection{\texorpdfstring{Ad\`elic Vorono\"i Summation for $\GL_3$}{Ad\`elic Voronoi Summation for $\GL(3)$}} \label{sec: adelic Voronoi}

Note that it is required in \cite{Ichino-Templier} that  $S$ contains the ramified places of $\psi$. In order to make their Vorono\"i summation useful when the class number $h_F \neq 1$ (or when $\frD$ is not principal), one has to relax this   condition. For this, we shall outline a proof of the following Vorono\"i summation for $\GL_3$ in Appendix \ref{app: adelic Voronoi, proof}. 

\begin{prop} \label{prop: Voronoi, adelic}   Let notation be as above. 
	Let $\zeta \in \BA^{    S}$ and $  \valpha  \in \BA^{\times   S}$.  Let $ R $ be the set of places $\varv$ such that $|\zeta/ \valpha  |_{\varv} > 1$. 
	Let $\varww_{\varv} \in C_c^{\infty} (F_{\varv}^{\times})$ for all $\varv \in S$. Then we have the identity
	\begin{equation}\label{2eq: Voronoi, adelic}
	\begin{split}
	\sum_{\gamma \shskip  \in F^{\sstimes}}  \hskip -2 pt & \psi^S (\gamma \zeta)  W_{\mathrm{o}}^S   
	(a (  1 / \delta, \valpha  \gamma ))
	\varww_{S} (\gamma) \\
	&	= \frac {   \overbar \omega_{R} (\zeta /\valpha   ) |   \zeta   |_{R}  |\valpha|^{S \cup R } } {  \omega^S (\delta) \sqrt{|\delta|^{S}}}   \sum_{\gamma \shskip  \in F^{\sstimes}} \hskip -2 pt K_{R} (\gamma, \zeta , \widetilde W_{\mathrm{o} \shskip R})  \widetilde W_{\mathrm{o}}^{S \cup R}   
	\big(a (1 / \delta, \gamma \delta /   \valpha) \big)
	\widetilde  \varww_S (\gamma),
	\end{split}
	\end{equation} 
	where $K_{R}   (\gamma, \zeta ,  \widetilde   W_{\mathrm{o} \shskip R})$ is defined to be the sum 
	\begin{equation}\label{2eq: K(...) = , 0}
	\begin{split}
	\sum_{\nu \, \in F^{\sstimes}_{R} /\frOO^{\sstimes}_{R}} \hskip -2 pt \widetilde{W}_{\oo \shskip R}  \big( a (  \nu   \valpha   /  \zeta \delta, \shskip     \gamma  \delta   / \nu^2    \zeta          ) \big)  \mathrm{Kl}_{R} (1,  -   \gamma  /   \zeta ; \nu),
	\end{split} 
	\end{equation}
	with  $\nu$ subject to 
	\begin{align}\label{2eq: condition on nu}
	1 \leqslant |\nu|_{\vv} \leqslant   | \zeta / \valpha     |_{\vv},   \qquad |  \gamma /    \zeta |_{\vv} \leqslant | \nu / \delta  |^2_{\vv}  
	\end{align}
	for all $\vv \in R$.  
\end{prop}

\subsection*{Notation in the Classical Setting} 

   Henceforth, we shall assume that     $\pi$ is {\it unramified} ({\it spherical}) at every non-Archimedean place  and that its  central character $\omega$ is {\it trivial}. We may thus choose $S = S_{\infty}$. 
   
   For  nonzero {\it integral} ideals $\frn_1 , \frn_2  $, we define  the Fourier coefficient
   \begin{align}\label{2eq: Fourier coefficients, r=3}
   A (\frn_1 , \frn_2 ) =    \mathrm{N} \big( \frn_1  \frn_2 \mathfrak{D}^{-2} \big) W_{\mathrm{o}}^{S_{\infty}}   
   \begin{pmatrix}
   \frn_1 \frn_2 \mathfrak{D}^{-2} & & \\
   & \frn_1 \mathfrak{D}\- & \\
   & & 1 
   \end{pmatrix}.
   \end{align}  
Normalize the Fourier coefficients so that $ A (1, 1) = 1$. We have the multiplicative relation 
\begin{align}\label{4eq: mult relation}
A (\frn_1 \frm_1, \frn_2 \frm_2) =  A (\frn_1  , \frn_2  ) A (  \frm_1,   \frm_2), \qquad (\frn_1 \frn_2, \frm_1 \frm_2) = (1), 
\end{align}
and the Hecke relation
\begin{align}\label{4eq: Hecke relation}
A (\frn_1, \frn_2) = \sum_{ \mathfrak{d} | \frn_1, \shskip \mathfrak{d} |\frn_2} \mu (\mathfrak{d}) A (\frn_1 \mathfrak{d}^{-1}, 1) A (1, \frn_2 \mathfrak{d}^{-1}),
\end{align}
where $\mu $ is the M\"obius function for $F$.  Moreover, it is known that $ \widetilde{A} (\frn_1, \frn_2) =  {A} (\frn_2, \frn_1) $ if  $\widetilde{A} (\frn_1, \frn_2)$  are the Fourier coefficients for $\widetilde{\pi}$.

It is known from \cite[\S 17]{Qi-Bessel} that for each $\vv | \infty$, $\widetilde f_{\vv} (y) =   | y|_{\vv}^{- 1} \widetilde w_{\vv} \lp - y \rp $ is the Hankel integral transform of $ f_{\varv} (x) =  | x|_{\vv}^{ - 1} w_{\vv} \lp x \rp  $ integrated against the Bessel kernel $J_{\pi_{\varv}} (xy)$ attached to $\pi_{\vv}$ (see \cite[(17.20)]{Qi-Bessel}).  Namely,
\begin{equation}\label{2eq: Hankel transform tilde u and u}
	\widetilde f_{\vv} \lp  y \rp  =   \int_{F^{\sstimes}_{\vv} }  f_{\vv} (x)  J_{\pi_{\vv}} ( x y) \nd x.  
\end{equation} 
The asymptotic expansion for   $J_{\pi_{\varv}} (x)$ will be given in \S \ref{sec: asymptotic Bessel}. 

\begin{rem}\label{rem: normalization Hankel}
	When $\vv$ is real, $\widetilde\varww_{\vv}(y) = |y|_{\vv} \widetilde{f}_{\vv} (-y)$ is equal to the $F(y)$ in \cite[Theorem 1.18]{Miller-Schmid-2006} {\rm(}if $f_{\vv} (x)$ is their $f(x)${\rm)}. The reason for our normalization of Hankel transforms is to get the 
	Fourier transform and the classical Hankel transform in the $\GL_1$ and $\GL_2$ settings, respectively.
\end{rem} 


\begin{defn}[Bessel kernel and Hankel transform]\label{defn: Hankel transform}
	For $x \in F_{\infty}$, we define the Bessel kernel
	\begin{align*}
	\SDJ  (x ) = \prod_{ \varv | \infty } J_{\pi_\varv} (x_{\varv} ). 
	\end{align*}
Let	$ \mathscr{C}^{\infty}_c (F^{\times}_{\infty}) $ denote the space of compactly supported smooth functions $f : F^{\times}_{\infty} \ra \BC $ that are of the  product form $
f (x) = \prod_{\varv | \infty} f_{\varv} (x_{\varv})$. For $f \in \mathscr{C}^{\infty}_c (F^{\times}_{\infty}) $, we define its Hankel transform $ \widetilde{f}  $ by
\begin{align}\label{3eq: Hankel, global}
\widetilde{f} (y) =    \int_{F^{\sstimes}_{\scalebox{0.55}{$\infty$} } }  f (x) \SDJ    ( x y)   \nd x, \qquad y \in F^{\times}_{\infty}. 
\end{align}
\end{defn}



\begin{defn}
  \label{defn: psi b ...}
	For $ \frb \subset  \frO$, define  the ring $$F_{\frb} = \big\{ \valpha \in F :  \valpha \in \frO_{\vv} \text{ for all } \frp_{\vv} | \frb \big\} . $$ Define $\psi_{\frb} = \prod_{\frp_{\vv} | \frb} \psi_{\vv} $ and $\fra_{\frb} = (\fra, \frb^{\infty}) =
	\prod_{\frp_{\vv} | \frb}   \frp_{\vv}^{\mathrm{ord}_{\vv} (\fra)} $.     
	 
\end{defn}

\begin{defn}
	[Kloosterman sum]  \label{defn: Kloosterman, globle}

	Let $\frb \subset \frq \subset \frO$.  For $\valpha  \in F_{\frb}$ and $ \beta  \in     (\frq \frD_{\frb})^{-2} F_{\frb} $, define  the Kloosterman sum  
	\begin{align}\label{4eq: Kloosterman sum, global}
	\mathrm{Kl}_{\frb} (\valpha, \beta; \frq  ) =  \sum_{    x  \shskip \in \shskip   ((\frq  \frD_{\frb})^{-1} / \frD_{\frb}^{-1})^{\times} } 
	\psi_{\frb}   \big( \valpha  x + \beta  x^{-1}   \big) ,
	\end{align} 
	where $((\frq  \frD_{\frb})^{-1} / \frD_{\frb}^{-1})^{\times}$ and $x^{-1} \in (\frq \frD_{\frb} / \frq^2 \frD_{\frb})^{\times}$ are defined as in Definition {\rm\ref{def: x inverse}}.
\end{defn}

We have $\psi_{\frb} = \psi_{R}$ and $\mathrm{Kl}_{\frb} (\valpha, \beta; \frq  ) = \mathrm{Kl}_{R} (\valpha, \beta;  \nu  )$ if $R $ is the set of primes dividing $\frb$ and $\frq$ is the integral ideal  generated by  $1/\nu  $. Moreover, it will be convenient to write $\mathrm{Kl}_{\frb} (\valpha, \beta; \frq  )$   as an integral on a certain compact homogeneous subspace of $F_{R}^{\times}$.

\begin{lem}\label{lem: Kloosterman as integral}
	Let notation be as above.  Define the Euler totient function $\varphi (\frq)$ as usual by 
	\begin{align}\label{4eq: Euler totient}
	\varphi (\frq) = \RN{(\frq)}  {\prod_{ \frp | \frq }  \big( 1 - \RN(\frp)^{-1}  \big)}.
	\end{align} Define $\widehat{\frO}{}^{\times}_{\frb} = \prod_{ \frp_{\vv} | \frb } \frO_{\varv}^{\times} $. Let $  \pi_{(\frq  \frD )^{-1} } \in \BA_f^{\times} $ be the chosen generator for $(\frq  \frD )^{-1}$.   
	Then 
	\begin{equation}\label{4eq: Kloosterman, integral}
	\begin{split}
	\mathrm{Kl}_{\frb} (\valpha, \beta; \frq  ) = 
	\varphi (\frq) \int_{\pi_{(\frq \frD )^{\scalebox{0.37}{$-1$}}   }  \shskip \widehat \frOO{\shskip}^{\sstimes}_{\frb} } \psi_{\frb} \big (    {\valpha  x + \beta x\-}  \big) \nd^{\times}\hskip -1pt x,
	\end{split}
	\end{equation} 
\end{lem}

\begin{proof}
	With our choice of Haar measure, the space $(\frq  \frD )^{-1} \widehat{\frO}{}^{\times}_{\frb}  $($=(\frq  \frD_{\frb})^{-1} \widehat{\frO}{}^{\times}_{\frb}$) has total mass $1$. 
	On the other hand, the $x$-sum in \eqref{4eq: Kloosterman sum, global} ranges over a set of size $\varphi (\frq)$. Then \eqref{4eq: Kloosterman, integral} follows from the definitions.  
\end{proof}

\subsection{\texorpdfstring{Classical Vorono\"i Summation for $\GL_3$}{Classical Voronoi Summation for $\GL(3)$}} \label{sec: classical Voronoi} 

In practice, we shall let    $\zeta \in \BA^{S_{\infty}} = \BA_f$   be the diagonal embedding of a fraction $a / c \in F  $, and it is preferable to have a   classical formulation of the Vorono\"i summation in terms of Fourier coefficients, exponential factors, Kloosterman sums, and Hankel transforms. 

\begin{prop}\label{prop: Voronoi}
	Let notation be as above. Let $a \in F$, $c \in F^{\times}$, and $\fra \subset \frO$. Let $R = \big\{ \vv \nnmid \infty : \mathrm{ord}_{\vv} (a/c) < \mathrm{ord}_{\vv}  ( \fra \frD\-  ) \big\}  $, and set $\frb = \prod_{ \vv \shskip  \in R } \frp_{\vv}^{ \mathrm{ord}_{\vv}  ( (c/a) \fra \frD\-  ) } ${\rm(}it is understood that $\frb = \frO$ if $R =  \text{{\rm\O}}${\rm)}. For $f \in \mathscr{C}^{\infty}_c (F^{\times}_{\infty})$ let its Hankel transform $\widetilde{f}$ be given by {\rm\eqref{3eq: Hankel, global}} as in Definition {\rm\ref{defn: Hankel transform}}. Then 
	\begin{equation}\label{2eq: Voronoi, classical}
	\begin{split}
	 \sum_{\gamma \shskip  \in \fra^{-1}   }  \hskip -2 pt \psi_{\infty} \left(- \frac {a \gamma} {c} \right)  A 
	(  1 , \fra \gamma ) &
	f (\gamma)   =   \frac { \RN  (\fra    ) } { {\RN(\frD)}^{3/2} } \sum_{ \frb \shskip \subset \frq \shskip \subset \frO } \frac {1 } {\RN(\frb \frq )   } \\
	 	\cdot & \mathop{ \sum}_{\gamma \shskip \in \fra   (\frb \frq^{2} \frD^{3})^{-1}  } A  (  \fra\- \frb \frq^2   \frD^3 \gamma,  \frb \frq\-  ) \mathrm{Kl}_{\frb} (1, \gamma c / a ; \frq  ) \widetilde{f} (\gamma) .
	\end{split}
	\end{equation}
\end{prop}

\begin{proof}
Let $\zeta \in \BA^{S_{\infty}} = \BA_f$   be the diagonal embedding of  $a / c \in F  $. Let $ \valpha   \in \BA_f^{\times}$ generate the ideal $\fra  \frD\-$. In the above settings, the left-hand side of \eqref{2eq: Voronoi, adelic} is translated into that of \eqref{2eq: Voronoi, classical} up to the constant $ \mathrm{N}  ( \fra^{-1} \mathfrak{D}^{2}   ) $. Note that $\psi_{f}  ( {a \gamma} / {c} ) = \psi_{\infty}  (-   {a \gamma} / {c} )$ as $\psi $ is trivial on $F$. For the right-hand side of   \eqref{2eq: Voronoi, adelic}, we set $\frq$ to be the ideal generated by $1/\nu $, then the conditions in \eqref{2eq: condition on nu} amounts to $ \frb \shskip \subset \frq \shskip \subset \frO $ and $\gamma \fra\- \frb    \frD  \subset \frq^{-2} \frD^{-2} $. The Kloosterman sum is $\mathrm{Kl}_{\frb} (1, \gamma c / a ; \frq  )$ as defined in Definition \ref{defn: Kloosterman, globle}. After changing $\gamma$ to $- \gamma$, we arrive at the right-hand side of  \eqref{2eq: Voronoi, classical}. 
\end{proof}
 
\begin{rem}
	When the class number $h_F = 1$, we may choose $a \in \frO$ and $c \in \frO \smallsetminus \{0\}$ such that $(a, c) = (1)$ and let $\fra = \frD$, $\frb = (c)$, and $\frq = (c/d)$ with $d | c$. In this way,
	upon changing $\gamma$ to $n$ or $ d^2 n /c^3 $ on the left or right, respectively, we obtain the classical Vorono\"i summation formula as in {\rm\cite{Miller-Schmid-2006}}. 
\end{rem}

\subsection{Averages of Fourier Coefficients} 
We recollect here some results from \cite[\S 4]{Qi-Wilton}.
First, as a consequence of the Rankin--Selberg theory (\cite{J-S-Rankin-Selberg}), it is well-known that for $X \geqslant 1$ 
\begin{align}\label{3eq: Ramanujan on average}
\mathop {\sum \sum}_{ \RN ( \frn_1^2 \frn_2)  \leqslant X } |A  (\frn_1, \frn_2) |^2 = O_{\pi} (X).  
\end{align}
As a consequence (\cite[(4.4)]{Qi-Wilton}), for $0 \leqslant c < 1$ we have
\begin{align}\label{3eq: Ramanujan on average, 1}
\mathop {\sum }_{ \RN ( \frn_2)  \leqslant X } \frac {|A (\frn_1, \frn_2) |} {{\RN (\frn_2)^{  c}}} = O_{c, \shskip \pi} (\RN(\frn_1) X^{1-c}),   
\end{align}  
Moreover, we have the following lemma for the average over $\gamma \in \fra^{-1} \smallsetminus \{0\}$. 

\begin{lem} \label{lem: averages} 
	For   $V \in \bfra_+$ and  $S \subset S_{\infty}$, define
$|V |_{S} = \prod_{ \vv \, \in S } V_{\vv}^{N_{\vv}}$ and 
\begin{align} \label{4eq: defn of FS (T)}
F_{\infty}^{S} (V) = \big\{ x \in F_{\infty} : |x|_{\vv} > V_{\vv}^{N_\vv} \text{ for all } \vv \in S, |x|_{\vv} \leqslant V_{\vv}^{N_\vv} \text{ for all } \vv \in S_{\infty} \smallsetminus S \big\} .
\end{align}
Let $0 \leqslant c < 1 < d$.  Then for any $0 < \vepsilon < d-1$ we have
\begin{align}\label{4eq: average over ideals, 2}  
\mathop{\sum_{\sstyle \gamma \, \in \, \Fx   \cap F_{\infty}^{S} (V)  }}_{  \sstyle  \gamma \fra \, \subset \frO }   
\frac { |A  (\frn , \gamma \fra) | } {|\RN \gamma|^{c} |\gamma |_S^{d-c}  } & = O_{\vepsilon, \shskip c, \shskip d, \shskip \pi} \bigg(  \frac { \RN (\frn)  \RN ( \fra )^{1 + \vepsilon} \RN(V)^{1-c + \vepsilon} }  {|V |_S^{d-c} }   \bigg).
\end{align}  
\end{lem}
 
\begin{proof}
	When $\RN (\fra) \RN(V) \geqslant 1$ is assumed, \eqref{4eq: average over ideals, 2}  follows from Lemma 4.2 or 4.3 in \cite{Qi-Wilton} in the case $S = \text{\O}$ or $S \neq \text{\O}    $, respectively. However, this assumption may be safely and conveniently removed. For example, the sum in \eqref{4eq: average over ideals, 2} has no terms if  $S = \text{\O}$ and $\RN (\fra) \RN(V) < 1$. 
\end{proof}
 
The next lemma is a generalization of (10) in \cite{Blomer}.

\begin{lem}\label{lem: second moment}
  Let $\theta \leqslant \frac 1 2$ be an  exponent such that \begin{align}\label{4eq: Ramanujan}
  | A  (\frn_1, \frn_2)| \leqslant \RN (\frn_1 \frn_2)^{\theta +\vepsilon}.
  \end{align}  Define $F_{\infty}^{\text{\rm\O} } (V) $ as in {\rm\eqref{4eq: defn of FS (T)}}. For $\frf \subset \frO$, we have
	\begin{align}\label{4eq: 2nd moment}
\mathop{\sum_{\sstyle \gamma \, \in \, \Fx   \cap F_{\infty}^{\text{\rm\O} } (V)  }}_{  \sstyle  \gamma \fra \, \subset \frf }   
  { |A  (\frn , \gamma \fra) |^2 }  \Lt_{\vepsilon, \shskip \pi} \RN (  \frf \frn)^{\theta + \vepsilon} \RN (\fra   \frf^{-1} )^{1 + \vepsilon} \RN (V) . 
	\end{align}
\end{lem}

\begin{proof}
	First of all, one can apply Lemma 4.1 in  \cite{Qi-Wilton} with \eqref{3eq: Ramanujan on average} to prove
	\begin{align}\label{4eq: second moment}
\mathop{\sum_{\sstyle \gamma \, \in \, \Fx   \cap F_{\infty}^{\text{\rm\O} } (V)  }}_{  \sstyle  \gamma \fra \, \subset \frO }   
{ |A  (1 , \gamma \fra) |^2 }  \Lt_{\vepsilon, \shskip \pi} \RN (\fra )^{ 1+ \vepsilon} \RN (V);
	\end{align} 
the proof is similar to that of Lemma 4.2 in \cite{Qi-Wilton}. The left-hand side of \eqref{4eq: 2nd moment} is bounded by
\begin{align*}
\sum_{ \frm | (\frf\frn)^{\infty} } \hskip -3 pt \mathop{\mathop{\sum_{\sstyle \gamma \, \in \, \Fx   \cap F_{\infty}^{\text{\rm\O} } (V)  }}_{  \sstyle  \gamma \fra \, \subset \frf \frm }}_{(\gamma \fra (\frf \frm)^{-1}, \frf \frm \frn ) = (1)} \hskip -5 pt
{ |A  (\frn , \gamma \fra) |^2 } \leqslant \sum_{ \frm | (\frf\frn)^{\infty} } \hskip -5 pt |A (\frn, \frf \frm) |^2  \hskip -4 pt \mathop{\sum_{\sstyle \gamma \, \in \, \Fx   \cap F_{\infty}^{\text{\rm\O} } (V)  }}_{  \sstyle  \gamma \fra (\frf \frm)^{-1} \, \subset \frO } \hskip -5 pt \big| A (1, \gamma \fra (\frf \frm)^{-1})\big|^2 ,
\end{align*}
and \eqref{4eq: 2nd moment} then follows from \eqref{4eq: Ramanujan} and \eqref{4eq: second moment}. 
\end{proof}

\subsection{Asymptotics for $\GL_3$-Bessel Kernels} \label{sec: asymptotic Bessel}

Assume further that  $\pi_{\infty}$ is a {\it spherical} representation of $\PGL_3 (F_{\infty})$ so that its Archimedean Langlands parameter is given by a triple $\boldsymbol{\mu} = (\mu_{ 1 }, \mu_{ 2 } , \mu_{ 3 } )$ in $\bfra_{\BC}^3$, with $\mu_1 + \mu_2 + \mu_3 = 0$. For each $\vv | {\infty}$, the Bessel kernel $J_{\pi_{\varv}} (x) = J_{\boldsymbol{\mu}_{\vv}} (x)$ depends only on $ \boldsymbol{\mu}_{\vv}$.  A main result of \cite{Qi-Bessel} is the following asymptotic expansions for $J_{\boldsymbol{\mu}_{\vv}} (x)$ when $x$ is large (see \cite[Theorem 14.1, 16.6]{Qi-Bessel}).

\begin{lem}\label{prop: asymptotic J}
	 Let $K$ be a non-negative integer.
	
	{\rm(1)} When $\vv$ is real, for $x \Gt_{K, \, \boldsymbol{\mu}_{\vv} } 1$, we have  the asymptotic expansion 
	\begin{align}\label{4eq: asymptotic, Bessel, R} 
	J_{\boldsymbol{\mu}_{\vv}}  (\pm x ) & =     \frac { { e \lp   \pm  3 x^{1/3}   \rp }}  {x^{1/3} }   \sum_{k= 0}^{K-1} \frac {B^{\pm}_{k } }  {x^{  k/3  }}   +  O_{K,\, \boldsymbol{\mu}_{\vv}} \bigg( \frac 1 {x^{ (K + 1) / 3}} \bigg), 
	\end{align}
	with the coefficients $B_k^{\pm}$ depending only on $ \boldsymbol{\mu}_{\vv} $. 
	
	{\rm(2)} When $\vv$ is complex,  for $|z| \Gt_{K, \, \boldsymbol{\mu}_{\vv} } 1$, we have the asymptotic expansion 
	\begin{equation}\label{4eq: asymptotic, Bessel, C} 
			J_{\boldsymbol{\mu}_{\vv}} (z) = \sum_{ \xi^3 = 1} \frac { e \big( 3 \big(\xi z^{1/3} + \widebar \xi \widebar z{}^{1/3}\big) \big) } {|z|^{2/3} }    \underset{ k + l \shskip \leqslant K -1 }{\sum \sum}  \frac{ B_{k } B_{l} }{\xi^{k-l} z^{ k/3} {\widebar z}{}^{\shskip  l/3}}
			         + O_{K,\, \boldsymbol{\mu}_{\vv}} \bigg( \frac 1 { |z|^{ (K + 2) / 3} } \bigg) , 
	\end{equation}
	with the coefficients $B_k $ depending only on $ \boldsymbol{\mu}_{\vv} $. 
\end{lem}


\delete{Following is a generalization of Lemma 5.1 in \cite{Qi-Wilton} to the derivatives of $ J_{\boldsymbol{\mu}_{\vv}} (x) $.

\begin{lem}\label{prop: asymptotic J, x < 1}
Suppose that all the components of $\boldsymbol{\mu}_{\vv} $ are in the left half plane $\{ s : \Re (s ) < \sigma \}  $. 

{\rm (1)} When $\vv$ is real, for $|x| \Lt 1$, we have
\begin{align}
x^i (\nd/\nd x)^{i} J_{\boldsymbol{\mu}_{\vv}} (x) \Lt_{i, \shskip \boldsymbol{\mu}_{\vv} } 1/|x|^{\sigma}. 
\end{align}

{\rm (2)} When $\vv$ is complex, for $|z| \Lt 1$, we have
\begin{align}
z^i \widebar{z}^j (\partial/\partial z)^{i} (\partial/\partial \widebar z)^{j} J_{\boldsymbol{\mu}_{\vv}} (z) \Lt_{i, \shskip j, \shskip \boldsymbol{\mu}_{\vv} } 1/|z|^{2\sigma}. 
\end{align}
\end{lem}
}

\subsection*{The Self-Dual Assumption}

Subsequently, we shall assume that $\pi$ is a {\it self-dual} {spherical} automorphic cuspidal representation of $ \PGL_3 (\BA)$. In particular, we have  $A(\frn_1, \frn_2) = A (\frn_2, \frn_1)$($= \overline { A (\frn_1, \frn_2) }$) and $\boldsymbol{\mu} = ( \mu, 0, - \mu)$ with $\mu \in i Y \subset \bfra_{\BC}$ ($  \mu_{\vv} \in i \shskip \BR$ or $ \mu_{\vv} \in \lp - \frac 1 2 , \frac 1 2 \rp$ for every $\vv | {\infty}$; see \S \ref{sec: Archimedean} for the definitions). 

It is known by \cite{GJ-GL(2)-GL(3)} that $\pi$ comes from the symmetric square lift of a Hecke--Maass form for $\GL_2 $. Thus the Kim--Sarnak bound  for $\GL_2 $ in \cite{Blomer-Brumley} implies that 
\begin{align}\label{3eq: Kim-Sarnak, A(n), GL3}
|A  (\frn_1, \frn_2)| \leqslant \RN (\frn_1 \frn_2)^{7/32 +\vepsilon},
\end{align} 
and
\begin{align}\label{3eq: Kim-Sarnak, GL3}
|\Re  (\mu_{\vv}) | \leqslant \frac 7 {32}, \qquad  \vv | \infty  .
\end{align}

\section{Preliminaries on $L$-Functions}

Let $f$ be a {spherical} Hecke--Maass cusp form on $ \PGL_2 (F) \backslash \PGL_2 (\BA) $ with Hecke eigenvalues $\lambdaup_f (\frn)$ and Archimedean parameter  $ \varnu_f \in  Y$. Let  $E (s)$ be the  {spherical} Eisenstein series on  $ \PGL_2 (F) \backslash \PGL_2 (\BA) $. Let  $\pi$ be a {fixed} {self-dual} spherical automorphic cuspidal representation of $ \PGL_3 (\BA)$ with Fourier coefficients $A (\frn_1, \frn_2)$ and Archimedean parameter $(\mu, 0, - \mu)$ ($\mu \in i Y$). 

\subsection{$L$-functions $L (s, \pi  )$,  $ L (s, \pi \otimes f ) $, and $L  (  s, \pi \otimes E  (      i t  )   )$} 
The $L$-function attached to $\pi    $ is defined by
\begin{equation}
L (s,  \pi   ) = {\sum_{\frn \shskip \subset \shskip \frO }}  \frac {A(1, \frn)  } {\RN(\frn)^{  s} }.
\end{equation}
The Rankin--Selberg $L$-function $L (s, \pi \otimes f )$ is  defined by 
\begin{align}\label{4eq: L (pi f chi)}
L (s, \pi \otimes f )  
= \underset{\frn_1, \frn_2 \shskip \subset \shskip \frO }{\sum \sum} \frac { A (\frn_1, \frn_2) \lambdaup_f  (\frn_2)  } {\RN  ( \frn_1^2 \frn_2  )^{  s} }. 
\end{align} 
The completed $L$-function for $ \pi $ is
$\Lambda (s,  \pi   ) =  \RN(\frD)^{3s/2} \gamma (s, \pi  ) L (s,  \pi  ) $, where $\gamma (s, \pi) = \gamma (s)$ is the product of 
\begin{equation}\label{4eq: defn of gamma (s)}
\gamma_{\vv} (s) =  (N_{\vv} \pi )^{-3 N_{\vv} s / 2} \Gamma \bigg(\frac {N_{\vv}(s-\mu_{\vv})} 2 \bigg) \Gamma \bigg(\frac {N_{\vv}s} 2 \bigg) \Gamma \bigg(\frac {N_{\vv}(s+\mu_{\vv})} 2 \bigg) .
\end{equation}
Recall that $N_{\vv} = 1$ if $\vv$ is real and  $N_{\vv} = 2$ if $\vv$ is complex. 
It is known that $\Lambda (s,  \pi )$ is entire and  has the functional equation
\begin{align*}
\Lambda (s,  \pi  ) = \Lambda   (1-s,  \pi   ).
\end{align*} 
We define  $\gamma (s, \varnu )$ to be the product of 
\begin{align}\label{4eq: defn of gamma (s, nu)}
\gamma_{\vv} (s, \varnu_{ \vv}) = \gamma_{\vv} (s - i \varnu_{  \varv}) \gamma_{\vv} (s + i \varnu_{  \varv}).
\end{align} 
Let $\gamma (s, \pi \otimes f) = \gamma (s, \varnu_f)$. 
Then  $ \Lambda (s, \pi \otimes f ) =  \RN(\frD)^{3s } \gamma (s , \pi \otimes f ) L (s, \pi \otimes f )$ is also entire and satisfies the functional equation
\begin{align*}
\Lambda (s, \pi \otimes f ) = \Lambda (1 - s, \pi \otimes f ).
\end{align*}

Similar to \eqref{4eq: L (pi f chi)},  we define
\begin{align}
 L  (  s, \pi \otimes E  (    i t  )   ) = \underset{\frn_1, \frn_2 \shskip \subset \shskip \frO }{\sum \sum} \frac {   A (\frn_1, \frn_2) \tau_{it}  (\frn_2  )    } {\RN  ( \frn_1^2 \frn_2  )^{  s} },
\end{align}
where $ \tau_{s}  (\frn  ) $ is defined as in \eqref{1eq: defn of eta (m, s)}. 
We have
\begin{align}\label{1eq: L (f x E) = L(f) L(f)}
L  (  s, \pi \otimes E  (    i t  )   ) = L \lp s+it, \pi  \rp L \lp s-it, \pi   \rp,
\end{align}
and hence
\begin{align}\label{1eq: L (f x E) = |L (f)|2}
L \big( \tfrac 1 2 , \pi \otimes E  (  i t  )  \big) = \left| L \big( \tfrac 1 2 + it, \pi \big) \right|^2.
\end{align}

\subsection{Approximate Functional Equations for  $ L (s, \pi \shskip \otimes f ) $ and $L (s, \pi )$} \label{sec: afeq}

Following Blomer \cite{Blomer}, for a large even integer $A' > 0$, we introduce the polynomial $ p (s, \varnu )$ as the product of $p_{\vv} (s, \varnu_{\vv} )$ ($\vv| {\infty}$) defined by 
\begin{align}\label{4eq: defn p(s, t)}
  \prod_{k=0}^{N_{\vv}A'/2-1}  \big( (s+ 2k/N_{\vv} -\mu_{\vv})^2 + \varnu_{\vv}^2   \big) \big((s+2k/N_{\vv})^2 +  \varnu_{\vv}^2 \big) \big( (s+2k/N_{\vv}+\mu_{\vv})^2 + \varnu_{\vv}^2   \big)    
\end{align}
so that $p_{\vv}  (s, \varnu_{\vv} )$ annihilates  the rightmost $N_{\vv}A'/2$ many poles of each of  the   gamma factors in $ \gamma_{\vv}  (s, \varnu_{\vv} ) $  defined by (\ref{4eq: defn of gamma (s)}, \ref{4eq: defn of gamma (s, nu)}). This polynomial will eventually be used to overcome the obstacle posed by the presence of
an infinite group of units.

We have the following Approximate Functional Equation for   $L (s, \pi \otimes f )$  (see 
\cite[Theorem 5.3]{IK}),
\begin{equation}
\label{1eq: approximate functional equation, 1} 
L \big(\tfrac 1 2,   \pi \otimes f \big)   =  2 \underset{\frn_1, \frn_2 \shskip \subset \shskip \frO }{\sum \sum} \frac { A (\frn_1, \frn_2) \lambdaup_f  (\frn_2)  } {\RN  ( \frn_1^2 \frn_2  )^{1/2} }     V  \big(  \RN  \big( \frn_1^2 \frn_2 \frD^{-3} \big); \varnu_f  \big)   , 
\end{equation}
with 
\begin{equation}
\begin{split}
V (y; \varnu ) = \frac 1  {2 \pi i} \int_{(3)} 
G  (u, \varnu)       y^{ -  u} \frac { \nd \shskip u } {u}, \qquad y > 0,  
\end{split}
\end{equation}
in which $G  (u, \varnu)$ is the product of
\begin{align}\label{4eq: def G}
G_{\vv}  (u, \varnu_{\vv}) = \frac {\gamma_{\vv} \big(\frac 1 2 + u , \varnu_{\vv} \big)  }  {\gamma_{\vv} \big(\frac 1 2, \varnu_{\vv}   \big)   } \cdot   \frac { p_{\vv} \left(\frac 1 2 + u , \varnu_{\vv}  \right) p_{\vv} \left(\frac 1 2 - u , \varnu_{\vv}  \right) e^{N_{\vv} u^2 / N}  }  {  p_{\vv} \big(\frac 1 2, \varnu_{\vv}   \big)^2 }   .
\end{align} 
Note that the second quotient in \eqref{4eq: def G}   is even in $u$ and is equal to $1$ when $u = 0$.
Similarly, the Approximate Functional Equation for  $L  ( s, \pi \otimes E  (   i t  )    )$, along with  \eqref{1eq: L (f x E) = |L (f)|2}, yields 
\begin{align}
\label{1eq: approximate functional equation, 2} 
\left| L \big(  \tfrac 1 2 + it, \pi  \big) \right|^2  =  2 \underset{\frn_1, \frn_2 \shskip \subset \shskip \frO }{\sum \sum} \frac { A (\frn_1, \frn_2) \tau_{it}  (\frn_2  )   } {\RN  ( \frn_1^2 \frn_2  )^{1/2} }    V  \big(  \RN  \big( \frn_1^2 \frn_2 \frD^{-3} \big); t  \big) .
\end{align} 

Properties of $ V (y; \varnu ) $ and $ G  (u, \varnu) $ are contained in the following lemma (see \cite[Lemma 1]{Blomer} and \cite[Lemma 3.7]{Qi-Gauss}).

\begin{lem}\label{lem: afq}  
	Let  $ U  > 1 $,  $A > 0$, and $\vepsilon > 0$. Suppose that $\varnu \in Y$ satisfies the Kim--Sarnak bounds {\rm\eqref{2eq: Kim-Sarnak}}. Let $\RC (\varnu)$ 
	be defined as in {\rm\eqref{1eq: defn conductor}}.
	
	{\rm(1)}  We have
	\begin{align}\label{1eq: derivatives for V(y, t), 1}
	V (y; \varnu ) \Lt_{\,A, \shskip A' } 
	\bigg(  1 + \frac {y} {\RC (\varnu)^{3}} \bigg)^{-A} ,
	\end{align}  
	and 
	\begin{align}\label{1eq: approx of V}
	V  (y; \varnu) = \frac 1 {2   \pi i } \int_{ \vepsilon - i U}^{\vepsilon + i U}  G (u, \varnu)     y^{-  u}   \frac {\nd \shskip u} {u} + O_{\vepsilon, \shskip A'} \bigg( \frac {\RC (\varnu)^{3 \vepsilon} } {y^{ \vepsilon} e^{U^2 / 2} } \bigg) .
	\end{align}  
	
	{\rm(2)}  When $\Re (u) > 0$, the function $ G_{\vv} (u, \varnu_{\vv}) p_{\vv} \lp \tfrac 1 2, \varnu_{\vv} \rp^2$ is even and holomorphic on the region $ |\Im (\varnu_{\vv}) | \leqslant  A' + \frac {9} {32} = A' + \frac 1 2 - \frac {7} {32} $  {\rm(}see {\rm\eqref{3eq: Kim-Sarnak, GL3})}, and it satisfies in this region the uniform bound
	\begin{align}\label{1eq: bound for p G}
	G_{\vv} (u, \varnu_{\vv}) p_{\vv} \lp \tfrac 1 2, \varnu_{\vv} \rp^2  \Lt_{\shskip A', \shskip \Re(u)}     (1 + |\varnu_{\vv}|)^{ 3 N_{\vv}  ( \Re (u) + 2 A')}, 
	\end{align}
	and, more generally,
	\begin{align}\label{1eq: bound for p G, 2}
	 \frac { \partial^{j} } {\partial \varnu_{\vv}^j} \lp G_{\vv} (u, \varnu_{\vv}) p_{\vv} \lp \tfrac 1 2, \varnu_{\vv} \rp^2 \rp
	  \Lt_{\shskip j, \shskip A', \shskip \Re(u)}  (1 + |\varnu_{\vv}|)^{ 3 N_{\vv}  ( \Re (u) + 2 A' ) - j }   (1 + \Im (u))^{  j} . 
	\end{align}
\end{lem}

\begin{proof}
The estimate in	\eqref{1eq: derivatives for V(y, t), 1} may be found in 
	\cite[Proposition 5.4]{IK}. 
	The expression of $V (y; \varnu)$   in \eqref{1eq: approx of V} is essentially due to Blomer \cite[Lemma 1]{Blomer}.   The bounds in \eqref{1eq: bound for p G} and \eqref{1eq: bound for p G, 2} follow  readily from Stirling's formulae for  $\log \Gamma$ and its derivatives (see for example \cite[\S\S 1.1, 1.2]{MO-Formulas}).  
\end{proof}

\section{Choice of the Test Function}\label{sec: choice of h}

\begin{defn}
	Let $T, M \in \bfra_+ $ be such that $1 \Lt T_{\vv}^{\vepsilon} \leqslant M_{\vv} \leqslant T_{\vv}^{1-\vepsilon} $ for each $\vv | \infty$. Define the function $k (\varnu) = k_{T, \shskip M} (\varnu)$ to be the product of 
	\begin{align}\label{5eq: defn k(nu)}
	k_{\vv} (\varnu_{\vv}) = e^{- (\varnu_{\vv}  - T_{\vv})^2 / M_{\vv}^2} + e^{-(\varnu_{\vv} + T_{\vv})^2 / M_{\vv}^2} .
	\end{align}  
\end{defn}

We modify $k (\varnu)$ slightly, and let $k^{\snatural} (\varnu) = k_{T, \shskip M}^{\snatural} (\varnu)$ be the product of
\begin{align}\label{6eq: defn of k*}  
k^{\snatural}_{\vv}  (\varnu_{\vv}) = k_{\vv}  (\varnu_{\vv}) p_{\vv} \big(\tfrac 1 2, \varnu_{\vv} \big)^2 / T_{\vv}^{6  N_{\vv} A'},
\end{align}
with $p_{\vv}  \big(\tfrac 1 2, \varnu_{\vv} \big)$   defined as in  \eqref{4eq: defn p(s, t)}. 
For $\varnu \in \bfra$,
we have $
k^{\snatural} (\varnu) >  0  $,  and  $
k^{\snatural}  (\varnu )  \Gt 1 $ if $ |\varnu_{\varv} - T_{\vv}| \leqslant M_{\vv} $ for all $\vv | {\infty}$. Note also that $ k^{\snatural}_{\vv} (\varnu_{\vv}) = o \big(e^{- T_{\vv}^2/M_{\vv}^2} \big) $ if $ i \varnu_{\vv} \in \left( - \frac 1 2, \frac 1 2 \right)$. 

For $\Re (u) = \vepsilon$, the function
\begin{align}\label{5eq: defn h(nu)}
h (\varnu) = h_{T, \shskip M}  (\varnu) =  
G (u, \varnu)  k_{T, \shskip M}^{\snatural} (\varnu)
\end{align} 
lies in the space $ \mathscr{H} (S) $ as in Definition \ref{defn: test functions} with $S = A' + \frac {9} {32} $. It is clear that $h (\varnu)$ is the product of 
\begin{align}\label{6eq: defn of h, local}
h_{\vv} (\varnu_{\vv}) = k_{\vv}  (\varnu_{\vv}) G_{\vv} (u, \varnu_{\vv})  p_{\vv} \big(\tfrac 1 2, \varnu_{\vv} \big)^2 / T_{\vv}^{6  N_{\vv} A'} ,
\end{align} 
and it follows from  Lemma \ref{lem: afq} (2) that
\begin{align}\label{5eq: bound for h}
h_{\vv} (\varnu_{\vv}) \Lt_{\, A', \shskip \vepsilon}   (1 + |\varnu_{\vv}|)^{  \vepsilon} k_{\vv} (\varnu_{\vv}), \qquad \varnu_{\vv} \in \BR. 
\end{align}

Henceforth, in view of \eqref{1eq: approx of V}, we shall   assume that $\Re (u) = \vepsilon$ and $|\Im (u) | \leqslant \log \RN (T)$. 


\begin{appendices}

	\section{\texorpdfstring{Proof of the Ad\`elic Vorono\"i Summation for $\GL_3$}{Proof of the Ad\`elic Voronoi Summation for $\GL(3)$}} \label{app: adelic Voronoi, proof} 
	
	The purpose of this appendix is to prove the ad\`elic Vorono\"i summation for $\GL_3$  in Proposition \ref{prop: Voronoi, adelic} without the condition in \cite{Ichino-Templier} that $S$ contains the ramified places of $\psi$. 
	By directly modifying Theorem 1 in \cite{Ichino-Templier}, such a Vorono\"i summation for $\GL_2$ and $\GL_3$ was formulated in \cite[\S 3.1]{Qi-Wilton}, but it is not entirely correct in the $\GL_3$ case. A particular issue with the direct modification  is that the identity (2.4) in their proof as below
	\begin{align}\label{2eq: identity, unramified}
	\int_{F    _{\vv}^{n-2} } \widetilde{W}_{\oo \shskip \vv} (\begin{pmatrix}
	\gamma & \\ & 1_{n-1}
	\end{pmatrix} 
	\begin{pmatrix}
	1 & & \\ x & 1 & \\ & & 1
	\end{pmatrix}) \shskip \nd x = \widetilde{W}_{\oo \shskip \vv} \begin{pmatrix}
	\gamma & \\ & 1_{n-1}
	\end{pmatrix} 
	\end{align}
	is no longer valid if $n = 3$ and $\frD_{\vv} \neq \frO_{\vv}$ (though this is always true if $n = 2$ as the integration would be gone). To rectify this, one must replace the $\displaystyle {W}^S_{\oo} \begin{pmatrix}
	\gamma & \\ & 1_{2}
	\end{pmatrix} = {W}^S_{\oo} (a(1, \gamma)) $ on the left-hand side of (1.2) in \cite{Ichino-Templier} (for $n=3$) by $\displaystyle {W}^S_{\oo} \begin{pmatrix}
	\gamma / \delta & & \\ & 1/\delta   & \\ & & 1
	\end{pmatrix} = {W}^S_{\oo} (a(1/\delta , \gamma)) $ and thus resort to their Theorem 3  rather than Theorem 1. Nevertheless, the main results in \cite{Qi-Wilton} are not invalidated, for only several changes involving  $\frD$ or $\delta$ are needed. 
	

	\subsection*{Proof of Proposition \ref{prop: Voronoi, adelic}}
	We retain the ad\`elic notation in \S \ref{sec: Voronoi GL(3)}. 	  Let
	\begin{align*}
	&\varw' = \begin{pmatrix}
	1 & & \\ & & 1 \\ & 1 & 
	\end{pmatrix}, \qquad 
	\varw_2 = \begin{pmatrix}
	& 1 & \\ 1 & & \\ & & 1 
	\end{pmatrix}, \\
	&	n (x)  =  \begin{pmatrix}
	1 & x & \\ & 1 & \\ & & 1
	\end{pmatrix}, \quad \,  n^{\ssstyle -} (x)  = 
	\begin{pmatrix}
	1 &   & \\ x & 1 & \\ & & 1
	\end{pmatrix}. 
	\end{align*}
	By abuse of notation, we shall denote by $\gamma$, $\zeta$, $\delta$,  $\valpha $  their local components $\gamma_{\vv}$, $\zeta_{\vv}$, $\delta_{\vv}$, $\valpha_{  \vv}$  respectively. According to the proof of  Theorem 3 in \cite{Ichino-Templier}, the local integral at a place $\vv \notin S$ that we need to consider is 
	\begin{align*}
	\scalebox{1.05}{$I$}^{    \ssharp }_{\vv} (\gamma) = \int_{ F_{\vv} } \widetilde{W}_{\oo \shskip \vv} \big ( a (1, \gamma) n^{\ssstyle -} (x) \varw'  n^{\ssstyle -} (-\zeta) a ( 1 / \delta, \valpha )\- \big) \nd x,
	\end{align*}
	while at the  places in $S$  we have the   transform defined by \eqref{2eq: Ichino-Templier}, 
	\begin{align}\label{app: at v in S}
	\scalebox{1.05}{$I$}^{  \ssharp }_{\vv} (\gamma) = \widetilde \varww_{\vv} (\gamma);
	\end{align} see \cite[\S \S 2.7, 5.3]{Ichino-Templier}. Our goal is to prove that the sum over $\gamma \in F^{\times}$ of the product of $ \scalebox{1.05}{$I$}^{   \ssharp   }_{\vv} (\gamma) $ is equal to the right-hand side of \eqref{2eq: Voronoi, adelic}.

	
	For $\varv \notin S \cup R$, we have $|\zeta / \valpha |_{\vv} \leqslant 1$.  
	It follows   that $   a ( 1/ \delta, \valpha ) \varw' \cdot \varw' n^{\ssstyle -} (-\zeta) a ( 1 / \delta, \valpha )\-  = n^{\ssstyle -} (-\zeta / \valpha) \in \GL_3 (\frO_{\vv} ) $.  Thus
	\begin{align*}
	\scalebox{1.05}{$I$}^{  \ssharp }_{\vv} (\gamma) & = \int_{ F_{\vv} } \widetilde{W}_{\oo \shskip \vv}   ( \begin{pmatrix}
	\gamma & \\ & 1_2
	\end{pmatrix} 
	\begin{pmatrix}
	1 & & \\ x & 1 & \\ & & 1
	\end{pmatrix} 
	\begin{pmatrix}
	\delta /   \valpha   & & \\ & \hskip -1 pt 1 \hskip -1 pt & \\ & & \delta   
	\end{pmatrix}  ) \nd x \\ 
	& = \overbar \omega_{\vv} (\delta ) |    \valpha  / \delta|_{\vv}    \int_{ F_{\vv} } \widetilde{W}_{\oo \shskip \vv}   ( \begin{pmatrix}
	\gamma / \valpha   & & \\ &  1 / \delta & \\ & & 1
	\end{pmatrix}
	\begin{pmatrix}
	1 & & \\ x & 1 & \\ & & 1
	\end{pmatrix}   ) \nd x .
	\end{align*}
	For $|x|_{\vv} > 1$, we have the Iwasawa decomposition 
	\begin{align*}
	\begin{pmatrix}
	1 & & \\ x & 1 & \\ & & 1
	\end{pmatrix} & =   \begin{pmatrix}
	1 & 1/x  & \\   & 1 & \\ & & 1
	\end{pmatrix} \begin{pmatrix}
	1/x  &   & \\   & x & \\ & & 1
	\end{pmatrix} \begin{pmatrix}
	& -1 & \\ 1  & 1/x  & \\ & & 1
	\end{pmatrix}.
	\end{align*}
	Therefore the integrand above is equal to 
	\begin{align*}
	\overbar{\psi} (\gamma \delta /   \valpha  x) \widetilde{W}_{\oo \shskip \vv}    \begin{pmatrix}
	\gamma / \valpha  x & & \\ & x   / \delta & \\ & & 1
	\end{pmatrix} ,
	\end{align*}
	and it vanishes as $  |x / \delta |_{\vv} > 1 / |\delta|_{\vv} $. Then
	we infer that the integrand is nonzero  only if $x \in \frO_{\vv}$. Consequently,
	\begin{align}\label{app: Iv = , unramified}
	\scalebox{1.05}{$I$}^{  \ssharp }_{\vv} (\gamma) & = \frac { \overbar \omega_{\vv} (\delta  )   |\valpha  |_{\vv} } {  \sqrt{| \delta|_{\vv}}}   \cdot \widetilde{W}_{\oo \shskip \vv}  \begin{pmatrix}
	\gamma / \valpha   & & \\ &  1 / \delta & \\ & & 1
	\end{pmatrix}.   
	\end{align}
	Note that \eqref{app: Iv = , unramified} is reduced to \eqref{2eq: identity, unramified} if $\delta   = \valpha  = 1$. 
	
	Now let $\varv \in  R$ so that $| \zeta / \valpha  |_{\vv} > 1 $. It is the second case in the proof of  \cite[Theorem 3]{Ichino-Templier}. We adapt their computations for the $\GL_3$ case as follows. 
	
	We start with rewriting 
	\begin{align*}
	\scalebox{1.05}{$I$}^{  \ssharp }_{\vv} (\gamma) = |\gamma |_{\vv} \int_{ F_{\vv} } \widetilde{W}_{\oo \shskip \vv} \big (   n^{\ssstyle -}  (x) \varw'  a (1, \gamma)  n^{\ssstyle -} (-\zeta) a ( 1 / \delta, \valpha )\- \big) \nd x .
	\end{align*}
	The Iwasawa decomposition of $ n^{\ssstyle -} (-\zeta/\valpha ) $ yields
	\begin{align*}
	\scalebox{1.05}{$I$}^{  \ssharp }_{\vv} (\gamma) & = |\gamma |_{\vv} \int_{ F_{\vv} } \widetilde{W}_{\oo \shskip \vv} \big (   n^{\ssstyle -}  (x) \varw' a ( \delta   , \gamma/\valpha ) n (- \valpha /\zeta) a \big(   \zeta/\valpha , \valpha^2 / \zeta^2 \big) \big) \nd x \\
	& = |\gamma |_{\vv} \int_{ F_{\vv} } \widetilde{W}_{\oo \shskip \vv} \big (   n^{\ssstyle -}  (x)  \varw'   n (- \gamma / \zeta )  a \big(    \zeta \delta/   \valpha , \gamma \valpha  / \zeta^2 \big) \big) \nd x .
	\end{align*}
	By 
	\begin{align*}
	n^{\ssstyle -}  (x)  \varw'   n (- \gamma / \zeta ) \varw' = \begin{pmatrix}
	1 &     &  - \gamma / \zeta \\
	& 1 & - x \gamma / \zeta   \\
	&& 1
	\end{pmatrix} n^{\ssstyle -}  (x),
	\end{align*}
	we have
	\begin{align*}
	\scalebox{1.05}{$I$}^{ \ssharp }_{\vv} (\gamma) = |\gamma |_{\vv} \int_{ F_{\vv} } \overbar \psi_{\vv} (- x \gamma / \zeta) \widetilde{W}_{\oo \shskip \vv} \big (   n^{\ssstyle -}  (x)  \varw'   a \big(   \zeta \delta/   \valpha, \gamma \valpha  / \zeta^2 \big) \big) \nd x .
	\end{align*}
	Since
	\begin{align*}
	n^{\ssstyle -}  (x) \varw'  a \big( \zeta \delta/   \valpha    , \gamma \valpha  / \zeta^2 \big) & = \varw_2 n(x) \varw_2 \varw'  a \big( \zeta \delta/   \valpha   , \gamma \valpha  / \zeta^2 \big) \\
	& = \varw_2  \begin{pmatrix}
	1 & & \\
	& \gamma       \delta/ \zeta      & \\
	& & \zeta \delta/   \valpha  
	\end{pmatrix} n( x  \gamma  \delta/ \zeta     )  \varw_2 \varw',
	\end{align*} 
	and $ \varw_2 \varw' \in \GL_3 (\frO)$, we have
	\begin{align}\label{app: case 2}
	\scalebox{1.05}{$I$}^{ \ssharp }_{\vv} (\gamma) =  | \zeta   / \delta |_{\vv} \int_{ F^{\sstimes}_{\vv}  }   {\psi}_{\vv} (        x /   \delta ) \widetilde{W}_{\oo \shskip \vv} (\varw_2 \begin{pmatrix}
	1 & & \\ & \hskip -1pt  \gamma \delta/ \zeta    & \\ & & \hskip -1pt  \zeta  \delta /   \valpha 
	\end{pmatrix}  n (x)  ) \nd x.
	\end{align}

	To compute the Kloosterman integral  in \eqref{app: case 2}, we invoke the following  result adopted  from \cite[\S 6]{Ichino-Templier} in the $\GL_3$ setting (see in particular (6.3) and Corollary 6.7). 
	
	\begin{lem}\label{lem: Kloosterman}
		Let $\psi_{\vv}$ be unramified and $\psi_{\vv}'$ be trivial on $\frO_{\vv}$. Let $\widetilde{W}_{\vv}$ be a $\overbar \psi_{\vv}$-Whittaker function invariant under $\GL_3 (\frO_{\vv})$. Let $\nd x$ be the Haar measure self-dual with respect to $\psi_{\vv}$ {\rm(}the measure of $\frO_{\vv}$ is $1${\rm)}. Let $\beta, \zeta \in \Fx_{\vv}$.  
		Then 
		\begin{align}
		\int_{ F_{\vv} }   \psi_{\vv}' (x) \widetilde{W}_{\vv}  ( \varw_2  \begin{pmatrix}
		1 & & \\ &    \beta & \\ & &   \zeta    
		\end{pmatrix}  n (x)  ) \nd x =  \hskip -3 pt \mathop{ \mathop{\sum_{\nu \, \in F^{\sstimes}_{\vv} /\frOO^{\sstimes}_{\vv}} }_{1 \leqslant |\nu|_{\vv} \leqslant |\zeta|_{\vv}} }_{ |\beta|_{\vv} \leqslant | \nu |^2_{\vv} }  \hskip -3 pt \widetilde{W}_{\vv} \hskip - 1 pt \begin{pmatrix}
		\beta / \nu \hskip - 1 pt  &  & \\ & \nu  & \\ & & \zeta 
		\end{pmatrix}  \mathrm{Kl} (\beta;  \nu ; \psi_{\vv}, \psi_{\vv}'), 
		\end{align}
		with
		\begin{align*}
		\mathrm{Kl} (\beta; \nu ; \psi_{\vv}, \psi_{\vv}') =  \sum_{ x \shskip \in \shskip  \nu \shskip \frOO_{\,\varv}^{\sstimes} / \frO_\varv } 
		\psi_\varv'   (  x) \psi_{\vv} ( -   \beta /x  ).
		\end{align*}  
		
	\end{lem}
	
	It is required that $\psi_{ \vv }$ is {\it unramified} in Lemma \ref{lem: Kloosterman}. To remove this condition, we have to re-scale $\psi_{ \vv }$, $ \widetilde{W}_{\oo \shskip \vv} $ and $\nd x$ so that $\psi_{ \vv } (x) = \psi_{ \vv }^{    \snatural } (\delta x)$, $   \widetilde{W}_{\oo \shskip \vv} (g) = \widetilde{W}^{    \snatural }_{\oo \shskip \vv} (a(\delta, \delta) g) $ and $\nd x = \hskip -1.5 pt \sqrt{|\delta|_{\vv} } \nd^{    \snatural } x$.  
	Applying Lemma \ref{lem: Kloosterman},
	we may transform   \eqref{app: case 2} into 
	\begin{align}\label{app: Kloosterman, case 2}
	\scalebox{1.05}{$I$}^{  \ssharp }_{\vv} (\gamma) = \frac { | \zeta     |_{\vv} \overbar  \omega_{\vv} (\delta) } {\sqrt { |\delta |_{\vv} } }  	  \mathop{ \mathop{\sum_{\nu \shskip \in F^{\sstimes}_{\vv} /\frOO^{\sstimes}_{\vv}} }_{1 \leqslant |\nu|_{\vv} \leqslant |\zeta   / \valpha_1 \valpha_2 |_{\vv}} }_{ |\gamma / \zeta \valpha_1 |_{\vv} \leqslant | \nu / \delta |^2_{\vv} }   
	\widetilde{W}_{\oo \shskip \vv} \hskip -1 pt  \begin{pmatrix}
	\gamma  / \nu \zeta       &  & \\ & \hskip - 1 pt \nu / \delta  & \\ & & \hskip - 1 pt \zeta   /   \valpha  
	\end{pmatrix}  \mathrm{Kl}_{\vv} (1, - \gamma  / \zeta , \nu) ,
	\end{align}
	where $\mathrm{Kl}_{\vv} (\valpha, \beta; \nu )$ is the local Kloosterman sum defined in Definition \ref{defn: Kloosterman, local}.

	Finally, our proof is completed by combining \eqref{app: at v in S}, \eqref{app: Iv = , unramified} and \eqref{app: Kloosterman, case 2}. 
\end{appendices}

{\large \part{Analysis over Archimedean Fields}}

\renewcommand{\baselinestretch}{1.11}

In the following sections, we shall do analysis over a local Archimedean field $F_{\vv}$  ($\vv |\infty$).  For simplicity, we shall suppress $\vv$ from our notation. Accordingly,
$F$ will be an Archimedean local field, and $N = [F : \BR]$. Henceforth, $x$,
$y$ will always stand for real variables, while $z$, $u$ for complex variables; in the complex setting, we shall usually use the polar coordinates $z = x e^{i \phi}$ and $u = y e^{i \theta}$. 

\vskip 5pt

\section{Stationary Phase Lemmas}  

For later use, we collect here some useful stationary phase lemmas in one dimension or two dimensions. 

\subsection{The One-Dimensional Case}

Consider one-dimensional oscillatory integrals in the   form
\begin{align*}
\int_a^b e (f(x))  \varww (x) \nd x.
\end{align*}
In practice, the phase function $f (x) = f (x; \lambdaup, ...)$ usually contains some (real) parameters. It is convenient to transform  the phase into the form $\lambdaup f (x)$  by   change of  variables, but clearly this can not always be done. For instance, one may consider a phase of the form $ \lambdaup^{1/3} x^2  -   x^3$ or $ x - \lambdaup \log x $.  

Firstly, we record here Lemma A.1 in \cite{AHLQ-Bessel}, which is an improved version of Lemma {\rm 8.1} in {\rm\cite{BKY-Mass}}. 

\begin{lem}\label{lem: staionary phase, dim 1, 2}
	Let $\varww (x)$ be a smooth function  with support in $( a, b)$ and $f (x)$ be a real smooth function on  $[a, b]$. Suppose that there
	are   parameters $P, Q, R, S, Z  > 0$ such that
	\begin{align*}
	f^{(i)} (x) \Lt_{ \, i } Z / Q^{i}, \qquad \varww^{(j)} (x) \Lt_{ \, j } S / P^{j},
	\end{align*}
	for  $i \geqslant 2$ and $j \geqslant 0$, and
	\begin{align*}
	| f' (x) | \Gt R. 
	\end{align*}
	Then for any $A \geqslant 0$ we have
	\begin{align*}
	\int_a^b e (f(x)) \varww (x) \nd x \Lt_{ A} (b - a) S \bigg( \frac {Z} {R^2Q^2} + \frac 1 {R Q} + \frac 1 {R P} \bigg)^A .
	\end{align*}
\end{lem}

\delete{
\begin{rem}
	The estimate in Lemma {\rm\ref{lem: staionary phase, dim 1, 2}} may be simplified slightly as
	\begin{align*}
	\int_a^b e (f(x)) \varww (x) \nd x \Lt_{ A} (b - a) S \bigg( \frac {1 + \sqrt Y} {PQ} + \frac 1 {P U}  \bigg)^A ,
	\end{align*}
	and Lemma {\rm 8.1} in {\rm\cite{BKY-Mass}} is recovered when $Y \geqslant 1$. 
\end{rem}
}

Secondly, the second derivative test as follows is usually sufficient for our purpose---it is as strong as the stationary phase estimate in most of the cases. See \cite[Lemma 5.1.3]{Huxley}. 

\begin{lem}\label{lem: 2nd derivative test, dim 1}
	Let $f (x)$ be a real smooth function on  $(a, b)$ with $ f'' (x) \geqslant \lambdaup > 0 $. Let $\tw (x)$ be a real smooth function on $[  a, b]$, and let $V$ be its total variation plus its maximum modulus. 
	Then
	\begin{align*}
	\left|\int_a^b e (f(x)) \tw (x) \nd x \right| \leqslant \frac {4 V} {\sqrt {\pi \lambdaup}}. 
	\end{align*}
\end{lem}

Finally, when the phase is of the form $\lambdaup f (x)$, we record here  a generalization of 
the stationary phase estimate in \cite[Theorem 1.1.1]{Sogge} ($X = 1$ in \cite{Sogge}). See \cite[\S 2.4]{Qi-Gauss}. 

\begin{lem}\label{lem: stationary phase estimates, dim 1}
	Let $S >0$ and $ \sqrt {\lambdaup} \geqslant X \geqslant  1$.  	Let $\tw (x; \lambdaup)$ be a smooth function  with support in $( a, b )$ for all $\lambdaup$, and $f (x)$ be a real smooth function on $[a, b]$. Suppose that  $   \lambdaup^{j} \partial_x^{i} \partial_\lambdaup^{\shskip j}  \tw  (x; \lambdaup) \Lt_{  i, \shskip j } S X^{i + j}
	$ and that  $f(x_0) = f'(x_0) = 0$ at a point  $  x_0 \in (a, b)$, with $ f'' (x_0) \neq 0$ and $f' (x) \neq 0$ for all $x \in [a, b] \smallsetminus \{x_0\} $. Then
	\begin{align*}
	\frac {\nd^{\shskip j}} {\nd \shskip \lambdaup^{\shskip j}} \int_a^b e (\lambdaup f(x)) \tw (x; \lambdaup) \nd x \Lt_{  j} \frac {S X^j}  {  \lambdaup^{   1/2 +  j } }.
	\end{align*}
\end{lem}

We have deliberately avoided here to use Lemma 6.3 in \cite{Young-Cubic} (or the asymptotic expansion in Proposition 8.2 in \cite{BKY-Mass}) with arbitrary phase function, because it does not currently have a generalization in two dimensions. 

\subsection{The Two-Dimensional Case}

Next, we turn to two-dimensional oscillatory integrals of the form
\begin{align*}
\viint_D e (f(x, y   )) \tw (x, y   ) \nd x \nd y   .
\end{align*} 
Firstly, we have the two-dimensional generalization of Lemma \ref{lem: staionary phase, dim 1, 2} as follows.

\begin{lem}\label{lem: staionary phase, dim 2, 2}
Let $D \subset \BR^2$	be a finite domain. Let $\tw (x, y   )$ be a  smooth function  with support on $D$ and $f (x, y   )$ be a real smooth function on the closure $ \overline{D} $. Suppose that there
	are   parameters $P, Q,  \varUpsilon, \varPhi, R, S, Z > 0$ such that
	\begin{align}\label{7eq: bounds for f and w}
	(\partial/\partial x)^{i} (\partial/\partial y   )^{j}	f  (x, y   ) \Lt_{   i, \shskip j } Z / Q^{i} \varPhi^{\shskip j}, \quad (\partial/\partial x)^{k} (\partial/\partial y   )^{l} \tw (x, y   ) \Lt_{  k, \shskip l }  S / P^{k} \varUpsilon^{\shskip l},
	\end{align}
	for  $i, j, k, l \geqslant 0$ with $i + j \geqslant 2$, and
	\begin{align}\label{7eq: lower bound for g}
	| f' (x, y   ) |^2 =    (\partial f (x, y   )/\partial x  )^2 +  (\partial f (x, y   )/\partial y    )^2   \Gt  R^2. 
	\end{align}
	Then   
	\begin{align}
	\begin{aligned}
		\viint_D e (f(x, y   ))  \tw (x,  y   )  & \nd x \nd y  \Lt_{  \shskip A}   \mathrm{Area}(D) S \\
	 \cdot & \Bigg\{  \frac {1} {R } \bigg(   \frac 1 {P} +   \frac 1 {\varUpsilon  } + \frac 1 {Q} + \frac 1 {  \varPhi}  \bigg) \hskip -1pt + \frac {Z^2} {R^3} \bigg( \frac 1 {Q^3} + \frac 1 {\varPhi^3} \bigg) \Bigg\}^A
	\end{aligned}
	\end{align}
	for any  $A \geqslant 0$.
\end{lem}

\delete{
\begin{rem}
	The bound in Lemma {\rm\ref{lem: staionary phase, dim 2, 2}} may be simplified slightly as
	\begin{align*}
	\viint_D e (f(x, y   ))  \tw (x, y   ) \nd x \nd y & \Lt_{  \shskip A}  \mathrm{Area}(D) S \\
	\cdot &  \bigg\{  \frac {Z} {R^2} \bigg( (1+\sqrt{Z}) \bigg( \frac 1 {Q^2} + \frac 1 {\varPhi^2} \bigg) + \frac 1 {PQ} + \frac 1 {\varUpsilon \varPhi} \bigg) \hskip -1pt  \bigg\}^A. 
	\end{align*} 
\end{rem}
}

\begin{proof}
	We start with a useful simple lemma. 
	
	\begin{lem}\label{lem: derivatives of powers}
		Let $f (x, y)$ be a smooth function. Let $ i, j, n \geqslant 0$. Then  $\partial_{x}^{i } \partial_{y   }^{\shskip j }  (  f  ( x , y    ) ^n  )$ is a linear combination of products in the form
		\begin{align*}
			f ( x , y    )^{n \shskip - \sum k_{\nu   \mu} } \prod_{ \nu, \, \mu } \big(\partial_x^{\nu} \partial_{y   }^{ \shskip \mu} f ( x , y    ) \big)^{k_{\nu \mu} } , \hskip 15pt k_{00} = 0,  \hskip 5pt   \sum_{\nu, \, \mu} \nu k_{\nu \mu} = i , \hskip 5pt \sum_{\nu, \, \mu} \mu k_{\nu \mu} = j, \hskip 5pt \sum_{\nu, \, \mu}   k_{\nu \mu} \leqslant n. 
		\end{align*} 
	\end{lem}
	
	For brevity, we write $ g (x, y   ) = 	| f' (x, y   ) |^2 $. By Lemma \ref{lem: derivatives of powers}, along with \eqref{7eq: bounds for f and w} and the trivial inequalities $ |\partial_{x} f  ( x , y    )|, |\partial_{y} f  ( x , y    )| \leqslant \sqrt{g (x, y   )}  $, we infer that
	\begin{align}\label{7eq: derivatives of g}
	\partial_{x}^{i } \partial_{y   }^{\shskip j } g (x, y   ) \Lt  \bigg(  \bigg( \frac {Z}   {Q  } + \frac {Z}   {\varPhi } \bigg) {\textstyle \sqrt{g (x, y)}} + \bigg( \frac {Z^2}   {Q^{2 }  } + \frac {Z^2}   {\varPhi^2} \bigg) \bigg) \frac {1} {Q^{i} \varPhi^{\shskip j}} 
	\end{align} 
for $i+j \geqslant 1$. 
Our idea is to repeatedly apply H\"ormander's elaborate partial integration (see \cite[Theorem 7.7.1]{Hormander}) as follows.	Define  the differential operator
	\begin{align*}
	\mathrm{D}  =       \frac {   \partial_x f(x , y     )} {  g (x , y     ) } \frac {\partial} {\partial x}      +    \frac {  \partial_{y   } f(x , y     )} {   g (x , y     ) } \frac {\partial} {\partial y   }   
	\end{align*}
	so that $\mathrm{D}  (  e (    f  ( x , y    )    ) )   =  2 \pi i \cdot       e (    f  ( x , y    ) )$. Consequently, 
	\begin{align*}
	\mathrm{D}^* \hskip -2 pt = - \frac 1 {2\pi i} \bigg( \frac {\partial} {\partial x}   \frac {   \partial_x f  ( x , y    )} {   g (x , y     ) }  + \frac {\partial} {\partial y   }   \frac {  \partial_{y   } f (x , y     )} {  g (x , y     ) } \bigg)   
	\end{align*}
	is the adjoint of $   (1/ 2\pi i) \cdot \RD$ and
	\begin{align*}
	\int_a^b \int_c^d e (f(x, y   )) & \tw (x, y   ) \nd x \nd y   =  \int_a^b \int_c^d  e \lp     f  ( x , y    ) \rp \mathrm{D}^{* \shskip  n}     \tw     (x ,   y    )   \nd x  \nd y   .
	\end{align*}
	for any  integer $n \geqslant 0$. 
	By a straightforward inductive argument, it may be shown that $ \mathrm{D} ^{* \shskip  n}    \tw ( x , y    )  $ is a linear combination of all the terms occurring in  the product-rule expansions of
	\begin{align*}
	\partial_{x}^{i } \partial_{y   }^{\shskip  j }  \mbox{\larger[1]\text{${\big\{}$}}   (  \partial_{x} f  ( x , y    )  )^{i}   (   \partial_{y   } f ( x , y    ) )^{ j} g( x , y    )^{n}   \shskip  \tw  ( x , y    )  \mbox{\larger[1]\text{${\big\}}$}} / g( x , y    )^{ 2 n}, \quad  i +  j = n.
	\end{align*}
	Now let $ i_1, \shskip i_2 \leqslant i$ and $j_1, \shskip j_2 \leqslant  j$. It follows from Lemma \ref{lem: derivatives of powers}, along with \eqref{7eq: bounds for f and w}, \eqref{7eq: lower bound for g}, and  the trivial inequalities $ |\partial_{x} f  ( x , y    )|, |\partial_{y} f  ( x , y    )| \leqslant \sqrt{g (x, y   )}  $, that 
	\begin{align*} 
	\partial_{x}^{i_1 } \partial_{y   }^{\shskip  j_1 } \mbox{\larger[1]\text{${\big\{}$}} \hskip -2 pt \lp   \partial_{x} f ( x , y    )  \rp^{i} &   \lp  \partial_{y   } f ( x , y    ) \rp^{ j} \hskip -1 pt \mbox{\larger[1]\text{${\big\}}$}} \Lt  \bigg\{ 1 + \frac Z R \bigg( \frac {1}   {Q  } + \frac {1}   {\varPhi } \bigg)  \bigg\}^{i_1+j_1}  \frac {g(x, y)^{( i +  j )/2  } }  {Q^{i_1} \varPhi^{\shskip    j_1 } } .
	\end{align*}  
	Similarly, Lemma \ref{lem: derivatives of powers},  \eqref{7eq: lower bound for g}, and \eqref{7eq: derivatives of g} yield 
	\begin{align*}
	  { \partial_{x}^{i_2 } \partial_{y   }^{\shskip j_2 } \big(  g  ( x , y    ) ^n \big)}    \Lt    \bigg\{ 1 + \frac {Z^2 }   {R^2  } \bigg( \frac 1 {Q^2} + \frac 1 {\varPhi^2} \bigg) \bigg\}^{ i_2 +  j_2} \frac { g  ( x , y    )^n} {Q^{i_2} \varPhi^{\shskip j_2}}.
	\end{align*}
	Thus
	\begin{align*}
	\mathrm{D}^{* \shskip  n}     \tw     (x ,   y    ) & \Lt \frac S {R^n} \hskip -2 pt  \sum_{ i +  j = n} \hskip -2 pt \frac 1 { P^{i } \varUpsilon^{\shskip  j  } }  \hskip -2 pt \mathop{\sum_{\sstyle i_1 + i_2 \shskip\leqslant \shskip i}}_{ \sstyle \sstyle j_1 + j_2 \shskip\leqslant \shskip  j } \hskip -2 pt  \frac {P^{i_1 + i_2} \varUpsilon^{\shskip j_1 + j_2}} {Q^{i_1 + i_2} \varPhi^{\shskip j_1 + j_2}}  \bigg\{ 1 + \frac {Z  }   {R   } \bigg( \frac 1 {Q} + \frac 1 {\varPhi} \bigg) \bigg\}^{i_1 +j_1 +  2i_2 +  2j_2} \\
	& \Lt S \hskip -2pt \left\{  \frac {1} {R } \bigg(   \frac 1 {P} +   \frac 1 {\varUpsilon  } + \frac 1 {Q} + \frac 1 {  \varPhi}  \bigg) \hskip -1pt + \frac {Z^2} {R^3} \bigg( \frac 1 {Q^3} + \frac 1 {\varPhi^3} \bigg)  \right\}^n , 
	\end{align*}
	as desired. 
\end{proof}

Secondly,  we need a two-dimensional generalization of the second derivative test in Lemma \ref{lem: 2nd derivative test, dim 1}. A very useful version in the literature  is Lemma 4 in \cite{Munshi-Circle-III} (see also Lemma 5 in \cite{Srinivasan-Lattice-3}), in which it is assumed that
\begin{equation}\label{app: conditions on f''}
\begin{split}
& \left|\partial^2 f / \partial x^2 \right| \Gt \lambdaup > 0, \hskip 15pt  \left|\partial^2 f / \partial y   ^2 \right| \Gt \rho > 0, \\
& |\det f''|  = \left|\partial^2 f / \partial x^2 \cdot \partial^2 f / \partial y   ^2 - (  \partial^2 f / \partial x \partial y    )^2 \right| \Gt \lambdaup \shskip \rho,
\end{split}
\end{equation} 
on the integration domain $D = [ a, b] \times [c, d]$. 
However,  their bound $1 / \sqrt{\lambdaup \shskip \rho}$ would not be desirable if $ (  \partial^2 f / \partial x \partial y    )^2 $ is very large compared to $   \partial^2 f / \partial x^2  \cdot  \partial^2 f / \partial y   ^2  $ so that the former dominates in $\det f''$. This is because the choice of coordinates is not quite appropriate. In general, it   seems that some work is required to find the optimal coordinates.  Fortunately, in our application, we shall have $ \partial^2 f / \partial x^2 = - \partial^2 f / \partial y   ^2  $ (see \eqref{6eq: f'' = }) and the change of coordinates may be simply chosen to be
\begin{align*}
\sqrt 2 x = x'+y   ', \hskip 15pt \sqrt 2 y    = x' - y   '.
\end{align*}

As in \cite{Munshi-Circle-III}, we first suppose that $ \tw (x, y)  \equiv 1$. Let  $f (x, y   )$ be a real smooth function on the rectangle  $[ a, b] \times [c, d]$ such that 
\begin{align}\label{app: f'' = f''}
\partial^2 f / \partial x^2 = - \partial^2 f / \partial y   ^2,  
\end{align} 
with 
\begin{align}\label{app: bounds for f''}
\max \big\{\left| \partial^2 f / \partial x^2 \right| , \left| \partial^2 f / \partial x \partial y    \right| \big\} \Gt \lambdaup > 0 . 
\end{align}   We would like to prove 
\begin{align*}
\int_a^b \int_c^d e (f(x, y   ))  \nd x \nd y     \Lt \frac { 1 } {  {\, \lambdaup }},
\end{align*}
with an absolute implied constant. Note that $|\det f'' | \Gt \lambdaup^2 $, so this is in essence the expected stationary phase estimate. 

For $ \left| \partial^2 f / \partial x^2 \right|  \geqslant \left| \partial^2 f / \partial x \partial y    \right| $,  Lemma 4 in \cite{Srinivasan-Lattice-2}  (with $D =  [ a, b] \times [c, d]$) gives us the bound $1/ \lambdaup$ as expected. 
Now assume that  $ \left| \partial^2 f / \partial x^2 \right|  <  \left| \partial^2 f / \partial x \partial y    \right| $. Let $x', y'$ be as above. Then  $ \partial^2 f / \partial x'^2 = - \partial^2 f / \partial y   '^2 =   \partial^2 f / \partial x \partial y $ and $\partial^2 f / \partial x' \partial y' =   \partial^2 f / \partial x^2$. By applying Lemma 4 in \cite{Srinivasan-Lattice-2}  (with $D$ the rotated rectangle) again to the integral after the change of variables, we also obtain the  bound $ 1/ \lambdaup $. 

To extend the result to smooth $w (x, y) $ with support in $( a, b)  \times (c, d)$, we apply partial integration once in each variable.

\begin{lem}\label{lem: 2nd derivative test, dim 2}
	Suppose that $f $,  $\tw $, and $\lambdaup$  are as above satisfying {\rm\eqref{app: f'' = f''}} and {\rm\eqref{app: bounds for f''}}. Let 
	\begin{align*}
	V = \int_a^b \int_c^d \left|  \frac {\partial^2 \tw(x, y   )} {\partial x \partial y   } \right| \nd x \nd y   .
	\end{align*}
	Then 
	\begin{align*}
	\int_a^b \int_c^d e (f(x, y   )) \tw (x, y   )  \nd x \nd y \Lt \frac { V  } {  {\, \lambdaup }},
	\end{align*}
	with an absolute implied constant.
\end{lem}

Finally, we remark that the generalization of Lemma \ref{lem: stationary phase estimates, dim 1} in two (or higher) dimensions as in  \cite[Theorem 1.1.4]{Sogge} is not sufficient for our purpose because of the angular argument. We refer the reader to \cite[\S \S 2.4, 6.1]{Qi-Gauss} for  discussions in this regard.

\section{Analysis of Bessel Integrals}

Let $ B_{\varnu}(x) $  and $B_{\varnu} (z)$  be the Bessel kernels for $F = \BR$ and $F = \BC$ as in Definition \ref{defn: Bessel kernel}, respectively.  For $1 \Lt T^{\vepsilon} \leqslant M  \leqslant T^{1-\vepsilon} $, let $h (\varnu)$ 
be (a local component of) the test function as defined in \S \ref{sec: choice of h}. 
Let $\SDH (x) $ or $\SDH (z)$ be the corresponding Bessel integral
\begin{align}
\SDH (x) = \int_{-\infty}^{\infty} h (\varnu) B_{i \varnu} (x) \varnu  \tanh (\pi \varnu ) \nd \shskip \varnu, \quad \SDH (z) = \int_{-\infty}^{\infty} h (\varnu) B_{i \varnu} (z) \varnu^2  \nd \shskip \varnu ; 
\end{align} see \eqref{1eq: defn Bessel integral}. 

\subsection{Analytic Properties of Bessel Integrals}  

We collect here estimates and integral representations for the Bessel integrals. Our attempt is to have a unified presentation, so several results are not necessarily optimal. For the details, we refer the reader to \cite{Qi-Liu-LLZ} (and also \cite{XLi2011,Young-Cubic} for the real case).  


\begin{lem}\label{lem: crude estimates, > 1} We have estimates for Bessel integrals of small argument as follows. 
	
{\rm(1)} When $F$ is real, for $|x| \leqslant 1$ we have 
\begin{equation}\label{7eq: crude bound for H, R}
\SDH (x) \Lt_{A', \shskip \vepsilon}   M |x|^{1/2} / T^{2 A'-1} . 
\end{equation}

{\rm(2)}	When  $F $ is complex, for $|z| \leqslant 1$  we have  
\begin{equation}\label{7eq: crude bound for H, C}
\SDH (z) \Lt_{A', \shskip \vepsilon} M |z| / T^{4 A'-2} .
\end{equation}  

\end{lem}

\begin{proof}
These estimates may be derived from modifying the proofs of Lemma 3.2, A.4, and A.6 in \cite{Qi-Liu-LLZ} by shifting the integral contour far right to $\Im (\varnu) = A' +   \vepsilon$.\footnote{In the notation of \cite{Qi-Liu-LLZ}, $t = \varnu $, $H_{T, \shskip M} (\hskip -1 pt \sqrt{z}) = \SDH (z)$, and $H^{\ssstyle \pm}_{T, \shskip M} (  \hskip -1 pt \sqrt{x}) = \SDH (\pm x)$.} In view of Lemma \ref{lem: afq} (2) and \eqref{6eq: defn of h, local}, the test function $h (\varnu)$ is holomorphic for $ |\Im (\varnu ) | \leqslant  A' + \frac {9} {32}   $, and, along with the bound
\begin{align*}
	\left| J_{\varnu} (z) \right| \Lt \frac {   \left|   z  ^{\varnu} \right|  } {\Gamma \big(\varnu + \tfrac 1 2 \big)},  \hskip 20 pt |z| \leqslant 4 \pi,
\end{align*} one may estimate the   residues and the integral after the contour shift. To be explicit, one has
\begin{align*}
	\SDH (x) \Lt  e^{-M^2/T^2} |x|^{1/2} \sum_{k=0}^{A'-1} |x|^{k} + M T^{1+\vepsilon} \big(   {|x|^{1/2}} / {T}  \big)^{2 A' + 2 \vepsilon}  \Lt     {M |x|^{1/2}} / { T^{2 A'-1} }, 
\end{align*}
and
\begin{align*}
	\SDH (z) \Lt  e^{-M^2/T^2} |z| \sum_{k=0}^{A'-1} |z|^{k} + M T^{2+\vepsilon} \big(   {|z|} / {T^2}  \big)^{2 A' + 2 \vepsilon}  \Lt     {M |z|} / { T^{4 A'-2} }.
\end{align*} 
\end{proof}


\begin{lem}\label{lem: H(x), |z|>1}
	
	There exists a Schwartz function $ g (r)$ satisfying $g^{(j)} (r) \Lt_{j, \shskip A, \shskip A', \shskip \vepsilon} (1 + |r| )^{-A}$ for any $j, A \geqslant 0$, and such that
	
	{\rm(1)} if $F $ is real,  then $ \SDH (x) = \SDH_{  +}^{ \shskip \ssharp }  (x) + \SDH_{  -}^{\ssharp } (x) + O   (T^{-A} ) $ for $|x| >  1$, with 
	\begin{equation}\label{8eq: H+natural}
	\SDH_{  \pm}^{\shskip \ssharp }  (  x^2) =   MT^{1+\vepsilon}  
	\int_{- M^{\vepsilon} / M}^{M^{\vepsilon}/ M}   g (    {   M r} )  e( Tr / \pi \mp 2 x \cosh r  ) \nd r,
	\end{equation}
	and 
	\begin{equation}\label{8eq: H-natural}
	\SDH_{  \pm}^{\shskip  \ssharp }  (- x^2) =   MT^{1+\vepsilon}  
	\int_{- M^{\vepsilon} / M}^{M^{\vepsilon}/ M}   g (    {   M r} )  e( Tr / \pi \pm 2 x \sinh r  ) \nd r,
	\end{equation}
	for $x > 1${\rm;}
	
	{\rm(2)} if $F $ is complex, then $ \SDH (z) = \SDH_{  +}^{ \shskip \ssharp }  (z) + \SDH_{  -}^{ \shskip \ssharp } (z) + O   (T^{-A} ) $ for $|z| > 1$, with 
	\begin{align}\label{8eq: H-sharp(z)}
\SDH_{\ssstyle \pm }^{ \shskip \ssharp}   (  x^2 e^{2i\phi} ) =  M T^{2+\vepsilon} \int_0^{ \pi}   \hskip -1 pt
	\int_{- M^{\vepsilon} / M}^{M^{\vepsilon}/ M}   g  ( M r )  
	e  (2 T r/ \pi \mp 4 x \hskip 1pt \trh ( r, \omega; \phi)      )    \nd r \shskip \nd \omega,
	\end{align} 
	for $x > 1$, 
	where $\trh  (r, \omega; \phi)$ is the ``trigonometric-hyperbolic" function defined by 
	\begin{align}\label{8eq: trh function}
	\trh  (r, \omega; \phi) =    \cosh r \cos \omega \cos \phi - \sinh r \sin \omega \sin \phi.
	\end{align} 
	
Furthermore, 

	{\rm (3)} for real $x$ with $1 < |x| \Lt T^2 $, we have $ \SDH (x) = O  (T^{-A})${\rm;}

{\rm (4)} for complex $z$ with $1 < |z| \Lt T^2 $, we have $ \SDH (z) = O  (T^{-A})$. 
\end{lem}

%


\begin{proof}
	See (3.2), (3.3), (A.16), and (A.21) in \cite{Qi-Liu-LLZ}\footnote{Strictly speaking, the test function in \cite{Qi-Liu-LLZ} is like the $k (\varnu)$ in \eqref{5eq: defn k(nu)}, while our test function $h (\varnu)$  has extra factors (see \eqref{6eq: defn of h, local}). However, these factors would not play an essential role.} for the integral representations in (1) and (2).   The important point is that the Fourier transform of $g (\pi r/ N)$ is equal to $ e^{- \varnu^2} $ (for real $\varnu$) up to a harmless factor; although the factor involves $T$ and $M$, one may easily verify that it is bounded by  $ \Lt_{ \shskip A', \shskip \vepsilon} (1 + |  \varnu  |)^{N (6A'+1) + \vepsilon}  $  with the implied constant independent on  $T$ and $M$, and so are its derivatives (see \eqref{1eq: bound for p G} and \eqref{1eq: bound for p G, 2}). 
	
	The statements in (3) and (4) follow from simple applications of (one-dimensional) stationary phase  to the integrals in (1) and (2); see Lemma 3.5, A.5, and A.8 in \cite{Qi-Liu-LLZ}. 
	
	For the real case, we also refer to \cite[\S \S 4, 5]{XLi2011} and \cite[\S 7]{Young-Cubic}. 
\end{proof}

\begin{rem}\label{rem: real, Bessel range}
	In the real case, it is easy to prove that $\SDH (x^2)$ or  $\SDH (- x^2)$ is negligibly small unless $ x \Gt T M^{1-\vepsilon}  $   or  $ x \sasymp T$, respectively. See {\rm\cite{XLi2011}}.
\end{rem}


\begin{cor}\label{cor: bound for H < T}
	We have uniform estimates for   Bessel integrals as follows.

	{\rm(1)} When $F$ is real,  we have 
	\begin{equation}\label{7eq: bound for H, R}
	\SDH (x) \Lt_{A', \shskip \vepsilon} \left\{ 
	\begin{aligned}
	&  T^{1+ \vepsilon}  , & & \text{ if } |x| \Gt T^2, \\
	& M |x|^{1/2} / T^{2 A'-1} , & & \text{ if } |x| \Lt T^2.
	\end{aligned}  \right.
	\end{equation}

	{\rm(2)}	When  $F $ is complex,  we have  
	\begin{equation}\label{7eq: bound for H, C}
	\SDH (z) \Lt_{A', \shskip \vepsilon} \left\{ 
	\begin{aligned}
	&  T^{2 + \vepsilon}   , & & \quad \text{ if } |z| \Gt T^2, \\
	& M |z| / T^{4 A'-2} , & & \quad \text{ if } |z| \Lt T^2.
	\end{aligned}  \right.
	\end{equation}  
\end{cor}

\subsection{Preliminary Analysis of the Trigonometric-Hyperbolic Function} Consider the trigonometric-hyperbolic function $\trh  (r, \omega; \phi)$ as in \eqref{8eq: trh function}. 
Since $\trh (r, \omega; \phi + \pi) = \trh  ( r, \omega + \pi; \phi) = - \shskip \trh  (r, \omega; \phi)$,  one may be restricted to $ \phi, \omega \in [0, \pi) $.  
When $(r, \omega) \neq (0,     \pi / 2)$,   $\trh (r, \omega; \phi)$ can be written in a unique way as
\begin{align}\label{3eq: trh polar}
\trh (r, \omega; \phi) =  \rho (r, \omega) \cos (\phi + \theta (r, \omega) ),
\end{align}
where $ \rho (r, \omega) > 0 $ is defined by
\begin{align}\label{3eq: rho (r, w)}
 \rho (r, \omega) =  \sqrt {\sinh^2 r + \cos^2 \omega} =   \sqrt {\cosh^2 r - \sin^2 \omega} =     \sqrt { \frac {\cosh 2 r + \cos 2 \omega} 2 }    ,
\end{align}
and  $\theta (r, \omega)$ is determined by
\begin{align}\label{3eq: theta (r, w)}
\cos \theta (r, \omega) = \frac {  \cosh r \cos \omega} { \rho (r, \omega)  }, \qquad   \sin \theta (r, \omega) = \frac {  \sinh r \sin \omega} { \rho (r, \omega) } .
\end{align}
By defining $\trh (r, \omega) = \rho (r, \omega) e^{i \theta (r, \shskip \omega)}$, the function $x \hskip 1pt  \trh ( r, \omega; \phi) $ in \eqref{8eq: H-sharp(z)} is $ \Re (z \hskip 1pt  \trh (r, \omega) ) $ for $z = x e^{i \phi}$. 

\begin{rem}
	Since
	\begin{equation*}
	\begin{split}
	  \trh (r, 0 ; 0) =  \trh (r, 0) = \cosh r   , \quad \trh ( r, \pi/2; \pi/2) = i \hskip 1pt  \trh ( r, \pi/2) = -  \sinh r  ,
	\end{split}
	\end{equation*}
	the reader should observe the resemblance between the $r$-integral in {\rm\eqref{8eq: H-sharp(z)}} for $\phi = \omega = 0$ or $ \phi = \omega =   \pi / 2 $ and the integral in {\rm\eqref{8eq: H+natural}} or {\rm\eqref{8eq: H-natural}} respectively. 
\end{rem}

\begin{lem}\label{lem: estimates for rho and theta}
	Suppose that  $(r, \omega) \neq (0,   \pi / 2)$ and $|r| < 1$.
	
	{\rm(1)} We have 
	\begin{align*}
	{ \partial \theta (r, \omega) } / { \partial r } =    {\sin 2 \omega} / {2 \rho (r, \omega)^2} ,\quad 
	{ \partial \theta (r, \omega) } / { \partial \omega } =    {\sinh 2 r} / {2 \rho (r, \omega)^2}. 
	\end{align*}
	
	{\rm(2)} We have  $$
   \frac {\partial^{i+j}} {\partial r^{\shskip i} \partial \omega^j } \bigg( \frac 1 {\rho (r, \omega)^2} \bigg) \Lt_{\shskip i, \shskip  j } \frac 1 { \rho (r, \omega)^{i+j + 2} }. $$
	
	{\rm(3)} Consequently, for $i +  j \geqslant 1$, we have $$ \frac {\partial^{i+j} \theta (r, \omega)} {\partial r^{\shskip i} \partial \omega^j }  \Lt_{\shskip i, \shskip  j } \frac 1 {\rho (r, \omega)^{i+j } } . $$
\end{lem}

\begin{proof}
	By \eqref{3eq: theta (r, w)}, we have $
	\tan \theta (r, \omega) = \tanh r \tan \omega, $
	so
	\begin{align*}
	\frac { \partial \theta (r, \omega) } { \partial r } 
	= \frac { \sin \omega \cos \omega } { \cosh^2 r \cos^2 \omega + \sinh^2 r \sin^2 \omega } = \frac {\sin 2 \omega} {\cosh 2 r + \cos 2 \omega},
	\end{align*}
	and similarly
	\begin{align*}
	\frac { \partial \theta (r, \omega) } { \partial \omega } =  \frac { \sinh r \cosh r } { \cosh^2 r \cos^2 \omega + \sinh^2 r \sin^2 \omega } = \frac {\sinh 2 r} {\cosh 2 r + \cos 2 \omega}.
	\end{align*}
	
	The estimates for $ \rho (r, \omega) $ in (2) readily follow from an inductive argument by using the identity obtained from applying the $i$-th $r$-derivative and the $ j$-th $\omega$-derivative to 
	\begin{align*}
	\frac { \cosh 2r + \cos 2 \omega} {  \rho (r, \omega)^2 } = 2,
	\end{align*} 
	along with the following  inequalities
	\begin{align*}
	&   {\sinh 2 r}, \, \sin  2 \omega \Lt  { \rho (r, \omega) }, \qquad    {\cosh 2 r}, \, \cos 2 \omega \Lt   1, 
	\end{align*}
	where  the expression $  \rho (r, \omega) =  \sqrt{ \sinh^2 r +   \cos^2 \omega} $ is used.
	
	Finally, combining the foregoing results, it is straightforward to bound the derivatives of $\theta (r, \omega)$ as in (3).
\end{proof}

\section{Remarks on Xiaoqing Li's Analysis} \label{sec: XQLi}

We briefly remind the reader here several aspects of Xiaoqing Li's  analysis in \cite{XLi2011}, and explain the issues for its generalization to the complex setting or the case when the number field has multiple infinite places. 

In the real setting of \cite{XLi2011}, $T^{3/8+\vepsilon} \leqslant M \leqslant T^{1/2}$ and $ T M^{1-\vepsilon} \Lt  x  \leqslant T^{3/2+\vepsilon} \leqslant M^4 $. By expanding $ \cosh r$ into Taylor series, and disregarding the non-oscillatory factors from the terms of order $\geqslant 4$, the integral $\SDH_{  \pm}^{\shskip \ssharp }  (  x^2)$ in \eqref{8eq: H+natural} essentially turns into 
\begin{align*}
MT^{1+\vepsilon}  
e(   \mp 2 x    )  \int_{- M^{\vepsilon} / M}^{M^{\vepsilon}/ M}   g (    {   M r} )  e( Tr / \pi \mp   x r^2   ) \nd r.
\end{align*} 
Xiaoqing Li's next step is to complete the square, getting 
\begin{align*}
MT^{1+\vepsilon}  
e \bigg(   \mp  2 x \pm \frac {T^2} {4\pi^2 x}  \bigg)  \int_{- M^{\vepsilon} / M}^{M^{\vepsilon}/ M}   g (    {   M r} )  e \bigg( \mp x \Big(r \pm \frac {T}  {2\pi x} \Big)^2   \bigg) \nd r;
\end{align*} 
by Parseval, the integral is seen to be a non-oscillatory function of $x$. The secondary exponential factor $e ( \pm T^2 / 4 \pi^2 x )$ plays an important role in her second application of Vorono\"i summation. 

In the complex setting, however, the corresponding conditions are  $T^{3/4+\vepsilon} \leqslant M \leqslant T $ and $ T  \Lt  x  \leqslant T^{3/2 +\vepsilon} \leqslant M^2 $. After expanding $\cosh r $ and $\sinh r$ into Taylor series, only the factors of order $0$ and $1$ are oscillatory, and the integral $\SDH_{  \pm }^{ \shskip \ssharp}   (  x^2 e^{2i\phi} )$ in \eqref{8eq: H-sharp(z)} is essentially 
\begin{align*} 
 M T^{2+\vepsilon} \int_0^{ \pi}   \hskip -1 pt
\int_{- M^{\vepsilon} / M}^{M^{\vepsilon}/ M}   g  ( M r )  
e  (2 T r/ \pi \mp 4 x (  \cos \omega \cos \phi -   r \sin \omega \sin \phi)    )  \, \nd r \shskip \nd \omega.
\end{align*}  
As such, we are unable to produce a secondary exponential factor. Even if there were such a factor, the analysis would be conceivably difficult, because $\cos \omega \cos \phi$ would come down with $x$ in the denominator.

Moreover, when the number field has more than one infinite place, a more serious issue is that the condition $x  \leqslant T^{3/2+\vepsilon}$ is not necessarily valid for every infinite place. 

At any rate, it is   better not to expand $\cosh r$ or $\sinh r$ into Taylor series at this stage, and to allow $M$ be a small power of $T$.

\section{Stationary Phase for the Hankel Transforms} \label{sec: stationary phase, Hankel}

In this section, we consider certain integrals that will arise from the Hankel transforms over $\BR$ and $\BC$.  For the real case, it is simply a matter of applying the method of stationary phase in one dimension. For the complex case, the double integral has  already been investigated in \cite[\S 6.1]{Qi-Gauss}, but there are certain difficulties in two dimensions---Lemma \ref{lem: staionary phase, dim 2, 2} is not applicable for the particular phase function, and Theorem 1.1.4 in \cite{Sogge} is not sufficient as we need to also differentiate the angular argument.

\subsection{The One-Dimensional Case} 

First, in the real setting, we need to consider the integral 
\begin{align}\label{9eq: defn I, R}
I (\lambdaup) = \int_{-\infty}^{\infty}  e \big(\lambdaup \big(3x^2- 2 x^3\big) \big) \varww (x; \lambdaup ) \nd x. 
\end{align}
Fix $\varDelta > 1$. Let $\rho, S > 0$ and $X \geqslant 1$.  Suppose that $\varww  (   x ; \lambdaup   )$ is supported in $\big\{ x : |x| \in  [ \rho  ,   \varDelta^{1/6}  \rho    ] \big\}$ and its derivatives satisfy 
\begin{align*}
x^{i} \lambdaup^{j} \partial_x^{i}  \partial_{\lambdaup}^{j}   \varww  (x  ; \lambdaup)   \Lt_{\, i, \shskip  j }        S X^{i+j}.
\end{align*} 
Define
\begin{align}\label{9eq: tilde I, R}
I^{\snatural} (\lambdaup) = e (  -   \lambdaup   ) I (\lambdaup ).
\end{align}

\begin{lem}\label{lem: I, R}  Let $A$ and $j$ be non-negative integers. 
	
	{\rm(1)}	For either $ \rho \geqslant \sqrt \varDelta  $ or $\rho \leqslant 1 / \sqrt \varDelta $, we have
	\begin{align*}
	I^{\snatural} (\lambdaup) \Lt_{  \shskip A}   S \rho \bigg(\frac { X  } {  |\lambdaup| \shskip \rho^2 (\rho+1) } \bigg)^A  . 
	\end{align*}

	{\rm(2)} Assume that $  X \leqslant \sqrt{|\lambdaup|}$.	For 
	$ 1 / \varDelta  \leqslant \rho \leqslant \varDelta $, we have
	\begin{align*}
	\lambdaup^{j} \frac {\nd^{j}} {\nd \lambdaup^{j} }   I^{\snatural} (\lambdaup) \Lt_{ j } \frac {S X^j } {\sqrt {|\lambdaup|}} . 
	\end{align*}
\end{lem}

\begin{proof}
	Note that the phase function $3x^2 - 2 x^3$ has a unique nonzero stationary point at $x_0= 1$. The estimates in (1) readily follow from Lemma \ref{lem: staionary phase, dim 1, 2}; in the case $\rho \geqslant \sqrt{\varDelta}$,   choose $P = \rho/X$,  $Q =  \rho$,  $Z = |\lambdaup| \rho^3$, $R = |\lambdaup | \rho^2 $, and, in the case $\rho \leqslant 1/ \sqrt{\varDelta}$,   choose $P = \rho/X$,  $Q =  1$, $Z = |\lambdaup|  $, $R = |\lambdaup | \rho $. The estimates in (2) essentially follow from Lemma  \ref{lem: stationary phase estimates, dim 1}. 
\end{proof}

\subsection{The Two-Dimensional Case} \label{sec: stationary phase, Hankel, C}

Second, consider the following double integral that will arise from the complex Hankel transform,
\begin{align}\label{9eq: defn I, C}
I  (\lambdaup, \psi ) = \int_{0}^{2 \pi}   \int_0^\infty    e (  2 \lambdaup f (x, \phi; \psi ) )    \varww  (   x ,   \phi; \lambdaup, \psi  )  \nd x \shskip \nd \phi, 
\end{align}
with  
\begin{align}\label{4eq: phase f}
f (x, \phi; \psi ) =   3 x^2  \cos (  2 \phi +      \psi   ) - 2  x^3  \cos   3  \phi   .
\end{align}
Fix $\varDelta > 1$. Let $\rho, S > 0$ and $X \geqslant 1$. Suppose that $\varww  (   x ,   \phi; \lambdaup, \psi    )$ is supported in $\big\{ (x, \phi) : x \in  [ \rho  ,   \varDelta^{1/6}  \rho    ] \big\}$ and its derivatives satisfy 
\begin{align*}
x^{i} \lambdaup^{k} \partial_x^{i} \partial_\phi^{\shskip j} \partial_{\lambdaup}^{k} \partial_{\psi}^{l} \varww  (x ,   \phi; \lambdaup, \psi)   \Lt_{\, i, \shskip  j, \shskip k, \shskip l  }        S X^{i+j+k+l}.
\end{align*}
Define
\begin{align}\label{9eq: tilde I, C}
 I^{\snatural} (\lambdaup, \psi) = e (  - 2   \lambdaup  \cos 3 \psi) I (\lambdaup, \psi).
\end{align}
Results from Lemma 6.1 and 6.3 in \cite{Qi-Gauss} are quoted in the following lemma
with slightly altered notation.

\begin{lem}\label{lem: I, C}   Let $A$, $k$, $l$ be non-negative integers. 
	
	{\rm(1)}	For either $ \rho \geqslant \sqrt \varDelta  $ or $\rho \leqslant 1 / \sqrt \varDelta $, we have
	\begin{align*}
	  I^{\snatural} (\lambdaup, \psi) \Lt_{  \shskip A}  S \rho \bigg(  \frac { X  } { \lambdaup \shskip \rho^2 (\rho+1)  } \bigg)^A . 
	\end{align*}

	{\rm(2)} Assume that $  X \leqslant \sqrt{\lambdaup}$.	For 
	$ 1 / \varDelta  \leqslant \rho \leqslant \varDelta $, we have
	\begin{align*}
	\lambdaup^{k} \frac {\partial^{k + l}} {\partial \lambdaup^{k} \partial \psi^{l } }   I^{\snatural} (\lambdaup, \psi) \Lt_{ k, \shskip l } \frac {S X^{k+l}} {\lambdaup} . 
	\end{align*}
\end{lem}

\section{Analysis of the Hankel Transforms, I}\label{sec: Hankel 1}

Let $\varww  (x) $ be a smooth function supported on $[1 , \varDelta ]$ satisfying $ \varww^{(i)} (x)  \Lt_{i} \log^{i} T $ for all $i \geqslant 0$. For  $|\varLambda| \Gt T^2$, define 
\begin{align}\label{11eq: defn of w (x, Lmabda), R}
\varww (x, \varLambda ) = \varww (|x|) \SDH  ( \varLambda x  ) ,
\end{align}
if $F$ is real, and
\begin{align}\label{11eq: defn of w (z, Lmabda), C}
\varww (z, \varLambda ) = \varww (|z|) \SDH  ( \varLambda  z ) ,
\end{align}
if $F$ is complex. Let $\widetilde {\varww} (y , \varLambda )$ and $\widetilde {\varww} ( u , \varLambda )$ be their Hankel transforms (see Definition \ref{defn: Hankel transform}) defined by
\begin{align}\label{11eq: Hankel}
 \widetilde {\varww} (y , \varLambda ) = \int {\varww} (x , \varLambda ) J_{\pi} (xy) \nd x, \quad \widetilde {\varww} (u , \varLambda ) = \viint {\varww} (z , \varLambda ) J_{\pi} (z u) \nd z, 
\end{align}
and modify $\widetilde {\varww} (y , \varLambda )$ and $\widetilde {\varww} (u , \varLambda )$ by exponential factors as follows,  
\begin{align}\label{10eq: defn of w nat}
\widetilde {\varww}^{\snatural} (y , \varLambda ) = e (- y / \varLambda) \widetilde {\varww}  (y , \varLambda ), \quad \widetilde {\varww}^{\snatural} (u , \varLambda ) = e (- 2 \Re (u / \varLambda) ) \widetilde {\varww}  (u , \varLambda ). 
\end{align}

Roughly speaking, our wish is to transform  $ \widetilde {\varww}^{\snatural} (y , \varLambda ) $ and $ \widetilde {\varww}^{\snatural} (u , \varLambda ) $ into the shape
\begin{align*}
	  \frac {M T^{1+\vepsilon}} {\sqrt{|y|} }   \Phi^{\sigmaup} (y / \varLambda), \qquad    \frac {M T^{2+\vepsilon}}  {|u|}  \Phi^{\sigmaup} (u / \varLambda) ,
\end{align*} 
with $\sigmaup = \oldstylenums{0}, -, +, \flat$ in various circumstances. It turns out that the analytic properties of $ \Phi^{\sigmaup} (x) $ or $\Phi^{\sigmaup} (z)$ depend only mildly on  $\varLambda$ and $M$, so, for brevity, this dependence will be suppressed from our notation.

\subsection{The Small-Argument Case}

We first consider the case when the Hankel transforms have relatively small argument. However, this case arises only when there are infinitely many units in the number field.

The following lemma is essentially due to \cite[Lemma 7]{Blomer} and \cite[Lemma 6.4]{Qi-Gauss} (As indicated in Remark \ref{rem: normalization Hankel}, Blomer has a slightly different normalization).
\begin{lem}\label{lem: Hankel, derivatives}
	For $\varww \in C_c^{\infty} (F^{\times} )$ define $ \|\tw\|_{L^\infty} $ to be its sup-norm.  If    $\varww  $ is supported in a fixed compact set $K \subset F^{\times}$, then its Hankel transform $\widetilde{\varww}$ has the following estimates{\rm:}
	
	{\rm(1)} when $F $ is real, 
	\begin{align*}
	y^{i}  (\nd /\nd y)^i  \widetilde{\varww} (y)  \Lt_{\shskip i, \shskip K }   \|\tw\|_{L^\infty} \cdot  ( |y|^{1/3} + 1  )^{i } / |y|^{  1/3} ;
	\end{align*} 
	
	{\rm(2)} when $F $ is complex, 
	\begin{align*}
	u^{i} \widebar u^{j} (\partial /\partial u)^i (\partial / \partial \widebar u  )^{j} \widetilde{\varww} (u)  \Lt_{\shskip i, \shskip j, \shskip K }   \|\tw\|_{L^\infty} \cdot  ( |u|^{1/3} + 1  )^{i + j} / |u|^{  2/3} .
	\end{align*}  
\end{lem}

As a consequence of Corollary \ref{cor: bound for H < T} and Lemma \ref{lem: Hankel, derivatives}, for $|y| \leqslant T^{\vepsilon}$, we have 
\begin{align*}
y^{i}  \frac { \nd^{i} \widetilde{\varww} (y, \varLambda) } {\nd y^{i}}  \Lt_{\shskip i }  \frac { T^{1+ (i+1)\vepsilon} }  {|y|^{1/3}} \Lt \frac { M T^{1+ (i+1)\vepsilon} }  {\sqrt{|y|} }, 
\end{align*}
and, in the polar coordinates 
\begin{align*}
y^{i}  \frac { \partial^{i+j}\widetilde{\varww} (y e^{i\theta}, \varLambda) } {\partial y^i \partial \theta^j } \Lt_{  \shskip i, \shskip j }  \frac { T^{2+ (i+j+1)\vepsilon} }  {y^{2/3}} \Lt \frac { M T^{2+ (i+j+1)\vepsilon} }  { y }. 
\end{align*}
 
\begin{cor}\label{cor: Hankel y<1}
Let $|\varLambda| \Gt T^{2}$. 
We artificially define  $\Phi^{\oldstylenums{0}} (x)$ and $\Phi^{\oldstylenums{0}} (z)$ by
\begin{align}\label{11eq: defn of Phi0}
\widetilde {\varww}^{\snatural} ( y , \varLambda) =  \frac {M T^{1+\vepsilon}} {\sqrt{|y|} }   \Phi^{\oldstylenums{0}} (y / \varLambda), \qquad \widetilde {\varww}^{\snatural} ( u , \varLambda) =  \frac {M T^{2+\vepsilon}}  {|u|}  \Phi^{\oldstylenums{0}} (u / \varLambda) ,
\end{align} 
with $x = y / \varLambda$ and $z = u / \varLambda$.
	
	{\rm(1)} When $F$ is real, for $ |x|  \leqslant T^{\vepsilon}/|\varLambda| $, we have 
	\begin{align}\label{11eq: Phi0, R}
	x^{i}  \frac { \nd^{i} \Phi^{\oldstylenums{0}} (x) } {\nd x^{i}} \Lt_{\shskip i }    T^{ i \vepsilon} . 
	\end{align}
	
	{\rm(2)} When $F$ is complex, for $ x \leqslant T^{\vepsilon}/|\varLambda| $, we have 
	\begin{align}\label{11eq: Phi0, C}
		x^{i}  \frac { \partial^{i+j}\Phi^{\oldstylenums{0}} (x e^{i\phi}) } {\partial x^i \partial \phi^j} \Lt_{  \shskip i, \shskip j }  \frac {T^{(i+j)\vepsilon}  |x \varLambda|^{1/3} } {M} \Lt T^{(i+j)\vepsilon} . 
	\end{align}

\end{cor}

\subsection{Application of Stationary Phase}

Our next goal is to deduce integral representations of $\widetilde {\varww}^{\snatural} (y , \varLambda )$ and $\widetilde {\varww}^{\snatural} ( u , \varLambda )$ from those of the $\GL_2$-Bessel integrals $ \SDH (x) $ and $\SDH (z)$ in Lemma \ref{lem: H(x), |z|>1} (1) and (2) along with the asymptotic formulae for the $\GL_3$-Bessel kernels $J_{\pi} (x)  $ and $J_{\pi} (z)$ in Lemma \ref{prop: asymptotic J}.

\begin{prop}\label{prop: Hankel after stationary phase, R}
	Suppose that $|y| > T^{\vepsilon} $ and    $|\varLambda| \Gt T^2$.   Define 
	 $\mathrm{hyp}_{\pm} (r)$ to be the hyperbolic function 
	\begin{align}\label{9eq: cosh and sinh}
	\mathrm{hyp}_+ (r) =   \cosh r, \qquad \mathrm{hyp}_- (r ) = -   \sinh r .
	\end{align} 
	There are  smooth functions $V_{\pm} ( r ; y , \varLambda) $ with support in the region defined by
	\begin{align}\label{6eq: support of V, R}
	(1 / \varDelta ) \cdot |y|^{1/3} /  |\varLambda|^{1/2}  \leqslant |\mathrm{hyp}_{\pm} (r)| \leqslant   \varDelta \cdot |y|^{1/3} /  |\varLambda|^{1/2}  ,
	\end{align}
	satisfying 
	\begin{align}\label{6eq: bounds for V, R}
	 ({\partial  }  / {\partial r})^{i}   V_{+} ( r; y, \varLambda) \Lt_{\shskip i }  \log^i T, \qquad r^i ({\partial  }  / {\partial r})^{i}   V_{-} ( r; y, \varLambda) \Lt_{\shskip i }  \log^i T,
	\end{align}
	such that 
	\begin{align}\label{10eq: tilde w = Phi, R}
	\widetilde {\varww}^{\snatural} ( y , \varLambda) =  \frac {M T^{1+\vepsilon}} {\sqrt{|y|} }   \big(\Phi_+ (y , \varLambda) + \Phi_- (y , \varLambda) \big) 
	+ O (T^{-A}) ,
	\end{align}
	with
	\begin{align}\label{10eq: Phi (y,...), R}
	\Phi_{\pm} (y, \varLambda) =  \int_{- M^{\varepsilon} / M}^{M^{\varepsilon}/ M}  \hskip -1 pt e ({T r} /\pi )  g  ( M r) e \big( \hskip -1 pt -  y \shskip \mathrm{hyp}_{\pm}^{\snatural} (r)^2 / {\varLambda}   \big) \hskip -1 pt V_{\pm} ( r; y , \varLambda) \nd r,
	\end{align} 
	in which   $g $ is a Schwartz function, and
	\begin{align}\label{11eq: hyp (r)}
	\mathrm{hyp}^{\snatural}_{+} (r) = \tanh r , \qquad  \mathrm{hyp}^{\snatural}_{-} (r) = \coth r .
	\end{align} 
\end{prop}

The reader may find the integral $\Phi_{\pm} (y, \varLambda)$   in \cite[(8.16)]{Young-Cubic} and \cite[(4.20)]{Huang-GL3}. Our analysis is slightly different, however,  as our strategy is to first transform the $x$-integral into $ I(\lambdaup)$ as in \eqref{9eq: defn I, R} with phase $\lambdaup (2x^3-3x^2)$ and then use the stationary phase results in Lemma \ref{lem: I, R}, while Young and Binrong Huang apply directly general stationary phase results (\cite[Lemma 6.3]{Young-Cubic}) with phase in an arbitrary form. A technical remark is that, in view of \eqref{6eq: bounds for V, R}, $V_{+} ( r; y, \varLambda)$ is more than an ``inert" function in the sense of Young \cite{Young-Cubic}.

\begin{proof}
To start with, let us assume $\varLambda > 0$, for $\widetilde {\varww}^{\snatural} (y , - \varLambda )= \widetilde {\varww}^{\snatural} ( - y , \varLambda )$.	

By Lemma \ref{prop: asymptotic J} (1), the contribution  to $  \widetilde {\varww} (  y , \varLambda ) $  from the leading term in \eqref{4eq: asymptotic, Bessel, R}  is the following   integral:
	\begin{align*}
	I  (  y^{1/3}, \varLambda  ) = \int  e \big(3 (xy)^{1/3} \big) \SDH (\varLambda x) \varww (|x|) \frac {\nd x} {|xy|^{1/3} } ; 
	\end{align*} 
	 the contributions from lower-order terms are similar and may be handled in the same manner. Also, it follows from Corollary \ref{cor: bound for H < T} (1) that the error term in \eqref{4eq: asymptotic, Bessel, R} yields an $O \big( T^{1+\varepsilon} / |y|^{(K+1)/3} \big) = O (T^{-A})$ for $|y|  > T^{\vepsilon}$ if we choose $K$ large, say $K >  3 (A+1) / \vepsilon$. 
	 
	 Next, we change  the variables $x$ and $y$ into $ \pm x^6 $ and $y^{3}$ so that
	 \begin{align*}
	 I   ( y , \varLambda )   =  \sum_{ \pm }   \frac 1 {|y|} \int_0^{\infty}   e  ( \pm 3   x^2 y    ) \SDH \big(\hskip -1 pt \pm \varLambda x^6 \big)   a   ( x   )   \nd x  ,
	 \end{align*}
	 where  $a (x ) = 6 x^3 \varww (x^6)   $ is supported on $ [1, \varDelta^{1/6}  ]$ and satisfies $ a^{(i)} (x ) \Lt_{i } \log^i T   $. By the formulae in \eqref{8eq: H+natural} and \eqref{8eq: H-natural} in Lemma \ref{lem: H(x), |z|>1}, we infer that $ I   ( y , \varLambda ) $ may be written as 
	 \begin{align*}
	 I   ( y , \varLambda ) = M T^{1+\vepsilon} \sum_{\pm}  \int_{- M^{\vepsilon} / M}^{M^{\vepsilon}/ M}  e( Tr / \pi   )  g (  M r  )  I_{\pm} (r  ; y, \varLambda)  \nd r + O (T^{-A}), 
	 \end{align*}
	 where
	 \begin{align*}
	 I_{\pm} (r  ; y, \varLambda) = \frac 1 {|y|} \int_{-\infty}^{\infty} e \big( \hskip -1 pt \pm 3 x^2 y \mp 2 \sqrt{\varLambda} x^3 \mathrm{hyp}_{\pm} (r)  \big) a(|x|) \nd x .
	 \end{align*}
On changing $x$ into $x y / \sqrt \varLambda \mathrm{hyp}_{\pm} (r)$ (need $r \neq 0$ for $ \mathrm{hyp}_{-} (r) \neq 0$), the inner integral $I_{\pm} (r  ; y, \varLambda)$ turns into
	 \begin{align*}
	 \frac {1} {\sqrt \varLambda  | \mathrm{hyp}_{\pm} (r)  |}  \int_{-\infty}^{\infty} e \bigg( \hskip -1 pt \pm \frac {y^3} {\varLambda \mathrm{hyp}_{\pm} (r)^2 } \big(3 x^2 - 2 x^3 \big) \bigg) a \bigg( \frac {|xy|} {\sqrt{\varLambda}  | \mathrm{hyp}_{\pm} (r)  | }  \bigg) \nd x ,
	 \end{align*}
	 and it is exactly the integral $I (\lambdaup) $ defined as in \eqref{9eq: defn I, R} if one lets $ \lambdaup = \pm y^3 / \varLambda \mathrm{hyp}_{\pm} (r)^2 $, $\rho =  \sqrt{\varLambda}  | \mathrm{hyp}_{\pm} (r)  | / |y|$($= \sqrt{|y / \lambdaup| }$), and 
	 \begin{align*}
	 \varww (x; \lambdaup) = {\textstyle \sqrt{|\lambdaup / y^3 | } } \cdot a \big(  |x| {\textstyle \sqrt{|\lambdaup / y|}} \big) ,
	 \end{align*} 
	 with 
	 \begin{align*} 
	x^{i} \lambdaup^{j} \partial_x^{i}  \partial_{\lambdaup}^{j}   \varww  (x  ; \lambdaup)   \Lt_{\, i, \shskip  j }        {\textstyle \sqrt{|\lambdaup / y^3| } } \cdot \log^{i+j} T.  
	\end{align*} 
	
	Let $ I^{\snatural} (\lambdaup) = e (  -   \lambdaup   ) I (\lambdaup )$ be as in \eqref{9eq: tilde I, R}.  
We introduce a smooth function $\varvv (x)$ such that $\varvv (x) \equiv 1$ on $ [1/\sqrt \varDelta, \sqrt \varDelta  ]$ and  $\varvv (x) \equiv 0$ on $(0, 1/\varDelta] \cup [ \varDelta, \infty)$. According to Lemma \ref{lem: I, R} (1), if $\rho = \sqrt{|y / \lambdaup| }$ is not in the interval $(1/\sqrt{\varDelta}, \sqrt{\varDelta})$, then
\begin{align*}
I^{\snatural}  (\lambdaup ) \Lt    \frac 1 {|y|} \bigg( \frac{\log T}{ \sqrt{|y^{3} / \lambdaup|}  + |y| } \bigg)^{K} < \frac{\log^K T }{|y|^{K + 1}},
\end{align*}
and hence $ I^{\snatural}  (\lambdaup )  ( 1 - \varvv  ( \lambdaup / y )  ) $ only contributes to the error term.  By    Lemma \ref{lem: I, R} (2), 
	\begin{align*}
	\lambdaup^{j} \frac {\nd^{j}} {\nd \lambdaup^{j} }  \big( I^{\snatural} (\lambdaup) \varvv (|\lambdaup / y|) \big)  \Lt_{ j } \frac {\log^j T   } {\sqrt {|y|^3}   }. 
	\end{align*} 
	
	Finally, let $$ V_{\pm} (r; y^3, \varLambda) = {\textstyle \sqrt {|y|^3}} I^{\snatural} (\lambdaup) \varvv (|\lambdaup / y|), \qquad \text{ ($\shskip \lambdaup = \pm y^3 / \varLambda \mathrm{hyp}_{\pm} (r)^2$)},$$
	\footnote{This $V_{\pm}$ function is only from the leading term  in \eqref{4eq: asymptotic, Bessel, R}, so, to be strict, one must also include those  $V_{\pm}$ functions constructed from the lower order terms  in \eqref{4eq: asymptotic, Bessel, R}.}then, after changing $y$ into $y^{1/3}$,  the expression of $  \widetilde {\varww}^{\snatural} ( y , \varLambda) $ (defined in \eqref{10eq: defn of w nat}) given by \eqref{10eq: tilde w = Phi, R} and \eqref{10eq: Phi (y,...), R} readily follow from the arguments above, along with the identity $$  -  \mathrm{hyp}^{\snatural}_{\pm} (r)^2 =  - 1 \pm \frac 1  {\mathrm{hyp}_{\pm} (r)^2}  ,$$ and, to deduce \eqref{6eq: bounds for V, R} one   needs the estimates\footnote{We also need Fa\`a di Bruno's formula for higher derivatives of composite functions (see \cite{Faa-di-Bruno}).} $$ \mathrm{hyp}_{\pm} (r) \frac { \nd^j} {\nd r^j } \bigg(\frac 1 { \mathrm{hyp}_{\pm} (r) } \bigg) \Lt_{j} \frac 1 { |\mathrm{hyp}_{\pm} (r)|^{j} } $$
	for $ |r | < 1$.
\end{proof}

\begin{prop}\label{prop: Hankel after stationary phase, C}
  Suppose that $|u| > T^{\vepsilon}$  and $|\varLambda| \Gt T^2$.  
  Recall the definition of $\rho (r, \omega)$ in {\rm\eqref{3eq: rho (r, w)}}.	There is a smooth function $V ( r, \omega; u, \varLambda) $ with support in the region defined by
	\begin{align}\label{6eq: support of V}
	(1 / \varDelta ) \cdot |u|^{1/3} /  |\varLambda|^{1/2}  \leqslant \rho (r, \omega) \leqslant  \varDelta  \cdot |u|^{1/3} /  |\varLambda|^{1/2}  ,
	\end{align}
	satisfying 
	\begin{align}\label{6eq: bounds for V}
	\rho (r, \omega)^{i+j} ({\partial  }  / {\partial r})^{i} ({\partial  } / {\partial \omega})^{j} V ( r, \omega; u, \varLambda) \Lt_{\shskip i, \shskip j}  \log^{i+j} T,
	\end{align}
	such that 
	\begin{align}\label{11eq: w tilde = Phi, C}
\widetilde {\varww}^{\snatural} ( u , \varLambda) =  \frac {M T^{2+\vepsilon}}  {|u|}  \Phi (u , \varLambda)  
	+ O  ( T^{-A} ) ,
	\end{align}
	with
	\begin{align}\label{6eq: Phi (u,...)}
	\Phi (u, \varLambda) \hskip -1 pt = \hskip -1 pt \int_0^{ \pi} \hskip -2 pt \int_{- M^{\varepsilon} / M}^{M^{\varepsilon}/ M}  \hskip -2 pt e ({2T r} /\pi )  g  ( M r) e \big( \hskip -2 pt - \hskip -1 pt 2 \Re  (u  \shskip \trh^{\snatural} (r, \omega) / \varLambda   )   \big) \hskip -1 pt V ( r, \omega; u, \varLambda) \nd r \shskip \nd \omega,
	\end{align}
in which $g $ is a Schwartz function,	and $ \trh^{\snatural} (r, \omega) = \rho^{\snatural} (r, \omega) e^{i \shskip \theta^{\snatural}  (r, \shskip \omega) } $ is defined by
	\begin{align}\label{6eq: rho and theta, natural}
	\rho^{\snatural} (r, \omega)   = \frac {\cosh 2 r -   \cos 2 \omega} {\cosh 2 r +   \cos 2 \omega}, \qquad 
	\tan  ( \theta^{\shskip \snatural}  (r, \omega)/ 2  ) =  \frac {\sin 2 \omega} {\sinh 2 r}   .
	\end{align}
\end{prop}

It is remarkable that the square-root signs in the formulae of   $\rho (r, \omega)$ in  \eqref{3eq: rho (r, w)} are no longer in the formula of $\rho^{\snatural} (r, \omega)$ in \eqref{6eq: rho and theta, natural}. This   makes our lives easier  in the polar coordinates. 

\begin{proof}
	 The first stage of proof will be similar to that of Proposition \ref{prop: Hankel after stationary phase, R}. 
	 
	 Let us assume without loss of generality that $\varLambda > 0$ and consider the integral 
	 \begin{align*}
	 I  (   u^{1/3}, \varLambda  ) \hskip -1 pt = \hskip -1 pt \sum_{ \xi^3 =  1 }  \viint   e \big( 6  \Re  \big( \xi ( zu )^{1/3} \big)  \big) \SDH ( \varLambda  z  )    \varww  (|z|) \frac {d z} {|z u|^{2/3} },
	 \end{align*} 
	 which is the contribution from the  three leading terms in \eqref{4eq: asymptotic, Bessel, C} in Lemma \ref{prop: asymptotic J} (2). Substituting the variables $z$ and $u$ by $ z^6$ and $u^{3}$,  we have
	 \begin{align*}
	 I   ( u , \varLambda )   = \frac 1 {  |u|^2} \viint_{\BC^{\times}/\{ \pm 1 \}} e \big(  6 \Re  (     z^2 u  ) \big) \SDH \big( \varLambda  z^6  \big) a (|z|) \nd z / |z|,
	 \end{align*}
	 where $ a (x) = 36 x^7 \varww (x^6) $  is supported on $ [1, \varDelta^{1/6}  ]$ and satisfies $ a^{(i)} (x ) \Lt_{i } \log^i T   $.
	 
	 Let $z = x e^{i \phi} $ and $u = y e^{i \theta}$. By the formula \eqref{8eq: H-sharp(z)} in Lemma \ref{lem: H(x), |z|>1}, we infer that $	I   ( y e^{i \theta} , \varLambda ) $ may be written as 
	 \begin{align*}
	 I   ( y e^{i \theta} , \varLambda ) =  {M T^{2+\vepsilon}}  \int_0^{ \pi}   \hskip -1 pt
	 \int_{- M^{\varepsilon} / M}^{M^{\varepsilon}/ M}  e ({2T r} /\pi )  g  ( M r) I (r, \omega; y e^{i \theta}, \varLambda) \nd r \nd \omega + O  (T^{-A}  ),
	 \end{align*}
	 in which 
	 \begin{align*}
	 I (r, \omega; y e^{i \theta}, \varLambda ) \hskip -1 pt = \hskip -1 pt \frac 2 {y^2} \hskip -2 pt \int_0^{2 \pi} \hskip -3 pt \int_0^{\infty}
	\hskip -2 pt e  \big( 6 x^2 y \cos (2\phi \hskip -1 pt + \hskip -1 pt \theta) \hskip -1 pt - \hskip -1 pt 2 \sqrt {\varLambda} x^3    \trh (r, \omega; 3 \phi)    \big)    a (x  ) \nd x \nd \phi. 
	 \end{align*}
At this point, we assume   $(r, \omega) \neq (0, \pi / 2)$ and invoke the expression of $\trh (r, \omega; 3 \phi)$ as in \eqref{3eq: trh polar}. On changing the variables $ x $ and $ \phi$ into $x y / \sqrt{\varLambda} \shskip \rho (r, \omega)$ and $\phi -  \theta (r, \omega) /3 $, respectively, the integral $I (r, \omega; y e^{i \theta}, \varLambda )$ turns into
\begin{align*}
\frac {2  } { \sqrt{\varLambda} \shskip \rho (r, \omega) y} \int_0^{2 \pi} \int_0^{\infty} e  \bigg(   \frac {2 y^3} {\varLambda \shskip \rho (r, \omega)^2} f (  x, \phi; \theta - 2 \theta (r, \omega)/3 ) \bigg)  a  \bigg(   \frac {x y  } {\sqrt{\varLambda} \shskip \rho (r, \omega)} \bigg)    \nd x \nd \phi .
\end{align*}
According to the notation in  \S \ref{sec: stationary phase, Hankel, C}, the phase function $ f (  x, \phi;   \psi ) $ is defined by \eqref{4eq: phase f}, and the integral above is of the form $I (\lambdaup, \psi)$  as in \eqref{9eq: defn I, C} if one lets $\lambdaup =   y^3 / \varLambda  \rho (r, \omega)^2$, $ \psi = \theta - 2 \theta (r, \omega)/3$, and $\rho = \sqrt{\varLambda} \shskip \rho (r, \omega) / y$ ($= \hskip -2pt \sqrt {y/\lambdaup}$); the weight function
\begin{align*}
\varww   (   x ; \lambdaup ) = 2 \hskip -1 pt {\textstyle \sqrt{  \lambdaup / y^5}} \cdot a \big(x \hskip -1 pt {\textstyle \sqrt {\lambdaup / y}}      \big),
\end{align*}
has bounds 
\begin{align*}
x^{i} \lambdaup^{k} \partial_x^{i}  \partial_{\lambdaup}^{k}  \varww   (   x ; \lambdaup) \Lt_{\shskip i, \shskip  k }        {\textstyle \sqrt{\lambdaup / y^5}} \log^{i+k} T .
\end{align*}
Next, we apply Lemma \ref{lem: I, C} to $   I^{\snatural} (\lambdaup, \psi) = e  (  - 2 \lambdaup \cos 3 \psi  ) I (\lambdaup, \psi)$ as in  \eqref{9eq: tilde I, C}. Let  $\varvv (x)$ be a smooth function such that $\varvv (x) \equiv 1$ on $ [1/\sqrt \varDelta, \sqrt \varDelta  ]$ and  $\varvv (x) \equiv 0$ on $(0, 1/\varDelta] \cup [ \varDelta, \infty)$. Lemma \ref{lem: I, C} (1) implies that $ I^{\snatural} (\lambdaup, \psi )  ( 1 - \varvv  ( \lambdaup / y )  ) $ only contributes an error term, while Lemma \ref{lem: I, C} (2) yields the estimates 
\begin{align*}
\lambdaup^{k} \frac {\partial^{k + l}} {\partial \lambdaup^{k} \partial \psi^{ l } } \big(   I^{\snatural}  (\lambdaup, \psi ) \varvv    ( \lambdaup / y ) \big) \Lt_{ \shskip k  , \shskip l } \frac { \log^{k+l} T } {  y^3} .
\end{align*}
Keeping in mind that $\psi = \theta - 2 \theta (r, \omega)/3$ and $ \lambdaup =   y^3 / \varLambda \shskip \rho (r, \omega)^2$, these estimates above along with those for $1/\rho (r, \omega)^2$ and $\theta  (r, \omega) $ in Lemma \ref{lem: estimates for rho and theta} imply that
\begin{align*}
\rho (r, \omega)^{i+j}   \omega ({\partial  }  / {\partial r})^{i} ({\partial  } / {\partial \omega})^{j} V  ( r, \omega; y^3 e^{3 i \theta}, \varLambda   ) \Lt_{\shskip i, \shskip j}  1 ,
\end{align*}
\footnote{We need to use here the simple fact: 
	For a composite function $ f (\lambdaup (r, \omega) , \theta (r, \omega) )$ in general, its derivative $ \partial_r^{i} \partial_\omega^{j } f (\lambdaup (r, \omega) , \theta (r, \omega) )  $ is a linear combination of 
	$$ 
	\partial_{\lambdaup}^{k}  \partial_{\theta}^{l} f (\lambdaup (r, \omega) , \theta (r, \omega) ) \prod_{  \nu = 1 }^{k} \partial_r^{i_{ \ssnu   } } \partial_\omega^{j_{\ssnu} } \lambdaup (r, \omega)  
	\prod_{  \mu = 1 }^{l} \partial_r^{i^{\hskip 0.5 pt \ssprime}_{\ssmu} } \partial_\omega^{j^{\hskip 0.5 pt \ssprime}_{\ssmu} } \theta (r, \omega)    , \quad 
	\sum  i_{\nu    } + \sum  i^{\shskip \prime}_{\mu} = i, \  \sum  j_{\nu} + \sum  j^{\shskip \prime}_{\mu} = j.$$ This is a two-dimensional Fa\`a di Bruno's formula in a less precise form.} with 
\begin{align*}
V  (r, \omega; y^3 e^{3 i \theta}, \varLambda  ) =  y^3   I^{\snatural}  (\lambdaup, \psi )  \varvv  (  \lambdaup / y ),
\end{align*}
supported in the region $  y /  \varDelta  \sqrt{\varLambda} \leqslant \rho (r, \omega) \leqslant  \varDelta y / \sqrt{\varLambda}  $.

In view of  \eqref{9eq: tilde I, C} and \eqref{10eq: defn of w nat}, we need to compute the exponential factor  $ e  ( 2 \lambdaup \cos 3 \psi - 2 y^3 \cos 3 \theta / \varLambda ) $, in which
\begin{align*}
2 \lambdaup \cos 3 \psi = \frac {2 y^3 \cos (3\theta - 2 \theta (r, \omega) ) } { \varLambda  \rho (r, \omega)^2 }.
\end{align*} After reverting $ y^3 e^{3 i \theta}$ to $ y e^{i \theta} $, the proof is completed if  we can prove \eqref{6eq: rho and theta, natural} for $\rho^{\snatural} (r, \omega)$ and $\theta^{\shskip \snatural}  (r, \omega)$ given by 
\begin{align*}
{  \cos  \theta   } - \frac {  \cos ( \theta - 2 \theta (r, \omega) ) } {   \rho (r, \omega)^2 }   =
  { \rho^{\snatural} (r, \omega) }   \cos  \big(   \theta + \theta^{\shskip \snatural}  (r, \omega) \big). 
\end{align*}
We have
\begin{align*}
\rho^{\snatural} (r, \omega) = \sqrt {\bigg( 1 - \frac {\cos 2 \theta (r, \omega)} { \rho (r, \omega)^2 } \bigg)^2 + \bigg(\frac {\sin 2 \theta (r, \omega)} { \rho (r, \omega)^2 }\bigg)^2 },
\end{align*}
\begin{align*}
\cos \theta^{\shskip \snatural}  (r, \omega) = \frac {1} {\rho^{\snatural} (r, \omega)} - \frac {\cos 2 \theta (r, \omega)} {\rho^{\snatural} (r, \omega) \rho (r, \omega)^2 } , \quad  \sin \theta^{\snatural}  (r, \omega) =  \frac {\sin 2 \theta (r, \omega)} { \rho^{\snatural} (r, \omega) \rho (r, \omega)^2 }.
\end{align*}
By the definitions of $\rho (r, \omega)$ and $ \theta (r, \omega) $ in \eqref{3eq: rho (r, w)} and \eqref{3eq: theta (r, w)}, we have
\begin{align*}
1 - \frac {\cos 2 \theta (r, \omega)} { \rho (r, \omega)^2 }  &  = 1 - \frac { 4 ( \cosh^2 \hskip -1 pt r \cos^2 \omega - \sinh^2 \hskip -1 pt r \sin^2 \omega  ) } { (\cosh 2 r +   \cos 2 \omega)^2 }   \\
&  = 1 - \frac { 2 ( \cosh 2 r \cos 2 \omega + 1 ) } { (\cosh 2 r +   \cos 2 \omega)^2 }   \\
& = \frac {  \cosh^2 2 r + \cos^2 2 \omega -2 } { (\cosh 2 r +   \cos 2 \omega)^2 } \\
& =  \frac { \sinh^2 2 r - \sin^2 2 \omega } { (\cosh 2 r +   \cos 2 \omega)^2 },
\end{align*}
and similarly
\begin{align*}
\frac {\sin 2 \theta (r, \omega)} { \rho (r, \omega)^2 } = \frac { 2 \sinh 2 r \cdot \sin 2 \omega } { (\cosh 2 r +   \cos 2 \omega)^2 }.
\end{align*}
We conclude that
\begin{align*}
\rho^{  \snatural} (r, \omega) = \frac {\sinh^2 2 r + \sin^2 2 \omega} { (\cosh 2 r +   \cos 2 \omega)^2 } = \frac {\cosh 2 r -   \cos 2 \omega} {\cosh 2 r +   \cos 2 \omega},
\end{align*}
\begin{align*}
\cos \theta^{ \shskip \snatural}  (r, \omega) =   \frac { \sinh^2 2 r - \sin^2 2 \omega } {\sinh^2 2 r + \sin^2 2 \omega}, \quad  \sin \theta^{ \shskip \snatural}  (r, \omega) =   \frac { 2 \sinh 2 r \cdot \sin 2 \omega } {\sinh^2 2 r + \sin^2 2 \omega},
\end{align*}
and hence
\begin{align*}
\tan  ( \theta^{\shskip \snatural}  (r, \omega)/ 2  ) =    \frac  {\sin 2 \omega}  {\sinh 2 r} .
\end{align*}
\end{proof}

\begin{cor} \label{cor: Hankel small} We have  $\widetilde \varww (y, \varLambda)$, $\widetilde \varww (u, \varLambda) = O (T^{-A})$   for  $ |y|$, $ |u| \Gt |\varLambda|^{3/2} $, respectively.  
\end{cor}

\begin{proof}
It is clear from \eqref{6eq: support of V, R} and \eqref{6eq: support of V}. 
\end{proof}

\subsection{Analysis of the New Trigonometric-Hyperbolic Function} For later use, we record here some results concerning the trigonometric-hyperbolic function $\trh^{\snatural} (r, \omega)$ that arose in Proposition \ref{prop: Hankel after stationary phase, C}. By \eqref{6eq: rho and theta, natural}, we have
\begin{align}\label{11eq: r-derivative trh2}
\frac {\partial \rho^{\snatural} (r, \omega) } {\partial r} =  \frac {4 \sinh 2r \cos 2\omega   } {(\cosh 2 r +   \cos 2 \omega)^2}, \quad \frac {\partial \rho^{\snatural} (r, \omega) } {\partial \omega} =  \frac {4 \cosh 2r \sin 2\omega   } {(\cosh 2 r +   \cos 2 \omega)^2} 
\end{align}
\begin{align}\label{11eq: w-derivative trh2}
\frac {\partial \theta^{\snatural} (r, \omega) } {\partial r}  =   - \frac {4 \cosh 2r \sin 2\omega } {\cosh^2 2 r - \cos^2 2\omega}, \quad
\frac {\partial \theta^{\snatural} (r, \omega) } {\partial \omega}  =   \frac {4 \sinh 2r \cos 2\omega } {\cosh^2 2 r - \cos^2 2\omega}.
\end{align}
Note that  $ \sinh^2 2r + \sin^2 2\omega = \cosh^2 2 r - \cos^2 2\omega $.

\begin{lem}\label{lem: derivatives of rho natural}
Define $ \mathrm{trh}^{\snatural} (r, \omega; \phi) = \rho^{\snatural} ( r, \omega ) \cos (\phi  + \theta^{\snatural} (r, \omega) ) $. 	For $|r| < 1$, we have  
	\begin{align*}
 	\frac {\partial^{i +  j}      \mathrm{trh}^{\snatural} (r, \omega; \phi)   }  {  \partial r^{\shskip i} \partial \omega^{ j} }    \Lt_{\shskip i, \shskip j } 
	\frac {\rho^{\snatural} ( r, \omega )} { (\cosh^2 2 r - \cos^2 2 \omega)^{(i+j)/ 2} }.
	\end{align*}  
\end{lem}

\begin{proof} 
Set  $\psi = \phi  + \theta^{\snatural} (r, \omega)$. From \eqref{11eq: r-derivative trh2} and \eqref{11eq: w-derivative trh2} we deduce that
	\begin{align*}
\frac	{\partial \shskip \mathrm{trh}^{\snatural} (r, \omega; \phi) }   {\partial r} & =   \frac { 4 \sinh 2r \cos 2\omega  } {(\cosh 2 r +   \cos 2 \omega)^2} \cos \psi + \frac { 4 \cosh 2r \sin 2\omega } {(\cosh 2 r +   \cos 2 \omega)^2} \sin \psi , \\
\frac	{\partial \shskip \mathrm{trh}^{\snatural} (r, \omega; \phi) }   {\partial \omega} & =   \frac { 4 \cosh 2r \sin 2\omega } {(\cosh 2 r +   \cos 2 \omega)^2} \cos \psi - \frac { 4 \sinh 2r \cos 2\omega  } {(\cosh 2 r +   \cos 2 \omega)^2} \sin \psi .
	\end{align*}
	For $ i +  j \geqslant 1$, we may prove by induction that $ {\partial^{ i +  j} \mathrm{trh}^{\snatural} (r, \omega; \phi) } / \partial r^{\shskip i} \partial \omega^{  j} $ is a linear combination of 
	\begin{equation*}
	\frac { \sinh^{ i_1  } 2 r \cosh^{  i_2 } \hskip -1pt 2 r \sin^{ j_1 } 2 \omega \cos^{\shskip  j_2} \hskip -1pt 2\omega } {(\cosh 2 r +   \cos 2 \omega)^{k_1 + l + 2 } 
		(\cosh 2 r -   \cos 2 \omega)^{ k_2 + l}} \hskip -1pt \cdot \hskip -1pt \left\{\begin{aligned}
	&\hskip -2pt \cos \psi \\
	& \hskip -2pt \sin \psi
	\end{aligned} \right\} ,
	\end{equation*}
	with
	\begin{align*}
	& k_1 + k_2 + l \leqslant  i +  j - 1, \hskip 9pt 2 (k_1 + k_2 + l) \leqslant i+j+i_1 + j_1-2     ,   \\
	& i_1  + i_2  + j_1 + j_2 =  k_1 + k_2 + 2 l + 2, \quad  i_2 \leqslant  i + 1, \quad  j_2 \leqslant  j + 1.
	\end{align*}
	Such a fraction is bounded by
	\begin{align*}
	& \frac {    {  ( \sinh^2 2 r + \sin^2 2 \omega  )^{ (i_1 + j_1)/2 } }  } { (\cosh 2 r +   \cos 2 \omega)^{k_1 + l + 2 } 
		(\cosh 2 r -   \cos 2 \omega)^{ k_2 + l} } \\
	=  \ &    \frac {    1  } { (\cosh 2 r +   \cos 2 \omega)^{k_1 + l - (i_1 + j_1)/2  + 2 } 
		(\cosh 2 r -   \cos 2 \omega)^{ k_2 + l - (i_1 + j_1)/2} } \\
	\Lt \ & \frac {1} { (\cosh 2 r +   \cos 2 \omega)^{( i +  j)/2 - k_2 + 1 } 
		(\cosh 2 r -   \cos 2 \omega)^{ ( i +  j)/2 - k_1-1} } \\
	\Lt \ & \frac {1} { (\cosh 2 r +   \cos 2 \omega)^{( i +  j)/2 + 1 } 
		(\cosh 2 r -   \cos 2 \omega)^{ ( i +  j)/2 - 1} },
	\end{align*} 
	as desired. Note that $2 (k_1 + k_2 + l) \leqslant i+j+i_1 + j_1-2$ is used here for the first inequality. 
\end{proof}

By  \eqref{11eq: r-derivative trh2} and \eqref{11eq: w-derivative trh2}, we have
\begin{align}\label{6eq: derivative of rho, theta, 1.1}
& \frac {\partial \log \rho^{\snatural} (r, \omega) } {\partial r} = \frac {\partial \theta^{\snatural} (r, \omega) } {\partial \omega} = \frac {4 \sinh 2r \cos 2\omega   } {\sinh^2 2 r + \sin^2 2\omega}, \quad  \\
\label{6eq: derivative of rho, theta, 1.2}
& \frac {\partial \theta^{\snatural} (r, \omega) } {\partial r} = - \frac {\partial \log \rho^{\snatural} (r, \omega) } {\partial \omega} = -  \frac {4 \cosh 2r \sin 2\omega } {\sinh^2 2 r + \sin^2 2\omega}.
\end{align}
Similar to Lemma \ref{lem: derivatives of rho natural}, one can establish the following lemma.

\begin{lem}\label{lem: derivatives of log rho, theta}
	For $|r| < 1$, we have  
	\begin{align*}
	\frac {\partial^{i +  j}  \log \rho^{\snatural} (r, \omega) }  {   \partial r^{\shskip i} \partial \omega^{\shskip j} }, \, \frac {\partial^{i +  j}    \theta^{\snatural} (r, \omega) }  {   \partial r^{\shskip i} \partial \omega^{\shskip j} }    \Lt_{\shskip  i, \shskip  j } 
	\frac 1 { (\sinh^2 2 r + \sin^2 2 \omega)^{( i+ j)/ 2} }
	\end{align*}  
	for $ i +  j \geqslant 1$.
\end{lem}

Furthermore, it follows from \eqref{6eq: derivative of rho, theta, 1.1} and \eqref{6eq: derivative of rho, theta, 1.2} that
\begin{equation}\label{6eq: derivative of rho, theta, 2.1}
\begin{split}
   \frac {\partial^2 \log \rho^{\shskip \snatural}   } {\partial r^2}   = \frac {\partial^2 \theta^{\shskip \snatural} } {\partial r \partial \omega} &  = - \frac {\partial^2 \log \rho^{\shskip \snatural}   } {\partial \omega^2}  = - \frac {8 \cosh 2r \cos 2\omega   (\sinh^2 2 r - \sin^2 2 \omega)} { (\sinh^2 2 r + \sin^2 2\omega )^2},  
\end{split}
\end{equation}
\begin{equation}
\label{6eq: derivative of rho, theta, 2.2}
\begin{split}
\frac {\partial^2 \theta^{\shskip \snatural}   } {\partial r^2} = - \frac {\partial^2 \log \rho^{\shskip \snatural}   } {\partial r \partial \omega} & = - \frac {\partial^2 \theta^{\shskip \snatural}  } {\partial \omega^2}  = \frac {8 \sinh 2r \sin 2\omega   (\cosh^2 2 r + \cos^2 2 \omega)} { (\sinh^2 2 r + \sin^2 2\omega )^2}.
\end{split}
\end{equation}

\delete{ 
\begin{cor} \label{lem: uniform estimates, Hankel} Let $|\varLambda| \Gt T^2$. Define $\widetilde \varww (y, \varLambda) $ and $\widetilde \varww (u, \varLambda) $  as in {\rm\eqref{11eq: defn of w (x, Lmabda), R}}, {\rm\eqref{11eq: defn of w (z, Lmabda), C}}, and {\rm\eqref{11eq: Hankel}}. 
		
		{\rm(1)} When $F$ is real,   we have uniformly
		\begin{equation}\label{11eq: unform bound for Hankel, R}
		\widetilde \varww (y, \varLambda) \Lt_{A', \shskip   \vepsilon } \frac {  T^{1+\vepsilon} } {\sqrt{|y|}}, \end{equation} 
		and $\widetilde \varww (y, \varLambda) = O (T^{-A})$  for any $A \geqslant 0$ if $ |y| \Gt |\varLambda|^{3/2} $.

		{\rm(2)}	When  $F $ is complex,  we have uniformly
		\begin{equation}\label{11eq: unform bound for Hankel, C}
		\widetilde \varww (u, \varLambda) \Lt_{A', \shskip   \vepsilon } \frac {  T^{2+\vepsilon} } { {|u|}}, 		\end{equation}  
	and $\widetilde \varww (u, \varLambda) = O (T^{-2A})$   for any $A \geqslant 0$ if $ |u| \Gt |\varLambda|^{3/2} $. 
\end{cor}

\begin{proof}
	When $ |y|, |u| \Gt 1 $, the estimates follow trivially from Proposition \ref{prop: Hankel after stationary phase, R} and \ref{prop: Hankel after stationary phase, C}. When $ |y|, |u| \Lt 1 $, we simply use Lemma \ref{prop: asymptotic J}  to bound the Bessel kernels $J_{\pi} (xy) \Lt 1 / \hskip -1 pt \sqrt{|y|} $, $J_{\pi} (zu) \Lt 1/  {|u|} $, and Corollary \ref{cor: bound for H < T} to bound the Bessel integrals $ \SDH (\varLambda x) \Lt T^{1+\vepsilon} $, $ \SDH (\varLambda z) \Lt T^{2+\vepsilon} $  for $|x|, |z| \in [1, \varDelta]$. 
\end{proof}
}

\subsection*{Notation} For simplicity of exposition, we introduce some non-standard notation as below. 

\begin{notation}\label{notation: sim}
	For $ \varDelta > 1$ and $X > 0$, let $ x \sim_{\varDelta} X $ stand for $ x \in [X, \varDelta X] $.  
\end{notation}

\begin{notation}\label{notation: approx}
	Let $ \varDelta > 1$.  
	We write $X \approx_{\varDelta} Y$ if $ 1/c_\varDelta \leqslant X /Y \leqslant c_{\varDelta}$ for some $c_{\varDelta} > 1$ such that $c_{\varDelta} \ra 1$ as $\varDelta \ra 1$. We write $X \llcurly_{\shskip \varDelta} Y$ if $|X| \leqslant \delta_{\varDelta}  Y$ for some $\delta_{\varDelta} > 0 $ such that $\delta_{\varDelta} \ra 0$ as $\varDelta \ra 1$.  
\end{notation}

\subsection{Preliminary Analysis of the $\Phi$-integrals}\label{sec: analysis of Phi-interals} For convenience of the further analysis by the Mellin technique in \S \ref{sec: Hankel, II}, we introduce certain partitions of the $\Phi$-integrals. 

For the real case, the partition for $\Phi_{+} (y, \varLambda)$ is hidden in the proof of \cite[Lemma 8.2]{Young-Cubic}, but the case of $\Phi^{+}_0 (y/\varLambda)$ seems to be missing there. 

\begin{cor}\label{cor: Hankel, R}
	Let $\varDelta > 1$ be fixed. Let $A \geqslant 1$. Let  $ |y| \sim_{\varDelta} Y$. Suppose that  $Y \Gt T^{\vepsilon} $  and $|\varLambda| \Gt T^2$. Define $\Phi_{\pm} (y, \varLambda)$ by {\rm\eqref{10eq: Phi (y,...), R}} and {\rm\eqref{11eq: hyp (r)}}. We have
	\begin{align}\label{11eq: Phi + = Phi+}
	\Phi_{+} (y, \varLambda) = \Phi^{+}  (y/\varLambda)  + O  (T^{-A} ), 
	\end{align}
	for $ Y^{2/3} \approx_{\varDelta}  |\varLambda|   $, where $\Phi^{+}  (x)$ is supported on $|x| \Gt M^{1-\vepsilon} T$ and 
	\begin{align}
		\Phi^{+}  (x) = \left\{ \begin{aligned}
			&\Phi^{+}_1 (x), & & \ \text{ if } M^{1-\vepsilon} T   \Lt |x| \Lt  T^{2-\vepsilon},  \\
			&\Phi^{+}_0 (x), & & \ \text{ if } |x| \Gt T^{2-\vepsilon}, 
		\end{aligned}\right. 
	\end{align} in which $\Phi^{+}_1 (x)$ is given by
	\begin{equation}\label{11eq: Phi+1, R}
	\Phi^{+}_1 (x) =  \int 
	e  ({T r} /\pi   -  x \tanh^2 r   )   V^+ ( r) \nd r,
	\end{equation} 
		with $ V^+ ( r) $  supported in $ r \approx_{\varDelta} T     / 2 \pi x $,  satisfying $ r^i (\nd /\nd r)^{i} V^+ ( r) \Lt_{i } 1 $, and  $\Phi^{+}_0 (x)$  satisfies
	\begin{align}\label{11eq: E(x), R}
	x^{i}   {(\nd/\nd x )^i \Phi^{+}_0(x) }   \Lt_{i}  {\textstyle T^{(i + 1) \vepsilon }  / \sqrt{|x|} } ,
	\end{align}  
while
\begin{align}\label{11eq: Phi - = Phi-}
\Phi_{-} (y, \varLambda) = \Phi^{-}  (y/\varLambda), 
\end{align}  
for $ Y^{2/3} \Lt |\varLambda| / M^{2-\vepsilon} $, with 
	\begin{align}\label{11eq: defn of Phi-, R}
	\Phi^{-}  (x) = \int  e  ({T r} /\pi  -  x \coth^2 r )  V^-  ( r) \nd r,
	\end{align} 
	where $V^-  ( r)$ is supported in  $ |r| \approx_{\varDelta} Y^{1/3} / |\varLambda|^{1/2} $,   satisfying $ r^i (\nd /\nd r)^{i} V^-  (r) \Lt_{ i } \log^{i} T $.  
\end{cor}

\begin{proof}
  The case of $\Phi_{-} (y, \varLambda)$ is obvious.  For $\Phi_{+} (y, \varLambda)$ it requires some discussions. The nature of the integral $\Phi_{+} (y, \varLambda)$ changes when $ x = y / \varLambda $ moves beyond $ T^{2-\vepsilon} $.  For $|x| \Lt T^{2-\vepsilon}$ or $|x| \Gt T^{2-\vepsilon}$, respectively,   the integral $\Phi_{+} (y, \varLambda)$ will be turned into   $\Phi^{+}_1 (x)$ or $\Phi^{+}_0 (x)$  by smoothly truncating the $r$-integration near $T/2\pi x$ or at $ \pm T^{\vepsilon} / T $.

Let $x = y / \varLambda$. The phase function $T r / \pi - x \tanh^2 r  $ has a unique stationary point $r_0 \approx_{\varDelta} T   / 2 \pi x $. For $ |r_0| \leqslant M^{\vepsilon} / M $, it is necessary that $|x| \Gt M^{1-\vepsilon} T$, and otherwise $\Phi_{+} (y, \varLambda)$ is negligible by Lemma \ref{lem: staionary phase, dim 1, 2} with $Z = |x|$, $Q = 1$, $R=T$, and $P = 1/M  $. When $ |x| \Lt T^{2-\vepsilon} $, we have $  |x| / T^{2} \Lt 1/T^{\vepsilon}$, and hence Lemma \ref{lem: staionary phase, dim 1, 2} (now $P = T/|x|  $)  implies that only a negligibly small error is lost if we restrict the integration on the interval $r \approx_{\varDelta} T     / 2 \pi x$ via a smooth partition of unity, giving $\Phi_1^+(x)$. 

Next assume $|x| \Gt T^{2-\vepsilon}$ so that $|r_0| < T^{\vepsilon} / \varDelta T$. On applying Lemma \ref{lem: staionary phase, dim 1, 2} again with $R = |x| /T^{1-\vepsilon} $ and $P = T^{\vepsilon}/T$, we are left to consider the integral $\Phi^{+}_0 (x)$ restricted on $|r| \leqslant T^{\vepsilon} / T$. The factor $e ( Tr/\pi )$ is no longer oscillatory and may be absorbed into the weight function.  To prove the  estimates in \eqref{11eq: E(x), R} for the derivatives of ${  \Phi^{+}_0 (x) }$, we differentiate the integral and then confine the integration to $ |r| \leqslant T^{\vepsilon}/\sqrt{|x|} $; Lemma \ref{lem: staionary phase, dim 1, 2} is used for the last time with $R = \hskip -1pt \sqrt{|x|} T^{\vepsilon}$ and $P =  T^{\vepsilon} / \hskip -1pt \sqrt{|x|} $. Alternatively, one can also use Lemma \ref{lem: 2nd derivative test, dim 1}.

Note that the fact that $V_{+} ( r; y, \varLambda)$ has almost bounded derivatives (see \eqref{6eq: bounds for V, R}) is  used implicitly  to determine the $P$'s.
\end{proof}


\begin{cor}\label{cor: Hankel, C}
	Let $\varDelta > 1$ be fixed.  Let $A \geqslant 1$. Let  $|u| \sim_{\varDelta} Y$. Suppose that  $Y \Gt T^{\vepsilon} $  and $|\varLambda| \Gt T^2$. Define $\Phi  (u, \varLambda)$ by {\rm\eqref{6eq: Phi (u,...)}} and {\rm\eqref{6eq: rho and theta, natural}}. We have
	\begin{equation}\label{11eq: partition of Phi, C}
	\begin{split}
	\Phi  (u, \varLambda) =   & \mathop{ \sum_{ T^{\vepsilon  } / T < \rho \shskip < 1/\hskip -1 pt \sqrt{2 } \varDelta } }_{\rho \shskip  \approx_{\varDelta} T/2\pi |\varLambda|^{1/2} } \Phi_{\rho}^{+} (u/\varLambda) + \Phi^{+}_0 (u/\varLambda) \\
	& + \mathop{ \sum_{  \rho \shskip < 1/\hskip -1 pt \sqrt{2 } \varDelta } }_{\rho \shskip  \approx_{\varDelta} Y^{1/3}/|\varLambda|^{1/2} } \Phi^{-}_\rho (u/\varLambda) +  \Phi^{\flat} (u/\varLambda)
+  O  (T^{-A} ), 
	\end{split}
	\end{equation}  
	where  $\rho = \varDelta^{- k / 2}$ for integers $k$, $ \Phi^{\flat} (u/\varLambda) $ exists only when $ |\varLambda | \approx_{\shskip \varDelta} T^2 / 2 \pi^2 $ and $ Y \approx_{\shskip \varDelta} T^{3} /8\pi^3 $, 
	 $\Phi_{\rho}^{+} (z)$, $ \Phi^{-}_\rho (z) $, and $ \Phi^{\flat} (z) $ are integrals of the form
	\begin{align}\label{11eq: Phi integrals}
	  \viint e ({2T r} /\pi - 2 \Re  (z \shskip \trh^{\snatural} (r, \omega) )     )  g  ( M r) V ( r, \omega ) \nd r \shskip \nd \omega, 
	\end{align}
	with weight functions $V =  V_{\rho}^{+} $, $V^{-}_\rho$, and $V^{\flat}$ supported in 
	\begin{align}\label{11eq: support V, +}
  \cos^2 \omega \approx_{\shskip \varDelta} Y^{2/3} /|\varLambda| , \qquad \sqrt{r^2 + \sin^2 \omega} \sim_{\shskip \varDelta} \rho, 
	\end{align}
	\begin{align}\label{11eq: support V, -}
 	\sqrt{r^2 + \cos^2 \omega} \sim_{\shskip \varDelta} \rho , 
	\end{align}
	and
	\begin{align}\label{11eq: support V, flat}
 |\cos 2 \omega| \leqslant  \varDelta - 1,
	\end{align} respectively, satisfying
	\begin{align}\label{11eq: estimates for V, C}
	\frac {\partial^{i +  j}  V^{\pm}_\rho (r, \omega) }  {   \partial r^{\shskip i} \partial \omega^{\shskip j} } \Lt_{   i, \shskip j } \frac {\log^{i+j} T} {\rho^{ i + j}}, \qquad  \frac {\partial^{i +  j} V^{\flat} (r, \omega) }  {   \partial r^{\shskip i} \partial \omega^{\shskip j} } \Lt_{   i, \shskip j } \log^{i+j} T, 
	\end{align}
	and $\Phi^{+}_0 (z)$ has bounds
	\begin{align}\label{11eq: bounds for E+, C}
x^{\shskip i}	\frac {\partial^{i + j}  \Phi^{+}_0 (x e^{i\phi}) }  {   \partial x^{\shskip i} \partial \phi^{\shskip j} } \Lt_{   i, \shskip j} \frac { T^{(i+j+1)\vepsilon} } {\max \big\{ T^{2 - \vepsilon} ,  {x} \big\} } . 
	\end{align}
\end{cor}

Clearly, the integral $\Phi^{\flat} (z)$ has no counterpart in the real case. 
A similar partition on $\Phi^{\flat} (z)$ will be needed, but it seems  more appropriate to introduce it when we apply the Mellin technique  in \S \ref{sec: Hankel, II}. 

\begin{proof}
We start with dividing the $\omega$-integral via a smooth partition of $[0, \pi]$ into the union of three  regions where the inequalities
	\begin{align*}
	\sin \omega \leqslant \frac 1 {\sqrt {2 \varDelta} } , \qquad |\cos \omega | \leqslant \frac 1 {\sqrt {2 \varDelta} }, \qquad |\cos 2\omega| \leqslant \varDelta - 1,
	\end{align*}
 are valid,	respectively. 
	
	The second integral turns into the sum of $\Phi^{-}_\rho (u/\varLambda)$ after employing a $\varDelta$-adic partition with respect to $   \sqrt{ r^2 + \cos^2 \omega} $. Note that \eqref{6eq: support of V} amounts to  the condition $ \rho  \approx_{\shskip \varDelta}  Y^{1/3}/|\varLambda|^{1/2} $ in \eqref{11eq: partition of Phi, C}, and \eqref{6eq: bounds for V} is required to deduce the estimates for $ V^{-}_\rho (r, \omega) $ in \eqref{11eq: estimates for V, C}.
	
	It remains to analyze the first and the third integrals.  Keep in mind that because of \eqref{6eq: support of V}  we have necessarily $ \cos^2 \omega \approx_{\shskip \varDelta} Y^{2/3} / |\varLambda| $ in the first case and $  Y^{2/3} \approx_{\shskip \varDelta}  |\varLambda| / 2 $ in the third case. Moreover,  \eqref{6eq: bounds for V} manifests that the weight functions have almost bounded derivatives  as $\rho (r, \omega ) \Gt 1$ for both cases. 
	
	Let   $z = u / \varLambda$ and write $z = x e^{i \phi}$. The phase function in \eqref{6eq: Phi (u,...)} or \eqref{11eq: Phi integrals} is equal to
	\begin{align*}
	f(r, \omega; x, \phi) =   {T r} / { \pi} - {x \rho^{\snatural} (r, \omega) \cos (\phi + \theta^{\snatural}  (r, \omega) ) }  
	\end{align*}
Set $\psi   = \phi + \theta^{\snatural}  (r, \omega)$ for brevity. In view of \eqref{11eq: r-derivative trh2} and \eqref{11eq: w-derivative trh2}, we have 
	\begin{align*}
	&    {\partial f } / {\partial r} =   T / \pi + x (A \cos \psi + B \sin \psi)   , \quad   {\partial f } / {\partial \omega} =  x (B \cos \psi - A \sin \psi)  ,
	\end{align*}
	where
	\begin{align*}
	A = \frac { 4 \sinh 2r \cos 2\omega  } {(\cosh 2 r +   \cos 2 \omega)^2}, \quad  B = \frac { 4 \cosh 2r \sin 2\omega } {(\cosh 2 r +   \cos 2 \omega)^2}.
	\end{align*} 
	It is clear that
	\begin{align}\label{11eq: f' square}
	   ({\partial f } / {\partial r})^2 + ({\partial f } / {\partial \omega})^2 
	& \geqslant \big(  T / {\pi} - x \sqrt{A^2+B^2}\big)^{\hskip -1pt 2},
	\end{align}
in which  
\begin{align}\label{11eq: A2 + B2}
A^2 + B^2 = \frac {16 (\cosh 2 r - \cos 2\omega)} {(\cosh 2 r + \cos 2\omega)^3} = \frac {4 (\sinh^2 r + \sin^2 \omega)} {(\cosh^2 r - \sin^2 \omega)^3} . 
\end{align}

In the first case, \eqref{11eq: f' square} and \eqref{11eq: A2 + B2} imply that
\begin{align*}
|f'(r, \omega; x, \phi)|^2 \Gt T^2 + x^2   { (r^2 + \sin^2 \omega) }  , 
\end{align*}
except for $ r^2 + \sin^2 \omega \approx_{\shskip \varDelta} T^2 / 4\pi^2 |\varLambda| $ (since $\cos^2 \omega \approx_{\shskip \varDelta} Y^{2/3} / |\varLambda|$ and $x \sim_{\shskip \varDelta} Y/ |\varLambda|$). By Lemma \ref{lem: derivatives of rho natural}, we have
\begin{align*}
\frac {\partial^{i + j}  f (r, \omega; x, \theta) }  {   \partial r^{\shskip i} \partial \omega^{\shskip j} } \Lt_{\shskip i, \shskip j } \frac {x  } { (r^2 + \sin^2 \omega)^{(i+j)/2 - 1} }
\end{align*}
for $i+j \geqslant 2$.  We truncate smoothly the first integral at $ \sqrt {r^2 + \sin^2 \omega} = T^{\vepsilon} / T $ and apply a $\varDelta$-adic partition of unity with respect to the value of $ \sqrt{r^2 + \sin^2 \omega}$ over $(T^{\vepsilon}/T, 1 / \sqrt{2 \varDelta})$. In this way, the integral splits into $\sum_{ \rho} \Phi_{\rho}^{+} (z) + \Phi^{+}_0 (z)$.  On applying Lemma \ref{lem: staionary phase, dim 2, 2} with $ Q = \varPhi = \varUpsilon = \rho $, $P = \min \{ \rho, 1/ M\}$,  $Z = x \rho^2 $, and $R = T + x \rho $,  we infer that  the integral $\Phi_{\rho}^{+} (z)$ is negligibly small unless $  \rho    \approx_{\varDelta} T/2\pi \sqrt{|\varLambda|}  $  ($\rho T > T^{\vepsilon} $ is required). When $ x \leqslant T^{2 - \vepsilon} $, the estimates for $\Phi^{+}_0 (x e^{i\phi})$ in \eqref{11eq: bounds for E+, C} follow from trivial estimation. When $ x > T^{2 - \vepsilon} $, we may further restrict the integration to $ \sqrt {r^2 + \sin^2 \omega} \leqslant T^{\vepsilon} / \sqrt {x} $. To see this, we absorb $e (2 T r / \pi)$ into the weight function, and apply Lemma \ref{lem: staionary phase, dim 2, 2} with $ Q = \varPhi = P = \varUpsilon = T^{\vepsilon} /\hskip -1 pt \sqrt {x}   $, $Z = T^{\vepsilon} $, and $R = \sqrt{x} T^{\vepsilon}$. Again,  \eqref{11eq: bounds for E+, C} follows trivially. 
	
	For the third case, it is left to prove that the integral restricted to $|\cos 2\omega| \leqslant \varDelta - 1$ is negligible unless $ x \approx_{\shskip \varDelta} T / 4 \pi $. To this end, observe that $\sqrt{A^2 + B^2} \approx_{\shskip \varDelta} 4$, and it follows from \eqref{11eq: f' square} that 
	\begin{align*}
	|f'(r, \omega; x, \phi)|^2 \Gt T^2 + x^2
	\end{align*}
	unless $ x \approx_{\shskip \varDelta} T / 4 \pi $. By Lemma \ref{lem: derivatives of rho natural},  we have
	\begin{align*}
	\frac {\partial^{i + j}  f (r, \omega; x, \theta) }  {   \partial r^{\shskip i} \partial \omega^{\shskip j} } \Lt_{\shskip i, \shskip j }   {x  }  
	\end{align*}
	for $i+j \geqslant 2$. The proof is completed by applying Lemma \ref{lem: staionary phase, dim 2, 2} with $ Q = \varPhi = 1$,  $P = 1 /M $,   $ \varUpsilon = 1 $, $Z = x $, and $R = T + x$. 
\end{proof}

\section{Stationary Phase for the Mellin Transforms}

In this section, we fix a smooth function $\varvv (x)$ such that $\varvv (x) \equiv 1$ on $ [  1 / 2, 2  ]$ and  $\varvv (x) \equiv 0$ on $(0,   1 / 3] \cup [ 3, \infty)$.

\subsection{The Real Case} As in \S \ref{sec: notation}, let $\widehat{\bfra} = \BR \times \{0, 1\}$ and define $\vchi_{  i \varnu, \shskip  m} (x) = |x|^{i\varnu} (x/|x|)^m$ for $(\varnu, m) \in \widehat{\bfra}$. The Mellin transform of  $f (x) \in C_c^{\infty} (\BR^{\times})$ is defined by
\begin{align*}
\breve {f} (\varnu, m) = \int_{ \BR^{\times} } f (x) \vchi_{  i \varnu, \shskip  m} (x) \nd^{\times} \hskip -1pt x,  
\end{align*}
and the Mellin inversion reads
\begin{align*}
	f (x) = \frac 1 {4\pi  } \viint_{\widehat{\bfra} }  \breve{f} (\varnu, m)  \overline{\vchi_{  i \varnu, \shskip  m} (x)} \nd \mu (\varnu, m). 
\end{align*}

\begin{lem}\label{lem: Mellin, inert, R} Let $R, S> 0$ and $X \geqslant 1$. Suppose that $ \varww (x) $ is   smooth and $ x^{i} \varww^{(i)} (x) \Lt_{  \shskip i } S X^{i  }$ for $|x| \in [R/3, 3R]$. We have 
	\begin{align*}
	\varww (x) =   \viint_{\widehat{\bfra} } \xiup  (\varnu, m)   \overline{\vchi_{  i \varnu, \shskip  m} (x)} \nd \mu (\varnu, m),
	\end{align*}
whenever $|x| \in [R/2, 2R]$, with the function $ \xiup  (\varnu, m) $ satisfying $\xiup  (\varnu, m) \Lt S$ and $\xiup  (\varnu, m) = O  (R S T^{-A}  )$ if $|\varnu| > T^{\vepsilon} X$. 
\end{lem}

\begin{proof}
Let $\xiup  (\varnu, m) $ be the Mellin transform of $4\pi \shskip \varvv (|x|/ R) \varww (x)$.  The first estimate for $\xiup  (\varnu, m) $ is trivial.	The second is an easy consequence of Lemma \ref{lem: staionary phase, dim 1, 2} with phase function $\varnu \log |x| / 2\pi$.
\end{proof}

\begin{lem}\label{lem: Mellin, R} Let $ R > 1$. 
	We have 
	\begin{align*}
	e (x) =   \viint_{\widehat{\bfra} } \xiup_{R}  (\varnu, m)   \overline{\vchi_{  i \varnu, \shskip  m} (x)} \nd \mu (\varnu, m),
	\end{align*}
	whenever $|x| \in [R/2, 2R]$, where  $\xiup_R   (\varnu, m) = O \big((R + |\varnu| )^{-A}\big)$ unless $|\varnu| \asymp R$, in which case $\xiup_R  (\varnu, m) \Lt 1/ \sqrt{R}$. 
\end{lem}

\begin{proof}
Let $\xiup_R (\varnu, m) $ be the Mellin transform of $4\pi  \shskip  \varvv (|x|/ R) e (x)$. To derive the estimates for $\xiup_R (\varnu, m) $, apply Lemma \ref{lem: staionary phase, dim 1, 2} and \ref{lem: 2nd derivative test, dim 1} (the second derivative test) with phase function $x + \varnu \log |x|/2\pi$.  
\end{proof}


\subsection{The Complex Case} As in \S \ref{sec: notation}, let $\widehat{\bfra} = \BR \times \BZ$ and define $\vchi_{  i \varnu, \shskip  m} (z) = |z|^{2 i\varnu} (z/|z|)^m$ for $(\varnu, m) \in \widehat{\bfra}$. The Mellin transform of  $f (z) \in C_c^{\infty} (\BC^{\times})$ is defined by
\begin{align*}
\breve {f} (\varnu, m) = \int_{ \BC^{\times} } f (z) \vchi_{  i \varnu, \shskip  m} (z) \nd^{\times} \hskip -1pt z,  
\end{align*}
and the Mellin inversion reads
\begin{align*}
f (z) = \frac 1 {4\pi^2  } \viint_{\widehat{\bfra} }  \breve{f} (\varnu, m)  \overline{\vchi_{  i \varnu, \shskip  m} (z)} \nd \mu (\varnu, m). 
\end{align*}
In the polar coordinates, 
\begin{align*}
\breve {f} (\varnu, m) = 2  \hskip -1 pt \int_0^{\infty} \hskip -1 pt \int_{0 }^{2\pi} f (x e^{i\phi}) x^{2i\varnu} e^{i m\phi} \frac {\nd  \phi \shskip \nd x} {x}. 
\end{align*}

\begin{lem}\label{lem: Mellin, inert, C} Let $R, S > 0$  and $X \geqslant 1$. Let $ \varww (z) $ be   smooth with $ x^{i} \partial_x^i \partial_{\phi}^j \varww (x e^{i\phi}) \Lt_{  \shskip i, \shskip j } S X^{i+j }$ for $ x \in [R/3, 3R]$. We have 
	\begin{align*}
	\varww (z) =   \viint_{\widehat{\bfra} } \xiup  (2 \varnu, m)   \overline{\vchi_{  i \varnu, \shskip  m} (z)} \nd \mu (\varnu, m),
	\end{align*}
	whenever $|z| \in [R/2, 2R]$, with the function $ \xiup  (\varnu, m) $ satisfying  $\xiup  (\varnu, m) \Lt S$ and $\xiup  (\varnu, m) = O  (R^2 S T^{-A} )$ if $ \sqrt{\varnu^2 + m^2} > T^{\vepsilon} X$. 
\end{lem}

\begin{proof}
Let $\xiup  (2 \varnu, m) $ be the Mellin transform of $4\pi^2 \varvv (|z|/ R) \varww (z)$. In the polar coordinates, apply Lemma \ref{lem: staionary phase, dim 1, 2} to the $x$- or $\phi$-integral with phase function $\varnu \log x / 2 \pi$ or $ m \phi / 2\pi$, respectively.
\end{proof}

The complex analogue of Lemma \ref{lem: Mellin, R} is as follows. However, its proof requires considerably more work.

\begin{lem}\label{lem: Mellin, C} Let $ R \Gt 1$. 	We have 
	\begin{align*}
e ( 2 \Re (z) ) =   \viint_{\widehat{\bfra} } \xiup_{R}  (2 \varnu, m)   \overline{\vchi_{  i \varnu, \shskip  m} (z)} \nd \mu (\varnu , m),
	\end{align*}
	whenever $|z| \in [R/2, 2R]$, where    $\xiup_R   (\varnu, m) = O \big((R + |\varnu| + |m|)^{-A}\big)$ unless $ \sqrt{\varnu^2 + m^2} \asymp R$, in which case $\xiup_R  (\varnu, m) \Lt \log R /  {R}$. 
\end{lem}

Let $\xiup_R  (2\varnu, m) = \xiup  (2\varnu, m)$ be the Mellin transform of $4\pi^2 \varvv (|z|/ R) e ( 2 \Re (z) )$.  Write 
\begin{align*}
\xiup (\varnu, m) = 8 \pi^2 R^{i\varnu}   \int_0^{\infty} \hskip -2 pt \int_{0 }^{2\pi}  \varvv (x  ) e ( f  (x, \phi; \varnu, m )) \frac  {\nd \phi \shskip \nd x} {x} ,
\end{align*}	
with 
\begin{align*}
f (x, \phi) =	f  (x, \phi; \varnu, m) = 2 R x \cos \phi + (\varnu \log x + m \phi  ) / 2 \pi .
\end{align*}
We have
\begin{align}\label{12eq: f'(x, phi)}
f' (x, \phi ) = (2 R \cos \phi + \varnu/2 \pi x, -2 R x \sin \phi + m / 2 \pi),
\end{align}
and hence there is a unique stationary point $(x_0, \phi_0)$ given by
\begin{align*}
x_0 = \frac {\sqrt {\varnu^2 + m^2} } {4 \pi R}, \quad  \cos \phi_0 = - \frac {\varnu} {\sqrt {\varnu^2+m^2}}, \quad  \sin \phi_0 =  \frac {m} {\sqrt {\varnu^2+m^2}}. 
\end{align*} 

\subsubsection{Applying H\"ormander's Partial Integration}
First, we prove $ \xiup (\varnu, m) = O  \big( (R + |\varnu| + |m|)^{-A}\big) $ for any $A \geqslant 1$ unless $ x_0 \in [1/4, 4] $, say. 
The arguments below are similar to those in \cite[\S 6.1]{Qi-Gauss}. Our idea is to modify H\"ormander's elaborate partial integration. To this end, we introduce
\begin{equation}\label{12eq: defn of g}
\begin{split}
g ( x, \phi ) & = x ( \partial_x f ( x, \phi ))^2  + (1/x) (\partial_x f ( x, \phi ))^2     \\
& = \bigg(2 R \sqrt x - \frac{\sqrt{\varnu^2 + m^2}} {2\pi \sqrt x } \bigg)^2 + \frac {2 R } {\pi } \big( \hskip -2 pt {\textstyle \sqrt{\varnu^2 + m^2} } + \varnu \cos \phi - m  \sin \phi \big) . 
\end{split}
\end{equation}
It is clear that 
\begin{equation}\label{12eq: lower bounds for g}
 g ( x, \phi ) \Gt \left\{ \begin{aligned}
& R^2 x , & & \text{ if } x \geqslant 4 x_0/3, \\
& (\varnu^2 + m^2) / x, & & \text{ if } x \leqslant  3 x_0 / 4.  
 \end{aligned} \right.  
\end{equation}  
Define  the differential operator
\begin{align*}
\mathrm{D}  =       \frac {  x  \partial_x f( x, \phi   )} {  g ( x, \phi   ) } \frac {\partial} {\partial x}      +    \frac {  \partial_{ \phi  } f( x, \phi  )} {  x g ( x, \phi  ) } \frac {\partial} {\partial \phi  }   
\end{align*}
so that $\mathrm{D}  (  e (    f  ( x, \phi  )    ) )   =  2 \pi i \cdot       e (    f  (x, \phi  ) )$; its  adjoint operator is given by
\begin{align*}
\mathrm{D}^* \hskip -2 pt = - \frac 1 {2\pi i} \bigg( \frac {\partial} {\partial x}   \frac {  x  \partial_x f  ( x , \phi  )} {   g (x , \phi ) }  + \frac {\partial} {\partial \phi }   \frac {   \partial_{\phi} f (x , \phi )} { x g (x , \phi ) } \bigg)  ,
\end{align*} 
and
\begin{align*}
\xiup (\varnu, m) =    \frac {8 \pi^2 R^{i\varnu}} {(2\pi i)^n}  \hskip -1 pt \int_{0}^{\infty} \hskip -2 pt  \int_{0 }^{2\pi} \mathrm{D}^{* \shskip  n}   (\varvv (x) /x) e      ( f  ( x , \phi    ) )         \nd \phi \shskip \nd x . 
\end{align*}
For integer $n \geqslant 0$,  $ \mathrm{D} ^{* \shskip  n}   (\varvv (x) /x)  $ is a linear combination of all the terms occurring in  the  expansions of
\begin{align*}
\partial_{x}^{i } \partial_{\phi  }^{\shskip  j }  \mbox{\larger[1]\text{${\big\{}$}} \hskip -1 pt  (  x  \partial_{x} f  ( x , \phi  )   )^{i}    (    \partial_{\phi } f ( x , \phi ) /x   )^{ j} g( x , \phi )^{n}   \shskip  (\varvv (x) /x) \mbox{\larger[1]\text{${\big\}}$}} / g( x , \phi )^{ 2 n}, \quad  i +  j = n.
\end{align*}
Moreover,  we have
\begin{align*}
 & x  \partial_{x} f  ( x , \phi  ) \Lt R x + |\varnu|, \quad   \partial_{\phi}^{j+1}   (  x  \partial_{x} f  ( x , \phi  )   ) \Lt Rx, \quad \partial_x \partial_{\phi}^j    (  x \partial_{x} f  ( x , \phi  )  ) \Lt R,   \\ 
 &      \partial_{\phi} f  ( x , \phi  ) /x   \Lt R + |m|/x, \ \partial_x^{i+1}   (   \partial_{\phi} f  ( x , \phi  ) /x   ) \Lt |m|/x^{i+2},  \  \partial_{\phi}^{j+1}  (    \partial_{\phi} f  ( x , \phi  ) / x   ) \Lt R, \\
 & x^2 \partial_x  g ( x , \phi  )   \Lt R^2 x^{2} \hskip -1 pt + \hskip -1 pt \varnu^2 \hskip -1 pt + \hskip -1 pt m^2, \ x^{i+3} \partial_x^{i+2}  g ( x , \phi  )   \Lt   \varnu^2 \hskip -1 pt + \hskip -1 pt m^2,   \  \partial_{\phi}^{j+1}   g ( x , \phi  )   \Lt \hskip -1 pt \sqrt{ \varnu^2 \hskip -1 pt + \hskip -1 pt m^2 },   \\
 & \partial_x^2   (  x\partial_{x} f  ( x , \phi  )   ) = 0, \quad \partial_x  \partial_{\phi}  (    \partial_{\phi} f  ( x , \phi  ) /x  ) = 0, \quad \partial_x  \partial_{\phi} g ( x , \phi  )   = 0,
\end{align*}
for $i, j \geqslant 0$. 
Now assume that $x \in [1/3, 3]$ and $  x_0 \notin [1/4, 4] $. Then \eqref{12eq: lower bounds for g} yields
\begin{align*}
g (x, \phi) \Gt R^2 + \varnu^2 + m^2.
\end{align*} 
Let $ i_1, i_2 \leqslant i$ and $j_1, j_2 \leqslant j $. From the estimates above, it is straightforward
to prove that
\begin{align*}
\partial_{x}^{i_1} \partial_{\phi}^{j_1} \mbox{\larger[1]\text{${\big\{}$}} \hskip -1 pt  (  x  \partial_{x} f  ( x , \phi  )   )^{i}    (    \partial_{\phi } f ( x , \phi ) /x   )^{ j}   \mbox{\larger[1]\text{${\big\}}$}} \Lt (R+|\varnu|)^i (R + |m|)^j ,
\end{align*}
and 
\begin{align*}
\frac { \partial_{x}^{i_2} \partial_{\phi}^{j_2}  g(x, \phi)^n}{g(x, \phi)^{2n}}    \Lt  \sum_{k_1 + 2k_2 \leqslant i_2} \sum_{l \leqslant j_2} \frac {(R^2+\varnu^2+m^2)^{k_1} (\varnu^2+m^2)^{k_2 + l/2} } {g(x, \phi)^{n+k_1+k_2+l}} . 
\end{align*}
Combining these, we conclude that
\begin{align*}
\xiup (\varnu, m) & \Lt   \sum_{k_1+2k_2+l \shskip \leqslant n} \int_{1/3}^3 \int_0^{2\pi} \frac { (R+|\varnu| +|m|)^{n + 2 (k_1 + k_2) + l }  } {g(x, \phi)^{n+k_1+k_2+l}} \nd \phi \shskip \nd x \\
 & \Lt \sum_{ l \shskip \leqslant n} \frac 1 { (R+|\varnu| +|m|)^{n + l} } \\
 & \Lt  \frac 1 { (R+|\varnu| +|m|)^{n} }, 
\end{align*}
as desired. 

\subsubsection{Applying Olver's Uniform Asymptotic Formula} \label{sec: Olver}
Next, we need to prove the bound  $\xiup  (\varnu, m) \Lt \log R/  {R}$ when $x_0 \in [1/4, 4]$. This may be easily deduced from the same bound for the unweighted integrals as follows. 

\begin{lem}\label{lem: I < 1/R}
	Suppose that  $x_0 \in [1/4, 4]$. For $b > a > 0$, define \begin{align*}
	I (a, b ) = \int_a^{b} \hskip -2 pt \int_{0 }^{2\pi}  e (2Rx\cos \phi) e^{  i m \phi } x^{i\varnu -1}  \nd \phi  \shskip  { \nd x} .
	\end{align*} Then for any $b > a \geqslant 1/8$, we have $ I (a, b) \Lt \log R/R $, where the implied constant is absolute.  
\end{lem}

Firstly, we  write 
\begin{align*}
 I (a, b ) = \int_a^{b} \hskip -2 pt \int_{0 }^{2\pi}  e ( f  (x, \phi  )) \frac  {\nd \phi \shskip \nd x} {x} ,
\end{align*}
and apply H\"ormander's elaborate partial integration once, obtaining
\begin{align*}
&  \frac 1 {2\pi i} \int_{0 }^{2\pi} \bigg( \frac {    \partial_x f( b, \phi   )} {  g ( b, \phi   ) } e ( f  (b, \phi  )) - \frac {    \partial_x f( a, \phi   )} {  g ( a, \phi   ) } e ( f  (a, \phi  )) \bigg) \nd \phi \\
     - \, &  \frac 1 {2\pi i} \int_a^{b} \hskip -1 pt \int_{0 }^{2\pi} \bigg( \frac {\partial} {\partial x}  \bigg( \frac {    \partial_x f  ( x , \phi  )} {   g (x , \phi ) } \bigg) + \frac 1 {x^2} \frac {\partial} {\partial \phi }  \bigg(  \frac {   \partial_{\phi} f (x , \phi )} {   g (x , \phi ) } \bigg) \bigg) e ( f  (x, \phi  )) \shskip \nd x \shskip \nd \phi .
\end{align*}
For $a \geqslant 2 x_0$\,($\geqslant 1/2$), one uses \eqref{12eq: f'(x, phi)}, \eqref{12eq: defn of g}, and the first lower bound in \eqref{12eq: lower bounds for g}  to bound this by $ 1/ a R  \Lt  1 / R$. 
The case when $ b \leqslant x_0/2 $ is similar, for which we use  the second lower bound in \eqref{12eq: lower bounds for g}. 

The problem is thus reduced to the case when $  2 x_0 \geqslant b > a \geqslant x_0/2 $. Assume    $m \geqslant 0$ for simplicity. We invoke the integral representation of Bessel for $J_{m} (z)$  as follows (see \cite[2.2 (1)]{Watson}),
\begin{align*}
&J_{m} (z) = \frac { 1 } {2 \pi i^{\hskip 0.5 pt m}} \int_0^{2\pi} e^{ i z \cos \phi + i m \phi } \nd \phi ,
\end{align*}
and hence
\begin{align}\label{12eq: I = int of J}
 I (a, b ) = 2\pi i^{m} \int_a^{b} J_{m} (4\pi R x) x^{i\varnu -1}    \shskip  { \nd x} . 
\end{align}
According to \cite[\S 7.13.1]{Olver}, 
\begin{align*}
J_{m} (x) = \Big(\frac 2 {\pi x} \Big)^{1/2}   \cos \Big(x- \frac {\pi m} 2 - \frac {\pi } 4\Big)   + O \bigg(\frac {m^2+1} {x^{3/2}} \bigg), \qquad x \Gt m^2+1. 
\end{align*} 
It follows that if  $m \leqslant R^{1/4}$, then $I (a, b; \varnu, m) = O (1/\hskip -1 pt \sqrt{R|\varnu|}) = O(1/R)$ by Lemma \ref{lem: 2nd derivative test, dim 1}.  

For $ m  > R^{1/4}$, we employ Olver's uniform asymptotic formula, in particular, Lemma \ref{lem: Olver} in Appendix \ref{app: Olver}.  

Recall that $x_0 =   {\sqrt {\varnu^2 + m^2} } / {4 \pi R}$. For brevity, set $c = 4 \pi R / m$, $x_0' = m / 2 \pi R$, and $x_0^{+} = (m + m^{1/3}) / 4\pi R$.

When $ R^{1/4} < m \leqslant \pi R x_0   $, so that $2 \leqslant c x \Lt m^{3} $ for all $x \in [x_0/2, 2 x_0]$,  by Lemma \ref{lem: Olver} (4), the integral in \eqref{12eq: I = int of J} turns into
\begin{align*}
I (a, b ) =   \sum_{ \pm}  \int_a^b e (f_{\pm} (x ) / 2\pi) \varww_{\pm} (x ) \nd x + O  (1/R), 
\end{align*}
with 
\begin{align*}
f_{\pm} (x ) = \pm m \gamma (c x) + \varnu \log x ,  
\end{align*}
and 
\begin{align*}
\varww_{\pm} (x ) =  2 \sqrt{2} \pi i^{m} \frac{W_{\pm} (m \gamma (c x))} { m^{1/2} ( (c x)^2 - 1 )^{1/4} x}, 
\end{align*}
in which $\gamma (x) = \sqrt{x^2-1} - \mathrm{arcsec} x$. We have
\begin{align*}
\gamma ' (x) = \frac {\sqrt{x^2-1}} {x}, \qquad \gamma '' (x) = \frac 1 {x^2 \sqrt{x^2-1}}. 
\end{align*}
Therefore
\begin{align*}
 f_{\pm}' (x ) = \frac 1 {x} \big( \hskip -1 pt \pm m {\textstyle \sqrt{(cx)^2 - 1}} + \varnu  \big), \quad   f_{\pm}''  (x ) = \frac 1 {x^2} \bigg( \hskip -1 pt \pm \frac m {  \sqrt{(cx)^2 - 1}} - \varnu  \bigg).
\end{align*} 
Moreover, 
\begin{align*}
 \varww_{\pm} (x ), \ \varww_{\pm}' (x ) \Lt \frac 1 {\sqrt{R}};
\end{align*}
note that $\gamma (x) = x + O(1)$ for $x \geqslant 2$ (see \eqref{12eq: zeta large}). 
If $ \pm \varnu > 0$, then $ |f_{\pm}' (x )| \Gt m + |\varnu| \Gt R $, and the integral is  $O \big( 1 / R^{3/2} \big)$ by partial integration (the first derivative test). If $ \pm \varnu < 0$, then $ |f_{\pm}'' (x )| \Gt m + |\varnu| \Gt R $, and the integral is $ O (1/R)$ by the second derivative test in Lemma \ref{lem: 2nd derivative test, dim 1}. 

Suppose now $   x_0 / 2 < x_0' $ so that we  have necessarily $ m \asymp R$. When $ \max \big\{ x_0/2, x_0^{+} \big\} \leqslant a < b \leqslant x_0' $,  we apply Lemma \ref{lem: Olver} (3) and then divide the integral in \eqref{12eq: I = int of J} by a dyadic partition with respect to $c x - 1$; the error term is $O (1/m) = O (1/R)$, and the resulting integrals can be treated in a manner similar to the above. We just need to notice that $ m / \sqrt{(cx)^2 - 1} $ would dominate $\varnu$ in $f_{\pm}''(x)$ when $ cx - 1 $ is small, in which case only Lemma  \ref{lem: 2nd derivative test, dim 1} is applied. However, by doing the  dyadic partition, we might lose a $\log R$. 

Finally, assume that $ x_0/2 < x_0^{+}$, and consider the case when $ x_0/2 \leqslant a < b \leqslant x_0^+ $. We use Lemma \ref{lem: Olver} (1) and (2) to bound the integral $I(a, b)$ as  in \eqref{12eq: I = int of J} by
\begin{align*}
& \Lt \frac 1 {m^{1/3}} \int_{1/c}^{x_0^+}   {\nd x }   + \frac 1 {m^{1/3}} \int_{1/2c}^{1/c} \exp  \big(  \hskip -1 pt -  \tfrac 1 3 m (2-2cx))^{3/2}   \big)   {\nd x }   \\
& \Lt \frac 1 {m^{1/3}} \int_0^{  1/m^{2/3}}   {\nd y }  + \frac 1 {m^{1/3}} \int_0^{1}  \exp \big(  \hskip -1 pt -  \tfrac 1 3 m y^{3/2}   \big) \nd y \\
& \Lt \frac 1 {m}.
\end{align*}

\begin{rem}
	The $\log R$ in Lemma {\rm\ref{lem: Mellin, C}} or {\rm\ref{lem: I < 1/R}} could be removed on applying the stationary phase method {\rm(}Lemma {\rm 5.5.6} in {\rm\cite{Huxley}} for example{\rm)} instead of the second derivative test, as revealed by the formula
	\begin{align*}
		|I(0, \infty) | = \frac {2 \pi} {\sqrt{m^2 + \varnu^2}}, 
	\end{align*}
for $m \neq 0${\rm;} this may be seen from 
	\begin{align*}
	I (0, \infty) = \int_0^{\infty} \hskip -2 pt \int_{0 }^{2\pi}  e (2   R x \cos \phi ) e^{  i m \phi } x^{i\varnu -1}  \nd \phi  \shskip  { \nd x} = \frac { \pi i^{|m|} \Gamma \big(\frac 1 2 (|m|+i\varnu)\big)} {(2\pi R)^{i \varnu} \Gamma  \big(\frac 1 2 (|m|-i\varnu) + 1\big)} ,
	\end{align*}  which is a consequence of Weber's integral formula   in {\rm\cite[13.24 (1)]{Watson}}. 
\end{rem}

\section{Analysis of the Hankel Transforms, II}\label{sec: Hankel, II}
 
In this final analytic section, our primary object is to use the Mellin technique and the stationary phase method to analyze the $\Phi$-integrals in \S \ref{sec: analysis of Phi-interals}. We remind the reader that the expressions of these $\Phi$-integrals depend only mildly on $M$ and $\varLambda$.

\begin{defn}\label{defn: a (U), local}
Let $U \Gt 1 $ and  $ (\kappa , n) \in \widehat{\bfra} $. Define
\begin{align}
\widehat{\bfra}  (U) \hskip -1 pt = \hskip -1 pt \big\{ (\varnu, m) \in \widehat{\bfra} \hskip -1 pt : \hskip -1 pt {\textstyle \sqrt{\varnu^2 + m^2}}  \Lt U \big\}, \quad \widehat{\bfra}' (U) \hskip -1 pt = \hskip -1 pt \big\{ (\varnu, m) \in \widehat{\bfra} \hskip -1 pt : \hskip -1 pt {\textstyle \sqrt{\varnu^2 + m^2}}  \asymp U \big\},
\end{align}
and
\begin{align}
\widehat{\bfra}_{\kappa, \shskip n} (U) = \big\{ (\varnu, m) \in \widehat{\bfra} : {\textstyle \sqrt{(\varnu - \kappa)^2 + (m - n)^2}}  \Lt U \big\}. 
\end{align} 
\end{defn}

For convenience, we shall not distinguish $\xiup_R (\nu, m)$ and $\xiup_U (\nu, m)$  when $R \asymp U$; see Lemma \ref{lem: Mellin, R} and \ref{lem: Mellin, C}. 

\subsection{The Real Case}

\begin{lem}\label{lem: final, R}
Fix a constant $\varDelta > 1$ with $\log \varDelta$ small.	Let  $|\varLambda| \Gt T^2$. Let $|x| \sim_{\shskip \varDelta} X$. Let $\Phi^{\oldstylenums{0}} (x)$,  $\Phi^{-} (x)$ and $ \Phi^{+}  (x)   $   be given as in Corollary {\rm\ref{cor: Hankel y<1}} and {\rm\ref{cor: Hankel, R}}.  For  $\sigmaup = \oldstylenums{0}, -, +$, we have
	\begin{align}\label{13eq: Phi = Mellin, R}
	\Phi^{\sigmaup } (x) = \frac {T^{\vepsilon}} {\sqrt{A^{\sigmaup}}} \viint_{\widehat{\bfra}  (U^{\sigmaup})} \lambdaup^{\sigmaup} (\varnu, m)  {\vchi_{  i \varnu, \shskip  m} (x)} \nd \mu (\varnu, m) + O (T^{-A}), 
	\end{align}
for 
\begin{align}\label{13eq: ranges of X}
\begin{aligned}
& X \Lt T^{\vepsilon} / |\varLambda|,  & & \text{ if } \sigmaup = \oldstylenums{0}, \\
& T^{\vepsilon} / |\varLambda| \Lt X \Lt  {\textstyle \sqrt{|\varLambda|}} / M^{3-\vepsilon},  & & \text{ if } \sigmaup = -, \\
& X \approx_{\varDelta} {\textstyle \sqrt{|\varLambda|}}  \Gt M^{1-\vepsilon} T, & & \text{ if } \sigmaup = +, 
\end{aligned} 
\end{align} 
with
\begin{equation}
U^{\sigmaup} = \left\{ \begin{aligned}
& T^{\vepsilon},      \\
& |X \varLambda|^{1/3},    \\
& \min \big\{ T^2 / X, T^{\vepsilon} \big\},    
\end{aligned} \right. \quad A^{\sigmaup} = \left\{ \begin{aligned}
& 1,  & & \text{ if } \sigmaup = \oldstylenums{0}, \\
&{\textstyle   {|\varLambda|}},  & & \text{ if } \sigmaup = -, \\
& \min \big\{ T^{2-\vepsilon} ,  {X}    \big\}, & & \text{ if } \sigmaup = + ,
\end{aligned} \right.
\end{equation}  
and  $ \lambdaup^{\sigmaup} (\varnu, m) \Lt 1 $ for all $\sigmaup$. 
\end{lem}

Firstly, note that, according to Corollary {\rm\ref{cor: Hankel y<1}} and {\rm\ref{cor: Hankel, R}},  $ \Phi^{\sigmaup } (x) $ vanishes unless $X$ satisfies \eqref{13eq: ranges of X} in various cases. 

For $\Phi^{\oldstylenums{0}} (x)$ and  $\Phi^{+}  (x) = \Phi^{+}_0 (x)$, it is easy to establish \eqref{13eq: Phi = Mellin, R} by Lemma \ref{lem: Mellin, inert, R}, along with  \eqref{11eq: Phi0, R} and \eqref{11eq: E(x), R}.  This settles the case $\sigmaup = \oldstylenums{0}$ and partially the case $\sigmaup = +$  for $X \Gt T^{2-\vepsilon}$.

Next, we consider the integral $\Phi^{+}_1 (x)$ as defined in \eqref{11eq: Phi+1, R} for  $ M^{1-\vepsilon} T   \Lt |x| \Lt  T^{2-\vepsilon} $. Since $ |x \tanh^2 r| \approx_{\shskip \varDelta} T^2 / 4 \pi^2 X $ for $ |r| \approx_{\varDelta} T / 2 \pi X $, up to a negligible error,  we can rewrite $ \Phi^{+}_1 (x)$ using  Lemma \ref{lem: Mellin, R}  as follows, 
\begin{align*}
  \viint_{\widehat{\bfra}' (T^2/X) }   \overline{\xiup_{T^2 /   X} (\varnu, m)} {\vchi_{  i \varnu, \shskip  m} (x)} \int 
 e  ({T r} /\pi    )  \vchi_{  i \varnu, \shskip  m}( \tanh^2 r)   V^+ ( r) \nd r \, \nd \mu (\varnu, m), 
\end{align*} 
with $\xiup_{T^2 /   X} (\varnu, m) = O (\sqrt{X} / T)$. 
Write the inner integral as an exponential integral with phase   $f_+ (r) = (T r + \varnu \log |\tanh r| ) / \pi $. Note that $f_+'' (r) =  \varnu (\tanh^2 r - \coth^2 r)) / \pi   $ is of size $|\varnu| / r^2 \asymp X$. By the second derivative test (Lemma \ref{lem: 2nd derivative test, dim 1}), the  $r$-integral is $O (\log T /\sqrt{X})$, and $  \xiup_{T^2 /  X} (\varnu, m)  / \sqrt{X} = O (1 / T) $, leading to $\sqrt {A^{+}} = T $ for $X \Lt T^{2-\vepsilon}$ as claimed. 
 
Finally, let $ \Phi^{-} (x)$ be as defined in \eqref{11eq: defn of Phi-, R}. Note that  $ |r| \approx_{\varDelta} Y^{1/3} / |\varLambda|^{1/2} $ there amounts to $ |r| \approx_{\varDelta} X^{1/3} / |\varLambda|^{1/6} $ for $X = Y / |\varLambda|$, and hence $ |x \coth^2 r| \approx_{\shskip \varDelta}   |X \varLambda|^{1/3} $. By Lemma \ref{lem: Mellin, R}, up to a negligible error, the integral  $ \Phi^{-} (x)$ can be rewritten as 
\begin{align*}
 \viint_{\widehat{\bfra}' (|X \varLambda|^{1/3}) }   \overline{\xiup_{|X \varLambda|^{1/3}} (\varnu, m)} {\vchi_{  i \varnu, \shskip  m} (x)} \int 
e  ({T r} /\pi    )  \vchi_{  i \varnu, \shskip  m}(  \coth^2 r)   V^- ( r) \nd r \, \nd \mu (\varnu, m). 
\end{align*} 
with $\xiup_{|X \varLambda|^{1/3}} (\varnu, m) = O (1/|X \varLambda|^{1/6})$. Now the phase function of the inner integral is  $f_- (r) = (T r + \varnu \log |\coth r| ) / \pi $. Since $ f''_- (r) $($= - f''_{+} (r)$) is of size $ |\varnu| / r^2 \asymp  | \varLambda^2 / X  |^{1/3} $, by  the second derivative test, the  $r$-integral is $O (X^{1/6} \log T / |\varLambda|^{1/3} )$, and   $  \xiup_{|X \varLambda|^{1/3}} (\varnu, m)  \cdot  X^{1/6} / |\varLambda|^{1/3} = O (1/ \sqrt{|\varLambda|}) $, as desired. 

\begin{rem}\label{rem: case -}
	In the case $\sigmaup = -$,  $ f_- (r) $ has a stationary point at $ |\varnu| / T   \asymp   |X\varLambda|^{1/3} / T $, while $V^{-} (r)$ is supported on $ |r| \asymp X^{1/3} / |\varLambda|^{1/6} $, so a consistency check shows that $ |\varLambda| \asymp T^2 $. However, this would have been implied at an early stage when analyzing the Bessel integral $ \SDH   (- x^2) $  {\rm(}see Remark {\rm\ref{rem: real, Bessel range}}{\rm)}. 
\end{rem}

\subsection{The Complex Case} 

\begin{lem}\label{lem: final, C, 1}
	Fix a constant $\varDelta > 1$ with $\log \varDelta$ small.	Let  $|\varLambda| \Gt T^2$. Let $|z| \sim_{\shskip \varDelta} X$. Let $\Phi^{\oldstylenums{0}} (z)$,  $\Phi^{-}_{\rho} (z)$, $\Phi^{+}_{\rho} (z)$, and $\Phi^{+}_{0} (z)$   be given as in Corollary {\rm\ref{cor: Hankel y<1}} and {\rm\ref{cor: Hankel, C}}. Set $$ \Phi^{+}  (z) = \mathop{ \sum_{ T^{\vepsilon  } / T < \rho \shskip < 1/\hskip -1 pt \sqrt{2 } \varDelta } }_{\rho \shskip  \approx_{\varDelta} T/2\pi |\varLambda|^{1/2} } \Phi_{\rho}^{+} (z) + \Phi^{+}_0 (z) , \quad \Phi^{-}  (z) = \mathop{ \sum_{  \rho \shskip < 1/\hskip -1 pt \sqrt{2 } \varDelta } }_{\rho \shskip  \approx_{\varDelta} X^{1/3}/|\varLambda|^{1/6} } \Phi^{-}_\rho  (z) . $$  
	For  $\sigmaup = \oldstylenums{0}, -, + $, we have
	\begin{align}\label{13eq: Phi = Mellin, C}
	\Phi^{\sigmaup } (z) = \frac {T^{\vepsilon} } {A^{\sigmaup  }} \viint_{\widehat{\bfra}  (U^{\sigmaup})} \lambdaup^{\sigmaup} (\varnu, m)  {\vchi_{  i \varnu, \shskip  m} (z)} \nd \mu (\varnu, m) + O (T^{-A}), 
	\end{align}
	for 
	\begin{align}
	\begin{aligned}
	& X \Lt T^{\vepsilon} / |\varLambda|,  & & \text{ if } \sigmaup = \oldstylenums{0}, \\
	& T^{\vepsilon} / |\varLambda| \Lt X \Lt  {\textstyle \sqrt{|\varLambda|}} ,  & & \text{ if } \sigmaup = -, \\
	& X \asymp {\textstyle \sqrt{|\varLambda|}}  , & & \text{ if } \sigmaup = +, 
	\end{aligned} 
	\end{align} 
 	with
	\begin{equation}\label{13eq: U =, A =, C}
	U^{\sigmaup} = \left\{ \begin{aligned}
	& T^{\vepsilon},      \\
	& |X \varLambda|^{1/3},    \\
	& \min \big\{ T^2 / X, T^{\vepsilon} \big\},    
	\end{aligned} \right. \quad A^{\sigmaup} = \left\{ \begin{aligned}
	& 1,  & & \text{ if } \sigmaup = \oldstylenums{0}, \\
	&{\textstyle   {|\varLambda|}},  & & \text{ if } \sigmaup = -, \\
	& \min \big\{ T^{2-\vepsilon} ,  {X}   \big\}  , & & \text{ if } \sigmaup = + ,
	\end{aligned} \right.
	\end{equation}  
	and  $ \lambdaup^{\sigmaup} (\varnu, m) \Lt 1 $ for all $\sigmaup$. 
\end{lem}

\begin{lem}\label{lem: final, C, 2}
	Fix a constant $\varDelta > 1$ with $\log \varDelta$ small. Assume that  $M \leqslant T^{1/3}$. 	Let  $ |\varLambda | \approx_{\shskip \varDelta} T^2 / 2 \pi^2$  and $ X \approx_{\shskip \varDelta} T  /4\pi $. For $|z| \sim_{\shskip \varDelta} X$, let $\Phi^{\flat} (z)$   be given as in Corollary  {\rm\ref{cor: Hankel, C}}. We have
	\begin{align*}
	\Phi^{\flat } (z) = \sum_{ T^{\vepsilon}/ K^{\flat}  < \shskip \rho \shskip <  \varDelta - 1 } \Phi_{\rho}^{\flat } (z) + \Phi_0^{\flat } (z) + O (T^{-A}), 
	\end{align*} 
	with $K^{\flat} = \min \big\{ (T/M)^{1/2}, T^{1/4}  \big\} $, 
	\begin{align}\label{13eq: Phi = Mellin, C, 2}
	\Phi^{\flat }_{\rho} (z) = \frac {1} {A_{\rho}^{\flat}} \viint_{\widehat{\bfra}_{0, \lfloor T \rfloor}  (U_{\rho}^{\flat})\shskip \cup \shskip \widehat{\bfra}_{0, \lceil -T \rceil }  (U_{\rho}^{\flat})} \lambdaup^{\flat}_{\rho} (\varnu, m)  {\vchi_{  i \varnu, \shskip  m} (z)} \nd \mu (\varnu, m)  , 
	\end{align} 
	where $\rho = \varDelta^{-k/2}$ or $0$, 
	\begin{align}
	 U_{\rho}^{\flat} = T \rho^2,  \qquad  A_{\rho}^{\flat} = T^2 \rho,  
	\end{align}
	\begin{align}\label{13eq: U and A flat 0}
	 U_0^{\flat} = \left\{ \begin{aligned}
	 & T^{1/2+\vepsilon},      \\
	 & T^{1/2+\vepsilon} ,    \\
	 & M T^{\vepsilon} ,    
	 \end{aligned} \right.  \qquad      A_{0}^{\flat} = \left\{ \begin{aligned}
	 & T^{5/3 },  & & \text{ if } T^{\vepsilon} \leqslant M \leqslant T^{1/3}, \\
	 & M^{1/2} T^{3/2 },  & & \text{ if } T^{1/3} < M \leqslant T^{1/2}, \\
	 & M^{1/2} T^{3/2 } , & & \text{ if } T^{1/2} < M \leqslant T^{1 - \vepsilon} ,
	 \end{aligned} \right. 
	\end{align}
	and $ \lambdaup^{\flat}_{\rho} (\varnu, m), \lambdaup^{\flat}_{0} (\varnu, m)  \Lt 1$. 
\end{lem}

\begin{rem}
	 It is important that  $r$ and $\omega$ play symmetric roles in the arguments below. Note that the restriction $|r| \leqslant M^{\vepsilon} / M $ does not apply to the $\omega$-variable.  Nevertheless, we could let $M = T^{\vepsilon}$ so that not much symmetry is lost. This symmetry seems unique for  the first moment of $\GL_3 \times \GL_2$ or the cubic moment of  $ \GL_2$---it does not occur, for example, in the case of the second moment of $\GL_2$, and one saves less in the $M$-aspect. 
\end{rem}

For $\Phi^{\oldstylenums{0}} (z)$ and  $\Phi^{+}_0 (z)$, it is easy to establish \eqref{13eq: Phi = Mellin, C} by Lemma \ref{lem: Mellin, inert, C}, along with  \eqref{11eq: Phi0, C} and \eqref{11eq: bounds for E+, C}.  This settles the case $\sigmaup = \oldstylenums{0}$ and partially the case $\sigmaup = +$ for $X \Gt T^{2-\vepsilon}$.

In view of \eqref{6eq: rho and theta, natural},  
\begin{align*} 
\rho^{\snatural} (r, \omega)   = \frac {\sinh^2 r +  \sin^2 \omega} {\cosh^2 r - \sin^2 \omega} = \frac {\cosh^2 r -   \cos^2 \omega} {\sinh^2 r +   \cos^2 \omega}  = \frac {\cosh 2 r -   \cos 2 \omega} {\cosh 2 r +   \cos 2 \omega} .
\end{align*}
From  \eqref{11eq: support V, +}--\eqref{11eq: support V, flat} and the conditions for the $\rho$-sums in \eqref{11eq: partition of Phi, C}, we deduce that
\begin{align*}
x \rho^{\snatural} (r, \omega) \approx_{\shskip \varDelta}   \left\{ \begin{aligned}
&{\displaystyle  |X\varLambda|^{1/3} \rho^2  \asymp   {  T^2} / {  X }, } \  & &  \text{ if }  \sigmaup = +, \\
&  {\displaystyle X   ({1-\rho^2}) / {\rho^2} 
\asymp |X\varLambda|^{1/3}  ,} & & \text{ if }  \sigmaup = -, \\
& X   \approx_{\shskip \varDelta}  T/4 \pi , & & \text{ if }  \sigmaup = \flat . 
\end{aligned} \right.   
\end{align*} 
Applying Lemma \ref{lem: Mellin, C} to the exponential factor $e (- 2 \Re  (z \shskip \trh^{\snatural} (r, \omega) ))$ in the integral   \eqref{11eq: Phi integrals}, we have       
\begin{align}\label{13eq: Phi = iint, pm}
 \Phi_{\rho}^{\pm} (x e^{i\phi}) & = \viint_{\widehat{\bfra}' (U^{\pm}) }   \overline{\xiup_{U^{\pm}} (2\varnu, m)}     I^{\pm}_{\rho} (2\varnu, m) \vchi_{  i \varnu, \shskip  m} (x e^{i\phi})  \nd \mu (\varnu, m) + O(T^{-A}), \\
\label{13eq: Phi = iint, flat}   \Phi^{\flat} (x e^{i\phi}) & = \viint_{\widehat{\bfra}' (T) }   \overline{\xiup_{T} (2\varnu, m)}     I^{\flat} (2\varnu, m) \vchi_{  i \varnu, \shskip  m} (x e^{i\phi})  \nd \mu (\varnu, m) + O(T^{-A}),
\end{align}
where $\xiup_{U^{\pm}} (\varnu, m) = O(\log T /U^{\pm})$, $ \xiup_{T} (\varnu, m) = O(\log T / T) $,   
\begin{align}\label{13eq: defn of I, pm}
I^{\pm}_{\rho} (\varnu, m) & = \viint  
e  ( f (r, \omega; t, m) /2 \pi  ) g (M r) V^{\pm}_{\rho}   (r, \omega) \nd r \nd \omega, \\
 \label{13eq: defn of I, flat}
I^{\flat}  (\varnu, m) & = \viint  
e  ( f (r, \omega; t, m) /2 \pi  ) g (M r) V^{\flat}   (r, \omega) \nd r \nd \omega, 
\end{align}  
and 
\begin{align}\label{13eq: phase f (r,w;t,m)}
f (r, \omega; \varnu, m) = 4Tr  + \varnu \log \rho^{\snatural} (r,  \omega) + m \theta^{\snatural} (r,  \omega).
\end{align}

\vskip 5 pt

\subsubsection*{Analysis of $ f (r, \omega; \varnu, m) $} 

By \eqref{6eq: derivative of rho, theta, 1.1} and \eqref{6eq: derivative of rho, theta, 1.2}, we have
\begin{align}\label{13eq: partial f}
{\partial f } / {\partial r} =  4(T+\varnu A_1 - mB_1), \quad  {\partial f } / {\partial \omega} =  4(m A_1 + \varnu B_1) ,
\end{align}
with 
\begin{align} \label{13eq: A1 B1}
A_1 = \frac {  \sinh 2r \cos 2\omega   } {\sinh^2 2 r + \sin^2 2\omega}, \qquad B_1 =  \frac { \cosh 2r \sin 2\omega } {\sinh^2 2 r + \sin^2 2\omega}.
\end{align}
Since  \begin{align*}
\sinh^2 2r \cos^2 2\omega + \cosh^2 2r \sin^2 2\omega = \sinh^2 2 r + \sin^2 2\omega, 
\end{align*}  
the stationary point $(r_0, \omega_0)$ is given by the equations:
\begin{align}\label{6eq: stationary point}
\sinh 2r_0 \cos 2\omega_0 = \varnu /T, \quad  \cosh 2r_0 \sin 2\omega_0 = - m/T.
\end{align}
Also note that 
\begin{align}\label{13eq: A2+B2}
A_1^2 + B_1^2 = \frac 1 { \sinh^2 2 r + \sin^2 2\omega }.
\end{align} 
It follows from \eqref{13eq: partial f} that
\begin{align*}
   ({\partial f } / {\partial r})^2 + ({\partial f } / {\partial \omega})^2    
= \, & 16 \Big( T   - \hskip -1pt {\textstyle \sqrt {(\varnu^2+m^2) (A_1^2+B_1^2) }} \Big)^{  2} \\
&  +   32 T \Big( \hskip -1pt {\textstyle \sqrt {(\varnu^2+m^2)  (A_1^2+B_1^2)} }  +   \varnu A_1  - mB_1   \Big)   .
\end{align*}
From this, it is easy to prove the following lemma.

\begin{lem}\label{lem: away from the stationary point}
	We have 
	\begin{align}\label{13eq: lower bound for f'}
	|f' (r, \omega; \varnu, m)|^2 \Gt T^2 +   \frac {\varnu^2+m^2} {\sinh^2 2 r + \sin^2 2\omega}, 
	\end{align}
	unless 
	\begin{align}\label{6eq: condition 2}
	\left|    \sinh 2r \cos 2\omega \cdot T - \varnu \right|, \  \left| \cosh 2r \sin 2\omega \cdot T + m \right|  \llcurly_{\shskip \varDelta} \sqrt {\varnu^2 + m^2}. 
	\end{align}
\end{lem}

Note that the conditions in  \eqref{6eq: condition 2} imply 
\begin{align}\label{6eq: condition 1}
   \frac {\varnu^2+m^2} {\sinh^2 2 r + \sin^2 2\omega} \approx_{\shskip \varDelta} T^2,
\end{align}
and they describe a small neighborhood of  $(r_0, \omega_0)$. 

By \eqref{6eq: derivative of rho, theta, 2.1} and \eqref{6eq: derivative of rho, theta, 2.2}, we have
\begin{align}\label{6eq: f'' = }
f'' = - 8 \begin{pmatrix}
  \varnu A_2 - m B_2 &    mA_2 + \varnu B_2  \\
  mA_2 + \varnu B_2 &  - \varnu A_2 + m B_2
\end{pmatrix},
\end{align}
with 
\begin{equation} 
A_2 = \frac { \cosh 2r \cos 2\omega   (\sinh^2 2 r - \sin^2 2 \omega)} { (\sinh^2 2 r + \sin^2 2\omega )^2}, \quad B_2 = \frac {  \sinh 2r \sin 2\omega   (\cosh^2 2 r + \cos^2 2 \omega)} { (\sinh^2 2 r + \sin^2 2\omega )^2}. 
\end{equation}
It is clear that for $|r| < 1$ we have
\begin{align}\label{13eq: bounds for A2 B2}
	A_2, \ B_2 \Lt \frac { 1 } { \sinh^2 2 r + \sin^2 2\omega }. 
\end{align}
Some computations show that
\begin{align*}
f'' (r_0, \omega_0; \varnu, m) \hskip -1pt = \hskip -1pt - \frac {8 T } {\sinh^2 \hskip -1pt 2 r_0 \hskip -1pt + \hskip -1pt \sin^2 \hskip -1pt 2\omega_0 } \hskip -1pt \begin{pmatrix}
\sinh \hskip -1pt 2r_0 \cosh \hskip -1pt 2r_0 &  \hskip -1pt \sin \hskip -1pt 2\omega_0 \cos \hskip -1pt 2\omega_0  \\
\sin \hskip -1pt 2\omega_0 \cos \hskip -1pt 2\omega_0 & \hskip -1pt - \sinh \hskip -1pt 2r_0 \cosh \hskip -1pt 2r_0
\end{pmatrix} \hskip -1pt.
\end{align*}
In light of this, we have the following lemma. 

\begin{lem}\label{lem: lower bounds}
	For any $(r, \omega)$ satisfying {\rm\eqref{6eq: condition 2}}, we have
	\begin{align}\label{13eq: bound 1}
	\varnu A_2 - m B_2 - \frac { T \sinh 2r \cosh 2 r } {\sinh^2 2r + \sin^2 2 \omega } 
	\llcurly_{\shskip \varDelta} \frac { \sqrt{\varnu^2+m^2} } { \sinh^2 2 r + \sin^2 2\omega },
	\end{align}
	\begin{align}\label{13eq: bound 2}
	m A_2 + \varnu B_2 - \frac { T \sin 2\omega \cos 2\omega } {\sinh^2 2r + \sin^2 2 \omega } 
	\llcurly_{\shskip \varDelta} \frac { \sqrt{\varnu^2+m^2} } { \sinh^2 2 r + \sin^2 2\omega }. 
	\end{align}
\end{lem}

\begin{proof} 
	\eqref{13eq: bound 1} is a consequence of \eqref{6eq: condition 2} and \eqref{13eq: bounds for A2 B2}, since  its left hand side may be written as $$ \lp \varnu - \sinh 2r \cos 2\omega \cdot T \rp A_2 - \lp m + \cosh 2r \sin 2\omega \cdot T  \rp B_2, $$ 
	and, similarly,  so is \eqref{13eq: bound 2}. 
\end{proof}

Moreover, by Lemma \ref{lem: derivatives of log rho, theta}, 
\begin{align}\label{13eq: bounds for f i+j}
\frac {\partial^{i +  j}  f (r, \omega; \varnu, m) }  {   \partial r^{\shskip i} \partial \omega^{\shskip j} }  \Lt_{\shskip  i, \shskip  j } 
\frac {\sqrt{\varnu^2 + m^2}} { (\sinh^2 2 r + \sin^2 2 \omega)^{( i+ j)/ 2} }
\end{align}  
for $ i +  j \geqslant 2$. 

\vskip 5 pt

\subsubsection*{The Case $\sigmaup = \pm$} 

For Lemma \ref{lem: final, C, 1} it remains to prove the bounds:
\begin{align*}
I^{+}_{\rho} (\varnu, m) \Lt \log^2 T / X, \qquad I^{-}_{\rho} (\varnu, m) \Lt X^{1/3} \log^2 T / |\varLambda|^{2/3}. 
\end{align*}

Recall from \eqref{11eq: support V, +} and \eqref{11eq: support V, -} that $V^{\sigmaup}_{\rho} (r, \omega)$ is supported on 
\begin{align}\label{13eq: support of V}
\sqrt{r^2 + \sin^2 \omega} \sim_{\shskip \varDelta} \rho, \qquad \sqrt{r^2 + \cos^2 \omega} \sim_{\shskip \varDelta} \rho,
\end{align} 
according as $\sigmaup = +$ or $-$, and from   \eqref{11eq: estimates for V, C} that  
\begin{align}\label{13eq: bounds for V}
\partial_r^{i } \partial_{\omega}^{i } V^{\sigmaup}_\rho (r, \omega)     \Lt_{   i, \shskip j } 1 / {\rho^{ i + j}} . 
\end{align}   

In view of Lemma \ref{lem: away from the stationary point},  \eqref{6eq: condition 1}, and \eqref{13eq: support of V}, together with the identity $$ \sinh^2 2r + \sin^2 2\omega = 4 (\sinh^2 r +   \sin^2 \omega)(\sinh^2 r + \cos^2 \omega) ,$$ one would expect the integral $ I^{\sigmaup}_{\rho} (\varnu, m) $ to be negligibly small unless
\begin{align}\label{13eq: v rho = T}
(1-\rho^2) \rho^2  \approx_{\shskip \varDelta} \frac {\varnu^2+m^2} {2  T^2 } .
\end{align} To see this,  we apply Lemma \ref{lem: staionary phase, dim 2, 2}  with $Q = \varPhi =  \varUpsilon  = \rho$, $P = \min  \{ \rho, 1/M \}$,  $Z = \sqrt{\varnu^2 + m^2} $, and $R  = T  + \sqrt{\varnu^2 + m^2} / \rho  $, which are determined by \eqref{13eq: lower bound for f'},  \eqref{13eq: bounds for f i+j}, and \eqref{13eq: bounds for V}. For  $\sigmaup = +$, the condition $\rho > T^{\vepsilon}/T$ is required here. For $\sigmaup = -$, it is slightly easier because $R  \asymp \sqrt{\varnu^2 + m^2} / \rho $  in view of $ \sqrt{\varnu^2 + m^2} / \rho  \asymp \sqrt{|\varLambda|} \Gt T $. 

Next, some remarks on \eqref{13eq: v rho = T} are in order. For $\sigmaup = +$, it is pleasant to check that \eqref{13eq: v rho = T} is consistent with $  \rho \shskip  \approx_{\varDelta} T/2\pi |\varLambda|^{1/2}$, $  \sqrt{ \varnu^2 + m^2 } \asymp T^2/X $, and  $ X \asymp \sqrt{|\varLambda|} $. For $\sigmaup = -$, since $ \rho \shskip  \approx_{\varDelta} X^{1/3}/|\varLambda|^{1/6} $ and $ \sqrt{ \varnu^2 + m^2 } \asymp |X \varLambda|^{1/3}$, \eqref{13eq: v rho = T} would imply $|\varLambda |\asymp T^2$. We remark that  the condition $ |\varLambda| \asymp T^2 $ for $\sigmaup = -$ also arose in the real case (see Remark \ref{rem: case -}). 

Consider rectangular regions  of the form 
\begin{align}\label{13eq: region, 2}
\left|  \pm  2 r (1 - 2 \rho^2) \cdot T - \varnu \right|, \  \left|  \sin 2\omega \cdot T + m \right|  \llcurly_{\shskip \varDelta} \sqrt {\varnu^2 + m^2}. 
\end{align} 
Given \eqref{13eq: support of V} (and $|r| \leqslant M^{\vepsilon}/M$), the regions defined by   \eqref{6eq: condition 2} and \eqref{13eq: region, 2} contain one another if we suitably choose the implied constants. By  Lemma \ref{lem: staionary phase, dim 2, 2}, we  infer that only a negligibly small error is lost if the integral is restricted to the region \eqref{13eq: region, 2}. Next, Lemma \ref{lem: lower bounds}, along with \eqref{13eq: support of V} and \eqref{13eq: v rho = T}, implies that, when $\log \varDelta > 0$ is a small constant, $ |\varnu A_2 - m B_2| \Gt T / \rho   $ if $|r|  \Gt |\sin 2 \omega |$ and $ |	m A_2 + \varnu B_2 | \Gt T / \rho  $ if $|\sin 2 \omega | \Gt |r|$. Now we exploit the second derivative test as in Lemma \ref{lem: 2nd derivative test, dim 2} to deduce
\begin{align*}
I^{\pm}_{\rho} (\varnu, m) \Lt   {\rho} \log^2 T / T .
\end{align*}
Finally, since
\begin{align*}
	{\rho}   / T \Lt \left\{ \begin{aligned}
		& 1 /  {|\varLambda|^{1/2}}\Lt 1/X, & & \text{ if } \sigmaup = +,  \\
		& X^{1/3} / T {|\varLambda|^{1/6}} \Lt X^{1/3} / |\varLambda|^{2/3}, & & \text{ if } \sigmaup = -, 
	\end{aligned} \right. 
\end{align*}
we arrive at the desired estimates. Recall here that $ X \asymp \sqrt{|\varLambda|} $ if $\sigmaup = +$ and  $ T \asymp \sqrt{|\varLambda|} $ if $ \sigmaup = -$.

\vskip 5 pt

\subsubsection*{The Case $\sigmaup = \flat$} 

Set $K = \min \big\{ (T/M)^{1/2}, T^{1/3} \big\}$. We  introduce a smooth partition to the integral $I^{\flat} (\varnu, m)$   in \eqref{13eq: defn of I, flat} according to the value of $\sqrt{r^2 + \cos^2 2 \omega}$, 
\begin{align*}
I^{\flat} (\varnu, m) = \sum_{T^{\vepsilon}/K < \shskip \rho \shskip <  \varDelta - 1 } I_{\rho}^{\flat} (\varnu, m) + I_{0}^{\flat} (\varnu, m), 
\end{align*}
with  
\begin{align*}
I^{\flat}_{\rho} (\varnu, m) = \viint  
e  ( f (r, \omega; t, m) /2 \pi  ) g (Mr) V^{\flat}_{\rho}   (r, \omega) \nd r \nd \omega, 
\end{align*}
where  $\rho = \varDelta^{- k / 2}$ or $0$, $V^{\flat}_{\rho} (r, \omega)$ or $V^{\flat}_{0} (r, \omega)$ is supported on 
\begin{align*}
\sqrt{r^2 + \cos^2 2 \omega} \sim_{\shskip \varDelta} \rho, \qquad  \sqrt{r^2 + \cos^2 2 \omega} \Lt  T^{\vepsilon} /K, 
\end{align*}
respectively,  and 
\begin{align*}
\partial_r^{i } \partial_{\omega}^{i } V^{\flat}_\rho (r, \omega)     \Lt_{   i, \shskip j } (\log T / \rho)^{ i + j}, \qquad \partial_r^{i } \partial_{\omega}^{i } V^{\flat}_0 (r, \omega)     \Lt_{   i, \shskip j }   {(K / T^{\vepsilon})^{ i + j }} . 
\end{align*} 
Consequently, we have a partition of $\Phi^{\flat} (x e^{i\phi})  $ (see \eqref{13eq: Phi = iint, flat}) in the same fashion:
\begin{align}\label{13eq: Phi flat = }
\Phi^{\flat} (\varnu, m) = \sum_{1/K < \shskip \rho \shskip <  \varDelta - 1 } \Phi_{\rho}^{\flat} (\varnu, m) + \Phi_{0}^{\flat} (\varnu, m), 
\end{align}
with
\begin{align} 
\label{13eq: Phi = iint, flat, 2}   \Phi_{\rho}^{\flat} (x e^{i\phi}) & = \viint_{\widehat{\bfra}' (T) }   \overline{\xiup_{T} (2\varnu, m)}     I_{\rho}^{\flat} (2\varnu, m) \vchi_{  i \varnu, \shskip  m} (x e^{i\phi})  \nd \mu (\varnu, m).
\end{align} 

An obvious distinction between the cases $\sigmaup = \pm$ and $\flat$ is the scale of $ \sinh^2 2 r + \sin^2 2\omega $ which arose ubiquitously in the denominators of the derivatives of $ f   (r, \omega; \varnu, m) $. Its scale grows from $ \rho^2$ to $   1 $ when  $\sigmaup$ changes from $\pm$ to $ \flat$.  As a consequence, we lose $1/\rho^2$ in the stationary-phase bound for $ I^{\flat}_{\rho} (\varnu, m) $ (it could be even worse than the trivial bound if $\rho$ were very close to $0$). Fortunately, however,  we shall be able to recover the loss by shrinking the integral domain $\widehat{\bfra}' (T)$  into the union of $ \widehat{\bfra}_{0, \shskip  \lfloor T \rfloor}  (U_{\rho}^{\flat})  $ and $\widehat{\bfra}_{0, \shskip \lceil - T \rceil }  (U_{\rho}^{\flat})$.  

Now we return to the analysis of the integral $ I^{\flat}_{\rho} (\varnu, m) $. 

Similar to the case $\sigmaup = \pm$, we deduce from the second derivative test (Lemma \ref{lem: 2nd derivative test, dim 2}) that  $
I^{\flat}_{\rho} (\varnu, m) \Lt   \log^2 T / {T \rho}. $
On the other hand, we have the trivial bound $I_{0}^{\flat} (\varnu, m) \Lt T^{\vepsilon} / \max\big\{(M T)^{1/2}, \allowbreak T^{2/3} \big\}  $. 

Recall that $\sqrt{\varnu^2 + m^2} \asymp T$. Note that $\sinh 2 r, \cos 2 \omega \Lt \sqrt{r^2 + \cos^2 2 \omega} \Lt \rho$ and $  \sinh^2 2 r + \sin^2 2\omega = 1 + O (\rho^2) $. It follows from  \eqref{13eq: A1 B1} and \eqref{13eq: A2+B2} that   $A_1, B_1 \pm 1 = O (\rho^2)$.   Consequently, in view of \eqref{13eq: partial f} and \eqref{13eq: A1 B1}, we have
\begin{align*}
{\partial f } / {\partial r} =  4(T \mp m) + O (T \rho^2 ), \quad  {\partial f } / {\partial \omega} = \pm 4  \varnu  + O (T \rho^2 ).   
\end{align*}
Moreover, \eqref{13eq: bounds for f i+j} now reads $$ {\partial^{i +  j}  f   } / {   \partial r^{\shskip i} \partial \omega^{\shskip j} }  \Lt_{\shskip  i, \shskip  j } T $$ for $i+j \geqslant 2$. 
Set $U = \max \big\{ T \rho^2, T^{1/2 + \vepsilon}  \big\} = \max \big\{ T \rho^2, T^{1/2 + \vepsilon} , M T^{\vepsilon} \big\}$. 
Then $ |{\partial f } / {\partial \omega}| \Gt U $ for $|\varnu| \Gt U$, and $ |{\partial f } / {\partial r}| \Gt U $ for  $ |T \mp m| \Gt  U $. On applying Lemma \ref{lem: staionary phase, dim 1, 2} to the  $\omega$- or $r$-integral, with $P = \rho$ or $ \min\{\rho, 1/M\}$, $ Q = 1 $,  $Z = T$, and $R = U$, we find that $I^{\flat}_{\rho} (\varnu, m)$ is negligibly small for such $\varnu$ or $m$ (it is important here that $T/U^2 , M / U  \leqslant 1/T^{\vepsilon}$). Similarly, if we put $U_0  = \max \big\{  T^{1/2 + \vepsilon} , M T^{\vepsilon} \big\}$,  then $I^{\flat}_{0} (\varnu, m)$ is negligibly small unless $\varnu = O (U_0)$ and $   m =  \pm T + O(U_0) $. 


  Lemma  \ref{lem: final, C, 2} follows if we truncate the $\rho$-sum in \eqref{13eq: Phi flat = } at $\rho = T^{\vepsilon} / T^{1/4 }$ in the case $M < T^{1/2}$ and absorb the sum of $\Phi^{\flat}_{\rho} (\varnu, m)$ over smaller  $\rho $'s into $\Phi^{\flat}_{0} (\varnu, m)$.


\begin{appendices}
	\section{Olver's Uniform Asymptotic Formula for Bessel Functions}\label{app: Olver}
	
	In this appendix, we recollect Olver's uniform asymptotic formula for Bessel functions of large order and prove some of its implications that will be useful in \S \ref{sec: Olver} for our study of certain Mellin integrals over $\BC$. 
	For our purpose, we only consider here $J_{m} (m x)$ with large integer order $m  $ and positive real variable $x$. 
	
	According to the works of Olver \cite{{Olver-1},Olver-Bessel}, we have
	\begin{equation}\label{12eq: Jm(mx) = }
	\begin{split}
	J_m (m x) = \Big(\frac {4\zeta} {1-x^2} \Big)^{1/4} \Bigg\{   \frac {\mathrm{Ai} \big(m^{2/3} \zeta\big)} {m^{1/3}} \sum_{s=0}^k \frac {A_s(\zeta)} {m^{2s}}   +   \frac { \mathrm{Ai}' \big(m^{2/3} \zeta\big) } {m^{5/3}} \sum_{s=0}^{k-1} \frac {B_s(\zeta)} {m^{2s}} & \\
	+ O \lp \frac{\big|\exp  (   - \frac 2 3 m \zeta^{3/2}  ) \big|}{m^{2k+1} (1 + m^{1/6} |\zeta|^{1/4})  } \rp & \Bigg\},
	\end{split}
	\end{equation}
	where $\mathrm{Ai} (y)$ is the Airy function, 
	\begin{equation} \label{12eq: gamma}
	\begin{aligned}
	&\frac 2 3 \zeta^{3/2}   = \log \frac {1 + \sqrt{1-x^2}} x - \sqrt{1 - x^2},  \quad & &  0 < x \leqslant 1, \\
	& \frac 2 3 (-\zeta)^{3/2}   = \sqrt{x^2-1} - \mathrm{arcsec} x,   & & x > 1, 
	\end{aligned}
	\end{equation}
	and
	\begin{align}\label{12eq: As Bs}
	A_s (\zeta) = \sum_{j=0}^{2s} b_j \zeta^{-3j/2} U_{2s-j} (\varv), \quad \zeta^{1/2} B_s (\zeta) = - \sum_{j=0}^{2s+1}   a_j \zeta^{-3j/2} U_{2s-j+1} (\varv),
	\end{align}
	in which $a_0 = b_0 = 1$,  
	\begin{align*}
	a_s = \frac {1} {3^{2s} (2s)! } \frac { \Gamma \big(3s+\frac 1 2 \big)} {\Gamma \big(\frac 1 2 \big)}, \qquad b_s = - \frac {6s+1} {6s-1} a_s,
	\end{align*}
	and $U_s (\varv)$ are polynomials in $\varv = 1 / \sqrt{1-x^2}$, with the first three found to be
	\begin{align}\label{12eq: Us}
	U_0 = 1, \quad U_1 = (3\varv -5\varv^3)/24, \quad U_2 = (81\varv^2-462\varv^4+385\varv^6)/1152; 
	\end{align}
	see  \cite[\S  4]{Olver-Bessel}  and \cite[Theorem B]{Olver-1} for the expansion and the error term as in \eqref{12eq: Jm(mx) = }, and \cite[\S  6]{Olver-Bessel} for the coefficients $A_s (\zeta)$ and $B_s (\zeta)$ as in \eqref{12eq: As Bs}. 
	By \cite[(4.13), (4.14)]{Olver-Bessel}, 
	\begin{align}\label{12eq: zeta small}
	& x = 1 - \frac {\zeta}  {2^{1/3}} + \frac {3 \zeta^2} {10 \cdot 2^{2/3}} + O(\zeta^3), \qquad |\zeta| \Lt 1, \\ 
	\label{12eq: zeta large}
	& x = \frac 2 3 (-\zeta)^{3/2} + O (1), \hskip 69 pt - \zeta \Gt 1. 
	\end{align} As for the Airy function, if we set $\gamma = \frac 2 3 y^{3/2}$  for $y > 0$, then it is well-known (see  \cite[(10.4.14)--(10.4.17)]{A-S}) that
	\begin{equation*}
	\begin{split}
	& \mathrm{Ai} (y) = \frac {\sqrt{y } } {\sqrt 3 \pi}    K_{1/3} (\gamma), \qquad \mathrm{Ai} (-y) = \frac {\sqrt y } 3 \big(J_{1/3} (\gamma) \hskip -1 pt + \hskip -1 pt J_{-1/3} (\gamma)\big), \\
	& \mathrm{Ai}' (y) = - \frac { {y } } {\sqrt 3 \pi}    K_{2/3} (\gamma), \qquad \mathrm{Ai}' (-y) = \frac {  y } 3 \big(J_{2/3} (\gamma) \hskip -1 pt - \hskip -1 pt J_{-2/3} (\gamma)\big).
	\end{split}  
	\end{equation*} 
	For $|y| \Lt 1$, we have $\mathrm{Ai} (y) = O (1)$ (see \cite[(10.4.2)]{A-S}). For $y \Gt 1$, it follows from \cite[7.21 (1, 3), 7.23 (1), \S 7.3]{Watson} that 
	\begin{align}\label{12eq: Ai, K}
	\mathrm{Ai} (y) = O \bigg( \frac { \exp (- \gamma) } {  y^{1/4} } \bigg), \qquad \mathrm{Ai}' (y) = O \big( y^{1/4} { \exp (- \gamma) }  \big), 
	\end{align} 
	and
	\begin{align}\label{12eq: Ai, J}
	\mathrm{Ai} (- y) = \sum_{\pm} \frac { \exp ( \pm i \gamma) } {y^{1/4}} W_{\pm} (\gamma), \qquad \mathrm{Ai}' (- y) = O \big(y^{1/4}\big), 
	\end{align}
	with $ \gamma^i W_{\pm}^{(i)} (\gamma) \Lt_{\shskip i} 1 $. 
	
	\begin{lem}\label{lem: Olver}
		Let $m \Gt 1$  and $x > 0$. For $x > 1$, define $\gamma (x) = \sqrt{x^2-1} - \mathrm{arcsec} x$. 
		
		{\rm (1)} For $|x-1| \leqslant 1/m^{2/3}$,  we have $ J_{m} (m x) = O (1 /m^{1/3}) $. 
		
		{\rm(2)} For $ \frac 1 2 \leqslant  x \leqslant 1 - 1/m^{2/3} $, we have
		\begin{align}\label{12eq: Bessel, smaller, K}
		J_{m} (m x) = O  \Bigg(  \frac{\exp  \big(  \hskip -1 pt -  \frac 1 3 m (2-2x))^{3/2}   \big) }{  m^{1/2} (1-x)^{1/4}  } \Bigg).
		\end{align} 
		
		{\rm (3)} For $1 + 1/m^{2/3} \leqslant  x \leqslant 2$, we have
		\begin{align}\label{12eq: Bessel, larger, J}
		J_{m} (m x) = \sqrt{2} \sum_{ \pm } \frac { \exp ( \pm i m \gamma (x) ) } {m^{1/2} (x^2-1)^{1/4} } W_{\pm} (m \gamma (x))  + O \bigg(\frac 1 {m^{7/6} (x-1)^{1/4}} \bigg) ,
		\end{align}
		in which $ \gamma^i W_{\pm}^{(i)} (\gamma) \Lt_{\shskip i} 1 $ for $\gamma \Gt 1$. 
		
		{\rm (4)} For $2 \leqslant x \Lt m^{13/3}$, the asymptotic formula {\rm\eqref{12eq: Bessel, larger, J}} holds with an error term $ O \big(1/ m x\big) $. 
	\end{lem}
	
	\begin{proof}
		First let $k = 0$ in \eqref{12eq: Jm(mx) = }.  The estimate in (1) is clear.	For $ 0 < x \leqslant 1$, it is easy to prove (compare with \eqref{12eq: zeta small})
		\begin{align*}
		\log \frac {1 + \sqrt{1-x^2}} x - \sqrt{1 - x^2} \geqslant \frac {1} 3   (2-2x)^{3/2} . 
		\end{align*}
		Then \eqref{12eq: Bessel, smaller, K} is a direct consequence of \eqref{12eq: Jm(mx) = }, \eqref{12eq: zeta small}, and \eqref{12eq: Ai, K}.\footnote{Note that \eqref{12eq: Bessel, smaller, K} would also follow from Nicholson's asymptotic formula in \cite[\S 3.14.3]{MO-Formulas}.} The asymptotic formula \eqref{12eq: Bessel, larger, J} in (3) is obvious in view of \eqref{12eq: Jm(mx) = }, \eqref{12eq: zeta small}, and \eqref{12eq: Ai, J}. As for (4), we let $k=1$ in  \eqref{12eq: Jm(mx) = }. For $ x \geqslant 2 $, it follows from  \eqref{12eq: As Bs}, \eqref{12eq: Us}, and \eqref{12eq: zeta large} that $ B_0 (\zeta) = O (1/\zeta^2) $ and $A_1 (\zeta) = O (1/\zeta^3)$. By  \eqref{12eq: zeta large} and \eqref{12eq: Ai, J}, the two lower-order main terms  and the error term in \eqref{12eq: Jm(mx) = } are $ O \big(1/ (mx)^{3/2} \big)$, $O \big(1/(mx)^{5/2}\big)$, and $ O \big(1/ m^{19/6} x^{1/2} \big)  $, respectively; all of these are $O (1/mx)$ provided   $x \Lt m^{13/3}$.
	\end{proof}
\end{appendices}

{\large \part{Proof of Theorem \ref{thm: main}}}


\section{Setup}\label{sec: setup}

We start with introducing the spectral mean of $L$-values
\begin{align*}
\sum_{f \in \SB  }  \omega_f  k^{\snatural}  (\varnu_f)  L \big( \tfrac 1 2 ,   \pi \otimes  f \big)   +  \frac {c_0} {4 \pi}   \int_{-\infty}^{\infty} \omega (t) k^{\snatural}  (t)   \left| L \big(  \tfrac 1 2 + it,   \pi \big) \right|^2 \nd \shskip t,   
\end{align*}
in which $k^{\snatural}  (\varnu)$ is the   test function defined in \S \ref{sec: choice of h}. Recall that $
k^{\snatural} (\varnu) >  0  $ for $\varnu \in \bfra$,  and that $
k^{\snatural}  (\varnu )  \Gt 1 $ if $ |\varnu_{\varv} - T_{\vv}| \leqslant M_{\vv} $ for all $\vv | S_{\infty}$. When $f \in \SB$ is exceptional in the sense that $\varnu_{f, \, \vv}$ is not real for some $\vv | \infty$, the weight $ k^{\snatural}  (\varnu_f) $ would be negligibly small (although not necessarily positive), for at this place $\vv$ we have $ k^{\snatural}  (\varnu_{f, \, \vv}) = o \big(e^{-T_{\vv}^2 / M_{\vv}^2 } \big) $ and $ T_{\vv} \geqslant \RN (T)^{\vepsilon} $ by assumption.    Thus, in view of \eqref{1eq: bounds for omega} in Lemma \ref{lem: lower bounds for omega}, along with the non-negativity of the $L$-values, Theorem \ref{thm: main} follows if we are able to prove that the spectral mean
is bounded by $\RN^{\snatural} (M) \RN (T)^{5/4 + \vepsilon}$. 

Applying the Approximate Functional Equations \eqref{1eq: approximate functional equation, 1} and \eqref{1eq: approximate functional equation, 2},  the above spectral mean may be written as
\begin{align*}
2 \underset{\frn_1, \frn_2 \shskip \subset \shskip \frO }{\sum \sum} \frac { A (\frn_1, \frn_2)   } {\RN  ( \frn_1^2 \frn_2  )^{1/2} } \Bigg\{  & \sum_{f \in \SB  }  \omega_f  k^{\snatural}  (\varnu_f) \lambdaup_f  (\frn_2) V  \big(  \RN  \big( \frn_1^2 \frn_2 \frD^{-3} \big); \varnu_f  \big) \\
& + \frac {c_0} {4 \pi}   \int_{-\infty}^{\infty} \omega (t) k^{\snatural}  (t) \tau_{ it}  (\frn_2  )  V  \big(  \RN  \big( \frn_1^2 \frn_2 \frD^{-3} \big); t  \big) \nd \shskip t \Bigg\}.
\end{align*}
By \eqref{1eq: derivatives for V(y, t), 1} in Lemma \ref{lem: afq} (1), we may truncate the sum over $\frn_1, \frn_2 $ at $ \RN  ( \frn_1^2 \frn_2  ) \leqslant \RN (T)^{3 + \vepsilon}$. We then apply \eqref{1eq: approx of V} in Lemma \ref{lem: afq} (1), in which   we choose $U = \log \RN(T)$. The error term is again negligible, and we need to prove
\begin{align*}
 \underset{\RN  ( \frn_1^2 \frn_2  ) \shskip \leqslant \RN (T)^{3 + \vepsilon}}{\sum \sum}  \frac { A (\frn_1, \frn_2)   } {\RN  ( \frn_1^2 \frn_2  )^{1/2 + u} } & \Bigg\{     \sum_{f \in \SB  }  \omega_f  h (\varnu_f) \lambdaup_f  (\frn_2)   \\
&   + \frac {c_0} {4 \pi}   \int_{-\infty}^{\infty} \omega (t) h  (t) \tau_{ it}  (\frn_2  )   \nd \shskip t \Bigg\} \Lt \RN^{\snatural} (M) \RN (T)^{5/4 + \vepsilon}.  
\end{align*} 
{uniformly} in $u  \in \left[ \vepsilon - i \log \RN (T) , \vepsilon + i \log \RN (T)  \right]$. By the Hecke relation \eqref{4eq: Hecke relation}, the left-hand side is equal to 
\begin{align*}
\underset{\RN  ( \frd^3 \frn_1^2 \frn_2  ) \shskip \leqslant \RN (T)^{3 + \vepsilon}}{\sum \sum \sum}   \frac { \mu (\frd) A (\frn_1, 1) A (1, \frn_2)  } {\RN  ( \frd^3 \frn_1^2 \frn_2  )^{1/2 + u} }   \Bigg\{ &  \sum_{f \in \SB  }  \omega_f  h (\varnu_f) \lambdaup_f  (\frd \frn_2)   \\
&    + \frac {c_0} {4 \pi}   \int_{-\infty}^{\infty} \omega (t) h  (t) \tau_{ it}  ( \frd \frn_2 )   \nd \shskip t \Bigg\}. 
\end{align*}

We now apply the Kuznetsov trace formula \eqref{1eq: Kuznetsov} in Proposition \ref{prop: Kuznetsov}, with $\frm_1 = \frd \frn_2$ and $\frm_2 = \frO$, obtaining a diagonal sum 
\begin{align}\label{5eq: diag}
 c_1 \SDH \sum_{ \RN  ( \frn  ) \shskip \leqslant \RN (T)^{3/2 + \vepsilon} }   \frac {A (\frn, 1)   } {\,  \RN (\frn)^{1 +2 u } },
\end{align}
and an off-diagonal sum
\begin{equation}\label{5eq: off-diag}
\begin{split}
  c_2 \underset{ \RN  ( \frd^3 \frn^2  ) \shskip \leqslant \RN (T)^{3  + \vepsilon} }{\sum \sum} \sum_{ \frc   \shskip    \in \shskip \widetilde{C}_F}   \sum_{ \epsilon \, \in \shskip \frOO^{\sstimes} \hskip -1pt / \frOO^{\sstimes 2} } 
 & \mathop{ \sum_{\gamma \shskip \in \fra  ^{-1} / \frOO^{\sstimes} } }_{ \RN(\gamma) \shskip \leqslant \RN (T)^{3+\vepsilon} / \RN ( \fra \frd^3 \frn^2    ) }  \frac { \mu (\frd) A (\frn, 1 ) A (1, \gamma \fra  )   } {\RN  (  \gamma \fra \frd^3 \frn^2     )^{1/2 + u} } \\
& \cdot \sum_{c \, \in \shskip \frc^{-1} } \frac {\mathrm{KS} ( \epsilon \gamma, \fra \frd \frD^{-1}; 1  / \beta_{\frd} , \frD^{-1};   c, \frc ) } { \RN (c \frc) } \SDH \bigg( \frac {\epsilon   \gamma } {\beta_{\frd} c^2   }  \bigg), 
\end{split}
\end{equation}
in which $\fra \in \widetilde{C}_F$ is determined by $\fra \sim (\frc \frD)^2 \frd\-$, and $\beta_{\frd} = \beta_{\frc, \shskip \fra \frd \frD^{-1} \cdot \frD^{-1}} =   [(\frc  \frD)^2 (\fra \frd)^{-1}  ]$.

\begin{lem}\label{lem: bound for H}
	For the test function $h (\varnu) $ 
	defined as in {\rm\eqref{5eq: defn k(nu)}--\eqref{6eq: defn of h, local}} {\rm(}see also {\rm(\ref{4eq: defn p(s, t)}, \ref{4eq: def G})}{\rm)}, we have the following estimate for $\SDH$ {\rm(}defined by {\rm(\ref{1eq: defn Plancherel measure}, \ref{1eq: defn Bessel integral})}{\rm)},
	\begin{align}\label{5eq: bound for H}
	 \SDH \Lt \RN^{\snatural} (M) \RN (T)^{1+\vepsilon}. 
	\end{align}
\end{lem}

\begin{proof}
	In view of \eqref{1eq: defn Plancherel measure}, \eqref{1eq: defn Bessel integral}, the integral $\SDH$ splits into the product of
	\begin{align*}
  \int_{-\infty}^{\infty} h_{\vv} (\varnu)    {  \tanh (\pi \varnu) \varnu \shskip \nd \shskip \varnu }
	\end{align*}
if $\vv$ is real, and
	\begin{equation*}
\int_{-\infty}^{\infty} h_{\vv} (  \varnu  )    \varnu^2  \nd \shskip \varnu
	\end{equation*}
	if $\vv$ is complex, which, in view of \eqref{5eq: bound for h}, are bounded by
\begin{align*}
\int_{0}^{\infty} e^{- (\varnu-T_{\vv})^2/M_{\vv}^2}  {   \varnu^{1 + \vepsilon} \shskip \nd \shskip \varnu } + T_{\vv}^{-A} \Lt M_{\vv} T_{\vv}^{1+\vepsilon},
\end{align*} 
and
\begin{align*}
\int_{0}^{\infty} e^{- (\varnu-T_{\vv})^2/M_{\vv}^2}  {   \varnu^{2 + \vepsilon} \shskip \nd \shskip \varnu } + T_{\vv}^{-A} \Lt M_{\vv} T_{\vv}^{2+\vepsilon}, 
\end{align*} 
respectively. 
Then \eqref{5eq: bound for H} follows immediately.
\end{proof}

It follows from Cauchy--Schwarz, \eqref{3eq: Ramanujan on average}, and \eqref{5eq: bound for H}, that the diagonal sum in \eqref{5eq: diag} is bounded by $\RN^{\snatural} (M) \RN (T)^{1+\vepsilon}$, as expected. 

For the off-diagonal sum, our aim is to execute Vorono\"i summation in the $\gamma$-variable, so we must unfold the $\epsilon \gamma$-sum from a sum over $  \fra^{-1} / \frO^{\times 2} $ to a sum over $ \fra^{-1}  $. For this, we set $\frq = c \frc$ and fold the $c$-sum into a $\frq$-sum over ideals. Thus  \eqref{5eq: off-diag} is rewritten as 
 \begin{equation}\label{5eq: off-diag, 2}
 \begin{split}
 &  2 {c_2}   \underset{ \RN  ( \frd^3 \frn^2  ) \shskip \leqslant \RN (T)^{3  + \vepsilon} }{\sum \sum}  \sum_{ \frc   \shskip    \in \shskip \widetilde{C}_F}  \frac {\mu (\frd) A (\frn, 1 )} { \RN  ( \fra \frd^3 \frn^2     )^{1/2 + u} }  \\
 \cdot & \sum_{\frq \sim \frc}  \frac 1 {\RN (\frq)} \hskip - 3 pt \mathop{ \sum_{\gamma \shskip \in \fra^{-1}    } }_{ \RN(\gamma) \shskip \leqslant \RN (T)^{3+\vepsilon} / \RN (\fra \frd^3\frn^2    ) }  \hskip - 4 pt
 \frac { A (1, \gamma \fra  )   } {\RN  (  \gamma  )^{1/2 + u} } 
       {\mathrm{KS} ( \gamma, \fra \frd \frD^{-1}; 1  / \beta_{\frd} , \frD^{-1};   c_{\frq}, \frc ) }  \SDH \bigg( \frac {    \gamma } {\beta_{\frd} c_{\frq}^2   }  \bigg), 
 \end{split}
 \end{equation}
 where $c_{\frq} = [\frc\- \frq]$, and $\fra$, $\beta_{\frd}$ are defined after \eqref{5eq: off-diag}. 
 In view of \eqref{2eq: choice of [a]}, it will be convenient to introduce $V ({\frb}) \in \bfra_+$ for every nonzero ideal $\frb$ with
 \begin{align}\label{8eq: defn of Va}
 V (\frb)_{ \vv} = \RN (\frb)^{ \theta_{\vv}}, \qquad   \theta_{\vv} =   \log T_{\vv} / \log \RN(T), 
 \end{align}
 so that
 \begin{align}\label{8eq: local size of beta and c}
 1/ |\beta_{\frd}|_{\vv} \asymp V(\frd)_{ \vv}^{N_{\vv}}, \qquad |c_{\frq}|_{\vv} \asymp V(\frq)_{\vv}^{N_{\vv}}, 
 \end{align} 
 for each $\vv|\infty$.
The main actors are $\frq$ and $\gamma$, so we shall be concerned with the last two summations in the second line of \eqref{5eq: off-diag, 2}.

\section{First Reductions}\label{sec: first reductions}
 
Next, we need to do  a smooth $\varDelta$-adic   partition of unity in $|\gamma|_{\vv}$ for each $\vv | \infty$, where $\varDelta > 1$ is a fixed constant with $\log \varDelta $ small. However, when $F$ is neither rational nor imaginary quadratic, an issue with the infinitude of units is that one has $|\gamma|_{\vv} \ra 0$ when $\gamma$ ranges in $\fra^{-1} \smallsetminus \{0\}$. This   may be addressed by proving that if  $   { \left|\gamma \right|_{\vv} V (\frd)_{\vv}^{N_{\vv}} } / {  V(\frq)_{\vv}^{2 N_{\vv}} } \leqslant  T_{\vv}^{2 N_{\vv}} $ (so that $ \left|\gamma / \beta_{\frd} c_{\frq}^2 \right| _{\vv} \Lt T_{\vv}^{2 N_{\vv}}$ by \eqref{8eq: local size of beta and c}) for any given $\varv  |\infty$ then the contribution is negligibly small; critical are the second  estimates for Bessel integrals in  Corollary \ref{cor: bound for H < T} and the assumption that $T_{\vv} $ is large ($T_{\vv} \geqslant \RN (T)^{\vepsilon}$) for every $\vv |\infty$.  
To this end, we use Weil's bound for Kloosterman sums \eqref{3eq: Weil}  and the estimates for Bessel integrals in Corollary \ref{cor: bound for H < T} to bound the contribution by
\begin{align*}
  &{\sum_{S \subsetneq S_{\infty}} }  \frac {\RN^{\snatural} (M) \RN (T)^{1+\vepsilon} } {|T|_{S_{\infty} \smallsetminus S}^{2A'}} \underset{ \RN  ( \frd^3 \frn^2  ) \shskip \leqslant \RN (T)^{3  + \vepsilon} }{\sum \sum} \sum_{ \frc   \shskip    \in \shskip \widetilde{C}_F}    \frac {|A (\frn, 1 )|} { \RN  (   \frd^3 \frn^2     )^{1/2 + \vepsilon} } \\
  \cdot &  \sum_{\frq \shskip \sim \frc} \frac { 1 } { \RN (\frq)^{1/2 -\vepsilon} \left|V (\frd^{-1}\frq^2)\right|^{1/2}_{S_{\infty} \smallsetminus S  } }  \mathop{ \sum_{  \gamma \shskip \in \shskip  F_\infty^{S} (T^2 V (\frd^{-1} \frq^2)) }}_{   \RN(\gamma ) \shskip \Lt \RN (T)^{3+\vepsilon} / \RN(  \frd^3\frn^2) } \frac { |A (1, \gamma \fra  ) |   } {|\RN  (  \gamma  )|^{  \vepsilon}  |\gamma|_{S  }^{1/2}  },
\end{align*}
where $F_\infty^{S} (V) \subset F_{\infty}  $ is defined in \eqref{4eq: defn of FS (T)}. 
Because of the occurrence of $ |T|_{S_{\infty} \smallsetminus S}^{2A' } $, this sum is negligibly small on choosing $A'$ to be large. Note that if \eqref{4eq: average over ideals, 2}   in Lemma \ref{lem: averages} is applied  to bound the  $\gamma$-sum by
\begin{align*}
\frac {\RN (T)^{9/4+\vepsilon}}  {\RN ( \frd^3\frn^2    )^{3/4}  }  { \sum_{  \gamma \shskip \in \shskip  F_\infty^{S} (T^2 V (\frd^{-1} \frq^2)) }}  \frac { |A (1, \gamma \fra  ) |  } {|\RN  (  \gamma  )|^{3/4+\vepsilon} {|\gamma|}^{1/2}_{S} } \Lt \frac {\RN (T)^{11/4+\vepsilon}   \RN (  \frq )^{1/2-\vepsilon} }  {|T|_S \RN ( \frd)^{5/2} \RN(\frn)^{3/2}   |V(\frd\- \frq^2)|^{1/2}_S }  ,
\end{align*}
then $\left|V (\frd^{-1}\frq^2)\right|^{1/2}_{S_{\infty} \smallsetminus S  }$ and $ |V(\frd\- \frq^2)|^{1/2}_S $ are combined into $\RN (\frd^{-1}\frq^2)^{1/2}$, and the $\frq$-sum is convergent.

We may therefore impose the condition $     |\gamma  |_{\vv} V (\frd)_{\vv}^{N_{\vv}}  / {  V(\frq)_{\vv}^{2N_{\vv}}} >  T_{\vv}^{2 N_{\vv}} $ for all $\vv |\infty$. Note that one has necessarily $ |\gamma  |_{\vv} > T_{\vv}^{(1-\vepsilon) N_{\vv}} $ for all $\vv |\infty$, since $ \RN (\frd)  \leqslant \RN (T)^{1+\vepsilon} $.  By a smooth partition of unity on the $\gamma$-sum, the problem can be reduced to proving the following proposition.

\begin{prop}\label{prop: main}
 Let $\frd$  	be a square-free integral ideal with $\RN(\frd) \leqslant \RN (T)^{1+\vepsilon}$. 
Let $\fra, \frc  \in \widetilde{C}_F$ satisfy $\fra \frd \sim (\frc \frD)^2 $. Set $\beta_{\frd} =  [(\frc  \frD)^2 (\fra \frd)^{-1}  ]$.  Let $R \in \bfra_+$ be such that
	\begin{align}\label{15eq: R < T3}
 	\RN (R) \leqslant \RN (T)^{3+\vepsilon} / \RN (\frd)^3.
	\end{align} 	 
Fix $\varDelta > 1$ with $\log \varDelta$ sufficiently small. For each $\vv | \infty$, let $f_{\vv} (r) $ be a smooth function supported on $[R_{\vv}, \varDelta R_{\vv}]$ satisfying $ f_{\vv}^{(i)} (r)  \Lt_{i} (\log \RN (T) / R_{\vv})^{i} $ for all $i \geqslant 0$. Suppose that $\SDH (x)$  is the Bessel transform of $h (\varnu) $ given in {\rm \eqref{1eq: defn Bessel integral}}, with $h (\varnu)$ defined as in {\rm\eqref{5eq: defn k(nu)}}--{\rm\eqref{5eq: defn h(nu)}}. Define
	\begin{equation}\label{5eq: defn of S(N)}
	\begin{split}
\SS_{\frd} (  T, R)   =  \sum_{\frq \sim \frc}  \frac 1 {\RN (\frq)}  \sum_{\gamma \shskip \in \fra\-  }    A (1, \gamma \fra  )    {\mathrm{KS} ( \gamma, \fra \frd \frD^{-1}; 1  / \beta_{\frd} , \frD^{-1}; c_{\frq}, \frc ) }  f \bigg( \gamma , \hskip -1 pt \frac { 1 } {\beta_{\frd} c_{\frq}^2   }  \bigg), 
	\end{split}
	\end{equation}
	in which  $c_{\frq} = [\frc\- \frq]$, the $\frq$-sum is finite, subject to the conditions
	\begin{align}\label{15eq: condition on q}
	  V(\frd\- \frq^2)_{\vv}  \Lt  R_{\vv} /     T_{\vv}^{2 } 
	, \qquad  \vv | \infty ,
	\end{align}  
	with $V(\frd\- \frq^2) \in \bfra_+$ defined as in {\rm\eqref{8eq: defn of Va}}, and $f \big(x, 1/ \beta_{\frd} c_{\frq}^2 \big)$ is the product of 
	\begin{equation}\label{15eq: w (x; Lambda)}
	f_{\vv} \bigg(x_{\vv} , \frac 1  {\beta_{\frd} c_{\frq}^2} \bigg) = f_{\vv} (|x_{\vv}| )  \SDH \bigg(  \frac {x_{\vv} } {\beta_{\frd} c_{\frq}^2}\bigg ). 
	\end{equation}
	Then 
	\begin{align}\label{15eq: bound for S}
\SS_{\frd} ( T, R) \Lt_{\vepsilon, \shskip \pi, \shskip F } \RN^{\snatural} (M) \RN (T)^{1/2+\vepsilon}  (\RN(R) \RN(\frd))^{3/4} + \frac {\RN^{\snatural} (M)   {   {\RN(R)}} \RN(\frd) } {\RN(T)^{1/3-\vepsilon} } .
	\end{align} 
\end{prop}

To deduce Theorem \ref{thm: main} from Proposition \ref{prop: main}, we use \eqref{15eq: bound for S} to bound the sum in \eqref{5eq: off-diag, 2} by the sum of
\begin{align*}
 \RN^{\snatural} (M)  \RN (T)^{1/2 + \vepsilon}  \underset{ \RN  ( \frd^3 \frn^2  ) \shskip \leqslant \RN (T)^{3  + \vepsilon} }{\sum \sum}     \frac { | A (\frn, 1 )| } { \RN  (   \frd      )^{3/4   } \RN  (   \frn      ) }      \bigg(\frac { \RN (T)^{3+\vepsilon} } { \RN (\frd^3 \frn^2) }\bigg)^{1/4 }   
 \Lt   \RN^{\snatural} (M) \RN (T)^{5/4 + \vepsilon} 
\end{align*}
and
\begin{align*}
\frac{\RN^{\snatural} (M)   \RN (T)^{\vepsilon}} {\RN (T)^{1/3}}   \underset{ \RN  ( \frd^3 \frn^2  ) \shskip \leqslant \RN (T)^{3  + \vepsilon} }{\sum \sum}     \frac { | A (\frn, 1 )| } { \RN  (   \frd     )^{1/2  } \RN  (   \frn     )  }       \bigg(\frac { \RN (T)^{3+\vepsilon} } { \RN (\frd^3 \frn^2) }\bigg)^{1/2 }   
\Lt    \RN^{\snatural} (M) \RN (T)^{7/6 + \vepsilon} .
\end{align*}
 
 \section{\texorpdfstring{Application of the Vorono\"i Summation}{Application of the Voronoi Summation}}

By Definition \ref{defn: Kloosterman KS}, we open the Kloosterman sum as follows,
\begin{align*}
{\mathrm{KS} ( \gamma, \fra \frd \frD^{-1}; 1  / \beta_{\frd} , \frD^{-1};   c_{\frq}, \frc ) } = \sum_{x \, \in ( \fra  \frd (\frc\frD)^{-1} /   \fra \frd (\frc \frD)^{-1} \frq )^{\sstimes} } \psi_{\infty} \bigg( \frac { \gamma x } { c_{\frq}} + \frac {   x^{-1}} {\beta_{\frd} c_{\frq}} \bigg),
\end{align*}
in which $x^{-1} \in  (     (\fra  \frd)^{-1} \frc \frD  /    (\fra  \frd)^{-1} \frc \frD \frq   )^{\times}$ is as defined in Definition \ref{def: x inverse}. On applying the Vorono\"i summation formula in Proposition \ref{prop: Voronoi}, up to the constant ${ \RN  (\fra    ) } / { {\RN(\frD)}^{3/2} }$, the $\gamma$-sum in $\SS_{\frd} (  T, R)$ is transformed  into 
\begin{align}\label{9eq: after Voronoi}
\sum_{ \frb \shskip \subset \frq_1 \shskip \subset \frO } \frac {1 } {\RN(\frb \frq_1 ) }  \mathop{ \sum}_{\gamma \shskip \in \fra   (\frb \frq_1^{2} \frD^{3})^{-1} \smallsetminus \{0\} }   A \big(  \fra\- \frb \frq_1^2   \frD^3 \gamma,  \frb \frq_1\- \big) T_{\frd } (\gamma; \frq, \frq_1)  \widetilde f \bigg( \gamma , \hskip -1 pt \frac { 1 } {\beta_{\frd} c_{\frq}^2   }  \bigg),
\end{align}
where  $\frb =  (\frd, \frq)^{-1} \frq $, the function $\widetilde f \big( y , \hskip -1 pt   { 1 } / {\beta_{\frd} c_{\frq}^2   }  \big)$ is the Hankel transform of $  f \big( x , \hskip -1 pt   { 1 } / {\beta_{\frd} c_{\frq}^2   }  \big)$ ($x, y \in F_{\infty}^{\times}$) as in Definition \ref{defn: Hankel transform}, with 
\begin{align}\label{16eq: Hankel}
\widetilde {f}_{\vv}  \big(y_{\vv} ,   1 / {\beta_{\frd} c_{\frq}^2} \big) = \int_{F_{\vv}^{\times} }   f_{\vv} \big(x_{\vv} ,   1 /  {\beta_{\frd} c_{\frq}^2} \big) J_{\pi_\varv} (x_{\varv} y_{\vv} )  \nd x_{\vv}, 
\end{align}
and the exponential sum 
\begin{equation}\label{9eq: exponential sum}
\begin{split}
T_{\frd } (\gamma; \frq, \frq_1) =    \sum_{x \, \in ( \fra  \frd (\frc\frD)^{-1} /   \fra \frd (\frc \frD)^{-1} \frq )^{\sstimes} }   \psi_{\infty} \bigg( \frac {   x^{-1}} {\beta_{\frd} c_{\frq}} \bigg)  \mathrm{Kl}_{\frb} (1, - \gamma c_{\frq} / x ; \frq_1  ).
\end{split}
\end{equation} 
To see $\frb =  (\frd, \frq)^{-1} \frq $,   use $x/c_{\frq} \in  (        \fra \frD^{-1}   \frd \frq^{-1} /   \fra \frD^{-1} \frd   )^{\times}$ to deduce $R = \big\{ \vv : \mathrm{ord}_{\vv} (\frd \frq^{-1}) < 0 \big\} = \big\{ \vv : \frp_{\vv} | (\frd, \frq)^{-1} \frq \big\}$ and $\mathrm{ord}_{\vv}  ( (c_{\frq}/x) \fra \frD\-  ) = \mathrm{ord}_{\vv} (\frd^{-1} \frq ) = \mathrm{ord}_{\vv} ((\frd, \frq)^{-1} \frq )$ for every $\vv \in R$.  

We conclude that, up to a constant,   
\begin{equation}\label{16eq: after Voronoi}
\begin{split}
 \SS_{\frd} (  T, R) = \mathop{\mathop{\sum \sum \sum}_{\frq_1 | \frb, \, \frb |\frq, \, \frq \sim \frc  }}_{  (\frb, \frd \frb \frq^{-1}) = (1)}  \frac 1 {\RN (\frb \frq_1 \frq)}   \sum_{\gamma \shskip \in \fra   (\frb \frq_1^{2} \frD^{3})^{-1}   }   A \big(  \fra\- \frb \frq_1^2   \frD^3 \gamma,  \frb \frq_1\- \big) T_{\frd} (\gamma; \frq, \frq_1)  \widetilde f \bigg( \gamma , \hskip -1 pt \frac { 1 } {\beta_{\frd} c_{\frq}^2   }  \bigg),
\end{split}
\end{equation}
where the $\frq$-sum is subject to the conditions in \eqref{15eq: condition on q}, and so are the $\frq_1$- and $\frb$-sums.

\section{Transformation of Exponential Sums}

Next, we need to compute the exponential sum $T_{\frd } (\gamma; \frq, \frq_1) $ as in \eqref{9eq: exponential sum}. 

\subsection{The Special Case $F = \BQ$}

For purely expository purpose, we first compute the exponential sum in the case when $F = \BQ$. For this,  Nunes \cite{Nunes-GL3} quoted a result of Blomer \cite{Blomer} for the corresponding character sum and then set the character $\vchi = 1$. However, when $\vchi = 1$, some of Blomer's manipulations become unnecessary, so it is easier to just compute in a direct manner. 

More precisely,  in the notation of \cite{Blomer,Nunes-GL3}, let $\fra = \frc = \frc_1 = \frD = (1)$,   $\frd = (\delta)$, $\beta_{\frd} = 1/\delta$, $\frq = (c)$, $\frb = (c_1)$($=(c / (c, \delta))$), $\frq_1 = (c_1/n_1)$ ($n_1 | c_1$),  $ c_{\frq} = c $, $c_{\frq_1} = c_1 / n_1$, and $\gamma = n_1^2 n_2 /c_1^3$. After suitable changes, the exponential sum in \eqref{9eq: exponential sum} turns into
\begin{align}\label{17eq: exp sum}
\sumx_{d (\mod c)} e \Big(\frac {d} {c} \Big) S (d, \overline{\delta}_1 n_2; c_1/n_1),
\end{align}
where $\delta_1 = \delta / (c, \delta)$, and $c$ or $n_2$ could   have signs. For simplicity, set $f_1 = c_1/n_1$. Opening the Kloosterman sum, we obtain
\begin{align*}
\sumx_{d (\mod c)} \hskip -1 pt e \Big(\frac {d} {c} \Big) \sumx_{a (\mod f_1)} \hskip -1 pt e \bigg( \hskip -1 pt \frac { {a } {d}} {f_1} + \frac {  \overline{\delta}_1 \overline{a} n_2} { f_1 } \hskip -1 pt \bigg) \hskip -1 pt = \hskip -1 pt \sumx_{a (\mod f_1)} \hskip -1 pt e \bigg( \hskip -1 pt \frac {  \overline{\delta}_1 \overline{a} n_2} { f_1 } \hskip -1 pt \bigg) \sumx_{d (\mod c)} \hskip -1 pt e \bigg( \hskip -1 pt \frac {d ( a c  /f_1 + 1)} {c }  \hskip -1 pt \bigg). 
\end{align*}
The $d$-sum is a Ramanujan sum, and it may be evaluated with the aid of M\"obius inversion. We then arrive at
\begin{align*}
\sum _{c_2 | c}  {c_2} \mu (c/c_2) \mathop{\sumx_{ a (\mod f_1) }}_{ a c  /f_1  \equiv - 1 (\mod  c_2)} e \bigg(  \frac {  \overline{\delta}_1 \overline{a} n_2} { f_1 } \bigg). 
\end{align*}
By necessity, we have $( c_2, c/f_1) = 1$, and hence $c_2 | f_1$. Moreover, we may assume that $c/c_2$ is square-free. 
By introducing the new variable $b  = (\overline{a} + c/f_1) /c_2 $, the sum above is transformed into
\begin{align*}
e \bigg( \hskip -2 pt - \frac { \overline{\delta}_1 (c/f_1) n_2} { f_1 } \bigg) \mathop{\sum _{c_2 | f_1}}_{( c_2, \, c/f_1) = 1} {c_2} \mu (c/c_2)  \mathop{\sum_{ b \shskip (\mod f_1 /c_2 )} }_{ ( b  c_2 - c/f_1, \, f_1 ) = 1 } e \bigg(  \frac {  \overline{\delta}_1 b n_2} { f_1/ c_2  }  \bigg). 
\end{align*}
Finally, M\"obius inversion turns the   innermost sum into
\begin{align*}
\sum_{ f_2 |f_1} \mu (f_2) \mathop{\sum_{ b (\mod f_1/c_2) } }_{ b c_2 \equiv c/f_1 (\mod f_2) } e \bigg(  \frac {  \overline{\delta}_1 b n_2} { f_1/ c_2  }  \bigg). 
\end{align*}
By $b c_2 \equiv c/f_1 (\mod f_2)$, it is easy to see that if $p | f_2$, then  $p \nnmid c_2$ and  $p \| c$ (recall that  $( c_2, c/f_1) = 1$ and that  $c/c_2$ is square-free).  Let $\breve f_{1} $ denote the square-free part of  $f_{1}$. Then $f_2 | \breve f_{1 } $ and $(f_2, c_2) = 1$; in particular, $f_2$ divides $f_1/c_2 $. Consequently, the sum above is equal to
\begin{align*}
\mathop{\mathop{\sum_{ f_2 | \breve f_{1} }}_{(f_2, \, c_2) = 1}}_{ (f_1/c_2 f_2) | n_2 } \frac {f_1} {c_2 f_2} \mu (f_2)   e \bigg(  \frac {   \overline{\delta}_1 \overline{c}_2 (c / f_1) n_2} { f_1/ c_2  }  \bigg) ,
\end{align*} 
in which $\overline{c}_2 c_2 \equiv 1 (\mod f_2)$. We conclude that the exponential sum in \eqref{17eq: exp sum} is equal to 
\begin{align}\label{17eq: exp sum =, Q}
e \bigg( \hskip -2 pt - \frac { \overline{\delta}_1 (c/f_1) n_2} { f_1 } \bigg) \mathop{\mathop{\mathop{\sum \sum}_{c_2 | f_1 , \, f_2 | \breve f_{1} } }_{ ( c_2, \, c/f_1) = (c_2, \, f_2) = 1 }   }_{c_2 f_2 n_2 = f_1 n_2' } \frac{ f_1} {f_2} \mu (c/c_2) \mu (f_2)   e \bigg(  \frac {   \overline{\delta}_1 \overline{c}_2 (c / f_1) n_2'} { f_2 }  \bigg). 
\end{align}

\subsection{The General Case}

By Lemma \ref{lem: Kloosterman as integral}, the Kloosterman sum in \eqref{9eq: exponential sum} is
\begin{align*} 
 \mathrm{Kl}_{\frb} (1, - \gamma c_{\frq} / x ; \frq_1  ) = \varphi (\frq_1) \int_{\pi_{(\frq_{\scalebox{0.35}{$1$}}  \frD)^{\scalebox{0.37}{$-1$}}   }  \shskip \widehat \frOO{\shskip}^{\sstimes}_{\frb} }  \psi_{\frb} \bigg(  y -  \frac { c_{\frq} \gamma } {x y} \bigg) \nd^{\times}\hskip -1pt y  . 
\end{align*}
Let $\widehat{\frO}{}^{\times}  = \prod_{  {\vv} \shskip \nmid \infty } \frO_{\varv}^{\times} $. We may also transform the $x$-sum in \eqref{9eq: exponential sum} into an integral over $ \pi_{\fra  \frd (\frc\frD)^{-1}} \widehat {\frO}{}^{\times} $.  
Precisely, 
\begin{align}\label{17eq: T=I}
T_{\frd} (\gamma; \frq, \frq_1)   =      \varphi(\frq) \varphi(\frq_1) I_{\frd} (\gamma; \frq, \frq_1) . 
\end{align}
where 
\begin{align*}
I_{\frd} (\gamma; \frq, \frq_1) =   \int_{\pi_{\fra \frd (\frc \frD)^{\scalebox{0.37}{$-1$}} } \shskip \widehat \frOO{\shskip}^{\sstimes} } \hskip -1 pt \int_{\pi_{(\frq_{\scalebox{0.35}{$1$}} \frD)^{\scalebox{0.37}{$-1$}}   }  \shskip \widehat \frOO{\shskip}^{\sstimes}_{\frb} }  \psi_{f} \bigg(   \hskip - 1 pt -   \frac {  1 } {\beta_{\frd} c_{\frq} x} \bigg) \psi_{\frb} \bigg( y -  \frac {  c_{\frq} \gamma } {xy}     \bigg) \nd^{\times}\hskip -1pt y \shskip \nd^{\times}\hskip -1pt x.
\end{align*}
On changing $y $ into $ - 1 /   \beta_{\frd} x y $ and then $x$ into $- 1/\beta_{\frd} x$, 
\begin{align*}
I_{\frd} (\gamma; \frq, \frq_1) =    \int_{\pi_{   \frc^{\scalebox{0.37}{$-1$}} \frq_{ \scalebox{0.35}{$1$} } } \shskip \widehat \frOO{\shskip}^{\sstimes}_{\frb}    } \psi_{\frb}  (     { \beta_{\frd} c_{\frq} \gamma } y  ) \int_{\pi_{  (\frc \frD)^{\scalebox{0.37}{$-1$}} } \shskip \widehat \frOO{\shskip}^{\sstimes}    }  \psi_{f} \bigg(      \frac {  x} { c_{\frq}}     \bigg) \psi_{\frb} \bigg(  \frac x y \bigg)  \nd^{\times}\hskip -1pt x \shskip   \nd^{\times}\hskip -1pt y.
\end{align*}
It is clear that $I_{\frd} (\gamma; \frq, \frq_1)$ may be factored into a product of local integrals $I_{\vv} (\gamma; \frq_{\vv}, \frq_{1 \shskip \vv})$ (the $\frd$ is suppressed from the subscript for brevity). For non-Archimedean $\vv$, define $d_{\vv} = \mathrm{ord}_{\vv} (\frD)$, $ r_{\vv} =   \mathrm{ord}_{\vv} (  \frc  ) $, $s_{  \vv} = \mathrm{ord}_{\vv} (    \frq   ) $, and $s_{1 \shskip \vv} = \mathrm{ord}_{\vv} (    \frq_1  ) $. 
For $\frp_{\vv} \nnmid \frb$, the local integral is 
\begin{align}\label{17eq: local integral, 0}
I_{\vv} (\gamma; \frq_{\vv}, \frO_{\varv}) =   \int_{\vv (x)=-r_{\vv}-d_{\vv}  } \psi_{\vv} \bigg(   \frac { x} { c_{\frq}}     \bigg) \nd^{\times}\hskip -1pt x.
\end{align} 
For $ \frp_{\vv}  | \frb$, the local integral is
\begin{align}\label{17eq: local integral}
I_{\vv} (\gamma; \frq_{\vv}, \frq_{1 \shskip \vv}) = \int_{\vv (y)= s_{1 \vv} - r_{\vv} } \psi_{\vv}  (     { \beta_{\frd} c_{\frq} \gamma } y  ) \int_{\vv (x)=-r_{\vv}-d_{\vv}  } \psi_{\vv} \bigg(   x \bigg( \frac 1 y   +   \frac {  1} { c_{\frq}} \bigg)       \bigg) \nd^{\times}\hskip -1pt x \hskip 1 pt  \nd^{\times}\hskip -1pt y.
\end{align}

The following lemma is standard.

\begin{lem}\label{lem: int, 1}
We have
	\begin{equation*}
 (\RN(\frp_{\vv}) -1)	\int_{ {\vv} (x) = - r-d_{\vv} } \psi_{\vv} (a x) \nd^{\times} \hskip -1pt x = \left\{\begin{aligned}
& \RN(\frp_{\vv}) -1 , \ &  & \text{ if }    {\vv} (a) \geqslant r, \\
&  -1 , \  & & \text{ if }    {\vv} (a)= r - 1, \\
& 0, \ & & \text{ if otherwise.}
	\end{aligned} \right.
	\end{equation*}
\end{lem}

For the case $\frp_{\vv} \nnmid \frb$, Lemma \ref{lem: int, 1} implies that the integral in \eqref{17eq: local integral, 0} is just $\mu (\frq_{\vv}) / \varphi (\frq_{\vv})$. 

For the case  $\frp_{\vv} | \frb$, we first observe that
\begin{align*}
\vv ( \beta_{\frd} c_{\frq} \gamma) \geqslant r_{\vv} - 2 s_{1 \shskip \vv} - d_{\vv},
\end{align*} 
for $    ( \beta_{\frd}) = (\frc\frD)^2 (\fra \frd)^{-1}$, $ (c_{\frq} ) = \frc^{-1}\frq$, $ \gamma \in \fra (\frb \frq_1^2 \frD^3)^{-1}   $, and $\frb = (\frd, \frq)^{-1} \frq$. Hence   the  integral in \eqref{17eq: local integral} is reduced to that  in \eqref{17eq: local integral, 0}  if $ s_{1 \shskip \vv} = 0 $, and one may henceforth assume  $ s_{1 \shskip \vv} \geqslant 1 $.

Keep in mind that $\vv (c_{\frq}) = s_{\vv}- r_{  \vv} $ and $\vv (y) = s_{1 \vv}- r_{  \vv}$.  On applying Lemma \ref{lem: int, 1} to the $x$-integral in \eqref{17eq: local integral}, we obtain
\begin{align}
I_{\vv}  (\gamma; \frq_{\vv}, \frq_{1 \shskip \vv}) =   \sum_{ \nu = 0, \shskip1} \frac  { (-1)^{\nu} \RN(\frp_{\vv})^{1-\nu} } { \RN(\frp_{\vv}) -1 } I_{ \vv}^{\nu}  (\gamma; \frq_{\vv}, \frq_{1 \shskip \vv}) , 
\end{align}
with 
\begin{align}\label{17eq: I mu}
I_{  \vv}^{\nu}  (\gamma; \frq_{\vv}, \frq_{1 \shskip \vv}) = \int_{ {\sstyle \hskip -34 pt \vv (y) = s_{1 \vv}- r_{  \vv}} \atop {\sstyle \vv (y+ c_{\frq}) \geqslant \shskip s_{\vv} + s_{1 \vv} - r_{\vv} - \nu}  } \psi_{\vv}  (     { \beta_{\frd} c_{\frq} \gamma } y  ) \nd^{\times}\hskip -1pt y , \qquad \text{($\nu = 0, 1$)}. 
\end{align}
First, consider the case when $s_{\vv} > \nu$.  For $ s_{1 \shskip \vv} < s_{\vv} $, we have  $\vv (y+ c_{\frq}) = \vv (y) = s_{1 \shskip \vv} - r_{\vv}  < s_{\vv} + s_{1 \shskip \vv} - r_{\vv} - \nu $, and hence $I_{  \vv}^{\nu}  (\gamma; \frq_{\vv}, \frq_{1 \shskip \vv}) = 0$ as the integration is on an empty set. For  $ s_{1 \shskip \vv} = s_{\vv} $, we introduce the new variable $w =  y + c_{\frq} $ so that the resulting $w$-integral is on $ \frp_{\vv}^{2s_{\vv}  - r_{\vv} - \nu} $ (since $\vv (w) > \vv (y) $ by the condition  $s_{\vv} > \nu$):
\begin{align*}
	I_{  \vv}^{\nu}  (\gamma; \frq_{\vv}, \frq_{  \vv}) = \frac {\RN(\frD_{\vv})^{1/2} \RN (\frp_{\vv})^{s_{\vv} - r_{\vv}+1} } {\RN (\frp_{\vv})-1} \cdot \psi_{ \vv } (- \beta_{\frd} c_{\frq}^2 \gamma ) \int_{  \frp_{\vv}^{2s_{\vv}  - r_{\vv} - \nu}  } \psi_{\vv}  (     { \beta_{\frd} c_{\frq} \gamma } w  ) \nd w .
\end{align*}
Recall that $\psi_{\vv}$ has conductor $\frD_{\vv}^{-1}$ and that $ \frp_{\vv}^{2s_{\vv}  - r_{\vv} - \nu} $ has measure $ \RN(\frD_{\vv})^{-1/2} \RN (\frp_{\vv})^{  r_{\vv} + \nu - 2s_{\vv}}$. It follows that $I_{  \vv}^{\nu}  (\gamma; \frq_{\vv}, \frq_{ \vv}) = 0$ unless $\vv (\beta_{\frd} c_{\frq} \gamma) \geqslant r_{\vv} - 2 s_{\vv} - d_{\vv} + \nu $, in which case 
\begin{align*}
I_{  \vv}^{\nu}  (\gamma; \frq_{\vv}, \frq_{  \vv}) = \psi_{ \vv } (- \beta_{\frd} c_{\frq}^2 \gamma )   \frac { \RN (\frp_{\vv})^{ \nu  } } {\varphi (\frq_{\vv})} . 
\end{align*}
Next, we consider the remaining case when $s_{\vv} = s_{1 \shskip \vv} = \nu = 1$. Then the second condition in \eqref{17eq: I mu} reads $ \vv (y+ c_{\frq}) \geqslant   1 - r_{\vv} $, and it can be dropped because it is implied by the first condition $ \vv (y ) = 1 - r_{\vv} $, along with $\vv (c_{\frq}) = 1 - r_{\vv}$. Thus
\begin{align*}
	I_{\vv}^{1}  (\gamma; \frp_{\vv}, \frp_{  \vv}) = \int_{    \vv (y) = 1 - r_{  \vv}  } \psi_{\vv}  (     { \beta_{\frd} c_{\frq} \gamma } y  ) \nd^{\times}\hskip -1pt y .
\end{align*} By Lemma  \ref{lem: int, 1},
\begin{align*}
I_{  \vv}^{1}  (\gamma; \frp_{\vv}, \frp_{  \vv}) = \mathop{\sum_{ \mu = 0, \shskip 1}}_{\vv (\beta_{\frd} c_{\frq} \gamma) \geqslant r_{\vv}-1-d_{\vv}- \mu }  \frac  { (-1)^{\mu} \RN(\frp_{\vv})^{1-\mu} } { \RN(\frp_{\vv}) -1 } . 
\end{align*}


\begin{lem}\label{lem: I =, local}
 	We have the following formulae for $ I_{\vv}  (\gamma; \frq_{\vv}, \frq_{1 \shskip \vv}) $.

	
	{\rm(1)} $I_{\vv} (\gamma; \frq_{\vv}, \frO_{\varv}) = \mu (\frq_{\vv}) / \varphi (\frq_{\vv})$.
	
	{\rm(2)} For $\mathrm{ord}_{\vv} (\frq) > 1$, we have $I_{\vv} (\gamma; \frq_{\vv}, \frq_{1 \shskip \varv}) = 0$ if $ \mathrm{ord}_{\vv} (\frq) > \mathrm{ord}_{\vv} (\frq_1) $, and
	\begin{align*}
	I_{\vv} (\gamma; \frq_{\vv}, \frq_{\varv}) = \psi_{ \vv } (- \beta_{\frd} c_{\frq}^2 \gamma )   \frac  {   \RN(\frp_{\vv})  } { \RN(\frp_{\vv}) -1 } \frac 1 {\varphi(\frq_{  \vv})} \mathop{\sum_{ \nu = 0, \shskip 1}}_{\vv (\beta_{\frd} c_{\frq}^2 \gamma \frD) \geqslant \shskip \nu -  {\vv} (\frq)   } (- 1)^{\nu} . 
	\end{align*}
	
	{\rm(3)} We have 
	\begin{align*}
	I_{\vv} (\gamma; \frp_{\vv}, \frp_{\varv}) = \psi_{ \vv } (- \beta_{\frd} c_{\frq}^2 \gamma )   \frac  {   \RN(\frp_{\vv})  } { (\RN(\frp_{\vv}) -1)^2 }  \mathop{\mathop{\sum\sum}_{0  \shskip \leqslant  \shskip  \mu \shskip \leqslant \shskip \nu \shskip \leqslant 1  }}_{\vv (\beta_{\frd} c_{\frq}^2 \gamma \frD) \geqslant \shskip \nu - \mu - 1  } \frac {(- 1)^{\nu + \mu}} {\RN(\frp_{  \vv})^{\mu}} \psi_{ \vv } (  \beta_{\frd} c_{\frq}^2 \gamma )^{\mu} . 
	\end{align*}
\end{lem}

As a consequence of \eqref{17eq: T=I} and Lemma \ref{lem: I =, local} above, it is straightforward to deduce the following formula for $T_{\frd} (\gamma; \frq, \frq_1)  $. The reader may compare \eqref{17eq: formula for T} with \eqref{17eq: exp sum =, Q}. 

\begin{cor}
  Let $\breve \frq_{1} $ denote the square-free part of $\frq_1$. We have $T_{\frd} (\gamma; \frq, \frq_1) = 0$ unless $(\frq_1, \frq \frq_1^{-1}) = (1)$, in which case
\begin{align}\label{17eq: formula for T}
T_{\frd} (\gamma; \frq, \frq_1) \hskip -1 pt = \hskip -1 pt \psi_{\frb} (- \beta_{\frd} c_{\frq}^2 \gamma )  \mu (\frq \frq_1^{-1}) \RN (\frq_1) \hskip -2 pt \mathop{ \mathop{\mathop{\sum \sum}_{\frq_2 | \frq_1 , \ \frf | \breve \frq_{1}  } }_{ (\frq_2,\, \frq \frq_1^{-1}) = (\frq_2, \shskip \frf) = (1)} }_{  \gamma \shskip \in \fra (\frb \frq_1 \frq_2 \frf \frD^3)^{-1}  } \hskip -2 pt \frac{\mu (\frq_1 \frq_2^{-1}) \mu (\frf)}{\RN (\frf)} \psi_{\frf} (\beta_{\frd} c_{\frq}^2 \gamma ),
\end{align} 
where $\psi_{\frb}$ and $\psi_{\frf}$ are defined as in Definition {\rm\ref{defn: psi b ...}}.  
\end{cor}

\section{Further Reductions}

Set $\mathfrak{r} = (\frd, \frq)^{-1} \frd $. Note that $\vv ( \beta_{\frd} c_{\frq}^2 \gamma \frD) \geqslant 0$ and hence $\psi_{ \vv } ( \beta_{\frd} c_{\frq}^2 \gamma) = 1 $ if $\frp_{  \vv} \hskip -1 pt \nnmid  \hskip -1 pt \mathfrak{r} \mathfrak{b} $. It follows that the $\psi_{\frb} (- \beta_{\frd} c_{\frq}^2 \gamma )$ in \eqref{17eq: formula for T} can be written as 
\begin{align}\label{18eq: psi = }
\psi_{\frb} (- \beta_{\frd} c_{\frq}^2 \gamma ) = \psi_{\infty} (\beta_{\frd} c_{\frq}^2 \gamma ) \psi_{\mathfrak{r}} ( \beta_{\frd} c_{\frq}^2 \gamma ), 
\end{align}
where we have used the fact that $\psi$ is trivial on $F$. 

Next, in view of \eqref{15eq: w (x; Lambda)}, \eqref{16eq: Hankel}, and \eqref{11eq: defn of w (x, Lmabda), R}--\eqref{11eq: Hankel}, we have
\begin{align*}
f_{\vv} \bigg(x_{\vv} , \frac 1  {\beta_{\frd} c_{\frq}^2} \bigg) = \varww_{\vv} \bigg(\frac {x_{\vv}} {R_{\vv}} , \frac {R_{\vv}} {\beta_{\frd} c_{\frq}^2} \bigg), \qquad \widetilde{f}_{\vv} \bigg(y_{\vv} , \frac 1  {\beta_{\frd} c_{\frq}^2} \bigg) = R_{\vv} \widetilde \varww_{\vv} \bigg( {R_{\vv}} {y_{\vv}}  , \frac {R_{\vv}} {\beta_{\frd} c_{\frq}^2} \bigg). 
\end{align*}
We combine the $\psi_{\infty} (\beta_{\frd} c_{\frq}^2 y )$ that occurred in \eqref{18eq: psi = } with  $\widetilde f \big( y , \hskip -1 pt   { 1 } / {\beta_{\frd} c_{\frq}^2   }  \big) $ to form
\begin{align}\label{18eq: tilde f nature}
\widetilde{f}^{\snatural} \bigg(y  , \frac 1  {\beta_{\frd} c_{\frq}^2} \bigg) =  \frac 1 {\RN(R)} \psi_{\infty} (\beta_{\frd} c_{\frq}^2 \gamma ) \widetilde{f}  \bigg(y  , \frac 1  {\beta_{\frd} c_{\frq}^2} \bigg)
\end{align}
so that, according to \eqref{10eq: defn of w nat}, 
\begin{align}\label{18eq: f = w}
 \widetilde{f}_{\vv}^{\snatural} \bigg(y_{\vv}  , \frac 1  {\beta_{\frd} c_{\frq}^2} \bigg) =   \widetilde \varww_{\vv}^{\snatural} \bigg( {R_{\vv}} {y_{\vv}}  , \frac {R_{\vv}} {\beta_{\frd} c_{\frq}^2} \bigg). 
\end{align}
In light of our analysis in \S    \ref{sec: Hankel 1}, 
it is very natural to introduce $ \Phi  (x)$ such that
\begin{align}\label{18eq: Phi = w tilde}
 \Phi  ({\beta_{\frd} c_{\frq}^2} y) = \frac  {{\textstyle \sqrt{\RN ( y ) \RN (R) } }  } {\RN^{\snatural} (M) \RN (T)^{1+\vepsilon}   } \widetilde{f}^{\snatural} \bigg(   y  , \frac 1  {\beta_{\frd} c_{\frq}^2} \bigg)    .  
\end{align}

Set $\frd_0 = (\frd, \frq)$, $\frb = \frb_1 \frq_1$, and $\frq_1 = \frf_1 \frq_2 $.  By  \eqref{17eq: formula for T}, \eqref{18eq: psi = }, \eqref{18eq: tilde f nature}, and \eqref{18eq: Phi = w tilde}, we can now reorganize the sum $\SS_{\frd} (  T, R) $ in \eqref{16eq: after Voronoi}   as follows:
\begin{align}\label{18eq: S (T, R), 3}
 \RN^{\snatural} (M) \RN (T)^{1+\vepsilon} {\textstyle \sqrt{\RN(R)}} \sum_{\frd =  \frd_0 \mathfrak{r}  } \frac {1} {\RN (\frd_0)} \mathop{\mathop{ \sum \sum   }_{ (\frb_1, \frf_1) = (1) } }_{ (\frb_1 \frf_1  , \frd) = (1) }  \frac {\mu (\frd_0 \frb_1     ) \mu (  \frf_1  ) } {\RN (\frb_1  \frf_1    )^2}  \sum_{  \frf | \frf_{1 } }   \frac{ \mu (\frf)}{\RN (\frf)} \SS^{\frd ,   \frr}_{ \frb_1,   \frf_1,   \frf } (T, R), 
\end{align}
with  
\begin{equation}
\SS^{\frd ,   \mathfrak{r}}_{ \frb_1,   \frf_1,   \frf } (T, R) \hskip -1 pt = \hskip -3 pt \mathop{\sum_{ (\frq_2,  \frd \frb_1  \frf )   = (1)} }_{ \frq = \frd_0 \frb_1 \frf_1 \frq_2  \sim \frc }   \sum_{\gamma \shskip \in \fra (\frb_1 \frf \frf_1^2 \frq_2^3 \frD^3)^{-1}   } \hskip -5 pt \frac {  A  (  \fra\- \frb_1 (\frf_1  \frq_2   \frD)^3 \gamma,  \frb_1  \hskip -1 pt )} {\RN(\frq_2)^2 \sqrt{\RN ( \gamma )   } }  \psi_{\frr \frf   } (\beta_{\frd} c_{\frq}^2 \gamma ) \Phi (\beta_{\frd} c_{\frq}^2 \gamma )   ,
\end{equation} 
where the sums over $\frb_1$, $\frf_1$, and $\frq_2$ must be subject to the conditions  (see  \eqref{15eq: condition on q}):
\begin{align}\label{15eq: condition on q2}
V(   \frb_1 \frf_1 \frq_2 )_{\vv}  \Lt  {\textstyle \sqrt{R_{\vv}V(\frd )_{\vv} }} /     T_{\vv} V( \frd_0 )_{\vv}  
, \qquad  \vv | \infty .
\end{align}  
Let $\frc_2 \in \widetilde C_F $ be such that $ \frd_0 \frb_1 \frf_1 \cdot \frc_2  \sim \frc $. Set $c_{\frq_2} = [\frc_2^{-1} \frq_2]$ and $\frb_2 =  (\frd_0  \frb_1)^{-1}    \frr   \frf_1    \frc_2   \frD $. By \eqref{2eq: choice of [a]} and \eqref{8eq: defn of Va}, we have
\begin{align}\label{18eq: cq2}
|c_{\frq_2}|_{\vv} \asymp V(\frq_2)_{\vv}^{N_{\vv}} , \quad \vv|\infty. 
\end{align} 
 Substituting $\gamma$ by $  \gamma / \beta_{\frd} c_{\frq}^2 c_{\frq_2}  $, up to a  harmless factor, we have
\begin{equation}\label{18eq: S... (T, R)}
 \SS^{\frd ,   \mathfrak{r}}_{ \frb_1,   \frf_1,   \frf } (T, R)   =  \frac {\RN (\frd_0 \frb_1 \frf_1)} {\sqrt{\RN (\frd)}} \mathop{\sum_{ \frq_2 \sim \frc_2 } }_{ (\frq_2,  \frd \frb_1  \frf )   = (1) }   \sum_{\gamma \frb_2 \shskip \subset \frf_1 \frf^{-1}   }   \frac {  A  (  \gamma \frb_2,  \frb_1  \hskip -1 pt )} {\sqrt{\RN(\gamma \frq_2)} }  \psi_{\frr \frf   } (  \gamma / c_{\frq_2} ) \Phi (  \gamma / c_{\frq_2} ).
\end{equation} 

By Corollary \ref{cor: Hankel small}, \eqref{8eq: local size of beta and c}, and \eqref{18eq: cq2}, for each $\vv | \infty$, we have $ \Phi_{\vv} (   \gamma / c_{\frq_2} ) = O (T_{\vv}^{-A}) $ when $ |\gamma|_{\vv}^{1/N_{\vv}} >  \sqrt{R_{\vv} V(\frd)_{ \vv}  }/ V(\frq \frq_2^{-1})_{\vv}  $.  Arguing as in \S \ref{sec: first reductions}, we can impose the condition $ \gamma \in F_{\infty}^{ \text{\O}  } \big(\hskip -1 pt \sqrt{R  V(\frd) } / V(\frq \frq_2^{-1})  \big)$ (see \eqref{4eq: defn of FS (T)}), with the loss of  a negligible error. Since we also have $\RN (\gamma) \Gt \RN (\frd_0  \frb_1) / \RN  ( \frr   \frf )  \Gt 1 / \RN (\frd \frq) \Gt 1 / \RN (T)^{1/2+ \vepsilon}$, due to \eqref{15eq: R < T3} and \eqref{15eq: condition on q}, each $|\gamma|_{\vv} \Gt T_{\vv}^{- A} $, so there is no issue with a $\varDelta$-adic partition in the $\gamma$-sum as we had in  \S \ref{sec: first reductions}.  Again, the assumption that $T_{\vv} \geqslant \RN (T)^{\vepsilon}$ for all $\vv|\infty$ is required here. 

\begin{defn}\label{defn: F D (V)}
	For   $V \in \bfra_+$ and $\varDelta > 1$, define 
	\begin{align} \label{18eq: defn of F(V)}
	F_{\infty}^{\varDelta} (V) = \big\{ x \in F_{\infty} : |x_{\vv}| \in [V_{\vv}, \sqrt{\varDelta} V_{\vv} ) \text{ for all } \vv \in S_{\infty}   \big\} .
	\end{align}
\end{defn}

Let $\sigmaup \in \{ \oldstylenums{0}, -, +, \flat \}^{|S_{\infty}|}$ ($\sigmaup_{\vv} = \flat$ only if $\vv$ is complex). As we have seen in \S\S  \ref{sec: Hankel 1} and \ref{sec: Hankel, II}, each local $\Phi_{\vv}  $ equals (the sum of) $\Phi^{\sigmaup_{\vv}}$ for $\sigmaup_{\vv} = \oldstylenums{0}, -, +, \flat$ in various circumstances; with abuse of notation, $\Phi^{\flat} = \Phi^{\flat}_{\rho}$ or $\Phi^{\flat}_0$. 
The product of such $ \Phi^{\sigmaup_{\vv}} (x_{\vv}) $ will be denoted by $\Phi^{\sigmaup}  (x)$. 

\begin{lem}\label{lem: final}
	Let notation be as above. Fix $\varDelta > 1$ as in {\rm \S  \ref{sec: Hankel, II}}. Let $C, \varGamma \in \bfra_+$ satisfy
	\begin{align}
	1 \Lt \RN (C) \Lt \frac {\textstyle \sqrt{\RN(R) \RN  (\frd )  }}   {\RN (T)  \RN( \frd_0 \frb_1 \frf_1 ) }, \quad  \frac {\RN (\frd_0  \frb_1)}   {\RN  ( \frr   \frf )}  \Lt {\RN (\varGamma)} \Lt \frac {\sqrt{\RN (R) \RN (\frd) }  } {\RN (\frd_0 \frb_1 \frf_1)   },
	\end{align} 
	and
	\begin{align}\label{19eq: range of C}
1 \Lt C_{\vv} \Lt \frac {\textstyle \sqrt{R_{\vv} V  (\frd )_{\vv}  }}   {T_{\vv}  V ( \frd_0 \frb_1 \frf_1 )_{\vv} },	 
\qquad 
\frac 1	{T_{\vv}^{A}} \Lt  \varGamma_{\vv}  \Lt \frac {\sqrt{R_{\vv} V(\frd)_{ \vv} }}  {V( \frd_0 \frb_1 \frf_1 )_{\vv} }. 
	\end{align}
	Define
	\begin{equation}
	\SS^{\sigmaup}  (T, R; \varGamma, C)   =  \mathop{ \mathop{\sum_{ \frq_2 \sim \frc_2 } }_{ (\frq_2,  \frd \frb_1  \frf )   = (1) } }_{c_{\frq_2} \in F_{\infty}^{\varDelta} (C) }  \mathop{\sum_{\gamma \frb_2 \shskip \subset \frf_1 \frf^{-1}    } }_{\gamma \shskip \in F_{\infty}^{\varDelta} (\varGamma) }  \frac {  A  (  \gamma \frb_2 ,  \frb_1  \hskip -1 pt )} {\sqrt{\RN(\gamma \frq_2)} }  \psi_{ \mathfrak{r} \frf } (   \gamma / c_{\frq_2} ) \Phi^{\sigmaup}  (  \gamma / c_{\frq_2} ).
	\end{equation} 
	Then 
	\begin{align}\label{18eq: bound for S}
\begin{aligned}
\SS^{\sigmaup}  (T, R; \varGamma, C)  \Lt & \  \frac {  \RN (\frd)^{5/4}  \RN (\frf)^{57/64}  } {\RN (\frd_0)^2 \RN (\frb_1)^{57/64} \RN (  \frf_1  )^{25/64}}          \frac { \RN (T)^{\vepsilon}  {\RN(R)^{1/4}    }}   {\RN (T)^{1/2}    } \\
& + \frac {  \RN (\frd)^{3/2}  \RN (\frf)^{57/64}  } {\RN (\frd_0)^{5/2} \RN (\frb_1)^{89/64} \RN (  \frf_1  )^{57/64}}  \frac {  \RN (T)^{\vepsilon}  {\RN(R)^{1/2}     }}   {\RN (T)^{4/3}    }  . 
\end{aligned}
	\end{align}
\end{lem}

Granted that Lemma \ref{lem: final} holds, we can now finish the proof of Proposition \ref{prop: main}. 
By the discussions above Lemma  \ref{lem: final},  the   bound in \eqref{18eq: bound for S} applies to the double sum in  \eqref{18eq: S... (T, R)}, and hence 
\begin{align*} 
\begin{aligned}
\SS^{\frd ,   \mathfrak{r}}_{ \frb_1,   \frf_1,   \frf } (T, R)  \Lt & \      \frac {  \RN (\frd)^{3/4} \RN (\frb_1)^{7/64} \RN (  \frf_1  )^{39/64} \RN (\frf)^{57/64}  } {\RN (\frd_0)   }          \frac { \RN (T)^{\vepsilon}  {\RN(R)^{1/4}    }}   {\RN (T)^{1/2}    } \\
& +      \frac {  \RN (\frd)  \RN (  \frf_1  )^{7/64} \RN (\frf)^{57/64}  } {\RN (\frd_0)^{3/2} \RN (\frb_1)^{25/64} }  \frac {  \RN (T)^{\vepsilon}  {\RN(R)^{1/2}     }}   {\RN (T)^{4/3}    } .
\end{aligned}
\end{align*}
Applying this to \eqref{18eq: S (T, R), 3} leads to  the estimate in \eqref{15eq: bound for S}.  

\section{Completion: Proof of Lemma \ref{lem: final}}

For each $\sigmaup_{\vv} \in \{ \oldstylenums{0}, -, +, \flat \}$, we first apply Lemma \ref{lem: final, R}, \ref{lem: final, C, 1}, and \ref{lem: final, C, 2} to the local $ \Phi^{\sigmaup_{\vv} } (  \gamma / c_{\frq_2} ) $, so that  the sum $\SS^{\sigmaup}  (T, R; \varGamma, C)$ is reformulated in a form which is ready for the hybrid large sieve in Corollary \ref{cor: large sieve}. Then, after applying Corollary \ref{cor: large sieve}, we infer that  
\begin{align} \label{19eq: S = T x T}
\SS^{\sigmaup}  (T, R; \varGamma, C) \Lt \sqrt{\ST (T, R) } \sqrt{ \ST^{\sigmaup} (T, R; \varGamma, C)  }
\end{align} 
with 
\begin{align}
\ST (T, R) = \RN (T)^{\vepsilon} \frac {\RN (\frr \frf) } {\RN(\varGamma)}  \mathop{\sum_{\gamma \frb_2 \shskip \subset \frf_1 \frf^{-1}    } }_{\gamma \shskip \in F_{\infty}^{\varDelta} (\varGamma) } |A  (  \gamma \frb_2 ,  \frb_1  \hskip -1 pt )|^2 ,
\end{align}
and
\begin{align}
\ST^{\sigmaup} (T, R; \varGamma, C) =  \frac {\textstyle   { 1 }} {\RN (A^{\sigmaup})} \bigg( \RN(U^{\sigmaup}) + \frac {\RN (C)} {\RN (\frr \frf)} \bigg)   \bigg(\RN(U^{\sigmaup})  + \frac {\RN (\varGamma)}{\RN(\frd_0 \frb_1)} \bigg) .
\end{align}
Recall that    $\rho$ or $0$ is suppressed from subscripts when $\sigmaup = \flat$. 
By Lemma \ref{lem: second moment}, with the Kim--Sarnak exponent $\theta = \frac {7} {32}$ as in \eqref{3eq: Kim-Sarnak, A(n), GL3}, 
\begin{align}\label{19eq: T < }
\ST (T, R)  \Lt \RN (T)^{\vepsilon} \frac { \RN (\frr)^2 \RN (  \frf_1  )^{7/32} \RN (\frf)^{57/32}  } {\RN (\frd_0) \RN (\frb_1)^{25/32}} . 
\end{align}
Set $\sqrt{\varDelta} X = \varGamma / C $. We have
\begin{align}\label{20eq: T = 1+2+3}
\ST^{\sigmaup} (T, R; \varGamma, C) \Lt \frac {     \RN(U^{\sigmaup})^2 } {\RN (A^{\sigmaup})}  + \frac {     \RN(C U^{\sigmaup} )  } {\RN (A^{\sigmaup})} (1 + \RN (X)) + \frac {     \RN(C^2 X  )  } {\RN (A^{\sigmaup})} . 
\end{align}
For the three summands in \eqref{20eq: T = 1+2+3}, we claim: 
\begin{align}\label{20eq: bounds, globle, 1}
& \frac {     \RN(U^{\sigmaup})^2 } {\RN (A^{\sigmaup})} \Lt \RN (T)^{\vepsilon}, \\
\label{20eq: bounds, globle, 2} & \frac {     \RN(C U^{\sigmaup} )  } {\RN (A^{\sigmaup})} (1 + \RN (X)) \Lt    \frac {\RN (T)^{\vepsilon}  \sqrt{\RN(R) \RN  (\frd )  }}   {\RN (T)  \RN( \frd_0 \frb_1 \frf_1 ) } , \\
\label{20eq: bounds, globle, 3} &   \frac {     \RN(C^2 X  )  } {\RN (A^{\sigmaup})} \Lt  \frac {   {\RN(R) \RN  (\frd )  }}   {\RN (T)^{8/3}  \RN( \frd_0 \frb_1 \frf_1 )^2 } . 
\end{align}
Lemma \ref{lem: final} readily follows from \eqref{19eq: S = T x T}, \eqref{19eq: T < }--\eqref{20eq: bounds, globle, 3}, with the observation that  \eqref{20eq: bounds, globle, 1} can be absorbed into \eqref{20eq: bounds, globle, 2}. 

Clearly, \eqref{20eq: bounds, globle, 1}--\eqref{20eq: bounds, globle, 3}  follow from the local inequalities: 
\begin{align}\label{20eq: bounds, local, 1}
 \frac {       U^{\sigmaup_{\vv}}   } { \sqrt{A^{\sigmaup_{\vv}}} } \Lt T_{\vv}^{\vepsilon}, \qquad  \frac {       U^{\sigmaup_{\vv}}   } { A^{\sigmaup_{\vv}}  } (1 + X_{\vv} ) \Lt T_{\vv}^{  \vepsilon}    , \qquad   \frac {   C_{\vv}^2 X_{\vv}   } { A^{\sigmaup_{\vv} } } \Lt  \frac {  R_{\vv} Z_{\vv}^2 }   {T_{\vv}^{8/3}  },
\end{align}
in which 
 $Z_{\vv} = \sqrt{  V  (\frd )_{\vv}  } / V ( \frd_0 \frb_1 \frf_1 )_{\vv}$ so that the range of $C_{\vv}$ in \eqref{19eq: range of C} reads 
\begin{align}\label{20eq: range of Cv}
1 \Lt C_{\vv} \Lt \frac {\textstyle \sqrt{R_{\vv}    } Z_{\vv} }   {T_{\vv}    } .
\end{align}
Note that the third inequality in \eqref{20eq: bounds, local, 1}  can be deduced from
\begin{align} \label{20eq: bounds, local, 3}
  \frac {   X_{\vv}   }   { A^{\sigmaup_{\vv} } } \Lt    \frac { 1  }    {T_{\vv}^{2/3}   }  . 
\end{align}
Moreover, if we define $ \varLambda_{\vv} = R_{\vv} / \beta_{\frd} c_{\frq}^2 $ (see \eqref{18eq: f = w}), then 
\begin{align}\label{19eq: Lambda = }
|\varLambda_{\vv}| \asymp \frac { R_{\vv} Z_{\vv}^2 } { C_{\vv}^2 } . 
\end{align}
It remains to verify \eqref{20eq: bounds, local, 1} for various cases  (alternatively, \eqref{20eq: bounds, local, 3} for all the cases other than $\sigmaup = +$). 

\subsection{The Case $\sigmaup = \oldstylenums{0}$} 

In this case, we have $U^{\oldstylenums{0}} = T_{\vv}^{\vepsilon}$, $A^{\oldstylenums{0}} = 1$, and $X_{\vv}  \leqslant T^{\vepsilon}_{\vv} / |\varLambda_{\vv}| \Lt T_{\vv}^{\vepsilon} / T_{\vv}^2$.  Thus \eqref{20eq: bounds, local, 1} and \eqref{20eq: bounds, local, 3} are obvious. 

\subsection{The Case $\sigmaup = -$} In this case, we have $U^{-} = |X_{\vv} \varLambda_{\vv}|^{1/3}$, $A^{-} =  {|\varLambda_{\vv}|} $, and $X_{\vv} \Lt \sqrt{|\varLambda_{\vv}|}$. Therefore
\begin{align*}
\frac {       U^{-}   } { \sqrt{A^{-}}  } = \frac{X_{\vv}^{1/3}} {|\varLambda_{\vv}|^{1/6}} \Lt 1, \quad   \frac {       U^{-}   } {  {A^{-}}  } (1 + X_{\vv} ) \Lt \frac {X_{\vv}^{1/3}} {|\varLambda_{\vv}|^{1/6}} \Lt 1    , \quad    \frac {    X_{\vv}   } { A^{ - } } \Lt \frac 1 {\sqrt{|\varLambda_{\vv}|}} \Lt \frac 1 {T_{\vv}},
\end{align*}
as desired. 

\subsection{The Case $\sigmaup = +$} In this case, we always have $X_{\vv} \asymp \sqrt{|\varLambda_{\vv}|}$, and, by \eqref{19eq: Lambda = }, 
\begin{align}\label{19eq: C2 X = }
{   C_{\vv}  X_{\vv}   } \asymp   {  \sqrt{R_{\vv}    } Z_{\vv} }    .
\end{align} 

For   $ T_{\vv} \Lt X_{\vv} \Lt T_{\vv}^{2-\vepsilon}  $, we have $U^+ = T_{\vv}^2 / X_{\vv} $, $A^{+} = T_{\vv}^{2- \vepsilon}$. Consequently, 
\begin{align*}
\frac {       U^{+}   } { \sqrt{A^{+}}  } = \frac {T_{\vv}^{1+\vepsilon}} {X_{\vv}} \Lt   {T_{\vv}^{\vepsilon}}  , \qquad  \frac {       U^{+}   } {  {A^{+}}  } (1 + X_{\vv} )   \Lt T_{\vv}^{\vepsilon},  
\end{align*} 
and, it follows from \eqref{19eq: C2 X = }   that 
\begin{align*}
\frac {   C_{\vv}^2 X_{\vv}   } {A^{+} } = \frac {  T_{\vv}^{\vepsilon} C_{\vv}^2 X_{\vv}   } {T_{\vv}^{2} } \asymp \frac {T_{\vv}^{\vepsilon}     {R_{\vv}    } Z_{\vv}^2 } {T_{\vv}^2 X_{\vv} } \Lt \frac {T_{\vv}^{\vepsilon}    {R_{\vv}    } Z_{\vv}^2 } {T_{\vv}^3 } .
\end{align*}

For   $   X_{\vv} \Gt T_{\vv}^{2-\vepsilon}  $, we have $U^+ = T_{\vv}^{\vepsilon} $, $A^{+} = X_{\vv} $. Consequently, 
\begin{align*}
\frac {       U^{+}   } { \sqrt{A^{+}}  } = \frac {T_{\vv}^{ \vepsilon}} {\sqrt{X_{\vv}}}  \Lt \frac{T_{\vv}^{\vepsilon}} {T_{\vv} }, \qquad  \frac {       U^{+}   } {  {A^{+}}  } (1 + X_{\vv} )   \Lt T_{\vv}^{\vepsilon},  
\end{align*} 
and, it follows from \eqref{19eq: C2 X = }  that 
\begin{align*}
\frac {   C_{\vv}^2 X_{\vv}   } {A^{+} } =   {    C_{\vv}^2   }  \asymp \frac {   {R_{\vv}    } Z_{\vv}^2 } {X_{\vv}^2 } \Lt \frac {T_{\vv}^{\vepsilon}  {R_{\vv}    } Z_{\vv}^2 } {T_{\vv}^4 } .
\end{align*}

\subsection{The Case $\sigmaup = \flat$} In this case, we always have  $|\varLambda_{\vv}|\asymp T_{\vv}^2$, and $X_{\vv} \asymp T_{\vv}$. 

Firstly, for $T_{\vv}^{\vepsilon} / \min \big\{ (T_{\vv}/M_{\vv})^{1/2},  T_{\vv}^{1/4 } \big\}  < \rho \Lt 1$,  we have $ A_{\rho}^{\flat} = T_{\vv}^2 \rho$ and  $ U_{\rho}^{\flat} = T_{\vv} \shskip \rho^2 $. Therefore 
\begin{align*}
\frac {       U_{\rho}^{\flat}   } { \sqrt{A_{\rho}^{\flat} }  } = \rho^{3/2} \Lt 1, \quad \frac {       U_{\rho}^{\flat}  } {  {A_{\rho}^{\flat} }  } (1 + X_{\vv} ) \Lt   {\rho}   \Lt 1 , \quad \frac {     X_{\vv}   } {A_{\rho}^{\flat} } \Lt \frac 1 {T_{\vv} \rho} < \frac {1} {T_{\vv}^{3/4 + \vepsilon}} . 
\end{align*} 

Secondly, it follows from \eqref{13eq: U and A flat 0} that
\begin{align*}
 \frac {       U_{0}^{\flat}   } { \sqrt{A_{0}^{\flat} }  } = \frac {T_{\vv}^{\vepsilon}} {T_{\vv}^{1/3}} , \quad \frac {       U_{0}^{\flat}  } {  {A_{0}^{\flat} }  } (1 + X_{\vv} ) \Lt  \frac {T_{\vv}^{\vepsilon} } {T_{\vv}^{1/6}} , \quad \frac {     X_{\vv}   } {A_{0}^{\flat} } \Lt   \frac {1} {T_{\vv}^{2/3}},
\end{align*}
if $ T_{\vv}^{\vepsilon} \leqslant M \leqslant T_{\vv}^{1/3} $, 
\begin{align*}
\frac {       U_{0}^{\flat}   } { \sqrt{A_{0}^{\flat} }  } = \frac {T_{\vv}^{\vepsilon}} {(M_{\vv} T_{\vv})^{1/4}}   , \quad \frac {       U_{0}^{\flat}  } {  {A_{0}^{\flat} }  } (1 + X_{\vv} ) \Lt  \frac {T_{\vv}^{\vepsilon} } {M_{\vv}^{1/2}} , \quad \frac {     X_{\vv}   } {A_{0}^{\flat} } \Lt   \frac {1} {(M_{\vv} T_{\vv})^{1/2}}  < \frac {1 } {  T_{\vv}^{2/3}},
\end{align*}
if $ T_{\vv}^{1/3} < M_{\vv} \leqslant T_{\vv}^{1/2} $, and
\begin{align*}
\frac {       U_{0}^{\flat}   } { \sqrt{A_{0}^{\flat} }  } = \frac {M_{\vv}^{3/4} T_{\vv}^{\vepsilon}} { T_{\vv}^{3/4}}   , \quad \frac {       U_{0}^{\flat}  } {  {A_{0}^{\flat} }  } (1 + X_{\vv} ) \Lt  \frac {M_{\vv}^{1/2} T_{\vv}^{\vepsilon}} { T_{\vv}^{1/2}}  , \quad \frac {     X_{\vv}   } {A_{0}^{\flat} } \Lt   \frac {1 } {(M_{\vv} T_{\vv})^{1/2}}  < \frac {1} {  T_{\vv}^{3/4}},
\end{align*}
if $ T_{\vv}^{1/2} < M_{\vv} \leqslant T_{\vv}^{1 - \vepsilon} $, all of which are   satisfactory.

\begin{appendices}
	 \section{Gallagher's Hybrid Large Sieve over Number Fields}\label{sec: Gallagher}
	 
In this appendix, we establish Gallagher's hybrid large sieve \cite[\S 1]{Gallagher-LS} over number fields. 

Let notation be as in \S \ref{sec: notation}. Recall that $\widehat{\bfra}    = \prod_{\vv |\infty} \widehat{\bfra}_{\vv}$ is  the unitary dual of $F^{\times}_{\infty}$; $ \widehat{\bfra}_{\vv} = \BR \times \{0, 1\} $ if $\vv$ is real and $ \widehat{\bfra}_{\vv} = \BR \times \BZ  $ if $\vv$ is complex.  

\delete{ For $f \in L^1 (\widehat{\bfra}_{\vv}) \cap L^2 (\widehat{\bfra}_{\vv})$,  its Fourier transform, as function on $ \bfrf_{\vv}$, is defined by
\begin{align}\label{20eq: Fourier}
\hat{f} (\vlambda, \phi) = \viint_{\widehat{\bfra}_{\vv} } f (\varnu, m) e^{- i \varnu \slambda   - i m \phi} \nd \mu (\varnu, m),
\end{align}
and the Fourier inversion reads
\begin{align}\label{20eq: Fourier inverse}
f (\varnu, m) = \frac {1} {2\pi c_{\vv}} \viint_{ {\bfrf}_{\vv} } \hat{f} (\vlambda, \phi) e^{i \varnu \slambda  + i m \phi} \nd \mu (\vlambda, \phi),
\end{align}
where $\nd \mu$ is the usual Lebesgue measure on either $\widehat{\bfra}_{\vv}$ or ${\bfrf}_{\vv}$, and $c_{\vv} = 2$ or $2 \pi$ according as $\vv$ is real or complex. Moreover, by Plancherel's theorem,
\begin{align}\label{20eq: Plancherel} 
\viint_{\widehat{\bfra}_{\vv} } |f (\varnu, m)|^2 \nd \mu (\varnu, m) = \frac {1} {2\pi c_{\vv}} \viint_{ {\bfrf}_{\vv} } \big|\hat{f} (\vlambda, \phi)\big|^2 \nd \mu (\vlambda, \phi) . 
\end{align}
In an obvious way, the formulae \eqref{20eq: Fourier}--\eqref{20eq: Plancherel} extend onto $ \widehat{\bfra} $ and ${\bfrf}$. }

Let $\bfra = \BR^{|S_{\infty}|}$ and $ \bfra_+ = \BR_+^{|S_{\infty}|} $. For  $U \in \bfra_+$, define 
\begin{align}\label{20eq: defn of a (T)}
\widehat{\bfra} (U) = \big\{ (\varnu, m) \in \widehat{\bfra} : |\varnu_{\varv}|, |m_{\vv}| \leqslant U_{\vv} \text{ for all } \vv | \infty \big\}.
\end{align}  
For $ y \in F^{\times}_{\infty} $ and $\delta \in \bfra_+$, with $\delta_{\vv} < \pi$ for each $\vv|\infty$, define 
\begin{align}\label{20eq: defn of f (delta)}
F^{\times}_{\infty}  (y; \delta) \hskip -1 pt =  \hskip -1 pt \big\{  \hskip -1 pt x \in F^{\times}_{\infty}  : N_{\vv}\big|\log |x|_{\varv} - \log |y|_{\vv} \big|, \big| \arg (x_{\varv}) - \arg (y_{\vv}) \big| \leqslant \delta_{\vv} \text{ for all } \vv | \infty  \hskip -1 pt \big\},
\end{align}
where $ \arg (x_{\vv}) $ lies on the circle $ \BR / 2\pi \BZ $ ($ \arg (x_{\vv}) = 0$ or $\pi$ if $\vv$ is real).

Let $\fra  $ be a fractional ideal of $F$. Let $S  (\varnu, m)$ be an absolutely convergent  series as follows, 
\begin{align}\label{20eq: series}
S_{\hskip -1 pt \fra} (\varnu, m) = \sum_{\gamma \shskip \in \fra  \smallsetminus \{0\} } a_{\gamma} \vchi_{i \varnu, \shskip m} (\gamma), \qquad    \ a_{\gamma} \in \BC. 
\end{align}

\begin{prop} 
	Let $U \in \bfra_+$ be such that  $U_{\vv} \geqslant 1$ for all $\vv |\infty$.  We have
	\begin{align}\label{20eq: Gallagher, 1}
	\viint_{\widehat{\bfra} (U)} |S_{\hskip -1 pt \fra} (\varnu, m)|^2 \nd \mu (\varnu, m) \Lt_{F} \RN (U)^2  \int_{ F^{\times}_{\infty} } \raisebox{- 0.2 \depth}{$\Bigg|$}    \sum_{\gamma \shskip \in   F^{\times}_{\infty}  (y; \shskip 1/  U) \cap \fra  } a_{\gamma}   \raisebox{- 0.2 \depth}{$\Bigg|^2$}  \nd^{\times} \hskip -1 pt y. 
	\end{align}

\end{prop}

\begin{proof}
	Set  $\delta = 1 / U$. The right-hand side of \eqref{20eq: Gallagher, 1} may be written as 
	\begin{align*}
	 \int_{ F^{\times}_{\infty} } \big| C_{\fra}^{\delta} (y) \big|^2 \nd^{\times} \hskip -1 pt y, \qquad C_{\fra}^{ \delta} (y) = \frac 1 {\RN (\delta)} \sum_{\gamma \shskip \in   F^{\times}_{\infty}  (y;\shskip \delta) \cap \fra  } a_{\gamma} .
	\end{align*}
	Put $ F_{\delta} (x) = 1/\RN (\delta)$ or $0$ according as $x \in F^{\times}_{\infty}  (1; \delta)$ or not. Then
	\begin{align*}
C_{\fra}^{\delta} (y) = \sum_{\gamma \shskip \in \fra \smallsetminus\{0\}} a_{\gamma} F_{\delta} (  y / \gamma ). 
	\end{align*}
	Taking the Mellin transform (see \eqref{2eq: Mellin}), we get $ \breve{C}_{\fra}^{\delta} = S_{\hskip -1 pt \fra} \cdot \breve{F}_{\delta} $. Since the series \eqref{20eq: series} converges absolutely, $C_{\fra}^{\delta}$ is a bounded integrable function, and hence is square-integrable. By Plancherel's theorem (see \eqref{2eq: Plancherel}),
	\begin{align*}
	\int_{ F^{\times}_{\infty} } \big| C_{\fra}^{\delta} (y) \big|^2 \nd^{\times} \hskip -1 pt y = \frac 1 {c} \viint_{\widehat{\bfra} } \big|S_{\hskip -1 pt \fra} (\varnu, m) \breve F_{\delta} (\varnu, m) \big|^2 \nd \mu (\varnu, m) ,
	\end{align*}
	for a certain constant $c$ (explicitly, $c = 2^{2r_1 + 2 r_2}\pi^{r_1+2r_2}$). Since $\breve F_{\delta} (\varnu, m)$ is the product of 
	\begin{align*}
	\left\{\begin{aligned}
& \displaystyle \frac {2 \sin ( \delta_{\vv} \varnu_{\vv}  )} { \delta_{\vv} \varnu_{\vv}},  \ && \text{ if } \vv \text { is real,} \\
& \displaystyle \frac { 2 \sin (  \delta_{\vv} \varnu_{\vv}  )} { \delta_{\vv} \varnu_{\vv}} \frac {2 \sin (  \delta_{\vv} m_{\vv} )} {\delta_{\vv} m_{\vv}  }, \ && \text{ if } \vv \text { is complex,} 
	\end{aligned}\right.	
	\end{align*}
we have $ \breve F_{\delta} (\varnu, m) \Gt 1 $ for $ (\varnu, m) \in \widehat {\bfra} (1/\delta) $, and the result follows.	
\end{proof}


We shall apply \eqref{20eq: Gallagher, 1} to sums of the form
\begin{align}\label{20eq: series, chi}
S_{\hskip -1 pt \fra} (\vchi ; \varnu, m) = \sum_{\gamma \shskip \in \fra  \smallsetminus \{0\} } a_{\gamma} \vchi  (\gamma) \vchi_{i \varnu, \shskip m} (\gamma). 
\end{align}
where $\vchi \in \widehat {(\fra / \fra \frn )}{}^{\hskip -1 pt\times}$ is induced from a character   $\vchi : (\frO /  \frn)^{\hskip -1 pt \times} \ra \BC^{\times}$ via a (fixed) isomorphism $(\fra / \fra \frn)^{\hskip -1 pt \times} \ra (\frO /  \frn)^{\hskip -1 pt \times}    $ (see Definition \ref{def: x inverse}), and $ \vchi (\gamma) = \vchi (\gamma + \fra \frn) $ or $0$ according as $\gamma + \fra \frn \in (\fra / \fra \frn)^{\hskip -1 pt \times}$ or not. 

\begin{lem}
	We have
	\begin{align}\label{20eq: sum over chi}
	\sum_{ \vchi \shskip \in \widehat {(\fra / \fra \frn)}{}^{\hskip -1 pt \times} } \raisebox{- 0.2 \depth}{$\Bigg| $} \sum_{\gamma \shskip \in   F^{\times}_{\infty}  (y; \shskip \delta) \cap \fra} a_{\gamma} \vchi (\gamma)  \raisebox{- 0.2 \depth}{$\Bigg|^2$}   \Lt_{F}  \bigg(\RN (\frn) + \frac {\RN (y) \RE(\delta)} {\RN (\fra)}  \bigg) \sum_{\gamma \shskip \in   F^{\times}_{\infty}  (y; \shskip \delta) \cap \fra }  |a_{\gamma} |^2, 
	\end{align}
	where $ \mathrm{E}(\delta) = \prod_{\vv|\infty} \big(e^{ \delta_{\vv}} - e^{- \delta_{\vv}} \big) (2 \delta_{\vv})^{N_{\vv}-1}$ so that $\RN (y) \RE(\delta)$ is the area of $F^{\times}_{\infty}  (y; \shskip \delta)$. 
\end{lem}

\begin{proof}
	By the orthogonality relations and the Cauchy--Schwarz inequality, the left-hand side of \eqref{20eq: sum over chi} is
	\begin{align*}
	\varphi (\frn) \hskip -1 pt \sum_{ x \shskip \in (\fra / \fra \frn){}^{\hskip -1 pt \times}} \hskip -1 pt \raisebox{- 0.2 \depth}{$\Bigg| $}  {\sum _{\gamma \shskip \in   F^{\times}_{\infty}  (y;\shskip \delta) \cap (x + \fra \frn)}} a_{\gamma} \raisebox{- 0.2 \depth}{$\Bigg|^2$}  \hskip -1 pt \Lt \varphi (\frn)  \sum_{ x \shskip \in (\fra / \fra \frn){}^{\hskip -1 pt \times}} \hskip -2 pt \bigg( \frac {\RN (y) \mathrm{E} (\delta)} {\RN (\fra \frn)} + 1 \bigg) \hskip -2 pt \sum_{\gamma \shskip \in   F^{\times}_{\infty}  (y;\shskip \delta) \cap (x + \fra \frn)} \hskip -1 pt |a_{\gamma} |^2. 
	\end{align*}
\end{proof}

\begin{prop}\label{prop: Gallagher}
	Let  $U \in \bfra_+$ be such that  $U_{\vv} \geqslant 1$ for all $\vv |\infty$. We have
	\begin{align}\label{20eq: main}
	\sum_{ \vchi \shskip \in \widehat {(\fra / \fra \frn)}{}^{\hskip -1 pt \times} } \viint_{\widehat{\bfra} (U)} |S_{\hskip -1 pt \fra} (\vchi; \varnu, m)|^2 \nd \mu (\varnu, m) \Lt_{F} \sum_{ \gamma  \shskip \in \fra \smallsetminus \{0\} } (\RN(U) \RN(\frn)  + \RN (\gamma \fra^{-1})) |a_{\gamma}|^2. 
	\end{align}
\end{prop}

\begin{proof}
	Using \eqref{20eq: Gallagher, 1} and \eqref{20eq: sum over chi}, the left-hand side of \eqref{20eq: main} is bounded by
	\begin{align*}
& \RN (U)^2 \int_{ F^{\times}_{\infty} } \sum_{ \vchi \shskip \in \widehat {(\fra / \fra \frn)}{}^{\hskip -1 pt \times} } \raisebox{- 0.2 \depth}{$\Bigg| $}  \sum_{\gamma \shskip \in   F^{\times}_{\infty}  (y; \shskip 1/  U) \cap \fra } a_{\gamma} \vchi (\gamma) \raisebox{- 0.2 \depth}{$\Bigg|^2$}  \nd^{\times} \hskip -1 pt y \\
\Lt \  & \RN (U)^2 \int_{ F^{\times}_{\infty} } \bigg(\RN (\frn) + \frac {\RN (y) \RE(1/U)} {\RN(\fra)} \bigg) \sum_{\gamma \shskip \in   F^{\times}_{\infty}  (y; \shskip 1/U) \cap \fra }  |a_{\gamma} |^2 \nd^{\times} \hskip -1 pt y .
	\end{align*}
	The coefficient of $|a_{\gamma}|^2$ here is 
	\begin{align*}
& \RN (U)^2 \RN (\frn)	\viint _{ F^{\times}_{\infty}  (\gamma ; \shskip 1/U)} \nd^{\times} \hskip -1 pt y + \frac {\RN (U)^2 \RE(1/U)} {\RN(\fra)} \viint _{ F^{\times}_{\infty}  (\gamma ; \shskip 1/U)} \nd  y \\
= \ & \RN (2 U) \RN (\frn) + \frac {\RN (U)^2 \RE(1/U)^2 \RN (\gamma)} {\RN(\fra)} \\
\Lt \ &  \RN (  U) \RN (\frn) +  \RN (\gamma \fra^{-1}) .
	\end{align*} 
\end{proof}

\subsection{A Corollary of Gallagher's Large Sieve}

\begin{defn}\label{defn: a(U)}
	Let $U \in \bfra_+ $ and  $ (\kappa, n)  \in \widehat {\bfra} $. Define
	\begin{align}\label{20eq:  a(U)}
	\widehat{\bfra}_{\kappa, \shskip n} (U) = \big\{ (\varnu, m) \in \widehat{\bfra} : {\textstyle \sqrt{(\varnu_{\vv} - \kappa_{\vv})^2 + (m_{\vv} - n_{\vv})^2}}  \Lt U_{\vv} \text{ for all } \vv | \infty  \big\}. 
	\end{align} 
\end{defn}

Note that $\widehat{\bfra}_{0, \shskip 0} (U) = \widehat{\bfra} (U) $ if we slightly modify the definition of $\widehat{\bfra} (U) $ in \eqref{20eq: defn of a (T)}. See also Definition \ref{defn: a (U), local}. 

\begin{cor}\label{cor: large sieve}
Let $\frc \in \widetilde{C}_F$. For $\frq \sim \frc$, define $c_{\frq} =   [\frc^{-1} \frq]$.	Let $\fra$ and $\frn$ be ideals, with $\frn \subset \frO$   and $ (\fra   \frc  \frD )_{\frn}  = \frn^{-1}$ {\rm(}see Definition {\rm\ref{defn: psi b ...}}{\rm)}. Let  $C, \varGamma \in \bfra_+$  and $F_{\infty}^{\varDelta} (C), F_{\infty}^{\varDelta} (\varGamma) $ be defined as in Definition {\rm\ref{defn: F D (V)}}. Let $a_{\gamma}$ and $b_{\frq}$ be sequences of complex numbers for $\gamma \in F_{\infty}^{\varDelta} (\varGamma) $ and $\frq \sim \frc$ with $ c_{\frq}   \in F_{\infty}^{\varDelta} (C) $. 
Define 
\begin{align}\label{20eq: Sn (...)}
 S_{\frn} (C, \varGamma; \varnu, m; a, b ) =  \mathop{\mathop{\sum_{ \frq \sim \frc } }_{(\frq, \shskip \frn) = (1)}}_{ c_{\frq}   \in F_{\infty}^{\varDelta} (C) } \mathop{\sum_{\gamma \shskip \in \fra  } }_{\gamma \shskip \in F_{\infty}^{\varDelta} (\varGamma) }   a_{\gamma} b_{\frq} \psi_{\frn}  (  {  \gamma} / { c_{\frq}}  ) \vchi_{  i \varnu, \shskip  m}  (   {\gamma} / {c_{\frq}}  )  .
\end{align} 
Let  $U \in \bfra_+$ be such that  $U_{\vv} \Gt 1$ for all $\vv |\infty$. For $ \kappa, n  \in \bfra $, define $\widehat{\bfra}_{\kappa, \shskip n} (U)$ as in {\rm\eqref{20eq:  a(U)}}.  Then 
\begin{align}\label{20eq: int of S}
\begin{aligned}
& \  \viint_{\widehat{\bfra}_{\kappa, \shskip n} (U)}  \left|S_{\frn} (C, \varGamma; \varnu, m; a, b ) \right| \nd \mu (\varnu, m) \\
  \Lt & \ {\textstyle   {\RN (\frn)^{1/2+\vepsilon}}} ( \RN(U) + \RN (C)/\RN (\frn))^{1/2}  (\RN(U)  + \RN (\varGamma)/\RN(\fra\frn))^{1/2} \|a\|_2 \|b\|_2 ,
\end{aligned} 
\end{align}
where $\|a\|_2^2 =  {\sum_{ \gamma} |a_{\gamma}|^2 }$ and  $\|b\|_2^2 =  {\sum_{ \frq} |b_{\frq}|^2 }$. 
\end{cor}

\begin{proof}
By changing $a_{\gamma}$ and $b_{\frq}$ into $a_{\gamma} \overline{\vchi_{  i \kappa, \shskip    n   } (\gamma)}$ and $b_{\frq}  {\vchi_{  i \kappa, \shskip    n   } (c_{\frq})}$ if necessary, we may assume with no loss of generality that $ \widehat{\bfra}_{\kappa, \shskip n} (U) = \widehat{\bfra}  (U) $. 
Moreover, we set $a_{\gamma} = b_{\frq} = 0$ if $\gamma \notin F_{\infty}^{\varDelta} (\varGamma) $ or $ c_{\frq}   \notin F_{\infty}^{\varDelta} (C) $. It will be convenient to view the $\frq$-sum as a sum over $ \{ c_{\frq}: \frq \sim \frc \} \subset \frc^{-1} $. 
	
	 Next, we 
	 reformulate $S_{\frn} (C, \varGamma; \varnu, m; a, b )$ in \eqref{20eq: Sn (...)}  as 
	 \begin{align*}
	\sum_{ \frm | \frn}   {\mathop{\sum_{ \frq \sim \frc } }_{(\frq, \shskip \frn) = (1)}}  \mathop{\sum_{\gamma \shskip \in \fra \frn \frm^{-1} } }_{ (\gamma \frc \frD)_{\frm} = \frm^{-1} }   a_{\gamma} b_{\frq} \psi_{\frm}  (  {  \gamma} / { c_{\frq}}  ) \vchi_{  i \varnu, \shskip  m}  (   {\gamma} / {c_{\frq}}  ). 
	 \end{align*}  
	For $\vchi \in \widehat {\big( (\frm \frD_{\frm}) ^{-1} \hskip -1 pt /   \frD_{\frm}^{-1} \big)}{}^{\hskip -1 pt \times}$, define the Gauss sum  
	\begin{align*}
	\tau (\vchi) = \sum_{x \shskip \in ((\frm \frD_{\frm}) ^{-1} \hskip -1 pt /   \frD_{\frm}^{-1} )^{\times}} \vchi (x) \psi_{\frm} (x) .
	\end{align*} 
	It is well-known that $ |\tau (\vchi)| \leqslant \sqrt{\RN (\frm)} $. 
	By the orthogonality relation, 
	\begin{align*}
	\psi_{\frm}  (  {  \gamma} / { c_{\frq}}  )  = \frac 1 {\varphi (\frm)} \sum_{ \vchi \in \widehat { ( (\frm \frD_{\frm}) ^{-1} \hskip -1 pt /   \frD_{\frm}^{-1}  )}{}^{\hskip -1 pt \times} } {\vchi} (\gamma / c_{\frq}) \tau (\overline{\vchi})  ,  
	\end{align*} 
	for $(\frq, \frm) = (1)$ and $ (\gamma \frc \frD)_{\frm} = \frm^{-1} $. From these, we deduce that the left-hand side of \eqref{20eq: int of S} is bounded by 
	\begin{align*}
\sum_{ \frm | \frn} \frac {\textstyle \sqrt{\RN (\frm)}} {\varphi (\frm)} \sum_{ \vchi \in \widehat {( (\frm \frD_{\frm}) ^{-1} \hskip -1 pt /   \frD_{\frm}^{-1} )}{}^{\hskip -1 pt \times} } 	\viint_{\widehat{\bfra}  (U)}   \big|S_{\fra \frn \frm^{-1} } (\vchi ; \varnu, m) \overline{S_{\frc^{-1}} (\vchi ; \varnu, m)} \big| \nd \mu ( \varnu, m), 
	\end{align*}
	where
	\begin{align*}
S_{\fra \frn \frm^{-1} } (\vchi ; \varnu, m) \hskip -1 pt = \hskip -3 pt \mathop{\sum_{\gamma \shskip \in \fra \frn \frm^{-1} } }_{ (\gamma \frc \frD)_{\frm} = \frm^{-1} } \hskip -3 pt a_{\gamma} 	{\vchi} (\gamma  )  \vchi_{  i \varnu, \shskip  m}  (   {\gamma}   ), \quad
S_{\frc^{-1}} (\vchi ; \varnu, m) \hskip -1 pt = \hskip -3 pt {\mathop{\sum_{ \frq \sim \frc } }_{(\frq, \shskip \frn) = (1)}} \hskip -3 pt \overline{b}_{\frq} {\vchi} (c_{\frq}) {\vchi}_{  i \varnu, \shskip  m} (c_{\frq}).
	\end{align*}
Finally, the bound in \eqref{20eq: int of S} readily follows from  Cauchy--Schwarz and Proposition \ref{prop: Gallagher}.

\end{proof}
	 
\end{appendices}


\begin{thebibliography}{AHLQ}
	
	\bibitem[Agg]{Aggarwal-GL3}
	K.~Aggarwal. 
	\newblock A new subconvex bound for {$\rm GL(3)$} {$L$}-functions in the
	{$t$}-aspect.
	\newblock {\em Int. J. Number Theory}, 17(5):1111--1138, 2021.
	
	\bibitem[AHLQ]{AHLQ-Bessel}
	K.~Aggarwal, R.~Holowinsky, Y.~Lin, and Z.~Qi.
	\newblock A {B}essel-delta method and exponential sums for {$\mathrm{GL} (2)$}.
	\newblock {\em Q. J. Math.}, 71(3):1143--1168, 2020.
	
	\bibitem[AS]{A-S}
	M.~Abramowitz and I.~A. Stegun.
	\newblock {\em Handbook of {M}athematical {F}unctions with {F}ormulas,
		{G}raphs, and {M}athematical {T}ables},  {National Bureau of
		Standards Applied Mathematics Series, Vol. 55}.
	\newblock Washington, D.C., 1964.
	
	\bibitem[BB]{Blomer-Brumley}
	V.~Blomer and F.~Brumley.
	\newblock On the {R}amanujan conjecture over number fields.
	\newblock {\em Ann. of Math. (2)}, 174(1):581--605, 2011.
	
	\bibitem[BH]{Blomer-Harcos-TR}
	V.~Blomer and G.~Harcos.
	\newblock Twisted {$L$}-functions over number fields and {H}ilbert's eleventh
	problem.
	\newblock {\em Geom. Funct. Anal.}, 20(1):1--52, 2010.
	
	\bibitem[BKY]{BKY-Mass}
	V.~Blomer, R.~Khan, and M.~P. Young.
	\newblock Distribution of mass of holomorphic cusp forms.
	\newblock {\em Duke Math. J.}, 162(14):2609--2644, 2013.
	
	\bibitem[Blo]{Blomer}
	V.~Blomer.
	\newblock Subconvexity for twisted {$L$}-functions on {${\rm GL}(3)$}.
	\newblock {\em Amer. J. Math.}, 134(5):1385--1421, 2012.
	
	\bibitem[BM1]{BM-Kuz-Spherical}
	R.~W. Bruggeman and R.~J. Miatello.
	\newblock Sum formula for {${\rm SL}_2$} over a number field and {S}elberg type
	estimate for exceptional eigenvalues.
	\newblock {\em Geom. Funct. Anal.}, 8(4):627--655, 1998.
	
	\bibitem[BM2]{B-Mo}
	R.~W. Bruggeman and Y.~Motohashi.
	\newblock Sum formula for {K}loosterman sums and fourth moment of the
	{D}edekind zeta-function over the {G}aussian number field.
	\newblock {\em Funct. Approx. Comment. Math.}, 31:23--92, 2003.
	
	\bibitem[BR]{Bernstein-Reznikov}
	J.~Bernstein and A.~Reznikov.
	\newblock Subconvexity bounds for triple {$L$}-functions and representation
	theory.
	\newblock {\em Ann. of Math. (2)}, 172(3):1679--1718, 2010.
	
	\bibitem[Bum]{Bump}
	Daniel Bump.
	\newblock {\em Automorphic {F}orms and {R}epresentations},  {Cambridge Studies in Advanced Mathematics, Vol. 55}.
	\newblock Cambridge University Press, Cambridge, 1997.
	
	\bibitem[CI]{CI-Cubic}
	J.~B. Conrey and H.~Iwaniec.
	\newblock The cubic moment of central values of automorphic {$L$}-functions.
	\newblock {\em Ann. of Math. (2)}, 151(3):1175--1216, 2000.
	
	\bibitem[Gal]{Gallagher-LS}
	P.~X. Gallagher.
	\newblock A large sieve density estimate near {$\sigma =1$}.
	\newblock {\em Invent. Math.}, 11:329--339, 1970.
	
	\bibitem[GJ1]{GJ-GL(2)-GL(3)}
	S.~Gelbart and H.~Jacquet.
	\newblock A relation between automorphic representations of {${\rm GL}(2)$}\
	and {${\rm GL}(3)$}.
	\newblock {\em Ann. Sci. \'Ecole Norm. Sup. (4)}, 11(4):471--542, 1978.
	
	\bibitem[GJ2]{Gelbart-Jacquet}
	S.~Gelbart and H.~Jacquet.
	\newblock Forms of {${\rm GL}(2)$} from the analytic point of view.
	\newblock   {Automorphic {F}orms, {R}epresentations and {$L$}-{F}unctions,
		{P}art 1}, Proc. Sympos. Pure Math., XXXIII, pages 213--251. Amer. Math.
	Soc., Providence, R.I., 1979.
	
	\bibitem[HN]{HoNe-ZeroFr}
	R.~Holowinsky and P.~D. Nelson.
	\newblock Subconvex bounds on {$\rm GL_3$} via degeneration to frequency zero.
	\newblock {\em Math. Ann.}, 372(1-2):299--319, 2018.
	
	\bibitem[H{\"o}r]{Hormander}
	L.~H{\"o}rmander.
	\newblock {\em The {A}nalysis of {L}inear {P}artial {D}ifferential {O}perators.
		{I}: {D}istribution {T}heory and {F}ourier {A}nalysis},  {Grundlehren der Mathematischen Wissenschaften, Vol. 256}.
	\newblock Springer-Verlag, Berlin, 1983.
	
	\bibitem[Hua]{Huang-GL3}
	B.~Huang.
	\newblock Hybrid subconvexity bounds for twisted {$L$}-functions on {${\rm
			GL}(3)$}.
	\newblock {\em Sci. China Math.}, 64(3):443--478, 2021.
	
	\bibitem[Hux]{Huxley}
	M.~N. Huxley.
	\newblock {\em Area, {L}attice {P}oints, and {E}xponential {S}ums},   {London Mathematical Society Monographs, Vol. 13. New Series}.
	\newblock The Clarendon Press, Oxford University Press, New York, 1996.
	\newblock Oxford Science Publications.
	
	\bibitem[IK]{IK}
	H.~Iwaniec and E.~Kowalski.
	\newblock {\em Analytic {N}umber {T}heory},   {American
		Mathematical Society Colloquium Publications, Vol. 53}.
	\newblock American Mathematical Society, Providence, RI, 2004.
	
	\bibitem[IT]{Ichino-Templier}
	A.~Ichino and N.~Templier.
	\newblock On the {V}orono\u\i\ formula for {${\rm GL}(n)$}.
	\newblock {\em Amer. J. Math.}, 135(1):65--101, 2013.
	
	\bibitem[Ivi]{Ivic-t-aspect}
	A.~Ivi\'{c}.
	\newblock On sums of {H}ecke series in short intervals.
	\newblock {\em J. Th\'{e}or. Nombres Bordeaux}, 13(2):453--468, 2001.
	
	\bibitem[Iwa1]{Iwaniec-L(1)}
	H.~Iwaniec.
	\newblock Small eigenvalues of {L}aplacian for {$\Gamma_0(N)$}.
	\newblock {\em Acta Arith.}, 56(1):65--82, 1990.
	
	\bibitem[Iwa2]{Iw-Spectral}
	H.~Iwaniec.
	\newblock {\em Spectral {M}ethods of {A}utomorphic {F}orms},   {Graduate Studies in Mathematics, Vol. 53}.
	\newblock American Mathematical Society, Providence, RI, {S}econd edition,
	2002.
	
	\bibitem[Joh]{Faa-di-Bruno}
	W.~P. Johnson.
	\newblock The curious history of {F}a\`a di {B}runo's formula.
	\newblock {\em Amer. Math. Monthly}, 109(3):217--234, 2002.
	
	\bibitem[JS]{J-S-Rankin-Selberg}
	H.~Jacquet and J.~A. Shalika.
	\newblock On {E}uler products and the classification of automorphic
	representations. {I}.
	\newblock {\em Amer. J. Math.}, 103(3):499--558, 1981.
	
	\bibitem[Kuz]{Kuznetsov}
	N.~V. Kuznetsov.
	\newblock {P}etersson's conjecture for cusp forms of weight zero and {L}innik's
	conjecture. {S}ums of {K}loosterman sums.
	\newblock {\em Math. Sbornik}, 39:299--342, 1981.
	
	\bibitem[Lan]{Lang-ANT}
	S.~Lang.
	\newblock {\em Algebraic {N}umber {T}heory}.
	\newblock {Graduate Texts in Mathematics, Vol. 110}. Springer-Verlag, New York,
	second edition, 1994.
	
	\bibitem[Lap]{Lapid}
	E.~M. Lapid.
	\newblock On the nonnegativity of {R}ankin-{S}elberg {$L$}-functions at the
	center of symmetry.
	\newblock {\em Int. Math. Res. Not.}, (2):65--75, 2003.
	
	\bibitem[LG]{B-Mo2}
	H.~Lokvenec-Guleska.
	\newblock {\em {S}um {F}ormula for $\mathrm{SL}_2$ over {I}maginary {Q}uadratic
		{N}umber {F}ields}.
	\newblock Ph.D. Thesis. Utrecht University, 2004.
	
 
	
	\bibitem[Li]{XLi2011}
	X.~Li.
	\newblock Bounds for {${\rm GL}(3)\times {\rm GL}(2)$} {$L$}-functions and
	{${\rm GL}(3)$} {$L$}-functions.
	\newblock {\em Ann. of Math. (2)}, 173(1):301--336, 2011.
	
	
	
	\bibitem[Lin]{Lin-GL3}
	Y.~Lin.
	\newblock Bounds for twists of {$\rm GL(3)$} {$L$}-functions.
	\newblock {\em J. Eur. Math. Soc. (JEMS)}, 23(6):1899--1924, 2021.
	
	\bibitem[LQ]{Qi-Liu-LLZ}
	S.-C. Liu. and Z.~Qi.
	\newblock Low-lying zeros of {$L$}-functions for {M}aass forms over imaginary
	quadratic fields.
	\newblock {\em Mathematika}, 66(3):777--805, 2020.
	
	\bibitem[LS]{LinSun-GL3}
	Y.~Lin and Q.~Sun.
	\newblock Analytic twists of {${\rm GL}_3 \times {\rm GL}_2$} automorphic
	forms.
	\newblock {\em Int. Math. Res. Not. IMRN}, (19):15143--15208, 2021.
	
	\bibitem[Mag1]{Maga-Sub}
	P.~Maga.
	\newblock Subconvexity for twisted {$L$}-functions over number fields via
	shifted convolution sums.
	\newblock {\em Acta Math. Hungar.}, 151(1):232--257, 2017.
	
	\bibitem[Mag2]{Maga-Shifted}
	P.~Maga.
	\newblock The spectral decomposition of shifted convolution sums over number
	fields.
	\newblock {\em J. Reine Angew. Math.}, 744:1--27, 2018.
	
	\bibitem[Mol]{Molteni-L(1)}
	G.~Molteni.
	\newblock Upper and lower bounds at {$s=1$} for certain {D}irichlet series with
	{E}uler product.
	\newblock {\em Duke Math. J.}, 111(1):133--158, 2002.
	
	\bibitem[MOS]{MO-Formulas}
	W.~Magnus, F.~Oberhettinger, and R.~P. Soni.
	\newblock {\em Formulas and {T}heorems for the {S}pecial {F}unctions of
		{M}athematical {P}hysics}.
	\newblock Third enlarged edition. Die Grundlehren der mathematischen
	Wissenschaften, Band 52. Springer-Verlag New York, Inc., New York, 1966.
	
	\bibitem[MS]{Miller-Schmid-2006}
	S.~D. Miller and W.~Schmid.
	\newblock Automorphic distributions, {$L$}-functions, and {V}oronoi summation
	for {${\rm GL}(3)$}.
	\newblock {\em Ann. of Math. (2)}, 164(2):423--488, 2006.
	
	\bibitem[MSY]{Ye-GL3}
	M.~McKee, H.~Sun, and Y.~Ye.
	\newblock Improved subconvexity bounds for {${\rm GL}(2)\times {\rm GL}(3)$}
	and {${\rm GL}(3)$} {$L$}-functions by weighted stationary phase.
	\newblock {\em Trans. Amer. Math. Soc.}, 370(5):3745--3769, 2018.
	
	\bibitem[Mun1]{Munshi-Circle-III}
	R.~Munshi.
	\newblock The circle method and bounds for {$L$}-functions---{III}:
	{$t$}-aspect subconvexity for {$GL(3)$} {$L$}-functions.
	\newblock {\em J. Amer. Math. Soc.}, 28(4):913--938, 2015.
	
	\bibitem[Mun2]{Munshi-Circle-IV}
	R.~Munshi.
	\newblock The circle method and bounds for {$L$}-functions---{IV}:
	{S}ubconvexity for twists of {$\rm GL(3)$} {$L$}-functions.
	\newblock {\em Ann. of Math. (2)}, 182(2):617--672, 2015.
	
	\bibitem[Mun3]{Munshi-Circle-IV2}
	R.~Munshi.
	\newblock Twists of {${\rm GL}(3)$} {$L$}-functions.
	\newblock {\em   arXiv:1604.08000}, 2016.
	
	\bibitem[Mun4]{Munshi-GL3xGL2}
	R.~Munshi.
	\newblock Subconvexity for {${\rm GL}(3) \times {\rm GL}(2)$} {$L$}-functions
	in $t$-aspect.
	\newblock {\em   arXiv:1810.00539}, 2018.
	
	\bibitem[MV]{Michel-Venkatesh-GL2}
	P.~Michel and A.~Venkatesh.
	\newblock The subconvexity problem for {${\rm GL}_2$}.
	\newblock {\em Publ. Math. Inst. Hautes \'Etudes Sci.}, (111):171--271, 2010.
	
	\bibitem[Nel]{Nelson-Eisenstein}
	P.~H. Nelson.
	\newblock Eisenstein series and the cubic moment for {$\text{PGL}_2$}.
	\newblock {\em   arXiv:1911.06310}, 2019.
	
	\bibitem[Nun]{Nunes-GL3}
	R.~M. Nunes.
	\newblock On the subconvexity estimate for self-dual {${\rm GL}(3)$}
	{$L$}-functions in the $t$-aspect.
	\newblock {\em   arXiv:\allowbreak1703.04424}, 2017.
	
	\bibitem[Olv1]{Olver-1}
	F.~W.~J. Olver.
	\newblock The asymptotic solution of linear differential equations of the
	second order for large values of a parameter.
	\newblock {\em Philos. Trans. Roy. Soc. London. Ser. A.}, 247:307--327, 1954.
	
	\bibitem[Olv2]{Olver-Bessel}
	F.~W.~J. Olver.
	\newblock The asymptotic expansion of {B}essel functions of large order.
	\newblock {\em Philos. Trans. Roy. Soc. London. Ser. A.}, 247:328--368, 1954.
	
	\bibitem[Olv3]{Olver}
	F.~W.~J. Olver.
	\newblock {\em Asymptotics and {S}pecial {F}unctions}.
	\newblock Academic Press, New York-London, 1974.
	
	\bibitem[Pet]{Petrow-Cubic}
	I.~N. Petrow.
	\newblock A twisted {M}otohashi formula and {W}eyl-subconvexity for
	{$L$}-functions of weight two cusp forms.
	\newblock {\em Math. Ann.}, 363(1-2):175--216, 2015.
	
	\bibitem[PY1]{PY-Cubic}
	I.~Petrow and M.~P. Young.
	\newblock A generalized cubic moment and the {P}etersson formula for newforms.
	\newblock {\em Math. Ann.}, 373(1-2):287--353, 2019.
	
	\bibitem[PY2]{PY-Weyl2}
	I.~Petrow and M.~P. Young. 
	\newblock The {W}eyl bound for {D}irichlet {$L$}-functions of cube-free
	conductor.
	\newblock {\em Ann. of Math. (2)}, 192(2):437--486, 2020.
	
	\bibitem[PY3]{PY-Weyl3}
	I.~Petrow and M.~P. Young.
	\newblock The fourth moment of {D}irichlet {$L$}-functions along a coset and
	the {W}eyl bound.
	\newblock {\em arXiv:1908.10346}, 2019.
	
	
	
	
	
	\bibitem[Qi1]{Qi-Wilton}
	Z.~Qi.
	\newblock Cancellation in the additive twists of {F}ourier coefficients for
	{$\rm GL_2$} and {$\rm GL_3$} over number fields.
	\newblock {\em Amer. J. Math.}, 141(5):1317--1345, 2019.
	
	\bibitem[Qi2]{Qi-Gauss}
	Z.~Qi.
	\newblock Subconvexity for twisted {$L$}-functions on {$\rm{GL}_3$} over the
	{G}aussian number field.
	\newblock {\em Trans. Amer. Math. Soc.}, 372(12):8897--8932, 2019.
	
	\bibitem[Qi3]{Qi-Bessel}
	Z.~Qi.
	\newblock Theory of fundamental {B}essel functions of high rank.
	\newblock {\em Mem. Amer. Math. Soc.}, 267(1303): vii+123, 2020.
	
	\bibitem[Sch]{Lin-Integral}
	R.~Schumacher.
	\newblock Subconvexity for {${\rm GL}_3({\mathbb{ R}})$} {$L$}-functions via
	integral representations.
	\newblock {\em   arXiv:\allowbreak2004.06791}, 2020.
	
	\bibitem[Sha]{Munshi-GL3GL2-q}
	P.~Sharma.
	\newblock Subconvexity for {${\rm GL}(3) \times {\rm GL}(2)$} twists in level
	aspect.
	\newblock {\em   arXiv:1906.09493}, 2019.
	
	\bibitem[Sog]{Sogge}
	C.~D. Sogge.
	\newblock {\em Fourier {I}ntegrals in {C}lassical {A}nalysis},  
	{Cambridge Tracts in Mathematics, Vol. 105}.
	\newblock Cambridge University Press, Cambridge, 1993.
	
	\bibitem[Sri1]{Srinivasan-Lattice-2}
	B.~R. Srinivasan.
	\newblock The lattice point problem of many-dimensional hyperboloids. {II}.
	\newblock {\em Acta Arith.}, 8:173--204, 1962/1963.
	
	\bibitem[Sri2]{Srinivasan-Lattice-3}
	B.~R. Srinivasan.
	\newblock The lattice point problem of many dimensional hyperboloids. {III}.
	\newblock {\em Math. Ann.}, 160:280--311, 1965.
	
	\bibitem[SZ]{SZ-Depth}
	Q.~Sun and R.~Zhao.
	\newblock Bounds for {${\rm GL}_3$} {$L$}-functions in depth aspect.
	\newblock {\em Forum Math.}, 31(2):303--318, 2019.
	
	\bibitem[Ven]{Venkatesh-BeyondEndoscopy}
	A.~Venkatesh.
	\newblock ``{B}eyond endoscopy'' and special forms on $\mathrm{GL}(2)$.
	\newblock {\em J. Reine Angew. Math.}, 577:23--80, 2004.
	
	\bibitem[Wat]{Watson}
	G.~N. Watson.
	\newblock {\em A {T}reatise on the {T}heory of {B}essel {F}unctions}.
	\newblock Cambridge University Press, Cambridge, England; The Macmillan
	Company, New York, 1944.
	
	\bibitem[Wu1]{WuHan-GL2}
	H.~Wu.
	\newblock Burgess-like subconvex bounds for {$\text{GL}_2\times\text{GL}_1$}.
	\newblock {\em Geom. Funct. Anal.}, 24(3):968--1036, 2014.
	
	\bibitem[Wu2]{WuHan-2}
	H.~Wu.
	\newblock Explicit subconvexity for {$\text{GL}_2$} and some applications.
	\newblock {\em   arXiv:1812.04391}, 2018.
	\newblock 
	
	\bibitem[You]{Young-Cubic}
	M.~P. Young.
	\newblock Weyl-type hybrid subconvexity bounds for twisted {$L$}-functions and
	{H}eegner points on shrinking sets.
	\newblock {\em J. Eur. Math. Soc. (JEMS)}, 19(5):1545--1576, 2017.
	
\end{thebibliography}

\end{document}